\documentclass[11pt]{amsart}
\usepackage[utf8]{inputenc}
\usepackage[margin=1in]{geometry}
\usepackage{amsfonts,amssymb,amsmath,amsthm,tikz,comment,mathtools,setspace,float,stmaryrd,datetime,bbm,adjustbox}
\usepackage[toc]{appendix}
\usepackage{enumitem}

\usepackage[hidelinks]{hyperref}
\usepackage{scalerel} 
\numberwithin{equation}{section}
\usepackage{chngcntr}
\counterwithin{figure}{section}

\allowdisplaybreaks

\newcommand{\N}{\mathbb{N}}
\newcommand{\R}{\mathbb{R}}
\newcommand{\Q}{\mathbb{Q}}
\newcommand{\Z}{\mathbb{Z}}
\newcommand{\F}{\mathcal{F}}
\newcommand{\G}{\mathcal{G}}

\newcommand{\Y}{\mathcal{Y}}

\newcommand{\C}{\mathcal{C}}
\newcommand{\D}{\mathcal{D}}
\newcommand{\E}{\mathbb{E}}
\newcommand{\Hh}{\mathcal{H}}

\newcommand{\ve}{\varepsilon}
\newcommand{\Ll}{\mathcal{L}}
\newcommand{\Pp}{\mathbb P}
\newcommand{\f}{\frac}

\newcommand{\deq}{\overset{d}{=}}

\newcommand{\mbf}{\mathbf}

\newcommand{\wt}{\widetilde}

\newcommand{\Rup}{\R_{{\scaleobj{0.7}{\uparrow}}}^4}
\newcommand{\wh}{\widehat}
\newcommand{\li}{\;{\le}_{\rm inc}\;}

\newcommand{\Nor}{\mathcal N}
\newcommand{\Exp}{\operatorname{Exp}}
\newcommand{\W}{W}
\newcommand{\NU}{\operatorname{NU}}
\newcommand{\DLBusedc}{\Xi}
\newcommand{\dir}{\xi}
\newcommand{\Split}{\mathfrak S}
\newcommand{\h}{\mathfrak h}

\newcommand{\graph}{\mathcal G}
\def\shdif{J}
\def\ind{\mathbf{1}}
\newcommand{\UC}{\operatorname{UC}}

\newcommand{\sig}{{\scaleobj{0.8}{\boxempty}}} 
\newcommand{\sigg}{{\scaleobj{0.9}{\boxempty}}}

\newcommand{\be}{\begin{equation}}
\newcommand{\ee}{\end{equation}}

\def\tsp{\hspace{0.5pt}}  
\def\tspa{\hspace{0.7pt}}
\def\tspb{\hspace{0.9pt}}
\def\tspc{\hspace{1.1pt}}

\newcommand\abullet{{\raisebox{1.5pt}{\scaleobj{0.5}{\bullet}}}}  
\newcommand\aabullet{{\tspc\raisebox{1.5pt}{\scaleobj{0.5}{\bullet}}\tspc}}  

\def\bck#1{{\overleftarrow{#1}}}

\DeclareMathOperator*{\argmax}{arg\,max}   

\def\SHpp{\Gamma} 
\def\SHHpp{\Lambda} 
\def\XiSH{\Xi_G}  

\newtheorem{theorem}{Theorem}[section]

\newtheorem{corollary}[theorem]{Corollary}
\newtheorem{lemma}[theorem]{Lemma}

\theoremstyle{definition}
\newtheorem{definition}[theorem]{Definition}

\theoremstyle{remark}
\newtheorem{remark}[theorem]{Remark}

\title[Geodesics in the directed landscape]{The stationary horizon and semi-infinite geodesics in the directed landscape}

\author{Ofer Busani}
\address{Ofer Busani, Universit\"at Bonn,
Endenicher Allee 60,
Bonn, Germany}
\email{busani@iam.uni-bonn.de}
\author{Timo Sepp{\"a}l{\"a}inen}
\address{Timo Sepp{\"a}l{\"a}inen, University of Wisconsin-Madison, Mathematics Department, Van Vleck Hall, 480
Lincoln Dr., Madison WI 53706-1388, USA.}
\email{seppalai@math.wisc.edu}
\author{Evan Sorensen}
\address{Evan Sorensen, University of Wisconsin-Madison, Mathematics Department, Van Vleck Hall, 480
Lincoln Dr., Madison WI 53706-1388, USA.}
\email{es4203@columbia.edu}

\date{\today}

\begin{document}

\maketitle 
\begin{abstract}
     The stationary horizon (SH) is a stochastic process of coupled Brownian motions  indexed by their real-valued drifts. It was first introduced by the first author as the  diffusive scaling limit of the Busemann process of exponential last-passage percolation. It was independently discovered as the Busemann process of Brownian last-passage percolation by the second and third authors. We show that SH is the unique invariant distribution and an attractor of the KPZ fixed point under conditions on the asymptotic spatial slopes. It follows that SH describes the Busemann process of the directed landscape.  This gives control of semi-infinite geodesics simultaneously across all initial points and directions. The countable dense set $\DLBusedc$ of directions of discontinuity of the Busemann process is the set of directions in which not all geodesics coalesce and in which there exist at least two distinct geodesics from each initial point. This creates two distinct families of coalescing geodesics in each $\DLBusedc$ direction. In $\DLBusedc$ directions, the Busemann difference profile is distributed like Brownian local time. We describe the point process of directions   $\dir\in\DLBusedc$ and spatial locations where the $\dir\pm$ Busemann functions separate. 
\end{abstract}

\tableofcontents


\section{Introduction}
\subsection{KPZ fixed point and directed landscape}


The study of the Kardar-Parisi-Zhang (KPZ) class of 1+1 dimensional stochastic models of growth and interacting particles  has  advanced to the point where  the first  conjectured universal scaling limits  have been rigorously constructed. These two interrelated objects are the {\it KPZ fixed point}, initially derived as the limit of the totally asymmetric simple exclusion process (TASEP)~\cite{KPZfixed}, and  the {\it directed landscape} (DL), initially derived as the limit of Brownian last-passage percolation (BLPP)~\cite{Directed_Landscape}.
 The KPZ fixed point describes the height of a growing interface, while the directed landscape describes the random environment through which growth propagates. These two objects are related by a variational formula, recorded in \eqref{eqn:KPZ_DL_rep} below.  Evidence for the universality claim comes from  rigorous scaling limits of exactly solvable models~\cite{reflected_KPZfixed,heat_and_landscape,KPZ_equation_convergence,Dauvergne-Virag-21}.
 
 Our paper studies the global geometry of the directed landscape, through the analytic and probabilistic properties of its  Busemann process. Our construction of the Busemann process begins with the recent construction of individual Busemann functions by Rahman and Vir\'ag~\cite{Rahman-Virag-21}. The remainder of this introduction describes the context  and gives   previews of some results. The organization of the paper is  in Section \ref{sec:org}.  
 

\subsection{Semi-infinite geodesics and Busemann functions}
In growth models of first- and last-passage type,   \textit{semi-infinite geodesics} trace the paths of infection all the way to infinity and hence are central to the large-scale structure of the evolution.  Their study was initiated by Licea and Newman in first-passage percolation in the 1990s~\cite{licea1996,Newman} with  the first results on existence, uniqueness and coalescence.   Since the work of Hoffman~\cite{Hoffman-2005,hoffman2008}, Busemann functions have been a key tool for studying semi-infinite geodesics (see, for example~\cite{Damron_Hanson2012,Hanson-2018,Georgiou-Rassoul-Seppalainen-17a,Timo_Coalescence,Seppalainen-Sorensen-21a,Seppalainen-Sorensen-21b,Rahman-Virag-21,Ganguly-Zhang-2022a}, and Chapter 5 of~\cite{50years}). 


Closer to the present work, the study of semi-infinite geodesics began in directed last-passage percolation with the application of the Licea-Newman techniques  to the exactly solvable exponential model  by Ferrari and Pimentel~\cite{ferr-pime-05}. 
Georgiou, Rassoul-Agha, and the second author~\cite{Georgiou-Rassoul-Seppalainen-17a,Georgiou-Rassoul-Seppalainen-17b} showed the existence of semi-infinite geodesics in directed  last-passage percolation with general weights under mild moment conditions. Using this, Janjigian, Rassoul-Agha, and the second author~\cite{Janjigian-Rassoul-Seppalainen-19} showed that geometric properties of the semi-infinite geodesics can be found by studying analytic properties of the Busemann process. In the special case of exponential weights,   the distribution of the Busemann process from~\cite{Fan-Seppalainen-20} 
was used to show that all geodesics in a given direction coalesce if and only if that direction is not a discontinuity of the Busemann process. 


In~\cite{Seppalainen-Sorensen-21b} the second and third author extended this work to the semi-discrete setting, by deriving the distribution of the Busemann process and analogous results for semi-infinite geodesics in BLPP. Again, all semi-infinite geodesics in a given direction coalesce if and only if that direction is not a discontinuity of the Busemann process. In each direction of discontinuity there are two coalescing families of semi-infinite geodesics  and from each initial point  \textit{at least} two semi-infinite geodesics. Compared to LPP on the discrete lattice, the semi-discrete setting of BLPP gives rise to additional non-uniqueness. In particular, \cite{Seppalainen-Sorensen-21b} developed a new coalescence proof   to handle the non-discrete setting. 

In the directed landscape, Rahman and Vir\'ag~\cite{Rahman-Virag-21} showed the existence of semi-infinite geodesics, almost surely in a fixed direction across all initial points, as well as almost surely from a fixed initial point across all directions. Furthermore,  all semi-infinite geodesics in a fixed  direction coalesce almost surely. This allowed \cite{Rahman-Virag-21} to construct a Busemann function for a fixed direction.
After the first version of our present paper was posted, Ganguly and Zhang~\cite{Ganguly-Zhang-2022a} gave an independent construction of a Busemann function and semi-infinite geodesics, again for a fixed direction. They   defined a notion of ``geodesic local time'' which was key to understanding the global fractal geometry of geodesics in DL. Later in~\cite{Ganguly-Zhang-2022b}, the same authors showed that the discrete analogue of geodesic local time in exponential LPP converges to geodesic local time for the DL. 

Starting from the definition in~\cite{Rahman-Virag-21}, we construct the full  Busemann process across all directions. Through the properties of this process, we establish a classification of uniqueness and coalescence of semi-infinite geodesics in  the directed landscape. Similar constructions of the Busemann process and classifications for discrete and semi-discrete models have previously been achieved  \cite{Sepp_lecture_notes,Janjigian-Rassoul-2020b,Janjigian-Rassoul-Seppalainen-19,Seppalainen-Sorensen-21b}, but the procedure in the directed landscape is more delicate. One reason is that the  space is fully continuous.  Another difficulty is that Busemann functions in DL possess monotonicity only in horizontal directions, while    discrete and semi-discrete models  exhibit   monotonicity   in both  horizontal and vertical directions. A new perspective is needed to construct the Busemann process for arbitrary initial points. 

The full Busemann process is necessary for a complete understanding of the geometry of semi-infinite geodesics.  In particular, countable dense sets of initial points or directions cannot capture non-uniqueness of geodesics or the singularities of the Busemann process. 



\subsection{Stationary horizon as the Busemann process of the directed landscape}
 The {\it stationary horizon} (SH) 
 is a cadlag process indexed by the real line  whose states are Brownian motions with drift (Definition~\ref{def:SH} in Appendix~\ref{sec:stat_horiz}). SH  was first introduced by the first author~\cite{Busani-2021} as the diffusive scaling limit of the Busemann process of exponential last-passage percolation from~\cite{Fan-Seppalainen-20}, and was conjectured to be the universal scaling limit of the Busemann process of models in the KPZ universality class. Shortly afterward, the paper~\cite{Seppalainen-Sorensen-21b} of the last two authors was posted. To derive the aforementioned results about semi-infinite geodesics, they constructed the Busemann process in BLPP and made several explicit distributional calculations. Remarkably, after discussions with the first author, the second and third authors discovered that the Busemann process of BLPP has the same distribution as the SH, restricted to nonnegative drifts. Furthermore, due to a rescaling property of the stationary horizon, when the direction is perturbed on order $n^{-1/3}$ from the diagonal, this process also converges to the SH, in the sense of finite-dimensional distributions. These results were added to the second version of~\cite{Seppalainen-Sorensen-21b}. 
 
 The convergence of the full Busemann process of  exponential LPP to SH under the KPZ scaling, proven  in~\cite{Busani-2021}, is currently the only example of what we expect  to be a universal phenomenon: namely,  that  SH is the universal limit of the Busemann processes  of models in the KPZ class.  The present paper takes a step towards this universality,  by establishing that the stationary horizon is the Busemann process of the directed landscape, which itself is the conjectured  universal scaling limit of metric-like objects in the KPZ class. This is the central result that gives access to properties of the Busemann process.  In addition to  giving strong evidence towards the universality of SH conjectured by \cite{Busani-2021}, it provides us with computational tools for studying the geometric features of DL.  
 
 The characterization of the Busemann process of DL comes from a combination of two results. (i) The Busemann process evolves as a backward KPZ fixed point. (ii)  The stationary horizon is the unique invariant distribution of  the KPZ fixed point, subject to  an asymptotic slope condition  satisfied by the Busemann process (Theorem~\ref{thm:invariance_of_SH}). 
 Our invariance result is an infinite-dimensional extension of the previously proved invariance of Brownian motion with  drift  \cite{KPZfixed,Pimentel-21a,Pimentel-21b}. 
    For the invariance of a single Brownian motion, we have  a strengthened uniqueness statement   (Remark~\ref{rmk:k = 1 uniqueness}).
 Furthermore, under asymptotic slope conditions   on the   initial data, the stationary horizon is an attractor. 
 This  is analogous to the results of~\cite{Bakhtin-Cator-Konstantin-2014,Bakhtin-Li-19,Bakhtin-16,Bakhtin-16chapter,Bakhtin-2013} for stationary solutions of the Burgers equation with random Poisson and kick forcing.

\subsection{Non-uniqueness of geodesics and random fractals} Among the key questions   is the uniqueness of semi-infinite geodesics in the directed landscape. We show the existence of a countably infinite, dense  random  set $\DLBusedc$ of directions $\dir$ such that, from each initial point in $\R^2$,    two semi-infinite geodesics in direction $\dir$ emanate,   separate immediately or after some time, and never return back together. It is interesting to relate this result and its proof to earlier work on disjoint finite geodesics.  

The set of exceptional pairs of points between which there is a non-unique geodesic in DL was studied in~\cite{Bates-Ganguly-Hammond-22}. Their approach relied on~\cite{Basu-Ganguly-Hammond-21} which studied the random nondecreasing function $z \mapsto \Ll(y,s;z,t) - \Ll(x,s;z,t)$ for fixed $x < y$ and $s < t$. This process is locally constant except on an exceptional set of Hausdorff dimension $\f{1}{2}$. From here~\cite{Bates-Ganguly-Hammond-22} showed that for fixed $s < t$ and $x < y$, the set of $z \in \R$ such that there exist disjoint geodesics from $(x,s)$ to $(z,t)$ and from $(y,s)$ to $(z,t)$ is exactly the set of local variation of the function $z \mapsto \Ll(x,s;z,t) - \Ll(y,s;z,t)$, and therefore has Hausdorff dimension $\f{1}{2}$. Going further, they showed that for fixed $s < t$, the set of pairs $(x,y) \in \R^2$ such that there exist two disjoint geodesics from $(x,s)$ to $(y,t)$ also has Hausdorff dimension $\f{1}{2}$, almost surely. Later, this exceptional set in the time direction was studied in~\cite{Ganguly-Zhang-2022a},and was shown to have Hausdorff dimension $2/3$. Across the entire plane, this set has Hausdorff dimension $\f{5}{3}$.
In a similar spirit, Dauvergne~\cite{Dauvergne-23} recently posted a paper detailing all the possible configurations of non-unique point-to-point geodesics, along with the Hausdorff dimensions--with respect to a particular metric--of the sets of points with those configurations.

Our focus is on the  limit  of the measure studied in~\cite{Basu-Ganguly-Hammond-21},  namely, the nondecreasing function $\dir \mapsto \W_{\dir}(y,s;x,s)= \lim_{t \to \infty}[\Ll(y,s;t\dir,t) - \Ll(x,s;t\dir,t)]$, which is exactly the Busemann function in direction $\dir$. 
  The support of its Lebesgue-Stieltjes measure corresponds to the existence of disjoint geodesics (Theorem~\ref{thm:Buse_pm_equiv}), but in contrast to~\cite{Bates-Ganguly-Hammond-22},  the measure is supported on a countable discrete set instead of on a set of Hausdorff dimension $\f{1}{2}$ (Theorem~\ref{thm:DLBusedc_description}\ref{itm:DL_Buse_no_limit_pts} and Remark~\ref{rmk:shock_measure}). 

We encounter a Hausdorff dimension $\f{1}{2}$ set if we look along a fixed time level  $s$ for those space-time points $(x,s)$ out of  which there are disjoint semi-infinite geodesics in a {\it random, exceptional} direction (Theorem~\ref{thm:Split_pts}\ref{itm:Hasudorff1/2}). Up to the removal of an at-most countable set, this Hausdorff dimension $\f{1}{2}$ set is the support of the random measure defined by the function
\[
x \mapsto f_{s,\dir}(x) = \W_{\dir +}(x,s;0,s) - \W_{\dir -}(x,s;0,s),
\]
where $\W_{\dir\pm}$ 
are the right and left-continuous Busemann processes (Theorem~\ref{thm:random_supp}). This is a semi-infinite analogue of the result in~\cite{Bates-Ganguly-Hammond-22}.  

The distribution of $f_{s,\dir}$ is  delicate.  The set of directions $\dir$ such that $\W_{\dir -} \neq \W_{\dir +}$, or equivalently such that 
$\tau_\xi=\inf\{x>0: f_{s,\dir}(x)>0\}<\infty $,
is the set   $\DLBusedc$ mentioned above. A fixed direction $\dir$ lies in $\DLBusedc$ with probability $0$.   Theorem~\ref{thm:BusePalm} shows that the law of $f_{s,\dir}(\tau_\xi+\aabullet)$ on $\R_{\ge0}$, conditioned on $\dir \in \DLBusedc$ in the appropriate Palm sense,  is exactly that of the running maximum  of a Brownian motion, or equivalently, that of Brownian local time. This complements the fact that the function $z \mapsto \Ll(y,s;z,t) - \Ll(x,s;z,t)$ is locally absolutely continuous with respect to Brownian local time \cite{Ganguly-Hegde-2021}.   Furthermore, the point process $\{(\tau_\dir, \dir):\dir\in\DLBusedc\}$ has an explicit mean measure (Lemma~\ref{lm:ac} in Section \ref{sec:Palm}). 

Since the first version of the present article has appeared, Bhatia~\cite{Bhatia-22,Bhatia-23} has posted two papers that use our results as inputs. The first,~\cite{Bhatia-22} studies the Hausdorff dimension of the set of splitting points of geodesics, along a geodesic itself. The second,~\cite{Bhatia-23} answers an open problem presented in this paper. Namely, the sets $\NU_0^{\dir \sig}$ and $\NU_1^{\dir \sig}$ defined in~\eqref{NU0}--\eqref{NU1} are almost surely equal, and for a fixed direction $\dir$, this set almost surely has Hausdorff dimension $\f{4}{3}$ in the plane.




\subsection{Inputs} We summarize the inputs to this paper, besides the   basic   \cite{Directed_Landscape, KPZfixed,reflected_KPZfixed}. Four ingredients go into the invariance of SH under the KPZ fixed point:  (i)  
The invariance of the Busemann process of  the exponential corner growth model  under the LPP dynamics \cite{Fan-Seppalainen-20}. (ii) Convergence of this Busemann process to SH  \cite{Busani-2021}. Here,  the emergence of SH as a scaling limit  in the KPZ universality class plays a fundamental role. (iii) Exit point bounds for stationary exponential LPP  \cite{bala-busa-sepp-20, Balazs-Cator-Seppalainen-2006, Emrah-Janjigian-Seppalainen-20,Seppalainen-Shen-2020, Sepp_lecture_notes}. 
(iv)   Convergence of exponential LPP to  DL \cite{Dauvergne-Virag-21}.    
For the uniqueness, we use 
Lemma~\ref{lem:DL_crossing_facts}\ref{itm:KPZ_crossing_lemma}, originally from~\cite{Pimentel-21b}.

To construct the global Busemann process, we start from the results in~\cite{Rahman-Virag-21},   summarized in Section~\ref{sec:RV_summ}. After  the first version of our paper appeared,~\cite{Ganguly-Zhang-2022a} gave an independent construction of the Busemann function in a fixed direction.  Our results do not rely on~\cite{Ganguly-Zhang-2022a}. After characterizing the distribution of the Busemann process, we use the regularity of SH from~\cite{Busani-2021,Seppalainen-Sorensen-21b} to prove results about the regularity of the Busemann process and semi-infinite geodesics.

To describe the size of the  exceptional sets of points with non-unique geodesics (Theorems~\ref{thm:Split_pts} and ~\ref{thm:DLNU}\ref{itm:DL_NU_count}), we use results about point-to-point geodesics from~\cite{Bates-Ganguly-Hammond-22} and~\cite{Dauvergne-Sarkar-Virag-2020}. 
A result from~\cite{Dauvergne-22} implies  Lemma~\ref{lm:horiz_shift_mix} and the mixing in Theorem~\ref{thm:Buse_dist_intro}\ref{itm:stationarity}.

Our techniques  are probabilistic rather than integrable, but some results we use come from  integrable inputs. We use results of~\cite{Directed_Landscape,Dauvergne-Virag-21,Busani-2021}, which each utilized the continuous RSK correspondence 
\cite{O'Connell-2003,rep_non_colliding}. We also use results on  point-to-point geodesics in~\cite{Bates-Ganguly-Hammond-22,Dauvergne-Sarkar-Virag-2020} that rely on \cite{Hammond2}, who studied the   number of disjoint geodesics in BLPP using integrable inputs. For more about the connections between RSK and the directed landscape, we refer the reader to \cite{Dauvergne-Nica-Virag-2021,Dauvergne-Zhang-2021}.

\subsection{Organization of the paper}\label{sec:org} 
Section~\ref{sec:model_main_results}  defines the models and states three results   accessible without further definitions: Theorem~\ref{thm:invariance_of_SH} (proved in Section~\ref{sec:invariance}) on the  unique invariance and attractiveness  of SH under the KPZ fixed point, Theorem~\ref{thm:DLSIG_main} (proved in Section~\ref{sec:Buseextraproofs}) on the global structure of semi-infinite geodesics   in DL, and Theorem~\ref{thm:Split_pts} (proved in Section~\ref{sec:last_proofs}) on the fractal properties of the set of initial points with disjoint semi-infinite geodesics in the same direction. Section~\ref{sec:invariance} proves Theorem~\ref{thm:invariance_of_SH}.  Section~\ref{sec:RV_summ} summarizes the results of~\cite{Rahman-Virag-21} that we use as the starting point for constructing the Busemann process. 

The remainder of the paper covers finer results on the Busemann process and semi-infinite geodesics. Sections~\ref{sec:Buse_geod_results}--\ref{sec:meas_supp} each start with several theorems that are then proved later in the paper. 
The theorems can be read independently of the proofs. 
Each section depends on the sections that came before. Section~\ref{sec:Buse_geod_results} describes the construction of the Busemann process and infinite geodesics in all directions. Section~\ref{sec:LR_sig} gives a detailed discussion of non-uniqueness of geodesics. Section~\ref{sec:geometry_sec} is concerned with coalescence and connects the regularity of the Busemann process to the geometry of geodesics. This culminates in the proof of Theorem~\ref{thm:DLSIG_main}. Section~\ref{sec:meas_supp} develops the theory of random measures for the Busemann process, culminating in the proof of Theorem~\ref{thm:Split_pts}.    Section~\ref{sec:op} collects  open problems.   The appendices contain material from the literature. 

\subsection{Acknowledgements}
Duncan Dauvergne explained the mixing of the directed landscape, recorded as Lemma~\ref{lm:horiz_shift_mix}. E.S.\ thanks also Erik Bates, Shirshendu Ganguly, Jeremy Quastel, Firas Rassoul-Agha, and Daniel Remenik for helpful discussions. We also thank the two anonymous referees for incredibly helpful comments that have greatly improved the organization and exposition of this paper.  

The work of O.~Busani was funded by the Deutsche Forschungsgemeinschaft (DFG, German Research Foundation) under Germany’s Excellence Strategy--GZ 2047/1, projekt-id 390685813, and partly performed at  University of Bristol. T.\ Sepp\"al\"ainen was partially supported by National Science Foundation grants DMS-1854619 and DMS-2152362 and by the Wisconsin Alumni Research Foundation. E.~Sorensen was partially supported by T.~Sepp{\"a}l{\"a}inen under National Science Foundation grants DMS-1854619 and DMS-2152362.

\section{Model and main theorems} \label{sec:model_main_results}
\subsection{Notation}
\label{sec:notat} 
\begin{enumerate}
[label={\rm(\roman*)}, ref={\rm\roman*}]   \itemsep=2pt 
    \item  $\Z$, $\Q$ and $\R$ are restricted by subscripts, as in for example $\Z_{> 0}=\{1,2,3,\dotsc\}$.
    \item $\mbf e_1= (1,0)$ and $\mbf e_2 = (0,1)$ denote the standard basis vectors in $\R^2$.
    \item Equality in distribution is   $\tspb\deq\tspb$ and convergence in distribution $\tspb\Longrightarrow$.
    \item   $X \sim \Exp(\rho)$ means that 
    $\Pp(X>t)=e^{-\rho t}$ for $t>0$. 
    \item The increments of a function $f:\R \to \R$ are denoted by $f(x,y) = f(y) - f(x)$.
    \item Increment ordering of $f,g:\R \to \R$:   $f \li g$ means that  $f(x,y) \le g(x,y)$ for all $x < y$.
    \item For $s \in \R$,   $\Hh_s=\{(x,s): x \in \R\}$ is  the set of space-time points at time level $s$.
    \item A two-sided standard Brownian motion is a continuous random process $\{B(x): x \in \R\}$ such that $B(0) = 0$ almost surely and   $\{B(x):x \ge 0\}$ and $\{B(-x):x \ge 0\}$ are two independent standard Brownian motions on $[0,\infty)$.
    \item\label{def:2BMcmu} If $B$ is a two-sided standard Brownian motion, then 
    $\{c B(x) + \mu x: x \in \R\}$ is a two-sided Brownian motion with diffusivity $c>0$ and drift $\mu\in\R$. 
    \item The parameter domain of the directed landscape is  $\Rup = \{(x,s;y,t) \in \R^4: s < t\}$.
    \item The Hausdorff dimension of a set $A$ is denoted by $\dim_H(A)$. 
\end{enumerate}
\subsection{Geodesics in the directed landscape} \label{sec:DL_geod}
  The directed landscape, originally constructed in~\cite{Directed_Landscape}, is a random continuous function $\Ll:\Rup \to \R$ that arises as the scaling limit of a large class of models in the KPZ universality class, and is expected to be a universal limit of such models. We cite the theorem for convergence of exponential last-passage percolation in Theorem~\ref{thm:conv_to_DL} in Appendix~\ref{sec:LPP}, and summarize some key points from~\cite{Directed_Landscape} here. The directed landscape satisfies the metric composition law: for $(x,s;y,u) \in \Rup$ and $t \in (s,u)$,
\be \label{eqn:metric_comp}
\Ll(x,s;y,u) = \sup_{z \in \R}\{\Ll(x,s;z,t) + \Ll(z,t;y,u)\}.
\ee
This implies the reverse triangle inequality:  for $s < t < u$ and $(x,y,z) \in \R^3$, $\Ll(x,s;z,t) + \Ll(z,t;y,u) \le \Ll(x,s;y,u)$. 
Furthermore, over disjoint time intervals $(s_i,t_i)$, $1 \le i \le n$, the processes $(x,y) \mapsto \Ll(x,s_i;y,t_i)$ are independent.

Under the directed landscape, the length of  a continuous path $g:[s,t] \to \R$  is
\[
\Ll(g) = \inf_{k \in \Z_{>0}} \; \inf_{s = t_0 < t_1 < \cdots < t_k = t} \sum_{i = 1}^k \Ll(g(t_{i - 1}),t_{i - 1};g(t_i),t_i),
\]
where the second infimum is over all partitions $s = t_0 < t_1 < \cdots < t_k < t$.
By the reverse triangle inequality, $\Ll(g) \le \Ll(g(s),s;g(t),t)$. We call $g$ a \textit{geodesic} if equality holds. When this occurs, every  partition $s = t_0 < t_1 < \cdots < t_k = t$ satisfies 
\[
\Ll(g(s),s;g(t),t) = \sum_{i = 1}^k \Ll(g(t_{i - 1}),t_{i - 1};g(t_i),t_i).
\]
  For  fixed $(x,s;y,t) \in \Rup$, there exists almost surely a unique geodesic between $
(x,s)$ and $(y,t)$ \cite[Sect.~12--13]{Directed_Landscape}. Across all points, there exist leftmost and rightmost geodesics. The leftmost geodesic $g$ is such that, for each $u \in (t,s)$, $g(u)$ is the leftmost maximizer of $\Ll(x,s;z,u) + \Ll(z,u;y,t)$ over $z \in \R$.  The analogous fact holds for the rightmost geodesic. Geodesics in the directed landscape   have H\"older regularity  $\f{2}{3} - \ve$  but not  $\f{2}{3}$ \cite{Directed_Landscape,Dauvergne-Sarkar-Virag-2020}.

A \textit{semi-infinite geodesic}   from $(x,s) \in \R^2$ is a continuous path $g:[s,\infty) \to \R$ such that $g(s) = x$ and   the restriction of $g$ to each  domain $[s,t]\subseteq[s,\infty)$ is a geodesic between $(x,s)$ and $(g(t),t)$. Such an infinite path $g$ has {\it direction} $\dir \in \R$ if $\lim_{t \to \infty} g(t)/t=\dir$.  Two semi-infinite geodesics $g_1$ and $g_2$ \textit{coalesce} if there exists $t$ such that $g_1(u) = g_2(u)$ for all $u\ge t$. If $t$ is the minimal such time, then  $(g_1(t),t)$ is the \textit{coalescence point}.
Two semi-infinite geodesics $g_1,g_2:[s,\infty) \to \R$ are \textit{distinct} if $g_1(t) \neq g_2(t)$ for at least some  $t\in(s,\infty)$ and \textit{disjoint} if $g_1(t) \neq g_2(t)$ for all $t\in(s,\infty)$.


\subsection{KPZ fixed point}
The KPZ fixed point $h_t(\aabullet;\mathfrak h)$ started from initial state $\h$  is a Markov process on the space of upper semi-continuous functions. More precisely, its state space is defined as 
\be \label{UCdef}
\begin{aligned}
\UC &= \{\text{ upper semi-continuous functions }\h:\R \to \R \cup \{-\infty\}: \\\  &\quad \text{ there exist }a,b > 0 \text{ such that } \quad \h(x) \le a + b|x| \text{ for all }x \in \R, \\ &\quad \text{ and }\h(x) > -\infty \text{ for some }x \in \R\}.
\end{aligned}
\ee
The topology on this space is that of local Hausdorff convergence of hypographs. When restricted to continuous functions, this convergence is equivalent to uniform convergence on compact sets (Section 3.1 in~\cite{KPZfixed}). This subspace of continuous functions is preserved under the KPZ fixed point (\cite{KPZfixed}, Lemma~\ref{lem:max_restrict}). The process $\{h_t(\aabullet;\h)\}_{t \ge 0}$ can be  represented as  \cite{reflected_KPZfixed}  
\be \label{eqn:KPZ_DL_rep}
h_t(y;\h) = \sup_{x \in \R}\{\h(x) + \Ll(x,0;y,t)\}, \quad y\in\R, 
\ee
where $\Ll$ is the directed landscape. 
If $\h$ is a two-sided Brownian motion with diffusivity  $\sqrt 2$ and arbitrary drift, then $ h_t(\aabullet;\h) - h_t(0;\h) \deq \h(\aabullet)$  for each $t > 0$  \cite{KPZfixed,Pimentel-21a,Pimentel-21b}.

\subsection{Stationary horizon} \label{sec:SHintro} The stationary horizon (SH) is a process $G = \{G_\dir\}_{\dir \in \R}$ with values $G_\dir$ in the space $C(\R)$ of continuous $\R\to\R$ functions. $C(\R)$ has its Polish topology of uniform convergence on compact sets. The paths $\dir\mapsto G_\dir$ lie  in the Skorokhod space $D(\R,C(\R))$ of cadlag functions   $\R \to C(\R)$. This means that for each $\dir \in \R$, $\lim_{\beta \searrow \dir} G_\beta = G_\dir$, where convergence holds uniformly on compact sets. The limit $\lim_{\alpha \nearrow} G_\alpha$ also exists in the same sense, but is not necessarily equal to $G_\dir$. We use $G_{\dir -}$ to denote this limit. For each $\dir \in \R$, $G_{\dir}$ is a two-sided Brownian motion with diffusivity $\sqrt 2$ and drift $2\dir$. The distribution of a $k$-tuple  $(G_{\dir_1},\dotsc,G_{\dir_k})$ can be realized as an image of $k$  independent Brownian motions with drift, given in Definition~\ref{def:SH}. See  Appendix~\ref{sec:stat_horiz} for further properties of SH.  

For a compact set $K \subseteq \R$, the process $\dir \mapsto G_\dir |_K$ of functions restricted to $K$ is a jump process. Figure~\ref{fig:SH_sim} shows a simulation of   $G_\dir$. Each pair of  trajectories remains together in a  neighborhood of the origin before separating for good, both forward and backward on $\R$. 
\begin{figure}[t]
    \centering
    \includegraphics[width = 4in]{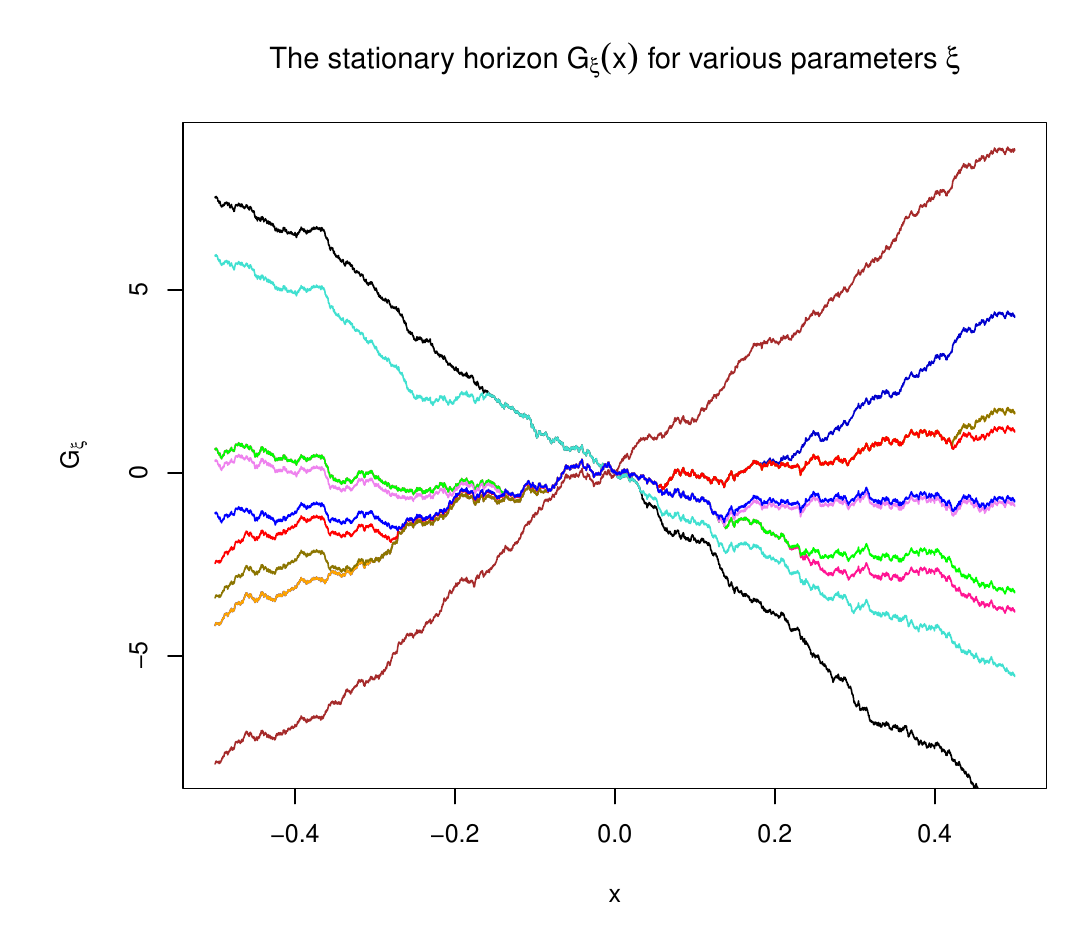}
    \caption{The stationary horizon. Each color represents a different parameter $\dir \in \{0,\pm 1,\pm 2,\pm 3,\pm 5,\pm 10\}$}
    \label{fig:SH_sim}
\end{figure}

Our first result is the unique  invariance and  attractiveness of SH  under the KPZ fixed point. This generalizes the invariance of a single Brownian motion with drift and provides a new uniqueness statement  (Remark~\ref{rmk:k = 1 uniqueness} below). Attractiveness is proved under these assumptions on the asymptotic drift $\dir \in \R$ of the initial function  $\h \in \UC$:
\begin{equation} \label{eqn:drift_assumptions}
    \begin{aligned}
    &\text{If } \dir = 0, \quad &\limsup_{x \to +\infty} \f{\h(x)}{x} \in [-\infty,0] \qquad &\text{and}\quad &\liminf_{x \to -\infty} \f{\h(x)}{x} \in [0,+\infty], \\
    &\text{if } \dir > 0,\quad &\lim_{x \to +\infty} \f{\h(x)}{x} = 2\dir\qquad&\text{and}\quad &\liminf_{x \to -\infty} \f{\h(x)}{x} \in (-2\dir,+\infty], \\
    &\text{and if } \dir < 0,\quad &\lim_{x \to -\infty} \f{\h(x)}{x} = 2 \dir\qquad&\text{and}\quad &\limsup_{x \to +\infty} \f{\h(x)}{x} \in [-\infty, -2\dir).
    \end{aligned}
\end{equation}
As spelled out in the theorem below, these conditions describe the basins of attraction for the KPZ fixed point. When $\dir > 0$ and $x>0$ is large, this condition forces $\h(x)$  to be approximated by $2\dir x$. The directed landscape $\Ll(x,s;y,t)$ can be approximated by $-\f{(x - y)^2}{t - s}$ (Lemma~\ref{lem:Landscape_global_bound}), so that $\h(x) + \Ll(x,0;y,t) \approx 2\dir x - \f{(y - x)^2}{t}$, which has its maximum at $x = y + \dir t$. Once we can control the maximizers, Lemma~\ref{lem:DL_crossing_facts} allows us to compare the KPZ fixed point from different initial conditions. This, of course, must be made precise. In the $\dir > 0$ case of the proof of Lemma~\ref{lem:unq}, the $\liminf$ condition as $x \to -\infty$ forces the maximizer to be positive, and an analogous statement holds for $\dir < 0$, although the condition is different.  These drift conditions are analogous to the conditions on the drift studied in~\cite{Bakhtin-Cator-Konstantin-2014} for stationary solutions of the Burgers equation with random Poisson forcing.

\begin{theorem} \label{thm:invariance_of_SH}
Let $(\Omega,\F,\Pp)$ be a  probability space on which the stationary horizon $G=\{G_\dir\}_{\dir \in \R}$ and directed landscape $\Ll$ are defined, and such that  the processes $\{\Ll(x,0;y,t):x,y \in \R, t > 0\}$ and $G$ are independent. For each $\dir \in \R$, let $G_\dir$ evolve under the KPZ fixed point in the same environment $\Ll$, i.e., for each $\dir \in \R$,
\[
h_t(y;G_\dir) = \sup_{x \in \R}\{G_\dir(x) + \Ll(x,0;y,t)\},\qquad\text{for all } y\in\R \text{ and } t > 0.
\]

{\rm(Invariance)} 
For each $t > 0$,   the equality in distribution  $\{h_t(\aabullet;G_\dir) - h_t(0;G_\dir)\}_{\dir \in \R}$ $\deq G$  holds between  random elements of $D(\R,C(\R))$.

\smallskip 

{\rm(Attractiveness)}
Let $k \in \Z_{>0}$ and $\dir_1 < \cdots < \dir_k$ in $\R$. Let $(\h^1,\ldots,\h^k)$ be a $k$-tuple of functions in $\UC$,  coupled with $(G,\Ll)$ {\rm arbitrarily},  and that almost surely satisfy \eqref{eqn:drift_assumptions} for $(\h, \dir) = (\h^i, \dir_i)$  for each  $i\in\{1,\dotsc,k\}$.  Then if $(\h^1,\ldots,\h^k)$ evolves   in the same environment $\Ll$,  for any $a > 0$, 
\[
\lim_{t \to \infty} \Pp\bigl\{   h_t(x;\h^i)-h_t(0;\h^i)= h_t(x;G_{\xi_i})-h_t(0;G_{\xi_i}) \ \, \forall x \in [-a,a],1 \le i \le k\bigr\} = 1.
\]
Consequently,  as $t \to \infty$, the distributional limit  
\[
\bigl(h_t(\aabullet;\h^1) -h_t(0;\h^1) ,\ldots,h_t(\aabullet;\h^k) - h_t(0;\h^k)\bigr)
\Longrightarrow
\bigl(G_{\dir_1}(\aabullet),\ldots,G_{\dir_k}(\aabullet)\bigr)
\]
holds in $\UC^k$ (or in $\C(\R)^k$ if the $\h^i$ are continuous). 

\smallskip 

{\rm(Uniqueness)}
In particular, on the space $\UC^k$,  $\bigl(G_{\dir_1}, \dotsc, G_{\dir_k})$ is the unique invariant distribution of the KPZ fixed point  such that for each  $i\in\{1,\dotsc,k\}$  the condition~\eqref{eqn:drift_assumptions} holds for $(\h, \dir) = (\h^i, \dir_i)$ almost surely. 
\end{theorem}
\begin{remark}
Theorem~\ref{thm:DL_Buse_summ}\ref{itm:global_attract} in  Section \ref{sec:Buse_geod_results} states that the Busemann process   is a global attractor of the backward KPZ fixed point. Namely, start the KPZ fixed point at time $t$ with initial data $\h$ satisfying \eqref{eqn:drift_assumptions} and run it backward in time to a fixed final time $s$.  Then, in a given a compact set,  for large enough $t$  the increments of the backwards KPZ fixed point at time $s$, started from initial data $\h$ at time $t$, match those of the Busemann function in direction $\dir$.  To prove Theorem~\ref{thm:DL_Buse_summ}\ref{itm:global_attract}, we first independently prove the attractiveness (and therefore uniqueness) of Theorem \ref{thm:invariance_of_SH}, then use this to characterize the Busemann process of the DL, which gives its regularity. This regularity is used in the proof of Theorem~\ref{thm:DL_Buse_summ}\ref{itm:global_attract}.   
\end{remark}
\begin{remark}
 The process $t \mapsto \{h_t(\aabullet;\h^\dir) - h_t(0;\h^\dir) \}_{\dir \in \R}$ is a well-defined Markov process on a state space which is a  Borel subset of  $D(\R,C(\R))$ (Lemma~\ref{lem:KPZ_preserve_Y}). By the uniqueness result for finite-dimensional distributions, $G$ is the unique invariant distribution on this space of $C(\R)$-valued cadlag paths. 
\end{remark}
\begin{remark} \label{rmk:k = 1 uniqueness}
In the above strength, the attractiveness result was previously unknown even in the case $k = 1$ (a single initial function).  Pimentel \cite{Pimentel-21a,Pimentel-21b} proved attractiveness for $k =1$ and $\dir = 0$  under the following condition on the initial data $\h$: there exist $\gamma_0 > 0$ and $\psi(r)$ such that for all $\gamma > \gamma_0$ and $r \ge 1$, 
\be \label{eqn:ergodicity_assumption}
\Pp(\gamma^{-1}\h(\gamma^2 x) \le r|x| \, \forall x \ge 1) \ge 1 - \psi(r), \  \text{ where } \lim_{r \to \infty} \psi(r) = 0.
\ee
\end{remark}

\subsection{Semi-infinite geodesics}
A significant consequence of Theorem~\ref{thm:invariance_of_SH} is that the stationary horizon characterizes the distribution of the  Busemann process of the directed landscape (Theorem \ref{thm:Buse_dist_intro}). 
The Busemann  process in turn is used to construct  semi-infinite geodesics called \textit{Busemann geodesics} simultaneously from all initial points and in all directions  (Theorem~\ref{thm:DL_SIG_cons_intro}).  The definition of Busemann geodesics, along with a detailed study, comes in Section \ref{sec:Buse_geod_results}. 

The next theorem    states  our conclusions for general semi-infinite geodesics.  The random countably infinite dense set $\DLBusedc$   of directions  is later characterized in \eqref{eqn:DLBuseDC_def} as the discontinuity set of the Busemann process, and its properties stated in  Theorem~\ref{thm:DLBusedc_description}. 

 We assume the probability space $(\Omega,\F,\Pp)$ of  the directed landscape $\Ll$  complete. All statements about semi-infinite geodesics are with respect to $\Ll$. Two geodesics are {\it disjoint} if they do not share any space-time points,  except possibly their common initial and/or final point.



\begin{theorem} \label{thm:DLSIG_main} The following statements hold on a single event of full probability. There exists a random countably infinite dense subset $\DLBusedc$ of $\R$ such that parts \ref{itm:good_dir_coal}--\ref{itm:bad_dir_split} below hold. 
\begin{enumerate} [label=\rm(\roman{*}), ref=\rm(\roman{*})]  \itemsep=3pt
    \item \label{itm:all_dir} Every semi-infinite geodesic has a direction $\dir \in \R$. From each initial point $p \in \R^2$ and in each direction $\dir \in \R$, there exists at least one semi-infinite geodesic from $p$ in direction $\dir$.
    \item \label{itm:good_dir_coal} When $\dir \notin \DLBusedc$, all semi-infinite geodesics in direction $\dir$ coalesce. There exists a random set of initial points, of zero planar Lebesgue measure, outside of which the semi-infinite geodesic in each direction $\dir \notin \DLBusedc$ is unique. 
    \item \label{itm:bad_dir_split} When $\dir \in \DLBusedc$, there exist at least two families of semi-infinite geodesics in direction $\dir$, called the $\dir -$ and $\dir +$ geodesics. From every initial point $p \in \R^2$ there exists both a $\dir -$ geodesic and a $\dir +$ geodesic which  eventually separate and never come back together. All $\dir -$ geodesics coalesce, and all $\dir +$ geodesics coalesce.
    \end{enumerate}
    \end{theorem}

 
    \begin{remark}[Busemann geodesics and general geodesics] 
Theorem \ref{thm:DLSIG_main} is proved   by controlling all semi-infinite geodesics with  Busemann geodesics. Namely,  from each initial point $p$ and in each direction $\dir$, all  semi-infinite geodesics lie between  the  leftmost and rightmost  Busemann geodesics (Theorem~\ref{thm:all_SIG_thm_intro}\ref{itm:DL_LRmost_SIG}). Furthermore, for  all   $p$ outside a random set of Lebesgue measure zero and all   $\dir \notin \DLBusedc$, the two extreme  Busemann geodesics coincide and thereby imply the uniqueness of the  semi-infinite geodesic from $p$ in direction $\dir$ (Theorem~\ref{thm:DLSIG_main}\ref{itm:good_dir_coal}). 
   Even more generally, whenever $\dir \notin \DLBusedc$,      all semi-infinite geodesics in direction $\dir$ are Busemann geodesics (Theorem~\ref{thm:DL_good_dir_classification}\ref{itm:DL_allBuse}).  This is presently unknown for $\dir \in \DLBusedc$, but may be expected   by virtue of what is known about  exponential LPP  \cite{Janjigian-Rassoul-Seppalainen-19}. 
   
   Our work therefore gives a nearly complete  description of the global behavior of semi-infinite geodesics in the directed landscape. The conjecture that all semi-infinite geodesics are Busemann geodesics is equivalent to the following statement: In Item~\ref{itm:bad_dir_split}, for $\dir \in \DLBusedc$, there are \textit{exactly} two families of coalescing semi-infinite geodesics in direction $\dir$. That is,  each $\dir$-directed semi-infinite geodesic  coalesces either with the $\dir-$ or the $\dir +$ geodesics.
\end{remark}

\begin{remark}[Non-uniqueness of geodesics]
The non-uniqueness of geodesics from initial points in a Lebesgue null set  in Theorem~\ref{thm:DLSIG_main}\ref{itm:good_dir_coal} is temporary in the sense that these geodesics eventually coalesce. This forms a  ``bubble."  The first point of intersection after the split is the coalescence point (Theorem~\ref{thm:DL_all_coal}\ref{itm:DL_split_return}). Hence, these particular geodesics form at most one bubble. This contrasts with  the non-uniqueness of  Theorem~\ref{thm:DLSIG_main}\ref{itm:bad_dir_split}, where geodesics do not return together (Figure~\ref{fig:non_unique_comp}).  Non-uniqueness is discussed in detail in Section~\ref{sec:LR_sig}.
\end{remark}

\begin{remark}
The authors of~\cite{Rahman-Virag-21} alluded to  non-uniqueness of geodesics. They showed that for a fixed initial point, with probability one, there are at most countably many directions with a non-unique geodesic. On page 23 of~\cite{Rahman-Virag-21}, they note that the set of directions with a non-unique geodesic ``should be dense over the real line.'' Our result is that this set is   dense and, furthermore, it is the set $\DLBusedc$ of discontinuities of the Busemann process. 
\end{remark}

    \begin{figure}[t]
    \centering
    \includegraphics[height = 1.5in]{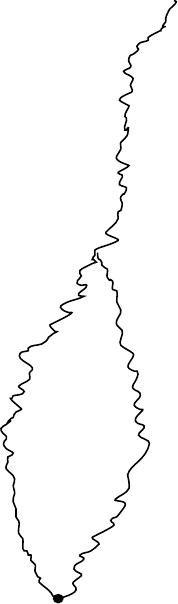} \qquad\qquad  
    \includegraphics[height = 1.5in]{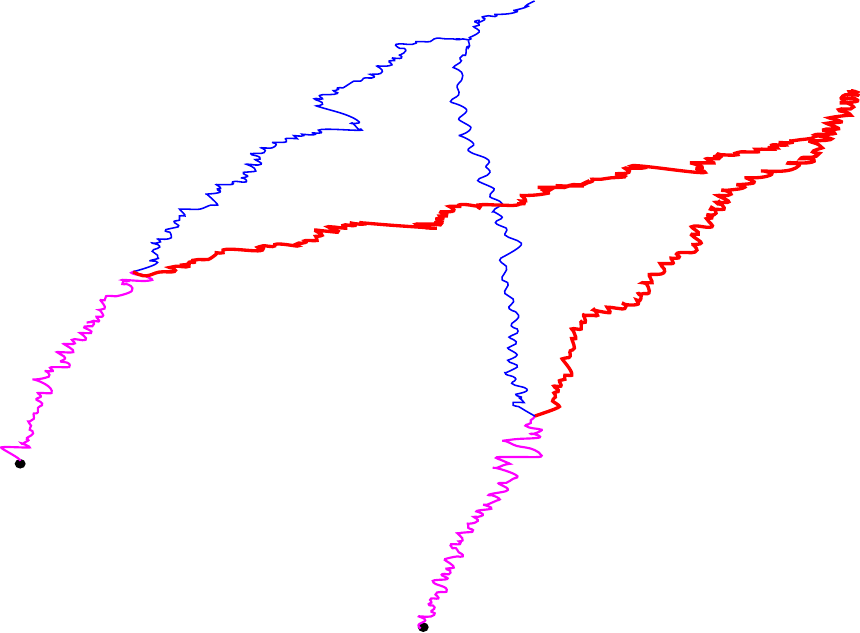}
    \caption{\small On the left, a depiction of the non-uniqueness in Theorem~\ref{thm:DLSIG_main}\ref{itm:good_dir_coal}: geodesics separate and coalesce back together, forming a bubble. After the first version of the present article was posted, Bhatia~\cite{Bhatia-23} and Dauvergne \cite{Dauvergne-23} proved that this is the only possible configuration for this type of non-uniqueness--that is, geodesics which split and later coalesce can only split at the initial point. On the right, $\dir\in \DLBusedc$. The  blue/thin paths depict the $\dir -$ geodesics, while the red/thick paths  depict the $\dir +$ geodesics. From each point, the  $\dir -$ and $\dir +$ geodesics separate at points of $\Split$. The $\dir -$ and $\dir +$ families each have a coalescing structure.}
    \label{fig:non_unique_comp}
\end{figure}

    The last theorem of this section describes   the  set of initial  points with disjoint geodesics in the same direction. Let $\DLBusedc$ be the random  set from Theorem~\ref{thm:DLSIG_main} (precisely characterized in~\eqref{eqn:DLBuseDC_def}).  Define the following random sets of splitting points.
    \begin{align}
    \Split_{s,\dir} &:= \{x \in \R:  \exists \text{ 
    \textbf{disjoint}}   \text{   }\text{semi-infinite  geodesics from  }(x,s) \text{ in direction }\dir\} \label{Split_sdir} \\
    \Split &:= \bigcup_{s \tspb\in\tspb \R, \, \dir\tspb \in\tspb \DLBusedc} \Split_{s,\dir} \times \{s\}. \label{eqn:gen_split_set}
    \end{align}
    \begin{remark}
    From Theorem~\ref{thm:DLSIG_main}\ref{itm:good_dir_coal}, $\Split_{s,\dir} = \varnothing$  whenever $\dir \notin \DLBusedc$.
    \end{remark}

    \newpage
    \begin{theorem} \label{thm:Split_pts} The following hold.
    \begin{enumerate} [label=\rm(\roman{*}), ref=\rm(\roman{*})]  \itemsep=3pt 
    \item \label{itm:split_dense}
    On a single event of full probability, the set $\Split$ is dense in $\R^2$.
    \item \label{itm:splitp0} For each fixed $p \in \R^2$, $\Pp(p \in \Split) = 0$.
    \item \label{itm:Hasudorff1/2} For each $s \in \R$, on an $s$-dependent full-probability event, for every $\dir \in \DLBusedc$, the set $\Split_{s,\dir}$ has Hausdorff dimension $\f{1}{2}$.
    \item \label{itm:nonempty} On a single event of full probability, simultaneously for every $s \in \R$ and $\dir \in \DLBusedc$, the set $\Split_{s,\dir}$ is nonempty and unbounded in both directions.
\end{enumerate}
\end{theorem}

\begin{remark} \label{rmk:supports}
For each $s \in \R$ and $\dir \in \DLBusedc$, the set 
$\Split_{s,\dir}$
has an interpretation as the support of a random measure, up to the removal of a countable set. Thus, since $\DLBusedc$ is countable, for each $s \in \R$, the set $\{x \in \R: (x,s) \in \Split\}$ is the countable union of supports of random measures, up to the removal of an at most countable set. By Item~\ref{itm:Hasudorff1/2}, this set also has Hausdorff dimension $\f{1}{2}$. 
Conditioning in the appropriate Palm sense on $\dir \in \DLBusedc$, the random measure whose support is ``almost'' $\Split_{s,\dir}$ is   equal to the local time of a Brownian motion  (Theorems~\ref{thm:random_supp},~\ref{thm:BusePalm},and~\ref{thm:indep_loc}). 
We expect that, simultaneously for all $s\in\R$, the set $\Split_{s,\dir}$ has Hausdorff dimension $\f{1}{2}$, but currently lack  a global result stronger than Item~\ref{itm:nonempty}.  
\end{remark}

\section{ Invariance and uniqueness of the stationary horizon under the KPZ fixed point} \label{sec:invariance}
In this section, we prove Theorem~\ref{thm:invariance_of_SH}. Take $\{G_\dir\}_{\dir \in \R}$ 
as the initial data of the KPZ fixed point, where $G$ is the stationary horizon, independent of $\{\Ll(x,0;y,t):x,y \in \R, t > 0\}$. For $\dir \in \R$, set
\[
h_t(y;G_\dir) = \sup_{x \in \R}\{G_\dir(x) + \Ll(x,0;y,t)\},\qquad\text{for all }y\in\R \text{ and } t > 0.
\]
Define the following state space: 
\be\label{Y}
\begin{aligned}
\Y &:= \bigl\{\{\h^\dir\}_{\dir \in \R} \in D(\R,C(\R)): \h^{\dir_1} \li \h^{\dir_2} \text{ for }\dir_1 < \dir_2,\\
&\qquad\qquad \text{ and  for all }\dir \in \R, \; \h^\dir(0) = 0 \text{ and } \h^\dir \text{ satisfies condition~\eqref{eqn:drift_assumptions}} \\
&\qquad\qquad\qquad  \text{ with all $\limsup$ and $\liminf$ terms finite}\bigr\}. 
\end{aligned}
\ee
\begin{lemma} \label{lem:KPZ_preserve_Y}
 The space $\Y$   defined in~\eqref{Y}  is a Borel subset of $D(\R,C(\R))$. Let  $\Ll$ be  the directed landscape, $\{\h^\dir\}_{\dir \in \R} \in \Y$, $h_0(\aabullet;\h^\dir)=\h^\dir$, and 
 \[
h_t(y;\h^\dir) = \sup_{x \in \R}\{\h^\dir(x) + \Ll(x,0,y;t)\} \qquad \text{for } \ t>0, \  y\in\R \text{ and } \dir \in \R. 
\]
Then 
$t \mapsto \{h_t(\aabullet;\h^\dir) - h_t(0;\h^\dir) \}_{\dir \in \R}$ is a Markov process on $\Y$. Specifically, on the event of full probability from Lemma~\ref{lem:Landscape_global_bound},    $\{h_t(\aabullet;\h^\dir) - h_t(0;\h^\dir) \}_{\dir \in \R} \in \Y$ for each $t > 0$.
\end{lemma}
\begin{proof}  Borel measurability of $\Y$ is standard and left to the reader. 
We show that $\{h_t(\aabullet;\h)-h_t(0;\h^\dir)\}_{\dir \in \R} \in \Y$ for all $t > 0$. Lemmas~\ref{lem:DL_crossing_facts}\ref{itm:KPZ_crossing_lemma} shows the preservation of the ordering of functions, Lemma~\ref{lem:KPZ_preserve_lim} shows the preservation of limits, and Lemma~\ref{lem:max_restrict}\ref{itm:KPZcont} shows that $h_t(\abullet;\h^\dir) \in C(\R)$ for all $\dir$. It remains to show that $\{h_t(\aabullet;\h^\dir)\}_{\dir \in \R} \in D(\R,C(\R))$ for each $t > 0$. Since $\h^{\dir_1} \li \h^{\dir_2}$,   Lemma~\ref{lemma:max_monotonicity} and the global bounds of Lemma~\ref{lem:Landscape_global_bound} imply  that, for each compact $K \subseteq \R$ and $\dir \in \R$, there exists a random $M = M(\dir,t,K) > 0$ such that for all $y \in K$, $\alpha \in (\dir - 1,\dir + 1)$,
\[
\sup_{x \in \R}\{\h^\alpha(x) + \Ll(x,0;y,t)\} = \sup_{x \in [-M,M]}\{\h^\alpha(x) + \Ll(x,0;y,t)\}.
\]
Then, it follows that $\{h_t(\aabullet;\h^\dir)\}_{\dir \in \R}$, as an $\R \to C(\R)$ function of $\dir$, is right-continuous with left limits because this is true of $\{\h^\dir\}_{\dir \in \R}$.

By the metric composition~\eqref{eqn:metric_comp}  of the directed landscape $\Ll$,  for $0 < s < t$,
\begin{align*} 
h_t(y;\h^\dir) - h_t(0;\h^\dir) 
&=\sup_{x \in \R}\{h_s(x;\h^\dir) -h_s(0;\h^\dir)+ \Ll(x,s;y,t)\}\\
&\qquad\qquad 
- \sup_{x \in \R}\{h_s(x;\h^\dir)- h_s(0;\h^\dir) + \Ll(x,s;0,t)\}.
\end{align*}  
The process $t \mapsto \{h_t(\aabullet;\h^\dir) - h_t(0;\h^\dir)\}_{\dir \in \R}$ is Markovian by the  independent temporal increments of   $\Ll$. 
\end{proof}

\begin{proof}[Proof of Theorem~\ref{thm:invariance_of_SH}]

\textbf{Invariance:} 
For the invariance of SH $G$, it suffices to prove the invariance of a  finite-dimensional marginal  $(G_{ \dir_1},\ldots,G_{ \dir_k})$ for given  $-\infty < \dir_1 < \cdots < \dir_k < \infty$. So for 
\be \label{605}
h_t(y;G_{\dir_i}) = \sup_{x \in \R}\{G_{\dir_i}(x) + \Ll(x,0;y,t)\}, \quad 1 \le i \le k, 
\ee
the goal is to show that for each $t > 0$, 
\be \label{606}
\bigl(h_t(\aabullet;G_{\dir_1})-h_t(0;G_{\dir_1}),\ldots, h_t(\aabullet;G_{\dir_k})-h_t(0;G_{\dir_k})\bigr) \deq (G_{\dir_1},\ldots,G_{\dir_k}).
\ee

We prove \eqref{606} via a limit using stability of discrete queues. For $N \in \Z_{>0}$ and $1 \le i \le k$, set $\rho_i = \f{1}{2} - 2^{-4/3}\dir_i N^{-1/3}$ and $\boldsymbol \rho^k = (\rho_1,\ldots,\rho_k)$.  Let $\mu^{\boldsymbol \rho^k}$ be the probability distribution on $(\R_{>0}^\Z)^k$  defined in~\eqref{mu_def} in Appendix \ref{sec:cgm-bus}. It is   the joint distribution of $k$ horizontal Busemann functions of the exponential corner growth model by Theorem \ref{thm:exp_Buse_dist}. Let $(I^{N,1},\ldots,I^{N,k})$ be a $\mu^{\boldsymbol \rho^k}$-distributed $k$-tuple of random, positive bi-infinite sequences $I^{N,i}=(I^{N,i}_j)_{j\in\Z}$.

 
 For $1 \le i \le k$, let $F^N_{i}:\R \to \R$ be the linear interpolation of the function  defined by 
 \[   F^N_i(0)=0  \ \ \text{ and } \ \     F^N_i(m)-F^N_i(k)= \textstyle\sum_{j = k+1}^m I^{N,i}_j  \quad\text{ for integers } k<m. 
 \]
  Its scaled and centered version is defined by 
\be \label{eqn:GN}
G_{i}^N(x) = 2^{-4/3}N^{-1/3}\bigl[F_{i}^N(2^{5/3} N^{2/3}x ) - 2^{8/3}N^{2/3} x\bigr]
\qquad\text{for } \ x\in\R.
\ee
Theorems~\ref{thm:exp_Buse_dist} and~\ref{thm:conv_to_SH} give the distributional limit  
\be \label{eqn:conv_of_finite_GN}
(G_{1}^N,\ldots,G_{k}^N) \Longrightarrow (G_{\dir_1},\ldots,G_{\dir_k}),
\ee
on the space  $C(\R,\R^k)$,  
under the Polish topology of uniform convergence of functions on compact sets.

For $N \in \N$ sufficiently large and $1 \le i \le k$, we consider discrete LPP with initial data $F_i^N$ and exponential weights, as in \eqref{eqn:LPP_bd} in Appendix~\ref{sec:LPP}. For $m \in \Z$ and $n \in \Z_{>0}$ let 
\[
d_i^N(m,n) = \sup_{\ell\,: \,\ell \le m}\{F_i^N(\ell) + d((\ell,1),(m,n))\}.
\]
The scaled and centered version is given by   $H_{i,0}^N = G_i^N$ and for $t > 0$ by letting  $H_{i,t}^N:\R \to \R$ be the linear interpolation of
\be \label{600}
H_{i,t}^N(y) = 2^{-4/3}N^{-1/3}\bigl[d^N_i(tN + 2^{5/3}N^{2/3} y,tN) - 4Nt - 2^{8/3}N^{2/3} y \bigr].
\ee
By Lemma~\ref{lem:D_and_LPP_bd} and  Theorem~\ref{thm:mu_invariant},   $\forall N \in \Z_{>0}$ and $t > 0$ such that $tN \in \Z$, 
\[
\bigl(H_{1,t}^N(\aabullet) - H_{1,t}^N(0),\ldots, H_{k,t}^N(\aabullet) - H_{k,t}^N(0)\bigr) \deq (G_1^N,\ldots,G_k^N).
\]
Then, using~\eqref{eqn:conv_of_finite_GN}, the proof of~\eqref{606} is completed by the following lemma. 
\begin{lemma}  Let $(G_{\dir_1},\ldots,G_{\dir_k})$ be 
independent of $\{\Ll(x,0;y,t):x,y \in \R, t > 0\}$ and  $h_t(y;G_{\dir_i})$ defined by \eqref{605}. 
Then for $t > 0$, as $N \to \infty$, in the topology of uniform convergence on compact sets of functions $\R \to \R^k$, 
we have the distributional limit 
\be\label{goal46}
(H_{1,t}^N(\aabullet),\ldots, H_{k,t}^N(\aabullet)) \Longrightarrow (h_t(\aabullet;G_{\dir_1}),\ldots,h_t(\aabullet;G_{\dir_k})).  
\ee

\end{lemma}
\begin{proof}
Replace the integer $\ell$ with a continuous variable $x$: 
\begin{align}
   \label{601}  &H_{i,t}^N(y) \; = \sup_{-\infty < \;\ell \;\le  tN + 2^{5/3} N^{2/3}y} \;
    2^{-4/3}N^{-1/3}\Bigl[F_i^N(\ell)    \\
     \nonumber  &\qquad\qquad\qquad  + \;  d\bigl((\ell,1),(tN + 2^{5/3}  N^{2/3}y,tN)\bigr) - 4Nt - 2^{8/3} N^{2/3} y\Bigr]  \\
    &= \sup_{-\infty \tspb<\tspb 2^{5/3}N^{2/3} x \;\le\; tN + 2^{5/3}N^{2/3}y}
    \; 2^{-4/3}N^{-1/3}\Bigl[F_i^N(2^{5/3} xN^{2/3}) - 2^{8/3} N^{2/3}x\nonumber \\
    &\qquad + d\bigl((2^{5/3} xN^{2/3},1),(tN + 2^{5/3}  N^{2/3}y,tN)\bigr) - 4Nt - 2^{8/3} N^{2/3} (y - x)\Bigr] \nonumber \\
    &= \; \sup_{x \in \R} \bigl\{G_i^N(x) + \Ll_N(x,0;y,t)\bigr\}, \label{602}
\end{align}
where $G_i^N$ is defined in~\eqref{eqn:GN} and 
\[
\Ll_N(x,0;y,t) =  
\f{d((2^{5/3} xN^{2/3},1),(tN + 2^{5/3}  N^{2/3}y,tN)) - 4Nt - 2^{8/3} N^{2/3} (y - x)}{2^{4/3}N^{1/3}}  \]
when $\tspb x \le y + 2^{-5/3} N^{1/3} t\,$ 
and $-\infty$ otherwise.

 Let $Z_i^N(y)$ denote  the largest maximizer of~\eqref{601}. It is precisely the exit point defined in Equation~\eqref{eqn:exit_pt}. These satisfy $Z_i^N(x) \le Z_i^N(y)$ for $x < y$. If there exists some $M > 0$ such that $|Z_i^N(y)| \le M 2^{5/3}N^{2/3}$, then 
\[  \text{\rm line \eqref{602}} 
=  \sup_{x \in [-M,M]}\{G_i^N(x) + \Ll_N(x,0;y,t)\}. \]

By the weak limit \eqref{eqn:conv_of_finite_GN},   Theorem~\ref{thm:conv_to_DL}, and independence, Skorokhod representation (\cite[Thm.~11.7.2]{dudl}, \cite[Thm.~3.1.8]{ethi-kurt}) gives  a coupling of copies of $\{(G_i^N)_{1 \le i \le k},\Ll_N\}$ and $\{(G_{\dir_i})_{1 \le i \le k},\Ll\}$ such that  $G_i^N \to G_{\dir_i}$   for $1 \le i \le k$ and  $\Ll_N \to \Ll$, almost surely and uniformly on compacts.   Then, for  $a<b$,  $M > 0$ and $\ve > 0$, in this coupling we have 
\begin{align}
    &\wh \Pp\Big(\max_{1 \le i \le k} \sup_{y \in [a,b]} |H_{i,t}^N(y) - h_t(y;G_{\dir_i})| > \ve \Big) \nonumber \\
    \label{605a}  &\le \  \wh\Pp\Big(\max_{1 \le i \le k} \sup_{y \in [a,b]} |\sup_{x \in [-M,M]}\{G_i^N(x) + \Ll_N(x,0;y,t)\}\\
    &\qquad\qquad\qquad - \sup_{x \in [-M,M]}\{G_{\dir_i}(x) + \Ll(x,0;y,t)\}| > \ve \Big) \nonumber \\
   \label{606a}  &\qquad + \; \wh \Pp\Bigl(\;\sup_{x \in \R}\{G_{\dir_i}(x) + \Ll(x,0;a,t)\} >  \sup_{x \in [-M,M]}\{G_{\dir_i}(x) + \Ll(x,0;a,t)\} \Bigr) \\
   \label{606b} &\qquad + \; \wh \Pp\Bigl(\;\sup_{x \in \R}\{G_{\dir_i}(x) + \Ll(x,0;b,t)\} >  \sup_{x \in [-M,M]}\{G_{\dir_i}(x) + \Ll(x,0;b,t)\} \Bigr) \\
    &\qquad  + \; \sum_{i = 1}^k \bigl[ \, \wh\Pp(Z_i^N(a) < -M2^{5/3} N^{2/3}) + \wh \Pp(Z_i^N(b) > M 2^{5/3}N^{2/3}) \,\bigr] . \label{607a}
\end{align}
Above, \eqref{605a} vanishes as $N\to\infty$ by the coupling.  \eqref{606a}--\eqref{606b} vanishes as $M\to\infty$ by Lemma~\ref{lem:Landscape_global_bound} because 
$G_{\dir_i}$ is a Brownian motion with drift, independent of $\{\Ll(x,0;y,t): x,y\in\R, t > 0\}$, which has leading order $-\f{(x- y)^2}{t}$ (Lemma~\ref{lem:Landscape_global_bound}).   Lemma~\ref{lemma:line_exit_pt} controls \eqref{607a}.  This combination verifies the goal \eqref{goal46}.  
\end{proof}

\noindent \textbf{Attractiveness and uniqueness:} The proof idea is similar to that of Theorem 3.3 in~\cite{Bakhtin-Cator-Konstantin-2014}. 
Let $k\in\N$ and let $\bar{\dir} = (\dir_1,\ldots,\dir_k) \in \R^k$ be a strictly increasing vector. Let $\bar{\h}=(\h^1,...,\h^k) \in \UC^k$ satisfy~\eqref{eqn:drift_assumptions} with $\h = \h^i$ and $\dir = \dir_i$ for $1 \le i \le k$.  
	Let $\ve > 0$. By Theorem~\ref{thm:SH10}\ref{itm:SH_j}, there exists $\delta > 0$ such that 
	\[
	\Pp\bigl\{G_{\dir_i \pm \delta}(x) = G_{\dir_i}(x) \ \,\forall x \in [-a,a], \, 1 \le i \le k\bigr\} \ge 1 - {\ve}/{2}.
	\]
	Then, by invariance of the stationary horizon under the KPZ fixed point, for all $t > 0$, 
	\be\label{lb1}\begin{aligned}
	\Pp\bigl\{h_t(x;G_{\xi_i\pm \delta})-h_t(0;G_{\xi_i\pm\delta})&= h_t(x;G_{\xi_i})-h_t(0;G_{\xi_i}) \\ &\qquad   \forall  x\in[-a,a], \, 1 \le i \le k\bigr\} \ge 1 - {\ve}/{2}.
\end{aligned}\ee
Recall the  sets $Z^{a,0,t}_f$ of exit points from~\eqref{exitpt}. Because $G_{\xi_i\pm\delta}$ is a Brownian motion with drift $2(\xi_i \pm \delta)$ (Theorem~\ref{thm:SH10}\ref{itm:SHpm}), it satisfies~\eqref{eqn:drift_assumptions} with drift $\xi_i\pm\delta$. By 
the temporal reflection symmetry of Lemma~\ref{lm:landscape_symm},  Lemma~\ref{lem:unq} implies that for   all $t$ sufficiently large,
\be \label{largp}
\Pp\bigl(Z^{a,0,t}_{G_{\xi_i-\delta}} \leq Z^{a,0,t}_{\h^i}\leq Z^{a,0,t}_{G_{\xi_i+\delta}} \ \, \forall \tspa 1 \le i \le k\bigr) > 1 - {\ve}/{2},
\ee
where    for $A,B\subseteq \R$ we say $A\leq  B$ if $\sup A\leq \inf B$.  By Lemma~\ref{lem:DL_crossing_facts}\ref{itm:KPZ_crossing_lemma},   on the event in~\eqref{largp} the following holds for all $x \in [0,a]$ and $1 \le i \le k$: 
\begin{equation}\label{lb2}
	h_t(x;G_{\xi_i-\delta})-h_t(0;G_{\xi_i-\delta})\leq h_t(x;\h^i)-h_t(0;\h^i)\leq h_t(x;G_{\xi_i+\delta})-h_t(0;G_{\xi_i+\delta}).  
\end{equation}
The reverse inequalities hold for $x \in [-a,0]$.

Combining \eqref{lb1}--\eqref{lb2}, we have that for sufficiently large $t$,
\[
	\Pp\bigl\{h_t(x;G_{\xi_i})-h_t(0;G_{\xi_i})= h_t(x;\h^i)-h_t(0;\h^i)\ \, \forall  x\in[-a,a],1 \le i \le k\bigr\} \ge 1 - \ve. 
\]
The proof of Theorem~\ref{thm:invariance_of_SH} is complete.  
\end{proof}

\section{Summary of the Rahman--Vir\'ag results}
\label{sec:RV_summ}
The paper~\cite{Rahman-Virag-21} shows existence of the Busemann function for a fixed direction.
Below is a summary of their results that we use.

\begin{theorem}[\cite{Rahman-Virag-21}]\label{thm:RV-SIG-thm}
The following hold.
\begin{enumerate}  [label=\rm(\roman{*}), ref=\rm(\roman{*})]  \itemsep=3pt
    \item \label{itm:p_fixed} For fixed initial point $p$, there exist almost surely leftmost and rightmost semi-infinite geodesics $g_p^{\dir,\ell}$ and $g_p^{\dir,r}$ from $p$ in every direction $\dir$ simultaneously. There are at most countably many directions $\dir$ such that $g_p^{\dir,\ell}\neq g_p^{\dir,r}$ 
    \item \label{itm:d_fixed} For fixed direction $\dir$, there exist almost surely leftmost and rightmost geodesics $g_p^{\dir,\ell}$ and $g_p^{\dir,r}$ in direction $\dir$  from every initial point $p$.
    \item \label{itm:pd_fixed} For fixed $p =(x,s) \in \R^2$ and $\dir \in \R$, $g := g_p^{\dir,\ell}= g_p^{\dir,r}$ with probability one. 
    \item \label{itm:fixed_coal} Given  $\dir\in\R$, all semi-infinite geodesics in direction $\dir$ coalesce with probability one.
\end{enumerate}
\end{theorem}
\begin{remark}

Article  \cite{Rahman-Virag-21} used $-$ and $+$ in place of the superscripts $\ell$ and $r$ used above.  We replaced    $-/+$   with $\ell/r$ to avoid confusion with our $\pm$ notation that links with the left- and right-continuous Busemann processes.  As demonstrated in Section~\ref{sec:LR_sig}, non-uniqueness of geodesics is properly characterized by two parameters $\sigg \in \{-,+\}$ and $S \in \{L,R\}$. 
\end{remark}

For fixed direction $\dir$,~\cite{Rahman-Virag-21}  defines $\kappa^\dir(p,q)$ as the coalescence point of the rightmost geodesics in direction $\dir$ from initial points  $p$ and $q$. Then, they define the Busemann function
\be\label{RVW-def}
\W_\dir(p;q) = \Ll(p;\kappa^\dir(p,q)) - \Ll(q;\kappa^\dir(p,q)).
\ee


\begin{theorem}[\cite{Rahman-Virag-21}, Corollary 3.3, Theorem 3.5, Remark 3.1] \label{thm:RV-Buse}
$ $ 
\begin{enumerate}[label=\rm(\roman{*}), ref=\rm(\roman{*})]  \itemsep=3pt
    \item \label{itm:DL_Buse_BM} For each $t \in \R$, the process $x \mapsto \W_\dir(x,t;0,t)$ is a two-sided Brownian motion with diffusivity $\sqrt 2$ and drift $2\xi$.
\end{enumerate} 
Given a direction $\dir$,  the following hold on a $\dir$-dependent event of  probability one.
\begin{enumerate} [resume, label=\rm(\roman{*}), ref=\rm(\roman{*})]  \itemsep=3pt
\item \label{itm:fixed_additive} 
Additivity: $\W_\dir(p;q) + \W_\dir(q;r) = \W_\dir(p;r)$ for all $p,q,r \in \R^2$. 

    \item \label{itm:DL_Buse_var} For all $s < t$ and $x,y \in \R$,
    \[
    \W_\dir(x,s;y,t) = \sup_{z \in \R}\{\Ll(x,s;z,t) + \W_\dir(z,t;y,t)\}.
    \]
    The supremum is attained exactly at those $z$ such that $(z,t)$ lies on a semi-infinite geodesic from $(x,s)$ in direction $\dir$. 
    \item \label{itm:DL_Buse_cont} The function $\W_\dir:\R^4 \to \R$ is continuous. 
\end{enumerate}
Moreover:  
\begin{enumerate} [resume, label=\rm(\roman{*}), ref=\rm(\roman{*})]  \itemsep=3pt
\item \label{itm:DL_Buse_mont} For a pair of fixed directions $\dir_1 < \dir_2$, with probability one, for every  $t \in \R$ and 
      $x < y$, $\W_{\dir_1}(y,t;x,t) \le \W_{\dir_2}(y,t;x,t)$.
      \end{enumerate}
\end{theorem}


\section{Busemann process and Busemann geodesics} \label{sec:Buse_geod_results}
With the intention of being accessible to a large audience, in this section, we first present a list of theorems regarding the Busemann process in Section~\ref{sec:Buse_results}. Section~\ref{sec:DL_SIG_intro} defines Busemann geodesics and states their main properties.  The proofs are found in Section~\ref{sec:DL_Buse_cons}, except for the proofs of Theorem~\ref{thm:DL_Buse_summ}\ref{itm:BuseLim1}-\ref{itm:global_attract} and the mixing in Theorem~\ref{thm:Buse_dist_intro}\ref{itm:stationarity}, which are proved in Section~\ref{sec:Buseextraproofs}, and Theorem~\ref{thm:DLBusedc_description}\ref{itm:Busedc_t}, which is proved in Section~\ref{sec:last_proofs}.
\subsection{The Busemann process} \label{sec:Buse_results}

The Busemann process $\{\W_{\dir \sig}(p;q)\}$ is indexed by points $p,q \in \R^2$, a direction $\dir \in \R$, and a sign $\sigg \in \{-,+\}$. 
The following theorems describe this global process. The parameter $\sigg \in \{-,+\}$ denotes the left- and right-continuous versions of this process as a function of $\dir$. 

\begin{theorem} \label{thm:DL_Buse_summ}
 On $(\Omega,\F,\Pp)$, there exists a process
\[
\{\W_{\dir \sig}(p;q): \dir \in \R, \,  \sigg \in \{-,+\}, \, p,q \in \R^2\}
\]
satisfying the following properties.  All the  properties below hold on a single event of probability one, simultaneously for all directions $\dir \in \R$, signs $\sigg \in \{-,+\}$, and points $p,q \in \R^2$, unless otherwise specified. Below, for $p,q \in \R^2$, we define the sets
\be \label{eqn:DLBuseDC_def}
\DLBusedc(p;q) = \{\dir \in \R: \W_{\dir -}(p;q) \neq \W_{\dir +}(p;q)\}\qquad\text{and}\qquad\DLBusedc = \textstyle\bigcup_{p,q \,\in\, \R^2} \DLBusedc(p;q).
\ee
\begin{enumerate} [label=\rm(\roman{*}), ref=\rm(\roman{*})]  \itemsep=3pt
\item{\rm(Continuity)} \label{itm:general_cts}  As an $\R^4 \to \R$ function,  $(x,s;y,t) \mapsto \W_{\dir \sig}(x,s;y,t)$ is  continuous. 
 \item {\rm(Additivity)} \label{itm:DL_Buse_add} For all $p,q,r \in \R^2$, 
    $\W_{\dir \sig}(p;q) + \W_{\dir \sig}(q;r) = \W_{\dir \sig}(p;r)$.   In particular, $\W_{\dir \sig}(p;q) = -\W_{\dir \sig}(q;p)$ and $\W_{\dir \sig}(p;p) = 0$.
    \item {\rm(Monotonicity along a horizontal line)}
    \label{itm:DL_Buse_gen_mont} Whenever $\dir_1< \dir_2$, $x < y$, and $t \in \R$,
    \[
    \W_{\dir_1 -}(y,t;x,t) \le \W_{\dir_1 +}(y,t;x,t) \le \W_{\dir_2 -}(y,t;x,t) \le \W_{\dir_2 +}(y,t;x,t).
    \]
    \item {\rm(Backwards evolution as the KPZ fixed point)}\label{itm:Buse_KPZ_description} For 
    all $x,y \in \R$ and $s < t$,
    \be\label{W_var}
    \W_{\dir \sig}(x,s;y,t) = \sup_{z \in \R}\{\Ll(x,s;z,t) + \W_{\dir \sig}(z,t;y,t)\}.
    \ee
    \item {\rm(Regularity in the direction parameter)}
    \label{itm:DL_unif_Buse_stick}
    The process $\dir\mapsto\W_{\dir +}$ is right-continuous in the sense of uniform convergence on compact sets of functions $\R^4 \to \R$, and $\dir\mapsto\W_{\dir -}$ is left-continuous in the same sense. The restrictions to compact sets are locally constant in the parameter $\dir$:  for each $\dir \in \R$ and compact set $K \subseteq \R^4$ there exists a random $\ve =\ve(\dir,K)>0$ such that, whenever $\dir - \ve < \alpha < \dir < \beta < \dir + \ve$ and $\sigg \in \{-,+\}$,   we have these  equalities  for all $(x,s;y,t) \in K$: 
    \be \label{208}
    \W_{\alpha \sig}(x,s;y,t) = \W_{\dir -}(x,s;y,t)\qquad\text{and}\qquad\W_{\beta \sig}(x,s;y,t) = \W_{\dir +}(x,s;y,t). 
    \ee
   
    \item \rm{(Busemann limits I)} \label{itm:BuseLim1}  If $\dir \notin \DLBusedc$, then, for any compact set $K \subseteq \R^2$ and any net $r_t= (z_t,u_t)_{t \in \R_{\ge 0}}$ with $u_t \to \infty$ and $z_t/u_t \to \dir$ as $t \to \infty$, there exists $R \in \R_{>0}$ such that, for all $p,q \in K$ and $t \ge R$, 
    \[
    \W_{\dir}(p;q) = \Ll(p;r_t) - \Ll(q;r_t).
    \]
    \item {\rm(Busemann limits II)} \label{itm:BuseLim2} For all $\dir \in \R$, $s \in \R$, $x < y \in \R$, and any net $(z_t,u_t)_{t \in \R_{\ge 0}}$ in $\R^2$ such that   $u_t \to \infty$ and $z_t/u_t \to \dir$ as $t \to \infty$,
    \begin{align*}
    \W_{\dir -}(y,s;x,s) &\le \liminf_{t \to \infty} \Ll(y,s;z_t,u_t) - \Ll(x,s;z_t,u_t) \\  &\le \limsup_{t \to \infty} \Ll(y,s;z_t,u_t) - \Ll(x,s;z_t,u_t) \le \W_{\dir +}(y,s;x,s).
    \end{align*}
    \item \rm{(Global attractiveness)} \label{itm:global_attract} Assume that $\dir \notin \DLBusedc$, and let $\h \in \UC$ satisfy condition~\eqref{eqn:drift_assumptions} for the parameter $\dir$. For $s < t$, let 
    \[
        h_{s,t}(x;\h) = \sup_{z \in \R}\{\Ll(x,s;z,t) + \h(z)\}.   
    \]
    Then, for any $s \in \R$ and $a > 0$, there exists a random $t_0  = t_0(a,\dir,s)<\infty$ such that for all $t > t_0$ and $x \in [-a,a]$, $h_{s,t}(x;\h) - h_{s,t}(0;\h) = \W_{\dir}(x,s;0,s)$.
    \end{enumerate}
\end{theorem}   
\begin{remark}
 Item~\ref{itm:BuseLim1} is novel in   that it shows the limits simultaneously for all $\dir \notin \DLBusedc$, uniformly over  compact subsets of $\R^2$.  The existence of Busemann limits in fixed directions
 is shown in~\cite{Rahman-Virag-21} and~\cite{Ganguly-Zhang-2022a}.
Item~\ref{itm:global_attract} is analogous to  Theorem 3.3 in~\cite{Bakhtin-Cator-Konstantin-2014} and Theorem 3.3 in~\cite{Bakhtin-Li-19} on the global solutions of the Burgers equation with random forcing.  
When comparing with ~\cite{Bakhtin-Cator-Konstantin-2014,Bakhtin-Li-19},  note  that our geodesics travel north while theirs head south. 

\end{remark}

We describe the distribution of the Busemann process. The key to   Item~\ref{itm:SH_Buse_process} is Theorem~\ref{thm:invariance_of_SH}.

\newpage
\begin{theorem} \label{thm:Buse_dist_intro}
The following hold.
\begin{enumerate} [label=\rm(\roman{*}), ref=\rm(\roman{*})] \itemsep=3pt
\item {\rm(Independence)} \label{itm:indep_of_landscape} For each $T \in \R$, these processes are independent: 
    \begin{align*}
    &\{\W_{\dir \sig}(x,s;y,t): \dir \in \R, \,\sigg \in \{-,+\}, \, x,y \in \R, \, s,t \ge T \} \\[4pt] 
    &\qquad\qquad\qquad \text{and } \ \ 
    \{\Ll(x,s;y,t): x,y \in \R,\, s < t \le T\}. 
    \end{align*}
    
 \item {\rm(Stationarity and mixing)} \label{itm:stationarity} The process
 \be \label{eqn:stat}
 \{\Ll(v),\W_{\dir \sig}(p;q):v \in \Rup, \, p,q \in \R^2, \,\dir \in \R, \,\sigg \in \{-,+\} \}
 \ee
 is stationary and mixing under shifts in any space-time direction. More precisely,  let $a,b \in \R$ not both $0$, and $z > 0$. Set $r_z = (az,bz)$.  Then, the process~\eqref{eqn:stat} is stationary and mixing {\rm(}for fixed $a,b$ as $z \to +\infty${\rm)} under the transformation
 \begin{align*}
 \bigl\{\Ll(v), \W_{\dir \sig}(p;q )\bigr\} \mapsto T_{z;a,b}\{\Ll,\W \} := \{\Ll(v + (r_z;r_z)), \W_{\dir \sig}(p + r_z;q +r_z)\},
 \end{align*}
 where, on each side, the process is understood as a function of $(v,(p,q))\in\Rup \times \R^4$.  Mixing means that, for all $k\in\Z_{>0}$, $\dir_1,\dotsc,\dir_k\in\R$,  and Borel subsets $A,B \subseteq C(\Rup,\R)\times C(\R^4,\R)^k$,
 \begin{align*}
&\lim_{z \to \infty}\Pp\Bigl(\{\Ll, \W_{\dir_{1:k}}\} \in A, \{T_{z;a,b} \Ll, T_{z;a,b}\W_{\dir_{1:k}}\} \in B\Bigr)  \\
     &\qquad\qquad\qquad =\Pp\bigl( \{\Ll, \W_{\dir_{1:k}}\} \in A\bigr) \Pp\bigl(\{\Ll, \W_{\dir_{1:k}}\} \in B \bigr) .
 \end{align*}
 Above $\W_{\dir_{1:k}}=(\W_{\dir_{1}},\dotsc,\W_{\dir_{k}})\in C(\R^4,\R)^k$. 
    \item  {\rm(Distribution along a time level)}\label{itm:SH_Buse_process} For each $t \in \R$,  the following equality in distribution holds between random elements of the Skorokhod space $D(\R,C(\R))$:
\[
\{\W_{\dir +}(\aabullet,t;0,t)\}_{\dir \in \R} \deq \bigl\{G_{\dir}(\aabullet) \bigr\}_{\dir \in \R},
\]
where $G$ is the stationary horizon   in Section \ref{sec:SHintro},   with diffusivity $\sqrt 2$ and drifts $2\xi$.
\end{enumerate}
\end{theorem}
\begin{remark}
 Combining  Items   \ref{itm:indep_of_landscape} and \ref{itm:SH_Buse_process} with  Theorem~\ref{thm:DL_Buse_summ}\ref{itm:Buse_KPZ_description} gives a description of the   Busemann process on the full plane $\R^2$. 
\end{remark}

We describe the random sets of Busemann  discontinuities defined in~\eqref{eqn:DLBuseDC_def}.

\begin{theorem} \label{thm:DLBusedc_description}
The following hold on a single event of probability one.
    \begin{enumerate} [label=\rm(\roman{*}), ref=\rm(\roman{*})]  \itemsep=3pt
    \item \label{itm:Busedc_horiz_mont} For each $t \in \R$, the set    $\DLBusedc(x,t;-x,t)$ is nondecreasing as a function of $x \in \R_{\ge 0}$.
    \item \label{itm:Busedc_t} For $s,\dir \in \R$, define the function
    \be \label{fsdir}
x \mapsto f_{s,\dir}(x) := \W_{\dir +}(x,s;0,s) - \W_{\dir -}(x,s;0,s).
\ee
Then, $\dir \in \DLBusedc$ if and only if, for all $s \in \R$,
\be \label{bad_ub}
\lim_{x \to \pm \infty} f_{s,\dir}(x) = \pm \infty.
\ee
In particular, simultaneously for all $s,x \in \R$ and all sequences  $|x_k|\to\infty$,
    \be \label{eqn:dcset_union1}
    \DLBusedc = \bigcup_k \tspb\DLBusedc(x_k,s;x,s).
    \ee
    \item \label{itm:DL_dc_set_count} The set $\DLBusedc$ is countably infinite and dense in $\R$, while for each fixed $\dir \in \R$, $\Pp(\dir \in \DLBusedc) = 0$. In particular, the full-probability event of the theorem can be chosen so that $\DLBusedc$ contains no directions $\dir \in \Q$.
    \item \label{itm:DL_Buse_no_limit_pts} For each $p \neq q$ in $\R^2$, the set $\DLBusedc(p;q)$ is discrete, that is, has no limit points in $\R$. The function $\dir \mapsto \W_{\dir -}(p;q) = \W_{\dir +}(p;q)$ is constant on each open interval $I \subseteq (\R \setminus \DLBusedc(p;q))$. For $t \in \R$, on a $t$-dependent full-probability event, for all $x < y$, $\DLBusedc(y,t;x,t)$ is infinite and unbounded, for both positive and negative $\dir$.    
\end{enumerate}
Furthermore, 
\begin{enumerate} [resume, label=\rm(\roman{*}), ref=\rm(\roman{*})]  \itemsep=3pt
\item \label{itm:DLBusedcinvar} For $x,y,t,\nu \in \R$ and $c > 0$, the sets $\DLBusedc(y,t;x,t)$ satisfy the following distributional invariances:
\[
\DLBusedc(y,t;x,t) \deq \DLBusedc(y,0;x,0) \deq -\DLBusedc(-y,0;-x,0) 
\deq c^{-1}\DLBusedc(c^{-2}y,0;c^{-2}x,0) -\nu.  
\]
\end{enumerate}
\end{theorem}
\begin{remark} \label{rmk:shock_measure}
Item~\ref{itm:Busedc_t} states that all discontinuities of the Busemann process are present on each  horizontal ray. By Item~\ref{itm:DL_Buse_no_limit_pts}  $\dir \mapsto \W_{\dir \pm}(p;q)$ are the left- and right-continuous versions of a jump process. This function defines a random signed measure  supported on a discrete set. When $p$ and $q$ lie on the same horizontal line, this function is monotone  (Theorem~\ref{thm:DL_Buse_summ}\ref{itm:DL_Buse_gen_mont}) and the support of the measure is exactly the set of directions   at which the properly chosen coalescence point of semi-infinite geodesics  jumps (see Definition~\ref{def:coal_pt} and Theorems~\ref{thm:DL_eq_Buse_cpt_paths}--\ref{thm:Buse_pm_equiv}).

The discreteness of Item~\ref{itm:DL_Buse_no_limit_pts} allows us to view the sets $\DLBusedc(y,t;x,t)$ as well-defined point processes, and gives the statements in Item~\ref{itm:DLBusedcinvar} meaning. The set $\DLBusedc$ itself is dense, and it is not easy, a priori, to interpret as a random object. However, By Items~\ref{itm:Busedc_horiz_mont} and~\ref{itm:Busedc_t}, $\DLBusedc$ is the increasing union of the sets $\DLBusedc(x_k,0;x,0)$, where $x_k$ is a monotone sequence converging to $+\infty$ or $-\infty$.
\end{remark}

\subsection{Busemann geodesics} \label{sec:DL_SIG_intro}
The study of Busemann geodesics starts with this definition. 
\begin{definition} \label{def:LR_maxes}
For $\dir  \in \R$, $\sigg \in \{-,+\}$,  $(x,s) \in \R^2$ and $t\in(s,\infty)$, let $g_{(x,s)}^{\dir \sig,L}(t)$ and $g_{(x,s)}^{\dir \sig,R}(t)$ denote, respectively, the leftmost and rightmost maximizer of $\Ll(x,s;y,t) + \W_{\dir \sig}(y,t;0,t)$ over $y \in \R$. For $t = s$, define $g_{(x,s)}^{\dir \sig,L/R}(s) = x$.
\end{definition}
\begin{remark}
The modulus of continuity bounds of the directed landscape recorded in Lemma~\ref{lem:Landscape_global_bound}, along with continuity of $W_{\dir \sig}$, imply that $\lim_{t \searrow s} g_{(x,s)}^{\dir \sig,L/R}(t) = x$, so the definition $g_{(x,s)}^{\dir \sig,L/R}(s) = x$ makes $g_{(x,s)}^{\dir \sig,L/R}$ continuous at $t = s$. In fact, the path is continuous everywhere because it is the leftmost or rightmost geodesic between any pair of points along the path (Theorem \ref{thm:DL_SIG_cons_intro}\ref{itm:DL_LRmost_geod}), and geodesics are continuous. As is seen in the proofs, we are relying on the existence of leftmost and rightmost point-to-point geodesics from \cite[Lemma 13.2]{Directed_Landscape}.
\end{remark}

As noted earlier, Rahman and Vir\'ag \cite{Rahman-Virag-21} showed the existence of semi-infinite geodesics, almost surely for a fixed initial point  across all directions and almost surely for a fixed direction across all initial points. We  extend this  simultaneously across both all initial points and directions. 
Theorem~\ref{thm:RV-Buse}\ref{itm:DL_Buse_var}, quoted from~\cite{Rahman-Virag-21}, states that for a {\it fixed} direction $\dir$, with probability one at times $t > s$, 
the maximizers $z$ of the function $\Ll(x,s;z,t) + \W_{\dir}(z,t;0,t)$ are exactly the points on semi-infinite $\xi$-directed geodesics from $(x,s)$.
Theorem~\ref{thm:DL_SIG_cons_intro} clarifies this on a global scale: across all directions, initial points and signs, one can construct semi-infinite geodesics from the Busemann process.   
Furthermore,  $g_{(x,s)}^{\dir \sig,L}$ and $g_{(x,s)}^{\dir \sig,R}$ both define  semi-infinite geodesics in direction $\dir$  
and give  the leftmost (or rightmost) geodesic between any two of their points. We use this heavily in the present paper. 

\begin{theorem} \label{thm:DL_SIG_cons_intro}
 The following hold on a single event of probability one across all initial points $(x,s) \in \R^2$, times $t>s$, directions $\dir  \in \R$, and signs $\sigg \in \{-,+\}$.
 \begin{enumerate} [label=\rm(\roman{*}), ref=\rm(\roman{*})]  \itemsep=3pt
 \item \label{itm:intro_SIG_bd} 
 All maximizers of $z\mapsto\Ll(x,s;z,t) + \W_{\dir \sig}(z,t;0,t)$ are finite. Furthermore, as  $x,s,t$ vary over a compact set $K\subseteq \R$ with $s \le t$, the set of all maximizers is bounded.
    \item \label{itm:arb_geod_cons} 
    Let $s = t_0 < t_1 < t_2 < \cdots$ be an arbitrary increasing sequence with $t_n \to \infty$. Set $g(t_0) = x$, and for each $i \ge 1$, let $g(t_i)$ be \textit{any} maximizer of $\Ll(g(t_{i - 1}),t_{i - 1};z,t_i) + W_{\dir \sig}(z,t_i;0,t_i)$ over $z \in \R$. Then, pick \textit{any} geodesic of $\Ll$ from $(g(t_{i - 1}),t_{i - 1})$ to $(g(t_i),t_i)$, and for $t_{i - 1} < t < t_i$, let $g(t)$ be the location of this geodesic at time $t$. Then, regardless of the choices made at each step, the following hold.
    \begin{enumerate} [label=\rm(\alph{*}), ref=\rm(\alph{*})]
        \item \label{itm:g_is_geod} The path $g:[s,\infty)\to \R$ is a semi-infinite geodesic.
        \item \label{itm:weight_of_geod} For all  $ t < u$ in $[s,\infty)$,
    \be \label{eqn:SIG_weight}
    \Ll(g(t),t;g(u),u) = \W_{\dir \sig}(g(t),t;g(u),u).
    \ee
    \item \label{itm:maxes} For all  $ t < u$ in $[s,\infty)$, $g(u)$ maximizes $\Ll(g(t),t;z,u) + \W_{\dir \sig}(z,u;0,u)$ over $z \in \R$. 
    \item \label{itm:geo_dir} The geodesic $g$ has direction $\dir$, i.e., $g(t)/t \to \dir$ as $t \to \infty$. 
    \end{enumerate}
    \item \label{itm:DL_all_SIG} For 
    $S \in \{L,R\}$, $g_{(x,s)}^{\dir \sig,S}:[s,\infty) \to \R$ is a semi-infinite geodesic from $(x,s)$ in direction $\dir$. Moreover, for any $s \le t < u$, we have that 
    \[
    \Ll\bigl(g_{(x,s)}^{\dir \sig,S}(t),t;g_{(x,s)}^{\dir \sig,S}(u),u\bigr) = \W_{\dir \sig}\bigl(g_{(x,s)}^{\dir \sig,S}(t),t;g_{(x,s)}^{\dir \sig,S}(u),u\bigr),
    \]
    and $g_{(x,s)}^{\dir \sig,S}(u)$ is the leftmost/rightmost {\rm(}depending on $S${\rm)} maximizer of \\$\Ll(g_{(x,s)}^{\dir \sig,S}(t),t;z,u) + \W_{\dir \sig}(z,u;0,u)$ over $z \in \R$.
    \item \label{itm:DL_LRmost_geod} 
    The path $g_{(x,s)}^{\dir \sig,L}$ is the leftmost geodesic between any two of its points, and $g_{(x,s)}^{\dir \sig,R}$ is the rightmost geodesic between any two of its points.
    \end{enumerate}
    \end{theorem}
\begin{definition}
We refer to the  geodesics constructed  in Theorem~\ref{thm:DL_SIG_cons_intro}\ref{itm:arb_geod_cons} as $\dir \sig$ \textit{Busemann geodesics}, or simply {\it $\dir \sig$ geodesics}. 
\end{definition} 
\begin{remark}
The geodesics  $g_{(x,s)}^{\dir \sig,L}$ and $g_{(x,s)}^{\dir \sig,R}$ are special  Busemann geodesics. By   Theorem~\ref{thm:DL_SIG_cons_intro}\ref{itm:DL_all_SIG}--\ref{itm:DL_LRmost_geod}, for any sequence $s=t_0 < t_1< t_2 < \cdots$ with $t_n \to \infty$, the path $g = g_{(x,s)}^{\dir \sig,L}$ can be constructed by choosing $g(t_i)$ as the leftmost maximizer of $\Ll(g(t_{i - 1}),t_{i - 1};z,t_i) + \W_{\dir \sig}(z,t_i;0,t_i)$ over $z \in \R$, and   for $t \in (t_{i - 1},t_i)$, taking $g(t)$ to be     the leftmost geodesic from $(g(t_{i - 1}),t_{i - 1})$ to $(g(t_i),t_i)$. The analogous statement holds for $L$ replaced with $R$ and ``leftmost'' replaced with ``rightmost''.
\end{remark}

\subsection{Construction and proofs for the Busemann process and Busemann geodesics} \label{sec:DL_Buse_cons}
This section proves the results of Sections~\ref{sec:Buse_results} and~\ref{sec:DL_SIG_intro}. The order in which the items are proved is somewhat delicate, so we outline that here. After proving some lemmas, we prove Theorem~\ref{thm:DL_Buse_summ}\ref{itm:general_cts}--\ref{itm:Buse_KPZ_description} and Theorem~\ref{thm:Buse_dist_intro}. We then skip ahead to constructing the semi-infinite geodesics, culminating in the proof of Theorem~\ref{thm:DL_SIG_cons_intro}. Afterward, we turn to the proof of the regularity in Theorem~\ref{thm:DL_Buse_summ}\ref{itm:DL_unif_Buse_stick}, then prove Theorem~\ref{thm:DLBusedc_description}, except for Item~\ref{itm:Busedc_t}, which is proved in Section~\ref{sec:last_proofs}.

We construct a full-probability event $\Omega_1$.  Later in~\eqref{omega2} and~\eqref{omega3} follow full-probability events $\Omega_3 \subseteq \Omega_2 \subseteq \Omega_1$. For the rest of the proofs, we work  almost exclusively on these events. Once the events are constructed and shown to have full probability, the remaining proofs are deterministic statements that hold on those events.
\be \label{omega1}
\text{We define $\Omega_1 \subseteq \Omega$ to be the event of probability one on which the following hold.}
\ee
\begin{enumerate} [label=\rm(\roman{*}), ref=\rm(\roman{*})]  \itemsep=3pt
    \item \label{om1lrf} Simultaneously for all $(x,s;y,t) \in \Rup$ there exist leftmost and rightmost geodesics {\rm(}possibly in agreement{\rm)} between $(x,s)$ and $(y,t)$ (see Section~\ref{sec:DL_geod}).
    \item \label{om1rsi} For each rational direction $\dir \in \Q$ and each point $p \in \R^2$, there exist leftmost and rightmost semi-infinite geodesics {\rm(}possibly in agreement{\rm)} from $p$ in direction $\dir$, and all semi-infinite geodesics in direction $\dir$ coalesce {\rm(}see Theorem~\ref{thm:RV-SIG-thm}, Items~\ref{itm:d_fixed} and~\ref{itm:fixed_coal}{\rm)}.
    \item \label{om1pdf} For each rational direction $\dir \in \Q$ and each rational point $p \in \Q^2$, there is a unique semi-infinite geodesic from $p$ in direction $\dir$ (see Theorem~\ref{thm:RV-SIG-thm}\ref{itm:pd_fixed}).
    \item \label{om1dirrat} For each rational direction $\dir\in \Q$, the Busemann process defined by~\eqref{RVW-def} satisfies conditions~\ref{itm:fixed_additive}--\ref{itm:DL_Buse_cont} of Theorem~\ref{thm:RV-Buse}. For any pair $\dir_1 < \dir_2$ or rational directions, Item~\ref{itm:DL_Buse_mont} of Theorem~\ref{thm:RV-Buse} holds.
    \item \label{om1agree} For each $(x,t,y,\dir) \in \Q^4$,
    $
        \lim_{\Q \ni \alpha \to \dir} \W_\alpha(y,t;x,t) = \W_{\dir}(y,t;x,t).    
    $
    \item \label{om1asym} For every rational time $t \in \Q$ and rational direction $\dir \in \Q$, 
    \be \label{eqn:rat_asymp}
        \lim_{x \to \pm \infty} x^{-1}  {\W_\dir(x,t;0,t)} = 2\dir. 
    \ee
    This holds with probability one by properties of Brownian motion and Theorem~\ref{thm:RV-Buse}\ref{itm:DL_Buse_BM}.
    \item \label{om1appB} The conclusions of Lemmas~\ref{lem:Landscape_global_bound},~\ref{lm:BGH_disj}, and~\ref{lem:geod_pp} hold for  $\Ll$. Note that then  Lemma~\ref{lem:Landscape_global_bound} holds also for the reflected version  
    $
    \{\Ll(y;-t,x;-s):(x,s;y,t) \in \Rup\}. 
    $   
\end{enumerate}

To justify $\Pp(\Omega_1)=1$, it remains only to check Item~\ref{om1agree}. By   Theorem~\ref{thm:RV-Buse}\ref{itm:DL_Buse_mont},  for $y \ge x$,
\be \label{lrlimit}
\lim_{\Q \ni \alpha \nearrow \dir} \W_{\alpha}(y,t;x,t) \le \W_{\dir}(y,t;x,t) \le \lim_{\Q \ni \alpha \searrow \dir} \W_{\alpha}(y,t;x,t).
\ee
By Theorem~\ref{thm:RV-Buse}\ref{itm:DL_Buse_BM}, $\W_{\alpha}(y,t;x,t) \sim \Nor(2\alpha(y - x),2(y - x))$. Hence, 
all terms in~\eqref{lrlimit} have the same distribution and are almost surely equal. 

\smallskip 

Now, on the full-probability event $\Omega_1$, we have defined the process
 \be \label{rat_process}
 \{\W_{\alpha}(p;q):p,q \in \R^2,\alpha \in \Q \}.
 \ee
 On this event, for an arbitrary direction $\dir$, and $t,x,y \in \R$, define
\be \label{eqn:Buse_def}
\W_{\dir -}(y,t;x,t) = \lim_{\Q \ni \alpha \nearrow \dir}\W_{\alpha }(y,t;x,t)\;\;\text{and}\;\; \W_{\dir+}(y,t;x,t) = \lim_{\Q \ni \alpha \searrow \dir} \W_{\alpha }(y,t;x,t).
\ee
By Theorem~\ref{thm:RV-Buse}\ref{itm:DL_Buse_mont}, these limits exist for all $t \in \R$. Complete the definition by setting,  
\be \label{eqn:gen_Buse_var}\begin{aligned} 
\text{  for $s < t$, } \ 
\W_{\dir \sig}(x,s;y,t) &= \sup_{z \in \R}\{\Ll(x,s;z,t) + \W_{\dir \sig}(z,t;y,t)\},\\
\text{ and finally for $s > t$, } \ \W_{\dir \sig}(x,s;y,t) &= -\W_{\dir \sig}(y,t;x,s).
\end{aligned} \ee
With this construction in place, we prove an intermediate lemma.
\begin{lemma} \label{lem:DL_horiz_Buse}
The following hold on the event $\Omega_1$, across all points, directions and signs.
\begin{enumerate} [label=\rm(\roman{*}), ref=\rm(\roman{*})]  \itemsep=3pt
\item \label{itm:DL_agree_horiz} For all $x,y,t \in \R$, and $\dir \in \Q$, $\W_{\dir -}(y,t;x,t) = \W_{\dir +}(y,t;x,t) = \W_\dir(y,t;x,t)$, where $W_\dir$ is the originally defined Busemann function from~\eqref{rat_process}. 
    \item \label{itm:DL_h_add} Horizontal Busemann functions are additive: $\forall\,x,y,z,t \in \R$, $\dir \in \R$, and $\sigg \in \{-,+\}$,
    \[
    \W_{\dir \sig}(x,t;y,t) + \W_{\dir \sig}(y,t;z,t) = \W_{\dir \sig}(x,t;z,t).
    \]
    \item \label{itm:DL_h_unif_conv} For every $t,\dir \in \R$, the limits   \eqref{eqn:Buse_def} hold uniformly over $(x,y)$ on compact sets. Further, for each $t,\dir \in \R$ and $\sigg \in \{-,+\}$,
    these limits hold in the same sense: 
    \be \label{itm:horiz_lim}
    \lim_{\alpha \nearrow \dir} \W_{\alpha \sig}(y,t;x,t) = \W_{\dir -}(y,t;x,t)\ \text{ and }\  \lim_{\alpha \searrow \dir}\W_{\alpha \sig}(y,t;x,t) = \W_{\dir +}(y,t;x,t). 
    \ee
    
    \item \label{itm:DL_lim} For every $\dir \in \R$, $\sigg \in \{-,+\}$, $(p,q) \mapsto \W_{\dir \sig}(p;q)$ is continuous, and for each $t \in \R$,
    \be \label{eqn:horiz_asymptotics}
    \lim_{x \to \pm \infty}  x^{-1} {\W_{\dir \sig}(x,t;0,t)} = 2\dir.
    \ee
    \end{enumerate}

\end{lemma}
\begin{proof}
We prove Item~\ref{itm:DL_agree_horiz} last.

\smallskip\noindent
\textbf{Item~\ref{itm:DL_h_add}} follows from the same property in rational directions (Theorem~\ref{thm:RV-Buse}\ref{itm:fixed_additive}).

\smallskip\noindent \textbf{Item~\ref{itm:DL_h_unif_conv}:} The monotonicity of the horizontal Busemann process from Theorem~\ref{thm:RV-Buse}\ref{itm:DL_Buse_mont} extends to all directions by limits. That is, for any two rational directions $\dir_1 < \dir_2$ and any real $x < y$, and $t$,
    \be \label{eqn:Buse_mont}
    \W_{\dir_1 -}(y,t;x,t) \le \W_{\dir_1}(y,t;x,t) \le  \W_{\dir_1 +}(y,t;x,t) \le \W_{\dir_2 -}(y,t;x,t), 
    \ee
    and when $\dir_1 \notin \Q$, the same monotonicity holds, removing the middle term that does not distinguish between $\pm$. 
Hence, the limits as $\alpha \nearrow \dir$ and $\alpha \searrow \dir$ exist and agree with the limits from rational directions (without the $\sigg$). Without loss of generality, we take the compact set to be $[a,b]^2$. Then, by~\eqref{eqn:Buse_mont} and Lemma~\ref{lem:ext_mont}, for $\alpha < \dir$, $\sigg \in \{-,+\}$, and $a \le x \le y \le b$,
\be \label{156}
0 \le \W_{\dir-}(y,t;x,t)- \W_{\alpha \sig }(y,t;x,t)   \le   \W_{\dir-}(b,t;a,t)- \W_{\alpha \sig }(b,t;a,t),
\ee
and for general $(x,y) \in [a,b]^2$, 
\[
|\W_{\dir-}(y,t;x,t)- \W_{\alpha \sig }(y,t;x,t)|   \le   |\W_{\dir-}(b,t;a,t)- \W_{\alpha \sig }(b,t;a,t)|,
\]
so the limit as $\alpha \nearrow \dir$ is uniform on compacts. An analogous argument applies to   $\alpha \searrow\dir$.

\smallskip\noindent \textbf{Item~\ref{itm:DL_lim}:} 
For $t,\dir \in \R$ and $\sigg \in \{-,+\}$,  the continuity of  
$
(x,y) \mapsto \W_{\dir \sig}(y,t;x,t)
$
follows from Item~\ref{itm:DL_h_unif_conv} and the continuity for rational $\xi$ in Theorem~\ref{thm:RV-Buse}\ref{itm:DL_Buse_cont}. Before showing the general continuity, we show   the limits~\eqref{eqn:horiz_asymptotics}. For   $\dir, t \in \Q$, \eqref{eqn:rat_asymp} holds by definition of $\Omega_1$.
Keeping $\dir \in \Q$, let $s \in \R$, and let $t > s$ be rational. By Theorem~\ref{thm:RV-Buse}\ref{itm:fixed_additive}--\ref{itm:DL_Buse_var}, 
\begin{align*}
\W_{\dir}(x,s;0,s) 
&= \W_{\dir}(x,s;0,t) + \W_{\dir}(0,t;0,s) \\
&= \sup_{z \in \R}\{\Ll(x,s;z,t) + \W_{\dir}(z,t;0,t)\}+ \W_{\dir}(0,t;0,s).
\end{align*}
 Then, by Lemma~\ref{lem:KPZ_preserve_lim} (for the temporally reflected $\Ll$),
 $
 \lim_{x \to \pm \infty} x^{-1}  {\W_{\dir}(x,s;0,s)} = 2\dir.
 $
 Now, let $\dir \in \R$, $\sigg \in \{-,+\}$, and $t \in \R$ be arbitrary. Then, the monotonicity of~\eqref{eqn:Buse_mont} implies that for $\alpha < \dir < \beta$ with $\alpha,\beta \in \Q$,
 \[
 \alpha \le \liminf_{x \to \infty} x^{-1}{\W_{\dir \sig}(x,t;0,t)} \le \limsup_{x \to \infty} x^{-1} {\W_{\dir \sig}(x,t;0,t)}  \le \beta.
 \]
 Sending $\Q \ni \alpha \nearrow \dir$ and $\Q \ni \beta \searrow \dir$ implies~\eqref{eqn:horiz_asymptotics} for $+\infty$. The case    $x \to -\infty$ follows a symmetric argument.

Lastly, the continuity of $(x,y) \mapsto \W_{\dir \sig}(y,t;x,t)$ and~\eqref{eqn:horiz_asymptotics} imply that $\W_{\dir \sig}(x,t;0,t) \le a + b|x|$ for some constants $a,b$. The general continuity follows from~\eqref{eqn:gen_Buse_var} and  Lemma~\ref{lem:max_restrict}\ref{itm:KPZcont}.

\smallskip 
\noindent \textbf{Item~\ref{itm:DL_agree_horiz}:} 
The statement holds for all $x,y,t,\dir \in \Q$ by Item~\ref{om1agree} of $\Omega_1$. The continuity proved in Item~\ref{itm:DL_lim} extends this to all $x,y,t \in \R$. 
\end{proof}

\noindent Recall Definition~\ref{def:LR_maxes} of  the extreme maximizers $g_{(x,s)}^{\dir \sig,L/R}(t)$.

\begin{lemma} \label{lem:bounded_maxes}
For each $\omega \in \Omega_1$, $(x,s;y,t) \in \Rup$, $\dir \in \R$, and $\sigg \in \{-,+\}$,
    \be \label{eqn:zto_pminfty}
    \lim_{z \to \pm \infty} \Ll(x,s;z,t) + \W_{\dir \sig}(z,t;y,t) = -\infty
    \ee
    so that   $g_{(x,s)}^{\dir \sig,L/R}$ are well-defined.
    Let $K \subseteq \R$ be a compact set,   $\dir \in \R$ and $\sigg \in \{-,+\}$.  Then, there exists a random $Z = Z(\dir \sig,K)  \in (0,\infty)$ such that for all $x,s,t \in K$ with $s < t$ and $S \in \{L,R\}$, $|g_{(x,s)}^{\dir \sig,S}(t)| \le Z$.
\end{lemma}
\begin{proof}
By the continuity and asymptotics of Lemma~\ref{lem:DL_horiz_Buse}\ref{itm:DL_lim}, $\forall t \in \R$  $\exists a,b > 0$ such that $|\W_{\dir \sig}(x,t;0,t)| \le a + b|x|$ $\forall x \in \R$.   Lemma~\ref{lem:Landscape_global_bound} implies  $\Ll(x,s;z,t) \sim -\f{(z - x)^2}{t - s}$,  which gives   \eqref{eqn:zto_pminfty}. 
Next we observe that
\be \label{107}
\begin{aligned}
&\inf_{x,s,t \in K, s < t}\sup_{z \in \R}\{\Ll(x,s;z,t) + \W_{\dir \sig}(z,t;0,t)\} \\[-3pt] 
 &\qquad\qquad\qquad \ge \inf_{x,s,t \in K, s < t} \Ll(x,s;x,t) + \W_{\dir \sig}(x,t;0,t) > -\infty.
\end{aligned}
\ee
The last inequality is justified as follows. Since $\W_{\dir \sig}(x,t;0,t)$ evolves backwards in time as the KPZ fixed point~\eqref{eqn:gen_Buse_var}, Lemma~\ref{lem:max_restrict}\ref{itm:KPZ_unif_line} implies that   $a$ and $b$ can be chosen uniformly for $t \in K$.   Lemma~\ref{lem:Landscape_global_bound} states that $\forall x,s,t \in \R$ with $s < t$, there is a constant $C$ such that 
\[
\Ll(x,s;x,t) \ge - C(t - s)^{1/3}\log^2\bigl(\tfrac{2\sqrt{2x^2 + s^2 + t^2} + 4}{(t - s)\wedge 1}\bigr).
\]
Taking the infimum over $x,s,t \in K$ with $s < t$ yields the last inequality in~\eqref{107}.  

To  contradict the last statement of the lemma, assume   maximizers $z_n$ of $\Ll(x_n,s_n;z,t_n) + \W_{\dir \sig}(z,t_n;0,t_n)$ over $z \in \R$ such that  $x_n,s_n,t_n \in K$  but $|z_n| \to \infty$.
Then, by~\eqref{107},
\be \label{9}
\liminf_{n \to \infty} \Ll(x_n,s_n;z_n,t_n) + \W_{\dir \sig}(z_n,t_n;0,t_n) > -\infty,
\ee
but since $z_n \to \infty$ and $x_n,s_n,t_n \in K$ for all $n$, $\Ll(x_n,s_n;z_n;t_n) \sim -\f{(z_n - x_n)^2}{t_n - s_n}$ by Lemma~\ref{lem:Landscape_global_bound}. By the bound $|\W_{\dir \sig}(x,t;0,t)| \le a + b|x|$ that holds uniformly for $t \in K$ and $x \in \R$, the inequality~\eqref{9} cannot hold.
\end{proof}

\begin{proof}[Proof of Theorem~\ref{thm:DL_Buse_summ}, Items ~\ref{itm:general_cts}--\ref{itm:Buse_KPZ_description}]
The full-probability event of these items is $\Omega_1$. The remaining items are proved later.

\smallskip\noindent
\textbf{Item~\ref{itm:general_cts} (Continuity):} This was proved in Lemma~\ref{lem:DL_horiz_Buse}\ref{itm:DL_lim}. 

\smallskip\noindent \textbf{Item~\ref{itm:DL_Buse_add} (Additivity):} 
First, we show that on $\Omega_1$ for $s < t$, $x \in \R$, $\dir_1 < \dir_2$, and $S \in \{L,R\}$,
    \be \label{eqn:mont_maxes}
    -\infty < g_{(x,s)}^{\dir_1 -,S}(t) \le g_{(x,s)}^{\dir_1 +,S}(t) \le g_{(x,s)}^{\dir_2 -,S}(t) \le g_{(x,s)}^{\dir_2 +,S}(t)  < \infty. 
    \ee
    The finiteness of the maximizers comes from Lemma~\ref{lem:bounded_maxes}. The rest of~\eqref{eqn:mont_maxes} follows from the monotonicity of~\eqref{eqn:Buse_mont} and Lemma~\ref{lemma:max_monotonicity}. Next, we show that for $(x,s;y,t) \in \R^4$ and $\dir \in \R$, $\W_{\alpha }(x,s;y,t)$ converges pointwise to $\W_{\dir -}(x,s;y,t)$ as $\Q \ni \alpha \nearrow \dir$. The same holds for limits from the right, with $\dir -$ replaced by $\dir +$ (Later we prove that the convergence is locally uniform).   By~\eqref{eqn:gen_Buse_var}, it suffices to assume $s < t$. By~\eqref{eqn:mont_maxes} and the additivity of Lemma~\ref{lem:DL_horiz_Buse}\ref{itm:DL_h_add} when $s = t$, for all $\alpha \in [\dir - 1,\dir + 1] \cap \Q$ and $\sigg \in \{-,+\}$,
\begin{align*}
\W_{\alpha }(x,s;y,t) &= \sup_{z \in \R}\{\Ll(x,s;z,t) + \W_{\alpha }(z,t;y,t)\} \\
&= \sup_{z \in \R}\{\Ll(x,s;z,t) + \W_{\alpha}(z,t;0,t)\} + \W_{\alpha}(0,t;y,t) \\
&= \sup_{z \in [g_{(x,s)}^{(\dir - 1)-,L}(t),g_{(x,s)}^{(\dir + 1)+,R}(t)]}\{\Ll(x,s;z,t) + \W_\alpha(z,t;0,t)\} + \W_{\alpha}(0,t;y,t).
\end{align*}
By Lemma~\ref{lem:DL_horiz_Buse}\ref{itm:DL_h_unif_conv}, $\W_{\alpha}(z,t;y,t)$ converges uniformly on compact sets to $\W_{\dir -}(x,t;y,t)$ as $\Q \ni \alpha \nearrow \dir$ and to $\W_{\dir +}(x,t;y,t)$ as $\Q \ni \alpha \searrow \dir$. This implies the desired pointwise convergence. The additivity follows from the additivity  for rational $\dir$  (Theorem~\ref{thm:RV-Buse}\ref{itm:fixed_additive}). 

\smallskip\noindent \textbf{Item~\ref{itm:DL_Buse_gen_mont} (Monotonicity along a horizontal line):} This was previously proven as Equation~\eqref{eqn:Buse_mont}.

\smallskip\noindent \textbf{Item~\ref{itm:Buse_KPZ_description} (Backwards evolution as the KPZ fixed point):} This follows directly from the construction~\eqref{eqn:gen_Buse_var}.

\smallskip\noindent  We postpone the proofs of Items~\ref{itm:DL_unif_Buse_stick}--\ref{itm:global_attract}. Item~\ref{itm:DL_unif_Buse_stick} is proved after the proof of Theorem~\ref{thm:Buse_dist_intro}, and Items~\ref{itm:BuseLim2}--\ref{itm:global_attract} are proved after the proof of Theorem~\ref{thm:DL_good_dir_classification}. No subsequent results  depend on Items~\ref{itm:BuseLim2}--\ref{itm:global_attract}, except  the mixing in Theorem~\ref{thm:Buse_dist_intro}\ref{itm:stationarity}, which is proven later.
\end{proof}

\begin{proof}[Proof of Theorem~\ref{thm:Buse_dist_intro} (Distributional properties of Busemann process)]
\smallskip\noindent \textbf{Item~\ref{itm:indep_of_landscape} (Independence):} We know that $\{\Ll(x,s;y,t):s,y \in \R, s < t \le T\}$ is independent of $\{\Ll(x,s;y,t):s,y \in \R, T \le s < t\}$ for $T \in \R$. From the definition of the Busemann process from geodesics and the extension \eqref{eqn:Buse_def}--\eqref{eqn:gen_Buse_var}, the process
\[
\{\W_{\dir \sig}(x,s;y,t): \dir \in \R, \,\sigg \in \{-,+\}, \, x,y \in \R, \, s,t \ge T \} 
\]
is a function of 
$\{\Ll(x,s;y,t):s,y \in \R, T \le s < t\}$, and independence follows.

\smallskip\noindent \textbf{Item~\ref{itm:stationarity} (Stationarity):}
Similarly as the previous item, the stationarity of the process follows from the stationarity of the directed landscape from Lemma~\ref{lm:landscape_symm}\ref{itm:time_stat}. The mixing properties  will be proven in Section~\ref{sec:Buseextraproofs}, along with Items~\ref{itm:BuseLim2}--\ref{itm:global_attract} of Theorem~\ref{thm:DL_Buse_summ}.

\smallskip\noindent \textbf{Item~\ref{itm:SH_Buse_process} (Distribution along a time level):} By the additivity of Theorem~\ref{thm:DL_Buse_summ}\ref{itm:DL_Buse_add} and the variational definition~\eqref{eqn:gen_Buse_var}, for $x \in \R$, $s < t$, and $\sigg \in \{-,+\}$, on the full-probability event $\Omega_1$,
\begin{align*}
&\W_{\dir \sig}(x,s;0,s) = \W_{\dir \sig}(x,s;0,t) - \W_{\dir \sig}(0,s;0,t) \\
&\qquad 
=\sup_{y \in \R}\{\Ll(x,s;y,t) + \W_{\dir \sig}(y,t;0,t)\} - \sup_{y \in \R} \{\Ll(0,s;y,t) + \W_{\dir \sig}(y,t;0,t)\}.
\end{align*}
By Item~\ref{itm:indep_of_landscape}, Theorem~\ref{thm:DL_Buse_summ}\ref{itm:DL_Buse_gen_mont},  and  Items~\ref{itm:DL_h_unif_conv} and~\ref{itm:DL_lim} of Lemma~\ref{lem:DL_horiz_Buse}, $\{\W_{\dir +}(\abullet,t;0,t):\dir \in \R\}_{t \in \R}$ is a reverse-time Markov process that almost surely lies in the state space $\Y$ defined in~\eqref{Y}. By the stationarity of Item~\ref{itm:stationarity}, the law of $\{\W_{\dir +}(\abullet,t;0,t):\dir \in \R\}$ must be invariant for this process.  By the temporal reflection invariance of the directed landscape  (Lemma~\ref{lm:landscape_symm}\ref{itm:DL_reflect}),  $\{\W_{\dir +}(\abullet,t;0,t):\dir \in \R\}$ is also invariant for the KPZ fixed point, forward in time. The uniqueness part of Theorem~\ref{thm:invariance_of_SH} completes the proof. 
\end{proof}

\begin{lemma} \label{lem:L_and_Buse_ineq}
For every $\omega \in \Omega_1$ and $(x,s;y,t) \in \Rup$, $\Ll(x,s;y,t) \le \W_{\dir \sig}(x,s;y,t)$, and equality occurs if and only if $y$ maximizes $\Ll(x,s;z,t) + \W_{\dir \sig}(z,t;0,t)$ over $z \in \R$. 
\end{lemma}
\begin{proof}
For $s < t$, Theorem~\ref{thm:DL_Buse_summ}\ref{itm:DL_Buse_add},\ref{itm:Buse_KPZ_description} gives
\begin{align} \label{106}
\W_{\dir \sig}(x,s;y,t) &= \sup_{z \in \R}\{\Ll(x,s;z,t) + \W_{\dir \sig}(z,t;y,t)\} \nonumber
\\ &= \sup_{z \in \R}\{\Ll(x,s;z,t) + \W_{\dir \sig}(z,t;0,t)\} + \W_{\dir \sig}(0,t;y,t).
\end{align}
Setting $z = y$ on the right-hand side of~\eqref{106}, it follows that $\W_{\dir \sig}(x,s;y,t) \ge \Ll(x,s;y,t)$, and equality holds if and only if $y$ is a maximizer.
\end{proof}

\begin{proof}[Proof of Theorem~\ref{thm:DL_SIG_cons_intro} (Construction of the Busemann geodesics)]
The full-probability event of this theorem is $\Omega_1$~\eqref{omega1}.

\smallskip\noindent
\textbf{Item~\ref{itm:intro_SIG_bd} (Finiteness of the maximizers):} This follows immediately from Lemma~\ref{lem:bounded_maxes}.

\smallskip\noindent 
 We prove \textbf{Items~\ref{itm:arb_geod_cons}--\ref{itm:DL_LRmost_geod}} together. By Lemma~\ref{lem:L_and_Buse_ineq}, for any such construction of a path from the sequence of times $s = t_0 < t_1 < \cdots$ and any $i \ge 1$,
\[
\Ll(g(t_{i - 1}),t_{i - 1};g(t_i),t_i) = \W_{\dir \sig}(g(t_{i - 1}),t_{i - 1};g(t_i),t_i).
\]
Furthermore, for any $t_{i - 1} \le t < u \le t_i$, it must hold that  
\[
\Ll(g(t),t;g(u),u) = \W_{\dir \sig}(g(t),t;g(u),u),
\]
for otherwise, by additivity of the Busemann functions (Theorem~\ref{thm:DL_Buse_summ}\ref{itm:DL_Buse_add}),
\begin{align*}
&\quad \Ll(g(t_{i - 1}),t_{i - 1};g(t_i),t_i)\\
&= \Ll(g(t_{i - 1}),t_{i - 1};g(t),t) +\Ll(g(t),t;g(u),u) + \Ll(g(u),u;g(t_i),t_i)\\
&<  \W_{\dir \sig}(g(t_{i - 1}),t_{i - 1};g(t),t) +\W_{\dir \sig}(g(t),t;g(u),u) + \W_{\dir \sig}(g(u),u;g(t_i),t_i) \\
&= \W_{\dir \sig}(g(t_{i - 1}),t_{i - 1};g(t_i),t_i),
\end{align*}
a contradiction. Additivity extends~\eqref{eqn:SIG_weight} to all $s \le t < u$. Therefore, the path is a semi-infinite geodesic because the weight of the path in between any two points is optimal by Lemma~\ref{lem:L_and_Buse_ineq}. From the equality~\eqref{eqn:SIG_weight} and Lemma~\ref{lem:L_and_Buse_ineq}, for \textit{every} $t \ge s$, $g(t)$ maximizes $\Ll(x,s;z,t) + \W_{\dir \sig}(z,t;0,t)$ over $z \in \R$.

\begin{figure}
    \centering
    \includegraphics[height = 2in]{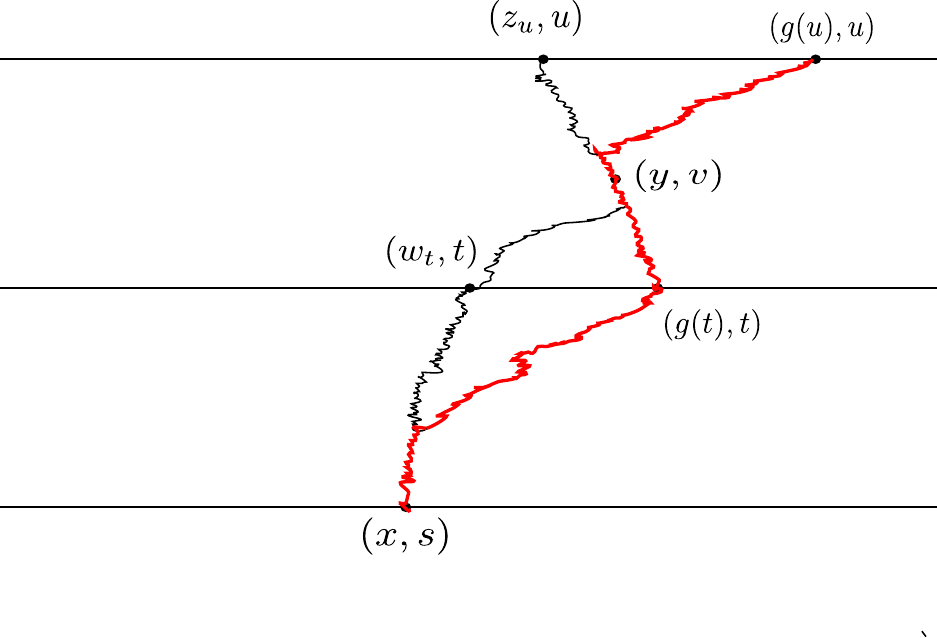}
    \caption{\small Illustration of the proof of Lemma~\ref{lem:RM_geod_SIG}. Here, the red/thick path denotes the path $\hat \gamma$ in the case $w_t < g(t)$, which is to the right of the rightmost geodesic between $(x,s)$ and $(g(u),u)$, which passes through $(w_t,t)$ by assumption. This gives the contradiction. }
    \label{fig:RM_geod}
\end{figure}

Before global directedness of all  geodesics, we show that $g_{(x,s)}^{\dir \sig,S}$ are semi-infinite geodesics and the leftmost/rightmost geodesics between any two of their points. Take $S = R$, and the result for $S = L$ follows similarly. Omit $x,s,\dir$, and $\sigg$ from the notation temporarily, and write $g(t) = g_{(x,s)}^{\dir \sig,R}(t)$. 
By what was just proved, it is sufficient to prove the following lemma.
\begin{lemma} \label{lem:RM_geod_SIG}
Let $g$ be as defined above. For $s < t < u$, let $z_u$ be the rightmost maximizer of $\Ll(g(t),t;z,u) + \W_{\dir \sig}(z,u;0,u)$ over $z \in \R$, and let $w_t$ be the rightmost maximizer of $\Ll(x,s;w,t) + \Ll(w,t;g(u),u)$ over $w \in \R$ {\rm(}Equivalently, the proof of \cite[Lemma 13.2]{Directed_Landscape} shows that $(w_t,t)$ is the point at level $t$ on the rightmost geodesic between $(x,s)$ and $(g(u),u)${\rm)}. Then, $g(t) = w_t$ and $g(u) = z_u$.
\end{lemma}
\begin{proof}
 By Lemma~\ref{lem:L_and_Buse_ineq} and Items~\ref{itm:arb_geod_cons}\ref{itm:weight_of_geod}--\ref{itm:maxes}, $w_t$ maximizes $\Ll(x,s;z,t)+ \W_{\dir \sig}(z,t;0,t)$ over $z \in \R$, and $z_u$ maximizes $\Ll(x,s;z,u) + \W_{\dir \sig}(z,u;0,u)$ over $z \in \R$. By definition of $g(u)$ and $g(t)$ as the rightmost maximizers, we have $w_t \le g(t)$ and $z_u \le g(u)$ in general. Assume, to the contrary, that $g(t) \neq w_t$ or $g(u) \neq z_u$. We first prove a contradiction in the case $w_t < g(t)$. For the proof, refer to Figure~\ref{fig:RM_geod} for clarity. Let $\gamma_1:[s,u]\to \R$ be the rightmost geodesic from $(x,s)$ to $(g(u),u)$ (which passes through $(w_t,t)$), and let $\gamma_2$ be the concatenation of the rightmost geodesic from $(x,s)$ to $(g(t),t)$ followed by the rightmost geodesic from $(g(t),t)$ to $(z_u,u)$.  By Item~\ref{itm:arb_geod_cons}\ref{itm:weight_of_geod} for $i = 1,2$, the weight of the portion of any part of $\gamma_i$ is equal to the Busemann function between the points. Since $w_t < g(t)$ and $z_u \le g(u)$, $\gamma_1$ and $\gamma_2$ must split before time $t$, and then meet again before or at time $u$. Let $(y,v)$ be a crossing point, where $t < v \le u$. Let $\hat \gamma: [s,u] \to \R$ be defined by $\hat \gamma(r) = \gamma_2(r)$ for $r \in [s,v]$ and $\hat \gamma(r) = \gamma_1(r)$ from $(y,v)$ to $(g(u),u)$. Then, by the additivity of Busemann functions, the weight $\Ll$ of any portion of the path $\hat \gamma$ is equal to the Busemann function between the two points. By Lemma~\ref{lem:L_and_Buse_ineq}, $\wh \gamma$ is then a geodesic between $(x,s)$ and $(g(u),u)$, which is to the right of $\gamma_1$, which was defined to be the rightmost geodesic between the points, a contradiction.

Now, we consider the case $z_u < g(u)$. Define $\gamma_1$ and $\gamma_2$ as in the previous case. Since $z_u < g(u)$, there is some point $(y,v)$ with $t \le v < u$ such that $\gamma_1$ splits from or crosses  $\gamma_2$ at $(y,v)$. Then, define $\hat \gamma$ as in the previous case. Again, the weight   $\Ll$ of any portion of the path $\hat \gamma$ is equal to the Busemann function between the two points. Specifically, $\Ll(g(t),t;g(u),u) = \W_{\dir \sig}(g(t),t;g(u),u)$, and by Item~\ref{lem:L_and_Buse_ineq}, $g(u)$ maximizes $\Ll(g(t),t;z,u) + \W_{\dir \sig}(z,u;0,u)$ over $z \in \R$. This contradicts the definition of $z_u$ as the rightmost such maximizer.
\end{proof}

Returning to the proof of Theorem~\ref{thm:DL_SIG_cons_intro}, we show the global directedness of all Busemann geodesics constructed in the manner described in Item~\ref{itm:arb_geod_cons}. By~\eqref{eqn:mont_maxes}, for $t \ge s$ and $\alpha < \dir < \beta$ with $\alpha,\beta \in \Q$,
\be \label{sig_sand}
g_{(x,s)}^{\alpha,L}(t) \le g_{(x,s)}^{\dir \sig,L}(t) \le g(t) \le g_{(x,s)}^{\dir \sig,R}(t) \le g_{(x,s)}^{\beta,R}(t).
\ee
Note that on $\Omega_1$ the $\pm$ distinction is absent for $\alpha,\beta \in \Q$ (Lemma~\ref{lem:DL_horiz_Buse}\ref{itm:DL_agree_horiz}).
By definition~\eqref{omega1} of the event $\Omega_1$ and Theorem~\ref{thm:RV-Buse}\ref{itm:DL_Buse_var}, $\forall\alpha\in\Q$, the maximizers of $\Ll(x,s;z,t) + W_{\alpha}(z,t;0,t)$ over $z \in \R$ are exactly the locations $z$ where an $\alpha$-directed geodesic goes through $(z,t)$. Therefore, $g_{(x,s)}^{\alpha,L}(t)/t \to \alpha$ and $g_{(x,s)}^{\beta,R}(t)/t \to \beta$ when $\alpha, \beta\in\Q$. By~\eqref{sig_sand},
\[
\alpha  \le \liminf_{t \to \infty} t^{-1}{g(t)} \le \limsup_{t \to \infty} t^{-1}{g(t)} \le \beta.
\]
Sending $\Q \ni \alpha \nearrow \dir$ and $\Q \ni \beta \searrow \dir$ completes the proof of   Theorem~\ref{thm:DL_SIG_cons_intro}.
\end{proof}

We now define the next   full-probability event. 
\be \label{omega2}
\text{Let $\Omega_2$ be the subset of $\Omega_1$ on which the following hold.} 
\ee
\begin{enumerate} [label=\rm(\roman{*}), ref=\rm(\roman{*})] \itemsep=3pt
    \item \label{itm:stickT} For each integer $T \in \Z$ and each compact set $K \subseteq \R^2$, there exists $\ve =\ve(\dir,T,K) > 0$ such that for $\dir - \ve < \alpha < \dir < \beta < \dir + \ve$ and $(x,y) \in K$,
    \be \label{465}
\W_{\alpha \sig}(y,T;x,T) = \W_{\dir -}(y,T;x,T)\ \ \text{ and }\ \  \W_{\beta \sig}(y,T;x,T) = \W_{\dir +}(y,T;x,T).
\ee
\item For each integer $T \in \Z$, the set
\be \label{infinit_dense}
\{\dir \in \R: \W_{\dir -}(x,T;0,T) \neq \W_{\dir +}(x,T;0,T) \text{ for some }x \in \R\}
\ee
is countably infinite and dense in $\R$.
\item For each $s < t \in \R$, $x,\dir \in \R$, $\sigg \in \{-,+\}$, and $S \in \{L,R\}$,
\be \label{eqn:dirtoinf}
    \lim_{\xi \to \pm \infty} g_{(x,s)}^{\xi \sig,S}(t) = \pm \infty.
\ee
\end{enumerate}

\begin{lemma}
$\Pp(\Omega_2) = 1.$
\end{lemma}
\begin{proof}
The fact that~\ref{itm:stickT} holds with probability one is a direct consequence of  Theorems~\ref{thm:Buse_dist_intro}\ref{itm:SH_Buse_process} and~\ref{thm:SH10}\ref{itm:SH_j}. The set~\eqref{infinit_dense} is countably infinite and dense for all $T \in \Z$ by  the distributional equality
$ 
\{W_{\dir+}(\abullet,T;0,T)\}_{\dir \in \R} \deq \{G_\dir\}_{\dir \in \R}
$ 
from Theorem~\ref{thm:Buse_dist_intro}\ref{itm:SH_Buse_process} and the properties of $G$ from Theorem~\ref{thm:SH10}\ref{itm:SH_j},\ref{itm:bad_dir_contained}.

Now, we prove that~\eqref{eqn:dirtoinf} holds with probability one. By the monotonicity of~\eqref{eqn:mont_maxes}, the limits $\lim_{\xi \to \infty} g_{(x,s)}^{\xi \sig,S}(t)$ and $\lim_{\xi \to -\infty} g_{(x,s)}^{\xi \sig,S}(t)$ exist in $\R \cup \{-\infty,\infty\}$. Furthermore, by this monotonicity, it is sufficient to show that 
\be \label{678}
\lim_{\xi \to \infty} g_{(x,s)}^{\xi -,L}(t) = \sup_{\dir \in \R}g_{(x,s)}^{\xi -,L}(t) = \infty\ \ \text{ and }\ \  \lim_{\dir \to -\infty} g_{(x,s)}^{\dir +,R}(t) = \inf_{\dir \in \R} g_{(x,s)}^{\dir +,R}(t) = -\infty.
\ee
First, we show that~\eqref{678} holds with probability one for a fixed initial point $(x,s)$ and fixed $t > s$. It is therefore sufficient to take $(x,s) = (0,0)$ and then $t > 0$. By the monotonicity, it suffices to take limits over $\dir \in \Q$ so that by Theorem~\ref{thm:RV-SIG-thm}\ref{itm:pd_fixed}, the $\pm$ and $L/R$ distinctions are unnecessary.   $\W_{\xi \sig}(z,t;0,t)$ is a two-sided Brownian motion with drift $2\xi$ and diffusivity $\sqrt 2$, independent of the random function $(x,y) \mapsto \Ll(x,0;y,t)$ (Theorem~\ref{thm:Buse_dist_intro}\ref{itm:indep_of_landscape}). Let $B$ be a standard Brownian motion, independent of $\Ll$. Using skew stationarity with $c = -\dir$ in the third equality below  and time stationarity in the fifth equality (Lemma~\ref{lm:landscape_symm}), we obtain, for $\dir \in \Q$,
\begin{align*}
   g_{(x,s)}^\dir(t) &= \argmax_{z \in \R} \{\Ll(x,s;z,t)  + \W_{\xi}(z,t;0,t) \} \\
   &\deq \argmax_{z \in \R} \{\Ll(x,s;z,t) + \sqrt 2 B(z) + 2\xi z\} \\
   &\deq \argmax_{z \in \R} \{\Ll(x - \dir s,s;z - \dir t,t) + 2\dir(x -z) + (t - s)\dir^2  + \sqrt 2 B(z) +2\dir z\}\\
   &= \argmax_{z \in \R} \{\Ll(x - \dir s,s;z - \dir t,t) + \sqrt 2(B(z) - B(\dir (t-s))) \} \\
    &\deq  \argmax_{z \in \R} \{\Ll(x,s;z - \dir (t-s),t) + \sqrt 2 B(z - \dir (t - s)) \} \\
    &= \argmax_{z \in \R}\{\Ll(x,s;z,t) + \sqrt 2 B(z) \} + \dir (t-s)  \deq  g_{(x,s)}^0(t) + \dir(t-s).
\end{align*}
Therefore, $\forall\dir \in \Q$, the distribution of $g_{(x,s)}^{\dir}(t)$ is that of a fixed, almost surely finite, random variable plus $\dir (t-s)$. Since we know  $\lim_{\Q \ni \dir \to \pm \infty} g_{(x,s)}^{\dir}(t)$ exists, the limit must be $\pm \infty$ a.s.

Now, consider the intersection of $\Omega_1$ with event of probability one on which for each triple $(w,q_1,q_2) \in \Q^3$ with $q_1 < q_2$, 
\be \label{692}
\lim_{\dir \to +\infty} g_{(w,q_1)}^{\dir-,L}(q_2) = + \infty\qquad\text{and}\qquad \lim_{\dir \to -\infty} g_{(w,q_1)}^{\dir+,R}(q_2) = - \infty.
\ee
On this event, let $(x,s,t) \in \R^3$ with $s < t$ be arbitrary. Assume, by way of contradiction, that 
\be \label{792}
z := \sup_{\dir \in \R} g_{(x,s)}^{\dir-,L}(t) < \infty,
\ee
and let $g:[s,t]$ denote the leftmost geodesic from $(x,s)$ to $(z,t)$. For this proof, refer to Figure~\ref{fig:fanning_proof} for clarity. By the assumption~\eqref{792} and the fact that $g_{(x,s)}^{\dir-,L}$ is the leftmost geodesic between any two of its points (Theorem~\ref{thm:DL_SIG_cons_intro}\ref{itm:DL_LRmost_geod}),  $g_{(x,s)}^{\dir-,L}(t) \le g(t)$ for all $\dir \in \R$ and $t > s$.
Let $q_1 \in (s,t)$ be rational. Choose  $w \in \Q$ such that $w < g(q_1)$. By continuity of geodesics, we may choose $q_2 \in (q_1,t) \cap \Q$ to be sufficiently close to $t$ so that $|g(q_2) - z| < 1$. Next, by~\eqref{692}, we may choose positive $\dir$ sufficiently large so that 
\be \label{892}
g_{(w,q_1)}^{\dir-,L}(q_2) > z + 1 > g(q_2) \ge g_{(x,s)}^{\dir -,L}(q_2).
\ee
Since $w < g(q_1)$,  $g_{(w,q_1)}^{\dir-,L}$ and $g_{(x,s)}^{\dir -,L}$ cross at some  $(\hat z,\hat t)$ with $\hat t \in (q_1,q_2)$.  By Theorem~\ref{thm:DL_SIG_cons_intro}\ref{itm:DL_all_SIG}, both $g_{(w,q_1)}^{\dir-,L}(q_2)$ and $g_{(x,s)}^{\dir -,L}(q_2)$ are the leftmost maximizer of $\Ll(\hat z,\hat t; y,q_2) + \W_{\dir -}(y,q_2;0,q_2)$ over $y \in \R$. This contradicts~\eqref{892}. The proof for   $\dir \to -\infty$ is analogous.
\end{proof}
\begin{figure}[t]
    \centering
    \includegraphics[height = 2 in]{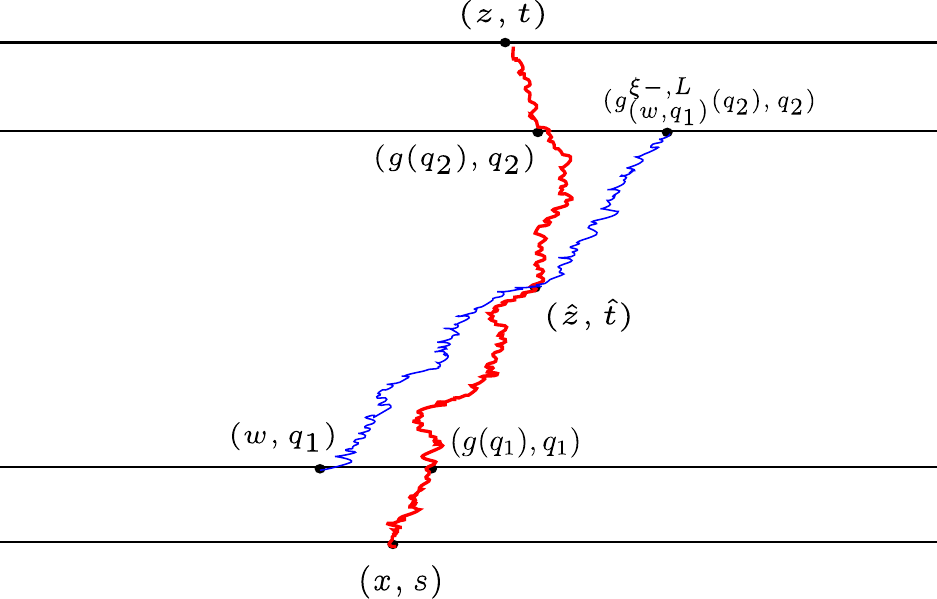}
    \caption{\small The blue/thin path represents $g_{(w,q_1)}^{\dir -,L}$ and the red/thick path represents $g$.}
    \label{fig:fanning_proof}
\end{figure}

\begin{proof}[Proof of Theorem~\ref{thm:DL_Buse_summ}\ref{itm:DL_unif_Buse_stick} (Regularity of the Busemann process)]

 By definition of the event $\Omega_2$~\eqref{omega2}, for each $\dir \in \R$, each integer $T$ and compact set $K\subseteq \R^2$, there is a $\ve > 0$ so that~\eqref{465} holds for all $(x,y) \in K$.

Now, let $\dir \in \R$, let $K$ be a compact subset of $\R^4$, and let $T$ be an integer greater than $\sup\{t \vee s: (x,s;y,t) \in K\}$. Let 
\begin{align*}
A &:= \inf\{g_{(x,s)}^{(\dir - 1)-,L}(T)\wedge g_{(y,t)}^{(\dir - 1)-,L}(T) : (x,s;y,t) \in K\}, 
\qquad\text{and} \\
B &:= \sup\{g_{(x,s)}^{(\dir + 1)+,R}(T)\vee g_{(y,t)}^{(\dir + 1)+,R}(T) : (x,s;y,t) \in K\}.
\end{align*}
By~\eqref{eqn:mont_maxes} and Lemma~\ref{lem:bounded_maxes}, $-\infty < A < B < \infty$.
By~\eqref{eqn:mont_maxes} and the additivity of Theorem~\ref{thm:DL_Buse_summ}\ref{itm:DL_Buse_add}, for all $(x,s;y,t) \in K$ and $\alpha \in (\dir - 1,\dir + 1)$, 
\be \label{Tdiff}
\begin{aligned}
&\W_{\alpha \sig}(x,s;y,t) = \W_{\alpha \sig}(x,s;0,T) - \W_{\alpha \sig}(y,t;0,T) \\
&
= \sup_{z \in \R}\{\Ll(x,s;z,T) + \W_{\alpha\sig}(z,T;0,T)\} - \sup_{z \in \R}\{\Ll(y,t;z,T) + \W_{\alpha \sig}(z,T;0,T)\} \\
&= \sup_{z \in [A,B]}\{\Ll(x,s;z,T) + \W_{\alpha \sig}(z,T;0,T)\} - \sup_{z \in [A,B]}\{\Ll(y,t;z,T) + \W_{\alpha \sig}(z,T;0,T)\}.
\end{aligned}
\ee
By~\eqref{465}, the conclusion follows. 
\end{proof}

\begin{proof}[Proof of Theorem~\ref{thm:DLBusedc_description} (Description of the discontinuity set)]
The full-probability event of this theorem is $\Omega_2$, except for Item~\ref{itm:Busedc_t} whose proof is postponed until Section~\ref{sec:last_proofs}. Proofs of  results that rely on Item~\ref{itm:Busedc_t} come afterwards.

\smallskip\noindent \textbf{Item~\ref{itm:Busedc_horiz_mont} (Monotonicity):}
 By the monotonicity of Theorem~\ref{thm:DL_Buse_summ}\ref{itm:DL_Buse_mont}, and by Lemma~\ref{lem:ext_mont}, for $a \le x \le y \le b$,
 \be \label{801}
 0 \le \W_{\dir +}(y,t;x,t) - \W_{\dir -}(y,t;x,t) \le \W_{\dir +}(b,t;a,t) - \W_{\dir -}(b,t;a,t).
 \ee
Thus, discontinuities of $\dir \mapsto \W_{\dir \sig}(y,t;x,t)$ are also discontinuities for $\dir \mapsto \W_{\dir \sig}(b,t;a,t)$. 

\smallskip\noindent \textbf{Item~\ref{itm:DL_dc_set_count} ($\DLBusedc$ is a countable dense set):}
Similarly as in~\eqref{Tdiff}, if $(x,s;y,t) \in \R^4$, then for $\dir \in \R$, $\sigg \in \{-,+\}$, and any integer $T > s\vee t$,
\be \label{880}\begin{aligned} 
\W_{\dir \sig}(x,s;y,t) &= \sup_{z \in \R}\{\Ll(x,s;z,T) + \W_{\dir \sig}(z,T;0,T)\}\\
&\qquad - \sup_{z \in \R} \{\Ll(y,t;z,T) + \W_{\dir \sig}(z,T;0,T)\}.
\end{aligned} \ee
So if $\W_{\dir -}(z,T;0,T) = \W_{\dir+}(z,T;0,T) \; \forall z \in \R$, then  $\W_{\dir -}(x,s;y,t) = \W_{\dir +}(x,s;y,t)$, and 
\be \label{881}
\DLBusedc = \bigcup_{T \in \Z} \{\dir \in \R: \W_{\dir - }(x,T;0,T) \neq \W_{\dir +}(x,T;0,T) \text{ for some }x \in \R\}. 
\ee
On $\Omega_2$,  $\DLBusedc$   is countably infinite and dense  by  \eqref{omega2}.  Lemma~\ref{lem:DL_horiz_Buse}\ref{itm:DL_agree_horiz}, along with~\eqref{881} imply that $\DLBusedc$ contains no rational directions $\dir$. For an arbitrary $\dir \in \R$, $\W_{\dir -}(\abullet,T;0,T)$ and $\W_{\dir +}(\abullet,T;0,T)$ are both Brownian motions with the same diffusivity and drift, and $\W_{\dir -}(y,T;x,T) \le \W_{\dir +}(y,T;x,T)$ for $x < y$ by Theorem~\ref{thm:DL_Buse_summ}\ref{itm:DL_Buse_gen_mont}. By~\eqref{881} and continuity,
\[
\Pp(\dir \in \DLBusedc) \le \sum_{T \in \Z, x \in \Q} \Pp(\W_{\dir - }(x,T;0,T) \neq \W_{\dir +}(x,T;0,T)) = 0,
\]
where $\Pp(\W_{\dir - }(x,T;0,T) \neq \W_{\dir +}(x,T;0,T)) = 0$ because the two random variables have the same law and are ordered.

\noindent \textbf{Item~\ref{itm:DL_Buse_no_limit_pts} ($\DLBusedc(p;q)$ is discrete):} The discreteness is a direct consequence of the regularity of 
 the Busemann process from Theorem~\ref{thm:DL_Buse_summ}\ref{itm:DL_unif_Buse_stick}.  The discreteness is a direct consequence of the regularity of 
 the Busemann process from Theorem~\ref{thm:DL_Buse_summ}\ref{itm:DL_unif_Buse_stick}. By Theorem \ref{thm:SH10}\ref{itm:SH_mont}, on a $t$-dependent full probability event, for each $x < y$, $\W_{\dir \sig}(y,t;x,t) \to \pm \infty$ as $\dir \to \pm \infty$.  Since the jumps are discrete, $\DLBusedc(y,t;x,t)$ is infinite and unbounded for both positive and negative $\dir$.

 \noindent \textbf{Item~\ref{itm:DLBusedcinvar} (Distributional invariances of $\DLBusedc$:)} The discreteness of Item~\ref{itm:DL_Buse_no_limit_pts} allows us to view the sets $\DLBusedc(y,t;x,t)$ as well-defined point processes.
 We recall that $\dir \in \DLBusedc$ if and only if $\W_{\dir -}(y,t;x,t) \neq \W_{\dir +}(y,t;x,t)$. Start with the distributional equality $\{W_{\dir +}(\abullet,t;0,t)\}_{\dir \in \R} \deq \{G_\dir\}_{\dir \in \R}$, which holds for all $t$ (Theorem~\ref{thm:Buse_dist_intro}\ref{itm:SH_Buse_process}). Furthermore, the additivity of the Busemann process (Theorem~\ref{thm:DL_Buse_summ}\ref{itm:DL_Buse_add}) implies 
\[
\{\W_{\dir +}(y,t;x,t):x,y \in \R\}_{\dir \in \R} \deq \{G_{\dir}(y) - G_{\dir}(x):x,y \in \R\}_{\dir \in \R}.
\]
This gives the first distributional equality $\DLBusedc(y,t;x,t) \deq \DLBusedc(y,0;x,0)$. The invariance $\DLBusedc(y,0;x,0) \deq -\DLBusedc(-y,0;-x,0)$ follows from the reflection invariance of $G$ (Corollary~\ref{cor:SH_reflect}). The invariance \\$\DLBusedc(y,0;x,0) \deq c^{-1} \DLBusedc(c^{-2}y,0;c^{-2}x,0) -\nu$ follows from the corresponding invariance for $G$ in Theorem~\ref{thm:SH10}\ref{itm:SH_sc}.
\end{proof}

\section{Non-uniqueness of semi-infinite geodesics} \label{sec:LR_sig} 
Theorem~\ref{thm:DL_SIG_cons_intro} established global existence of semi-infinite geodesics from each initial point and into each direction. 
We know from Theorem 3.3 of~\cite{Rahman-Virag-21}, recorded earlier in  Theorem~\ref{thm:RV-SIG-thm}\ref{itm:pd_fixed},  that for a fixed initial point and a fixed direction, there almost surely is a unique semi-infinite geodesic.  However, this uniqueness does not extend globally to all initial points and directions simultaneously.
 In fact, two qualitatively different types of non-uniqueness of    Busemann  geodesics from a given point into a given direction arise.  One is denoted by the $L/R$ distinction and the other by the  $\pm$ distinction. 
 All semi-infinite geodesics from $p$ in direction $\dir$ lie between the leftmost Busemann geodesic $g_{p}^{\dir -,L}$ and the rightmost Busemann geodesic   $g_p^{\dir +,R}$.  See Theorem~\ref{thm:all_SIG_thm_intro}\ref{itm:DL_LRmost_SIG}. We refer the reader back to Figure~\ref{fig:non_unique_comp} for the two types of non-uniqueness. The $L/R$ uniqueness is depicted on the left, where geodesics split and return to coalesce, while the $\pm$ non-uniqueness is depicted on the right in the figure, where geodesics split and stay apart, all the way to $\infty$.

The $L/R$ non-uniqueness is a feature of continuous space.  Only the $\pm$ non-uniqueness appears in the discrete corner growth model with exponential weights, while both $L/R$ and $\pm$  non-uniqueness are   present  in semi-discrete  BLPP~\cite{Seppalainen-Sorensen-21a,Seppalainen-Sorensen-21b}.  

To capture $L/R$ non-uniqueness, we  introduce  the following random sets of initial points.   For $\dir \in \R$ and $\sigg \in \{-,+\}$, let $\NU_0^{\dir \sig}$ be the set of points $p \in \R^2$ such that the $\dir \sigg$ geodesic from $p$ is not unique. Let $\NU_1^{\dir \sig}$ be the subset of $\NU_0^{\dir \sig}$ of those initial points at which   two $\dir \sig$ geodesics separate immediately. In notational terms, 
\begin{align}
\NU_0^{\dir \sig} &= \{(x,s) \in \R^2: g_{(x,s)}^{\dir \sig,L}(t) < g_{(x,s)}^{\dir \sig,R}(t) \text{ for some } t > s\}, \label{NU0}\qquad\text{and}\\
\label{NU1}
\NU_1^{\dir \sig} &= \{(x,s) \in \NU_0^{\dir \sig}: \exists\ve > 0 \text{ such that } 
g_{(x,s)}^{\dir \sig,L}(t) < g_{(x,s)}^{\dir \sig,R}(t) \  \forall t \in (s,s+\ve)\}. 
\end{align}
For $i = 0,1$, let 
\be \label{NU0_global}
\NU_i = \textstyle\bigcup_{\dir \tspb\in\tspb  \R,\,\sig \tspb\in\tspb \{-,+\}} \NU_i^{\dir \sig}.
\ee
Figure~\ref{fig:NU} illustrates   $\NU_0$ and $\NU_1$. 

\begin{figure}[t]
    \centering
    \includegraphics[height = 2in]{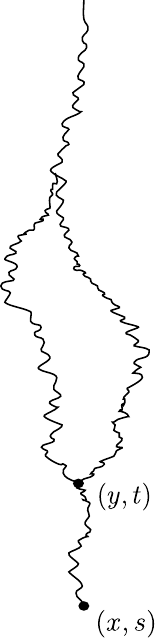}
    \caption{\small In this figure, $(x,s) \in \NU_0 \setminus \NU_1$ and $(y,t) \in \NU_1 \subseteq \NU_0$. It has since been shown by Bhatia~\cite{Bhatia-23} and Dauvergne~\cite{Dauvergne-23} that no such points $(x,s)$ exist. }
    \label{fig:NU}
\end{figure}

Theorem~\ref{thm:DLNU}\ref{itm:DL_NU_count} establishes that,   with probability one, for each $\dir \in \R$ and $\sigg \in \{-,+\}$, the restriction of $\NU_0^{\dir\sig}$ to each  time level $s$ is countably infinite. By Theorem~\ref{thm:DL_all_coal}\ref{itm:DL_allsigns_coal},  on a single event of probability one,  for each direction $\dir$ and sign $\sigg \in \{-,+\}$, all $\dir \sig$ geodesics coalesce. Therefore, from each $p \in \NU_0^{\dir \sig}$, two $\dir \sig$ geodesics separate but eventually   come back together. In particular,   the set of points $(x,s) \in \R^2$ such that $g_{(x,s)}^{\dir \sig,L}(t) < g_{(x,s)}^{\dir \sig,R}(t)$ for all $t\in(s,\infty)$ is empty and  the   $\ve > 0$ in the definition \eqref{NU1} of $\NU_1^{\dir\sig}$ is essential.

By definition  $\NU_1^{\dir \sig} \subseteq \NU_0^{\dir \sig}$. When this paper was first posted, we did not know whether $\NU_1^{\dir \sig}$ is a strict subset of  $\NU_0^{\dir \sig}$. Afterward, Bhatia~\cite{Bhatia-23} and Dauvergne \cite{Dauvergne-23} each independently proved that, in fact, $\NU_0^{\dir \sig} = \NU_1^{\dir \sig}$. In fact, something stronger is true: With probability one, there are no pairs of points $(x,s;y,t) \in \Rup$ and pairs of distinct geodesics $g_1,g_2$ from $(x,s)$ to $(y,t)$ satisfying, for some $\ve > 0$, $g_1(u) = g_2(u)$ for all $u \in (s,s + \ve)\cup (t - \ve,t)$ (\cite[Theorem 1]{Bhatia-23}, \cite[Lemma 3.3]{Dauvergne-23}).
 In BLPP, the set $\NU_1$ plays a significant role as the set of points from which the leftmost and rightmost competition interfaces have different directions (Theorem 4.32(ii) in~\cite{Seppalainen-Sorensen-21b}). Presently, we do not have  an analogous  characterization  in DL. 

Since   $\NU_0^{\dir -} \cup \NU_0^{\dir +}$  captures only the $L/R$ distinction and not the $\pm$ distinction, it does  \textit{not} in general contain all  the   initial points from which  the $\dir$-directed  semi-infinite geodesic  is not unique. However, when the $\dir\pm$ distinction is absent, Theorem~\ref{thm:all_SIG_thm_intro}\ref{itm:DL_LRmost_SIG} implies that $\NU_0^{\dir}=\NU_0^{\dir\pm}$ is exactly the set of points $p \in \R^2$ such that the semi-infinite geodesic from $p$ in direction $\dir$ is not unique. This happens under two scenarios: when $\dir \notin \DLBusedc$, and when we restrict attention to the $\dir$-dependent event of full probability on which $g_{p}^{\dir -,S} = g_p^{\dir +,S}$ for all $p \in \R^2$ and $S \in \{L,R\}$.

The failure to capture the $\pm$ non-uniqueness  is also evident from the size of $\NU_0$. Whenever $\dir \in \DLBusedc$, there are at least two semi-infinite geodesics with direction $\dir$ from {\it every} initial point. But along a fixed time level $\NU_0$ is  countable, and thereby  a strict subset of $\R^2$ (Theorem~\ref{thm:DLNU}\ref{itm:DL_NU_count} below). 




Recall that $\Hh_s=\{(x,s): x \in \R\}$ is the set of  space-time points at time level $s$. Theorem~\ref{thm:DLBusedc_description}\ref{itm:DL_dc_set_count} states that on a single event of full probability, $\DLBusedc \subseteq \R \setminus \Q$, so for $\dir \in \Q$,  we can drop the $\pm$ distinction and write $\NU_i^\dir =\NU_i^{\dir -} = \NU_i^{\dir +}$.
\begin{theorem} \label{thm:DLNU}
On a single event of probability one, for $i = 0,1$, the set $\NU_i$ satisfies  
    \be \label{109}
    \NU_i = \textstyle\bigcup_{\dir \in \Q}\NU_i^{\dir}.
    \ee In particular, the following hold.
\begin{enumerate} [label=\rm(\roman{*}), ref=\rm(\roman{*})]  \itemsep=3pt
    \item \label{itm:DL_NU_p0} 
     For each $p \in \R^2$, $\Pp(p \in \NU_0) = 0$, and the full-probability event of the theorem can be chosen so that $\NU_0$ contains no points of $\Q^2$.
    \item \label{itm:DL_NU_count} On a single event of full probability, simultaneously for every $s \in \R$, $\dir \in \R$ and $\sigg \in \{-,+\}$, the set $\NU_0^{\dir \sig} \cap\, \Hh_s$ is countably infinite and unbounded in both directions. Specifically, for each $s \in \R$, there exist  sequences $x_n \to -\infty$ and $y_n \to +\infty$ such that $(x_n,s),(y_n,s) \in \NU_0^{\dir \sig}$. 
    By~\eqref{109}, $\NU_0 \cap\, \Hh_s$ is also countably infinite.
\end{enumerate}
\end{theorem}
\begin{remark}
The set $\Q$ can be replaced by any countable dense subset of $\R$, by adjusting the full-probability event. In all applications in this paper, we use the set $\Q$. 
\end{remark}

The next theorem states properties of Busemann geodesics that involve the $L/R$ and $\pm$ distinctions.
\begin{theorem} \label{thm:g_basic_prop}
The following hold on a single event of full probability. 
\begin{enumerate} [label=\rm(\roman{*}), ref=\rm(\roman{*})]  \itemsep=3pt
    \item \label{itm:DL_mont_dir} For $s < t$, $x \in \R$, $\dir_1 < \dir_2$, and $S \in \{L,R\}$,
    \[
    g_{(x,s)}^{\dir_1 -,S}(t) \le g_{(x,s)}^{\dir_1 +,S}(t) \le g_{(x,s)}^{\dir_2 -,S}(t) \le g_{(x,s)}^{\dir_2 +,S}(t).  
    \]
     \item \label{itm:DL_SIG_unif} Let $\dir \in \R$, let $K \subseteq \R$ be a compact set, and let $T > \max K$. Then, there exists a random $\ve = \ve(\dir,T,K)>0$ such that, whenever $\dir - \ve < \alpha < \dir < \beta < \dir + \ve$, $\sigg \in \{-,+\}$, $S \in \{L,R\}$, and $x,s \in K$,
    \[
    g_{(x,s)}^{\alpha \sig,S}(t) = g_{(x,s)}^{\dir -,S}(t)\qquad\text{and}\qquad g_{(x,s)}^{\beta \sig,S}(t) = g_{(x,s)}^{\dir+,S}(t)\qquad\text{for all }t \in [s,T].
    \]
    \item \label{itm:limits_to_inf} For each $(x,s) \in \R^2$, $t > s$, $\sigg \in \{-,+\}$, and $S \in \{L,R\}$,
    $
    \lim_{\xi \to \pm \infty} g_{(x,s)}^{\xi \sig,S}(t) = \pm \infty.
    $
    \item \label{itm:DL_SIG_mont_x} For all $\dir \in \R$, $\sigg \in \{-,+\}$, $s < t$ and $x < y$, $g_{(x,s)}^{\dir \sig,R}(t) \le g_{(y,s)}^{\dir \sig,L}(t)$. More generally, if $x < y$, $s \in \R$, and  $g_1$ is a $\dir \sig$ geodesic from $(x,s)$ and $g_2$ is a $\dir \sig$ geodesic from $(y,s)$ such that $g_1(t) = g_2(t)$ for some $t > s$, then $g_1(u) = g_2(u)$ for all $u > t$. In other words, if   $g_1$ and $g_2$ intersect, they coalesce at their first point of intersection.
    \item \label{itm:DL_SIG_conv_x} For all $\dir \in \R$, $\sigg \in \{-,+\}$, $S \in \{L,R\}$, $x \in \R$, and $s < t$,
    \be \label{371}
    \lim_{w \nearrow x} g_{(w,s)}^{\dir \sig,S}(t) = g_{(x,s)}^{\dir \sig,L}(t),\qquad\text{and}\qquad \lim_{y \searrow x} g_{(y,s)}^{\dir \sig,S}(t) = g_{(x,s)}^{\dir \sig,R}(t),
    \ee
    and if \;  $g_{(x,s)}^{\dir \sig,L}(t) = g_{(x,s)}^{\dir \sig,R}(t) =: g_{(x,s)}^{\dir \sig}(t)$, then for $S \in \{L,R\}$,
    \be \label{372}
    \lim_{(w,u) \rightarrow (x,s)} g_{(w,u)}^{\dir \sig,S}(t) = g_{(x,s)}^{\dir \sig}(t).
    \ee
    Furthermore,
    \be \label{373}
    \lim_{x \to \pm \infty} g_{(x,s)}^{\xi \sig,S}(t) = \pm \infty.
    \ee
\end{enumerate}
\end{theorem}
\begin{remark} \label{rmk:mixing_LR_pm}
 In general, Theorem~\ref{thm:g_basic_prop}\ref{itm:DL_mont_dir} cannot be extended to mix $L$ with $R$. Pick a point $(x,s) \in \NU_0$, where $\NU_0$ is defined as in~\eqref{NU0_global}. Then, on the full-probability event of Theorem~\ref{thm:DLNU}, there exists a rational direction $\dir$ and $t > s$ such that 
 \[
 g_{(x,s)}^{\dir -,L}(t) = g_{(x,s)}^{\dir +,L}(t) < g_{(x,s)}^{\dir -,R}(t) = g_{(x,s)}^{\dir +,R}(t).
 \]
 By Theorem~\ref{thm:g_basic_prop}\ref{itm:DL_SIG_unif}, we may choose $\dir_1 < \dir < \dir_2$ sufficiently close to $\dir$ such that 
\[
g_{(x,s)}^{\dir_2 -,L}(t) = g_{(x,s)}^{\dir_2+,L}(t) = g_{(x,s)}^{\dir -,L}(t) < g_{(x,s)}^{\dir +,R}(t) = g_{(x,s)}^{\dir_1 -,R}(t) = g_{(x,s)}^{\dir_1+,R}(t).
\]

Item~\ref{itm:DL_SIG_mont_x} is an extension of Item 2 of Theorem 3.4 in~\cite{Rahman-Virag-21} to all directions and all pairs of initial points on the same horizontal level.
It is not true that for all $\dir \in \R$, $s < t$, and $x < y$, $g_{(x,s)}^{\xi +,R}(t) \le g_{(y,s)}^{\xi -,L}(t)$. This is discussed further in Remark~\ref{rmk:split_from_all_p} below.
\end{remark}

The  next theorem    controls   all  semi-infinite geodesics with Busemann geodesics.
    \begin{theorem} \label{thm:all_SIG_thm_intro}
    The following hold on a single event of probability one.  Let \\ $(x_r,t_r)_{r \in \R_{\ge 0}}$ be any net such that $t_r \to \infty$ and $x_r/t_r \to \dir$. 
    \begin{enumerate} [label=\rm(\roman{*}), ref=\rm(\roman{*})]  \itemsep=3pt
    \item  \label{itm:DL_LRmost_SIG} 
    Let $(x,s) \in \R^2$ and $\dir \in \R$. For each $r$ large enough so that $t_r > s$, let $g_r:[s,t_r] \to \R$ be a geodesic from $(x,s)$ to $(x_r,t_r)$.  Then, for each $t \ge s$, 
    \be \label{987}
    g_{(x,s)}^{\dir -,L}(t) \le \liminf_{r \to \infty} g_r(t) \le \limsup_{r \to \infty} g_r(t) \le g_{(x,s)}^{\dir +,R}(t).
    \ee
    In particular, $g_{(x,s)}^{\dir-,L}$ is the leftmost and $g_{(x,s)}^{\dir+,R}$ the rightmost among \textbf{all} semi-infinite geodesics from $(x,s)$ in direction $\dir$. 
   
    
    \item \label{itm:finite_geod_stick} Let $K \subseteq \R^2$ be compact.  Suppose that there is a level $t$ after which all semi-infinite geodesics from $(x,s) \in K$ in direction $\dir$ have coalesced. For $u \ge t$, let $g(u)$ be  this geodesic. Then, given $T > t$, there exists $R \in \R_{>0}$ such that for $r \ge R$ and all $(x,s) \in K$, if $g_r:[s,t_r]\to\R$ is a geodesic from $(x,s)$ to $(x_r,t_r)$, then
    \[
    g_r(u) = g(u)  \qquad\text{for all }u \in [t,T].
    \]
    In particular, suppose there is a unique semi-infinite geodesic from $(x,s)$ in direction $\dir$, denoted by  $g_{(x,s)}^\dir$. Then given  $T > s$, for sufficiently large $r$, we have 
    \[
   g_r(u) = g_{(x,s)}^\dir(u)   \qquad\text{for all }u \in [s,T].
    \]
\end{enumerate}
\end{theorem}
\begin{remark}
Theorem~\ref{thm:DL_all_coal}\ref{itm:DL_allsigns_coal} below states that the assumed coalescence in Item~\ref{itm:finite_geod_stick} occurs whenever $\dir \notin \DLBusedc$. The second statement of Item~\ref{itm:finite_geod_stick} is in Corollary 3.1 in~\cite{Rahman-Virag-21}. We provide a different proof that uses the regularity of the Busemann process.  
\end{remark}

\subsection{Proofs}
 In this section, we prove Theorems~\ref{thm:DLNU}, \ref{thm:g_basic_prop}, and~\ref{thm:all_SIG_thm_intro}. In each of these, the full-probability event is $\Omega_2$~\eqref{omega2}.
We start by proving parts of Theorem~\ref{thm:g_basic_prop}, then go to the proof of Theorem~\ref{thm:DLNU}.

\begin{proof}[Proof of Theorem~\ref{thm:g_basic_prop}, Items~\ref{itm:DL_mont_dir}--\ref{itm:limits_to_inf}]

\smallskip\noindent 
\textbf{Item~\ref{itm:DL_mont_dir} (monotonicity of geodesics in the direction parameter)} was already proven as Equation \eqref{eqn:mont_maxes}. In fact, this item holds on $\Omega_1$.

\smallskip\noindent \textbf{Item~\ref{itm:DL_SIG_unif} (geodesics agree locally for close directions): } This follows a similar proof as the proof of Theorem~\ref{thm:DL_Buse_summ}\ref{itm:DL_unif_Buse_stick}. Let $K$ be a compact subset of $\R$, and let $T$ be an integer greater than $\max K$. Set 
\[
A = \inf\{g_{(x,s)}^{(\dir - 1)-,L}(T):x,s \in K\},\qquad \text{and}\qquad B = \sup\{g_{(x,s)}^{(\dir + 1)+,R}(T):x,s \in K\}.
\]
By Lemma~\ref{lem:bounded_maxes} and Item~\ref{itm:DL_mont_dir}, $-\infty < A < B < \infty$. Then, for all $0 < \ve < 1$ sufficiently small, all $\dir- \ve < \alpha < \dir$, and all $x,s \in K$, the functions $z \mapsto \Ll(x,s;z,T) + \W_{\alpha \sig}(z,T;0,T)$ and  $z \mapsto \Ll(x,s;z,t) + \W_{\dir-}(z,T;0,T)$ agree on the set $[A,B]$, which contains all maximizers. Hence, for such $\alpha$ and $\sigg \in \{-,+\}$, and $S \in \{L,R\}$, $g_{(x,s)}^{\alpha \sig ,S}(T) = g_{(x,s)}^{\dir -,S}(T)$. 
Since $g_{(x,s)}^{\alpha \sig,L}:[s,\infty) \to \R$ and $g_{(x,s)}^{\alpha \sig,R}:[s,\infty) \to \R$ define semi-infinite geodesics that are, respectively, the leftmost and rightmost geodesics between any of their points (Theorem~\ref{thm:DL_SIG_cons_intro}\ref{itm:DL_all_SIG}-\ref{itm:DL_LRmost_geod}), it must also hold that for $S \in \{L,R\}$ and $t \in [t,T]$,
$g_{(x,s)}^{\alpha \sig,S}(t) = g_{(x,s)}^{\dir -,S}(t)
$.
Otherwise, taking $S = L$ without loss of generality, there would exist two distinct leftmost geodesics from $(x,s)$ to $(g_{(x,s)}^{\dir -,L}(T),T)$, a contradiction. The proof for the $\dir +$ geodesics where $\beta$ is sufficiently close to $\dir$ from the right is analogous. 

\smallskip\noindent \textbf{Item~\ref{itm:limits_to_inf} (limit of geodesics as direction goes to $\pm \infty)$:} This holds on $\Omega_2$ by definition~\eqref{omega2}. 

\smallskip\noindent We postpone the proof of Items~\ref{itm:DL_SIG_mont_x} and~\ref{itm:DL_SIG_conv_x} until after the following proof.
\end{proof}

\begin{proof}[Proof of Theorem~\ref{thm:DLNU} (Description of the sets $\NU_i$)]

By Theorem~\ref{thm:DLBusedc_description}\ref{itm:DL_NU_count}, on the event $\Omega_2$, $\alpha \notin \DLBusedc$ for all $\alpha \in \Q$, so we omit the $\pm$ distinction in this case. 
We first prove~\eqref{109}. 
If $(x,s) \in \NU_0^{\dir \sig}$ then 
$g_{(x,s)}^{\dir \sig,L}(t) < g_{(x,s)}^{\dir \sig,R}(t)$ for some $t > s$. By  Theorem~\ref{thm:g_basic_prop}\ref{itm:DL_SIG_unif}, there exists a rational direction $\alpha$ (greater than $\dir$ if $\sigg = +$ and less than $\dir$ if $\sigg = -$) such that
\[
g_{(x,s)}^{\alpha,L}(t) = g_{(x,s)}^{\dir \sig,L}(t) < g_{(x,s)}^{\dir \sig,R}(t) =g_{(x,s)}^{\alpha,R}(t).
\]
Hence, $(x,s) \in \NU_0^\alpha$. An analogous proof shows that $\NU_1 = \bigcup_{\dir \in \Q} \NU_1^\dir$.

\smallskip\noindent \textbf{Item~\ref{itm:DL_NU_p0}:} By
Theorem~\ref{thm:RV-SIG-thm}\ref{itm:pd_fixed}, for fixed direction $\dir$ and fixed initial point $p$, there is a unique semi-infinite geodesic from $p$ in direction $\dir$, implying $(x,s) \notin \NU_0^\dir$. The result now follows directly from~\eqref{109} and a union bound. In particular, by definition of the event $\Omega_1\supset \Omega_2$~\eqref{omega1}, for each $(q,r) \in \Q^2$ and $\dir \in \Q$, $(q,r) \notin \NU_0^\dir$. Then, by~\eqref{109}, on the event $\Omega_2$, \\$\NU_0 \subseteq \R^2 \setminus \Q^2$.

\smallskip\noindent 
We postpone the proof of Item~\ref{itm:DL_NU_count} until the end of this subsection.
\end{proof}


\begin{proof}[Remaining proofs of Theorem~\ref{thm:g_basic_prop}]

\smallskip\noindent 
 \textbf{Item~\ref{itm:DL_SIG_mont_x} (Spatial monotonicity of geodesics):} We first prove a weaker result. Namely, for $s \in \R$, $x < y$, $\dir \in \R$, $\sigg \in \{-,+\}$, and $S \in \{L,R\}$, 
\be \label{110}
g_{(x,s)}^{\dir \sig,S}(t) \le g_{(y,s)}^{\dir \sig,S}(t)\qquad\text{for all }t \ge s.
\ee
By continuity of geodesics, it suffices to assume that $z := g_{(x,s)}^{\dir \sig,L}(t) = g_{(y,s)}^{\dir \sig,L}(t)$ for some $t > s$, and then show that $g_{(x,s)}^{\dir \sig,L}(u) = g_{(y,s)}^{\dir \sig,L}(u)$ for all $u >  t$. 
By Theorem~\ref{thm:DL_SIG_cons_intro}\ref{itm:DL_all_SIG}, if $z := g_{(x,s)}^{\dir \sig,S}(t) = g_{(y,s)}^{\dir \sig,S}(t)$, then for $u > t$, both $g_{(x,s)}^{\dir \sig,L}(u)$ and $g_{(y,s)}^{\dir \sig,L}(u)$ are the leftmost maximizer of $\Ll(z,t;w,u) + \W_{\dir \sig}(w,u;0,u)$ over $w \in \R$, so they are equal.

Now, to prove the stated result, we follow a similar argument as Item 2 of Theorem 3.4 in~\cite{Rahman-Virag-21}, adapted to give a global result across all direction, signs, and pairs of points along the same horizontal line. Let $g_1$ be a $\dir \sig$ geodesic from $(x,s)$ and let $g_2$ be a $\dir \sig$ geodesic from $(y,s)$, and assume that $g_1(t) = g_2(t)$ for some $t > s$. By continuity of geodesics, we may take $t$ to be the minimal such time. Choose $r \in (s,t)\cap \Q$ and then choose $q \in (g_1(r),g_2(r)) \cap \Q$. See Figure~\ref{fig:choose_rational}. By Theorem~\ref{thm:DLNU}\ref{itm:DL_NU_p0}, on the event $\Omega_2$, there is a unique $\dir \sig$ Busemann geodesic from $(q,r)$, which we shall call $g = g_{(q,r)}^{\dir \sig,L} = g_{(q,r)}^{\dir \sig,R}$.
For $u \ge r$, 
\be \label{210}
 g_1(u)\le g_{(x,s)}^{\dir \sig,R}(u) \le g(u) \le  g_{(y,s)}^{\dir \sig,L}(u) \le g_2(u). 
\ee
The two middle inequalities come from~\eqref{110}.  The two outer inequalities come from the definition of $g_{(x,s)}^{\dir \sig,L/R}(u)$ as the left and rightmost maximizers.  

By assumption and~\eqref{210}, $z := g_1(t) = g(t) = g_2(t)$. By Theorem~\ref{thm:DL_SIG_cons_intro}\ref{itm:arb_geod_cons}\ref{itm:maxes}, for $u > t$, $g_1(u),g_2(u)$, and $g(u)$ are all maximizers of $\Ll(z,t;w,u) + \W_{\dir \sig}(w,u;0,u)$ over $w \in \R$. However, since there is a unique $\dir \sig$ geodesic from $(q,r)$, there can be only one such maximizer, so the inequalities in~\eqref{210} are equalities for $u \ge t$.  
\begin{figure}[t]
    \centering
    \includegraphics[height = 1.5in]{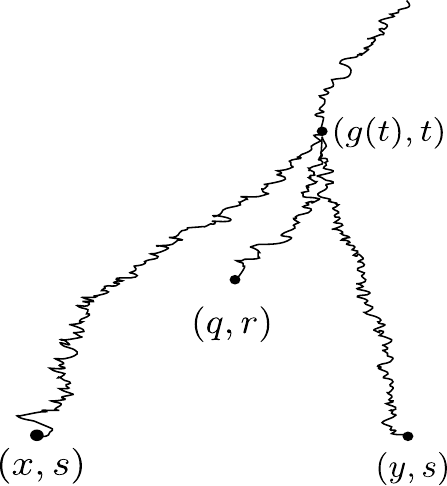}
    \caption{\small Choosing a point $(q,r) \in \Q^2$ whose $\dir \tiny{\boxempty}$ geodesic is unique}
    \label{fig:choose_rational}
\end{figure}

\smallskip\noindent \textbf{Item~\ref{itm:DL_SIG_conv_x} (limits of geodesics in the spatial parameter):} We start by proving~\eqref{371}. We prove the statement for the limits as $w \nearrow x$, and the limits as $w \searrow x$ follow analogously. By Item~\ref{itm:DL_SIG_mont_x}, $z := \lim_{w \nearrow x} g_{(w,s)}^{\dir \sig,S}(t)$ exists and is less than or equal to $g_{(x,s)}^{\dir \sig,L}(t)$. Further, by the same monotonicity, for all $w \in [x - 1,x]$, all maximizers of $\Ll(w,s;y,t) + \W_{\dir \sig}(y,t;0,t)$ over $y \in \R$ lie in the common compact set $[g_{(x - 1,s)}^{\dir \sig,L}(t),g_{(x,s)}^{\dir \sig,R}(t)]$. By continuity of the directed landscape (Lemma~\ref{lem:Landscape_global_bound}), as $w \nearrow x$, the function $y \mapsto \Ll(w,s;y,t) + \W_{\dir \sig}(y,t;0,t)$ converges uniformly on compact sets to the function $y \mapsto \Ll(x,s;y,t) + \W_{\dir \sig}(y,t;0,t)$. Hence, Lemma~\ref{lemma:convergence of maximizers from converging functions} implies that $z$ is a maximizer of $\Ll(x,s;y,t) + \W_{\dir \sig}(y,t;0,t)$ over $y \in \R$. Since $z \le g_{(x,s)}^{\dir \sig,L}(t)$, and $g_{(x,s)}^{\dir \sig,L}(t)$ is the leftmost such maximizer, equality holds. 

The proof of~\eqref{372} is similar: in this case, Lemma~\ref{lem:bounded_maxes} implies that for all $(w,u)$ sufficiently close to $(x,s)$, the maximizers of $y \mapsto\Ll(w,u;y,t) + \W_{\dir \sig}(y,t;0,t)$ lie in a common compact set. Then, by Lemma~\ref{lemma:convergence of maximizers from converging functions}, every subsequential limit of $g_{(w,u)}^{\dir \sig,S}(t)$ as $(w,u) \to (x,s)$ is a maximizer of $y \mapsto\Ll(x,s;y,t) + \W_{\dir \sig}(y,t;0,t)$. By assumption, there is only one such maximizer, so the desired convergence holds.  

Lastly, to show~\eqref{373}, we recall that the Busemann process evolves as the KPZ fixed point (Theorem~\ref{thm:DL_Buse_summ}\ref{itm:Buse_KPZ_description}).  The Busemann functions are continuous and satisfy the asymptotics prescribed in Lemma~\ref{lem:DL_horiz_Buse}\ref{itm:DL_lim}. Therefore, for each $t,\dir$, and $\sigg$, there exists constants $a,b > 0$ so that $|\W_{\dir \sig}(x,t;0,t)| \le a + b|x|$.  Lemma~\ref{lem:max_restrict}\ref{itm:KPZrestrict} applied to the temporally reflected version of $\Ll$ states that for sufficiently large $|x|$, $g_{(x,s)}^{\dir \sig,S}(t) \in (x - |x|^{2/3},x + |x|^{2/3})$. 
\end{proof}

\begin{proof}[Proof of Theorem~\ref{thm:all_SIG_thm_intro}]
We remind the reader that this theorem controls arbitrary geodesics via the Busemann geodesics. 
  
\smallskip\noindent \textbf{Item~\ref{itm:DL_LRmost_SIG}:} Let $\alpha < \dir < \beta$. By directedness of Busemann geodesics (Theorem~\ref{thm:DL_SIG_cons_intro}\ref{itm:DL_all_SIG}) and the assumption $x_r/r_r \to \dir$, for all sufficiently large $r$,
\[
g_{(x,s)}^{\alpha -,L}(t_r) < x_r < g_{(x,s)}^{\beta +,R}(t_r).
\]
Since $g_{(x,s)}^{\alpha -,L}$ is the leftmost geodesic between any of its points and $g_{(x,s)}^{\beta +,R}$ is the rightmost (Theorem~\ref{thm:DL_SIG_cons_intro}\ref{itm:DL_LRmost_geod}), it follows that for $u \in [s,t_r]$,
\be \label{504}
g_{(x,s)}^{\alpha -,L}(u) \le g_r(u) \le g_{(x,s)}^{\beta +,R}(u).
\ee
Hence, for all $t \ge s$,
\[
g_{(x,s)}^{\alpha -,L}(t) \le \liminf_{r \to \infty} g_r(t) \le \limsup_{r \to \infty} g_r(t) \le g_{(x,s)}^{\beta +,R}(t).
\]
By Theorem~\ref{thm:g_basic_prop}\ref{itm:DL_SIG_unif}, taking limits as $\alpha \nearrow \dir$ and $\beta \searrow \dir$ completes the proof. 

\smallskip\noindent \textbf{Item~\ref{itm:finite_geod_stick}:} Assume that all geodesics in direction $\dir$, starting from a point in the compact set $K$, have coalesced by time $t$, and for $u \ge t$, let $g(u)$ be the spatial location of this common geodesic. By Item~\ref{itm:DL_LRmost_SIG}, for all $p \in K$ and $u \ge t$,
\[
g(u) = g_p^{\dir -,L}(u) = g_p^{\dir +,R}(u).
\]
Let $T > t$ be arbitrary. By Theorem~\ref{thm:g_basic_prop}\ref{itm:DL_SIG_unif}, we may choose $\alpha < \dir < \beta$ such that, for all $p \in K$ and $u \in [t,T]$,
\be \label{301}
g_{(g(t),t)}^{\alpha-,L}(u) = g_p^{\alpha -,L}(u) = g(u) = g_p^{\beta +,R}(u) = g_{(g(t),t)}^{\beta +,R}(u).
\ee
The outer equalities hold because the geodesics pass through $(g(t),t)$. With this choice of $\alpha,\beta$, by the directedness of Theorem~\ref{thm:DL_SIG_cons_intro}\ref{itm:DL_all_SIG} and since $x_r/t_r \to \dir$, we may choose $r$ large enough so that $t_r \ge T$ and
$
g_{(g(t),t)}^{\alpha -,L}(t_r) < x_r < g_{(g(t),t)}^{\beta +,R}(t_r).
$
 Then,  as in the proof of Item~\ref{itm:DL_LRmost_SIG}, for all $u \in [t,t_r]$,
\[
g_{(g(t),t)}^{\alpha -,L}(u) \le g_r(u) \le g_{(g(t),t)}^{\beta +,R}(u).
\]
Combining this with~\eqref{301} completes the proof. 
\end{proof}

\noindent It remains to prove Theorem~\ref{thm:DLNU}\ref{itm:DL_NU_count}. We first prove a lemma.

\begin{lemma} \label{lem:NU_line}
Let $\omega \in \Omega_2$, $\dir \in \R$, $\sigg \in \{-,+\}$, $\Q \ni s < t \in \R$,  and assume that there is a nonempty interval $I = (a,b) \subseteq \R$ such that for all $x \in \Q$, $g_{(x,s)}^{\dir \sig}(t) \notin I$ {\rm(}By Theorem~\ref{thm:DLNU}\ref{itm:DL_NU_p0}, we may ignore the $L/R$ distinction when $(x,s) \in \Q^2${\rm)}. Then, there exists $\hat x \in \R$ such that 
\be \label{794}
g_{(\hat x,s)}^{\dir \sig,L}(t) \le a < b \le g_{(\hat x,s)}^{\dir \sig,R}(t). 
\ee
\end{lemma}
\begin{proof}
Choose some $y \in (a,b)$, and let 
\[
\hat x = \sup\{x \in \Q : g_{(x,s)}^{\dir \sig}(t) < y\}.
\]
By Equation~\eqref{373} of Theorem~\ref{thm:g_basic_prop}\ref{itm:DL_SIG_conv_x}, $\hat x \in \R$. By the monotonicity of Theorem~\ref{thm:g_basic_prop}\ref{itm:DL_SIG_mont_x}, for all $\Q \ni x < \hat x$, $g_{(x,s)}^{\dir \sig}(t) < y$, while for all $\Q \ni x > \hat x$, $g_{(x,s)}^{\dir \sig}(t) \ge y$. By assumption of the lemma,
this further implies that for $\Q \ni x < \hat x$, $g_{(x,s)}^{\dir \sig}(t) \le a$ while for $\Q \ni x > \hat x$, $g_{(x,s)}^{\dir \sig}(t) \ge b$.  By taking limits via Equation~\eqref{371} of Theorem~\ref{thm:g_basic_prop}\ref{itm:DL_SIG_conv_x}, we obtain~\eqref{794}.
\end{proof}

\begin{proof}[Proof of Theorem~\ref{thm:DLNU}\ref{itm:DL_NU_count} ($\NU_0^{\dir \sig} \cap \Hh_s$ is countably infinite and unbounded)] 
We prove the statement in three steps. First, we show that on $\Omega_2$, for all $s \in \Q$, $\dir \in \R$, $\sigg \in \{-,+\}$, the set $\NU_0^{\dir \sig} \cap\, \Hh_s$ is infinite and unbounded in both directions. Next, we show that, on $\Omega_2$, $\NU_0^{\dir \sig} \cap\, \Hh_s$ is in fact infinite and unbounded in both directions for all $s \in \R$. Lastly, we show that the set $\NU_0 \cap\, \Hh_s$ (the union over all directions and signs) is countable.

 For the first step, Theorem~\ref{thm:DLNU}\ref{itm:DL_NU_p0} states that, on the event $\Omega_2$, for each $(x,s) \in \Q^2$, $\dir \in \R$, and $\sigg \in \{-,+\}$, there is a unique $\dir \sig$ geodesic $g_{(x,s)}^{\dir \sig}$, and therefore this geodesic is both the leftmost and rightmost $\dir \sig$ geodesic from $(x,s)$. Since leftmost (resp. rightmost) Busemann geodesics are leftmost (rightmost) geodesics between any two of their points (Theorem~\ref{thm:DL_SIG_cons_intro}\ref{itm:DL_LRmost_geod}), it follows that  $g_{(x,s)}^{\dir \sig}$, restricted to times $t \in [s,s+2]$, is the unique geodesic from $(x,s)$ to $(g_{(x,s)}^{\dir \sig}(s + 2),s +2)$. By Lemma~\ref{lem:bounded_maxes}, for each compact set $K$, the set 
 \[
 \{g_{(x,s)}^{\dir \sig}(s + 1): x \in \Q \cap K\} 
 \]
is contained in some compact set $K'$.
Then, we have the following inclusion of sets:
\begin{align} 
&\quad \; \{g_{(x,s)}^{\dir \sig}(s + 1): x \in \Q \cap K \} \label{873}
\subseteq  \bigcup_{g \in \mathcal A_{K,K'}}\{g(s + 1) \} 
\end{align}
where 
\[
\mathcal A_{K,K'} = \{g: \text{$g$  is the unique geodesic from }(x,s) \text{ to }(y,s+2) \text{ for some } x\in K,y \in K'\}. 
\]
By Lemma~\ref{lem:geod_pp}, the set in the RHS of~\eqref{873} is finite, so the set on the LHS  is finite as well. Therefore, the set
\be \label{875}
\{g_{(x,s)}^{\dir \sig}(s + 1): x \in \Q \} = \bigcap_{k \in \Z_{>0}}\{g_{(x,s)}^{\dir \sig}(s + 1): x \in \Q \cap [-k,k] \}
\ee
is a union of finite nested sets. Further, by the ordering of geodesics from Theorem~\ref{thm:g_basic_prop}\ref{itm:DL_SIG_mont_x}, for each $k$, the difference 
\[
\{g_{(x,s)}^{\dir \sig}(s + 1): x \in \Q \cap [-(k + 1),k + 1] \} \setminus \{g_{(x,s)}^{\dir \sig}(s + 1): x \in \Q \cap [-k,k] \}
\]
lies entirely in the union of intervals
\[
\Bigl(-\infty, \inf \bigl\{g_{(x,s)}^{\dir \sig}(s + 1): x \in \Q \cap [-k,k] \bigr\}\Bigr] \cup \Bigl[\sup \bigl\{g_{(x,s)}^{\dir \sig}(s + 1): x \in \Q \cap [-k,k] \bigr\},\infty\Bigr).
\]
Therefore, the set~\eqref{875} has no limit points. 
Further, by Equation~\eqref{373} of Theorem~\ref{thm:g_basic_prop}\ref{itm:DL_SIG_conv_x}, the set~\eqref{875} is unbounded in both directions. These two facts imply that there exist infinitely many disjoint nonempty intervals whose intersection with the set~\eqref{875} is empty, and the set of endpoints of such intervals is unbounded. By Lemma~\ref{lem:NU_line}, for each $k > 0$, there exists $(x,s) \in \NU_0^{\dir \sig}$ such that $g_{(x,s)}^{\dir \sig,R}(s + 1) \ge k$, and there exists $(x,s) \in \NU_0^{\dir \sig}$ such that $g_{(x,s)}^{\dir \sig,L}(s + 1) \le -k$. Next, assume, by way of contradiction, that the set $\{x \in \R: (x,0) \in \NU_0^{\dir \sig}\}$ has an upper bound $b$. Then, by the monotonicity of Theorem~\ref{thm:g_basic_prop}\ref{itm:DL_SIG_mont_x}, for all $x \in \R$ with $(x,s) \in \NU_0^{\dir \sig}$, $g_{(x,s)}^{\dir \sig,R}(s + 1) \le g_{(b,s)}^{\dir \sig,R}(s +1)$. But this contradicts the fact we showed that $\{g_{(x,s)}^{\dir \sig,R}(s + 1): x \in \R\}$ is not bounded above. Hence, there exists a sequence $y_n \to \infty$ such that $(y_n,s) \in \NU_0^{\dir \sig}$ for all $n$. By a similar argument, there exists a sequence $x_n \to -\infty$ such that $(x_n,s) \in \NU_0^{\dir \sig}$ for all $n$.

Now, for arbitrary $s \in \R$, pick a rational number $T > s$. Pick $(z,T) \in \NU_0^{\dir \sig}$, and let
\begin{align*}
x_1 = \sup\{  x \in \R:   g_{(x,s)}^{\dir\sig, L}(T) \le  z  \}, \qquad \text{and}\qquad
x_2 = \inf\{  x \in \R:   g_{(x,s)}^{\dir\sig, R}(T) \ge  z  \}.
\end{align*}
By the limits in Equation~\eqref{373} of Theorem~\ref{thm:g_basic_prop}\ref{itm:DL_SIG_conv_x}, 
$x_1$ and $x_2$ lie in $\R$.

We first show that $x_2 \le x_1$. If not, then choose $x \in (x_1,x_2)$. Then, $g_{(x,s)}^{\dir\sig, R}(T) <   z <    g_{(x,s)}^{\dir\sig, L}(T)$, contradicting the meaning of L and R.  Hence  $x_2 \le x_1$. For any $x > x_2$, $g_{(x,s)}^{\dir \sig,R}(T) \ge z$, and by the limit in Equation~\eqref{371} of Theorem~\ref{thm:g_basic_prop}\ref{itm:DL_SIG_conv_x}, $g_{(x_2,s)}^{\dir \sig,R}(T) \ge z$ as well. By an analogous argument, for $x < x_1$, $g_{(x,s)}^{\dir \sig,L}(T) \le z$, and the inequality $g_{(x_1,s)}^{\dir \sig,L}(T) \le z$ holds by the same argument. Hence, for $x \in [x_2,x_1]$,
\[
g_{(x,s)}^{\dir \sig,L}(T) \le z,\qquad\text{and}\qquad g_{(x,s)}^{\dir \sig,R}(T) \ge z.
\]
Then, by the monotonicity of Theorem~\ref{thm:g_basic_prop}\ref{itm:DL_SIG_mont_x}, for $t \ge T$,
\be \label{1000}
g_{(x,s)}^{\dir \sig,L}(t) \le g_{(z,T)}^{\dir \sig,L}(t) \le g_{(z,T)}^{\dir \sig,R}(t) \le g_{(x,s)}^{\dir \sig,R}(t).
\ee
By assumption that $(z,T) \in \NU_0^{\dir \sig}$, there exists $t > T$ such that the middle inequality in~\eqref{1000} is strict, so $(x,s) \in \NU_0^{\dir \sig}$. Furthermore, by assumption, the set $\{z \in \R: (z,T) \in \NU_0\}$ has neither an upper or lower bound. Then, by the $t = T$ case of~\eqref{1000} and a similar argument as for the $s = 0$ case, the set $\{x \in \R: (x,s) \in \NU_0\}$ also has  neither an upper nor lower bound.

We lastly show countability of the sets. By~\eqref{109}, it suffices to show that for each $\dir \in \Q$ and $s \in \R$, $\NU_0^{\dir} \cap\, \Hh_s$ is countable. The proof is that of Theorem 3.4, Item 3 in~\cite{Rahman-Virag-21}, adapted to all horizontal lines simultaneously. For each $(x,s) \in \NU_0^\dir$, there exists $t > s$ such that $g_{(x,s)}^{\dir ,L}(t) < g_{(x,s)}^{\dir ,R}(t)$. By continuity of geodesics, the space between the two geodesics contains an open subset of $\R^2$. By the monotonicity of Theorem~\ref{thm:g_basic_prop}\ref{itm:DL_SIG_mont_x}, for $x < y$, $g_{(x,s)}^{\dir ,R}(t) \le g_{(y,s)}^{\dir ,L}(t)$ for all $t \ge s$. Hence, for $x < y$, with $(x,s),(y,s) \in \NU_0^\dir$, the associated open sets in $\R^2$ are disjoint, and $\NU_0^\dir \cap\, \Hh_s$ is at most countably infinite. 
\end{proof}

\section{Coalescence and the global geometry of geodesics} \label{sec:geometry_sec}
We can now describe the global structure of the semi-infinite geodesics, beginning with coalescence. 
\begin{theorem} \label{thm:DL_all_coal}
On a single event of full probability, the following hold across all directions $\dir \in \R$ and signs $\sigg \in \{-,+\}$. 
\begin{enumerate} [label=\rm(\roman{*}), ref=\rm(\roman{*})]  \itemsep=3pt
    \item \label{itm:DL_allsigns_coal} For all 
    $p,q \in \R^2$, if $g_1$ and $g_2$ are $\dir \sig$ Busemann geodesics from $p$ and $q$, respectively, then $g_1$ and $g_2$ coalesce. If the first point of intersection of the two geodesics is not $p$ or $q$, then the first point of intersection is the coalescence point of the two geodesics. 
    \item \label{itm:DL_split_return} 
    Let $g_1$ and $g_2$ be two distinct  $\dir \sig$ Busemann geodesics  from an initial point $(x,s)\in \NU_0^{\dir \sig}$.  Then, the set $\{t > s: g_1(t) \neq g_2(t)\}$ is a bounded open interval. That is, after the geodesics split, they coalesce exactly when they meet again. 
    \item \label{itm:unif_coal} For each 
    compact set $K \subseteq \R^2$, there exists a random $T = T(K,\dir,\sigg)<\infty$ such that for any two $\dir \sig$ geodesics $g_1$ and $g_2$ whose starting points lie in $K$, $g_1(t) = g_2(t)$ for all $t \ge T$. That is, there is a time level $T$ after which all semi-infinite geodesics started from points in $K$ have coalesced into a single path.
\end{enumerate} 
\end{theorem}
\begin{remark}
Theorem 1 of~\cite{Bhatia-23} and, independently, Lemma 3.3 of \cite{Dauvergne-23}, implies the following refinements of the results in this section. In Theorem~\ref{thm:DL_all_coal}\ref{itm:DL_split_return}, $\{t > s: g_1(t) \neq g_2(t)\}=(s, r)$ for some $r\in(s,\infty)$.  Under Condition~\ref{itm:DL_good_dir} of Theorem~\ref{thm:DL_good_dir_classification} below, the entire collection of semi-infinite geodesics in direction $\dir$ is a tree. 
\end{remark}

The following gives a full classification of the directions in which geodesics coalesce. We refer the reader to Theorems~\ref{thm:DL_eq_Buse_cpt_paths} and~\ref{thm:Buse_pm_equiv} below for the connection between coalescence and the regularity of the Busemann process.
 
\begin{theorem} \label{thm:DL_good_dir_classification}
On  a single event of probability one, the following are equivalent. 
\begin{enumerate} [label=\rm(\roman{*}), ref=\rm(\roman{*})]  \itemsep=3pt
    \item \label{itm:DL_good_dir} $\dir \notin \DLBusedc$.
    \item \label{itm:DL_LR_all_agree} $g_{p}^{\dir -,S} = g_{p}^{\dir +,S}$ for all $p \in \R^2$ and $S \in \{L,R\}$.
    \item \label{itm:DL_good_dir_coal} All semi-infinite geodesics in direction $\dir$ coalesce {\rm(}whether Busemann geodesics or not{\rm)}.
    \item \label{itm:DL_good_dir_unique_geod} For all $p \in \R^2 \setminus \NU_0$, there is a unique geodesic starting from $p$ with direction $\dir$.
    \item \label{itm:DL_good_dir_pt_unique}   There is a unique $\dir$-directed semi-infinite geodesic from some $p\in \R^2$.
    \item \label{itm:DL_good_dir_L_unique} There exists $p \in \R^2$ such that $g_{p}^{\dir -,L} = g_{p}^{\dir +,L}$.
    \item \label{itm:DL_good_dir_R_unique} There exists $p \in \R^2$ such that $g_{p}^{\dir -,R} = g_{p}^{\dir +,R}$
\end{enumerate}
Under these equivalent conditions, the following also holds.
\begin{enumerate}[resume, label=\rm(\roman{*}), ref=\rm(\roman{*})]  \itemsep=3pt
    \item \label{itm:DL_allBuse}  From any $p \in \R^2$, all semi-infinite geodesics in direction $\dir$ are Busemann geodesics. 
\end{enumerate}
\end{theorem}
\begin{remark} \label{rmk:split_from_all_p}
The equivalence~\ref{itm:DL_good_dir}$\Leftrightarrow$\ref{itm:DL_good_dir_L_unique} implies that $\forall\dir \in \DLBusedc$ and $p \in \R^2$,  geodesics $g_{p}^{\dir -,L}$ and $g_{p}^{\dir +,L}$ are distinct. The same is true when $L$ is replaced with $R$. Since $g_{p}^{\dir -,L}$ and $g_p^{\dir +,L}$ are both leftmost geodesics between any two of their points (Theorem~\ref{thm:DL_SIG_cons_intro}\ref{itm:DL_LRmost_geod}) 
then if $\dir\in\DLBusedc$,
these two geodesics must separate at some time $t\ge s$, and they cannot ever come back together. For each $\dir \in \DLBusedc$, there are two coalescing families of geodesics, namely the $\dir-$ and $\dir +$ geodesics. (See again Figure~\ref{fig:non_unique_comp}).  In particular, 
whenever $\xi \in \DLBusedc$, $s \in \R$,  and $x < y$, $g_{(x,s)}^{\xi +,L}(t) > g_{(y,s)}^{\xi -,R}(t)$ for sufficiently large $t$, as alluded to in Remark~\ref{rmk:mixing_LR_pm}.
\end{remark}

\subsection{Proofs}
In each of these theorems, the full-probability event is $\Omega_2$~\eqref{omega2}. We start by proving some lemmas that allow us to prove Theorem~\ref{thm:DL_all_coal}. The proof of Theorem~\ref{thm:DL_good_dir_classification} comes at the very end of this subsection. Section~\ref{sec:Buseextraproofs} proves Theorem~\ref{thm:DLSIG_main} as well as lingering results from Section~\ref{sec:Buse_geod_results}.
\begin{lemma} \label{lem:Buse_equality_coal}
Let $\omega \in \Omega_1$, $s \in \R$ and $x < y \in \R$. Assume, for some $\alpha < \dir$ and $\sigg_1,\sigg_2 \in \{-,+\}$, that $\W_{\alpha \sig_1}(y,s;x,s) = \W_{\dir \sig_2}(y,s;x,s)$. We also allow $\alpha = \dir$ if $\sigg_1 = -$ and $\sigg_2 = +$. If  $t > s$ and $g_{(x,s)}^{\dir \sig_2,R}(t) \le g_{(y,s)}^{\alpha \sig_1,L}(t)$, then for all $u \in [s,t]$,
\be \label{111}
g_{(x,s)}^{\alpha \sig_1,R}(u) = g_{(x,s)}^{\dir \sig_2,R}(u)\qquad\text{and}\qquad g_{(y,s)}^{\alpha \sig_1,L}(u) = g_{(y,s)}^{\dir \sig_2,L}(u).
\ee
\end{lemma}
\begin{proof}
By assumption, whenever $w < z$ and $t \in \R$, Theorem~\ref{thm:DL_Buse_summ}\ref{itm:DL_Buse_gen_mont} gives 
\be \label{100}
\W_{\alpha \sig_1}(z,t;w,t) \le W_{\dir \sig_2}(z,t;w,t).
\ee
For the rest of the proof, we  suppress the $\sigg_1,\sigg_2$ notation.
By Theorem~\ref{thm:DL_Buse_summ}\ref{itm:DL_Buse_add},\ref{itm:Buse_KPZ_description},
\begin{align}
\W_\dir(y,s;x,s) &= \W_\dir(y,s;0,t) - \W_\dir(x,s;0,t) \nonumber \\
&= \sup_{z \in \R}\{\Ll(y,s;z,t) + \W_\dir(z,t;0,t)\} - \sup_{z \in \R}\{\Ll(x,s;z,t) + \W_\dir(z,t;0,t)\}, \label{eqn:W_queue}
\end{align}
and the same with $\dir$ replaced by $\alpha$. Recall that $g_{(x,s)}^{\dir \sig,L}(t)$ and $g_{(x,s)}^{\dir \sig,R}(t)$ are, respectively, the leftmost and rightmost maximizers of $\Ll(x,s;z,t) + \W_{\dir \sig}(z,t;0,t)$ over $z \in \R$. 
Understanding that these quantities depend on $s$ and $t$, we use the shorthand notation $g_x^{\dir,R} = g_{(x,s)}^{\dir \sig_1,R}(t)$, and similarly with the other quantities. Then, we have
\begin{align}
    \Ll(x,s;g_x^{\dir,R},t) + W_\dir(g_x^{\dir,R},t;0,t) &- (\Ll(x,s;g_x^{\dir,R},t) + W_\alpha(g_x^{\dir,R},t;0,t)) \nonumber \\
    \ge \sup_{z \in \R}\{\Ll(x,s;z,t) + \W_\dir(z,t;0,t)\} &- \sup_{z \in \R}\{\Ll(x,s;z,t) + \W_\alpha(z,t;0,t)\} \label{101}\\
    =\sup_{z \in \R}\{\Ll(y,s;z,t) + \W_\dir(z,t;0,t)\} &- \sup_{z \in \R}\{\Ll(y,s;z,t) + \W_\alpha(z,t;0,t)\} \nonumber \\
    \ge \Ll(y,s;g_y^{\alpha,L},t) + W_\dir(g_y^{\alpha,L},t;0,t) &- (\Ll(y,s;g_y^{\alpha,L},t) + W_\alpha(g_y^{\alpha,L},t;0,t)), \label{102} 
\end{align}
where the middle equality came from the assumption that $\W_{\dir}(y,s;x,s) = \W_{\alpha}(y,s;x,s)$ and Equation~\eqref{eqn:W_queue} applied to both $\dir$ and $\alpha$.
Rearranging the first and last lines yields 
\[
\W_\dir(g_y^{\alpha,L},t;g_{x}^{\dir,R},t) \le \W_\alpha(g_y^{\alpha,L},t;g_{x}^{\dir,R},t).
\]
However, the assumption $g_x^{\dir,R} \le g_{y}^{\alpha,L}$
combined with~\eqref{100} implies that this inequality is an equality. Hence, inequalities~\eqref{101} and~\eqref{102} are also equalities. From the equality~\eqref{101},  
\[
\Ll(x,s;g_x^{\dir,R},t) + W_\alpha(g_x^{\dir,R},t;0,t) = \sup_{z \in \R}\{\Ll(x,s;z,t) + \W_\alpha(z,t;0,t)\},
\]
so $z=g_x^{\dir,R}$ is a maximizer of $\Ll(x,s;z,t) + \W_\alpha(z,t;0,t)$. By definition, $g_{x}^{\alpha,R}$ is the rightmost maximizer, and by geodesic ordering (Theorem~\ref{thm:g_basic_prop}\ref{itm:DL_mont_dir}), $g_x^{\dir,R} \ge g_x^{\alpha,R}$, so  $g_x^{\dir,R} = g_x^{\alpha,R}$. An analogous argument applied to~\eqref{102} implies  $g_y^{\alpha,L} = g_y^{\dir,L}$. We have shown that 
\[
g_{(x,s)}^{\alpha \sig_1,R}(t) = g_{(x,s)}^{\dir \sig_2,R}(t),\qquad\text{and}\qquad g_{(y,s)}^{\alpha \sig_1,L}(t) = g_{(y,s)}^{\dir \sig_2,L}(t).
\]
Since $g_{(x,s)}^{\alpha \sig_1,R}$ and $g_{(x,s)}^{\dir \sig_2,R}$ are both the rightmost geodesics between any two of their points and similarly with the leftmost geodesics from $(y,s)$ (Theorem~\ref{thm:DL_SIG_cons_intro}\ref{itm:DL_LRmost_geod}), Equation~\eqref{111} holds for all $u \in [s,t]$, as desired. 
\end{proof}

\begin{lemma} \label{lem:DL_LR_coal}
Let $\omega \in \Omega_2$, $s \in \R$, and $x < y$. If, for some $\alpha < \dir$ and $\sigg_1,\sigg_2 \in \{-,+\}$ we have that $\W_{\alpha \sig_1}(y,s;x,s) = \W_{\dir \sig_2}(y,s;x,s)$, then $g_{(x,s)}^{\alpha \sig_1,R}$ coalesces with $g_{(y,s)}^{\alpha \sig_1,L}$, $g_{(x,s)}^{\dir \sig_2,R}$ coalesces with $g_{(y,s)}^{\dir \sig_2,L}$, and the coalescence points of the two pairs of geodesics are the same. 
\end{lemma}
\begin{proof}
By Theorem~\ref{thm:DL_SIG_cons_intro}\ref{itm:DL_all_SIG}, $g_{(x,s)}^{\dir \sig_2,R}(t)/t \to \dir$ while $g_{(y,s)}^{\alpha \sig_1,L}(t)/t \to \alpha$ as $t \to \infty$. By this and continuity of geodesics, there exists a minimal time $t > s$ such that $z := g_{(x,s)}^{\dir \sig_2,R}(t) = g_{(y,s)}^{\alpha \sig_1,L}(t)$. By Lemma~\ref{lem:Buse_equality_coal}, 
\[
g_{(x,s)}^{\alpha \sig_1,R}(u) = g_{(x,s)}^{\dir \sig_2,R}(u)\qquad\text{and}\qquad g_{(y,s)}^{\alpha \sig_1,L}(u) = g_{(y,s)}^{\dir \sig_2,L}(u) \qquad \text{for all }u \in [s,t].
\]
Since $t$ was chosen to be minimal, Theorem~\ref{thm:g_basic_prop}\ref{itm:DL_SIG_mont_x} implies that the pair $g_{(x,s)}^{\alpha \sig_1,R}$, $g_{(y,s)}^{\alpha \sig_1,L}$ and the pair $g_{(x,s)}^{\dir \sig_2,R}$,  $g_{(y,s)}^{\dir \sig_2,L}$ both coalesce at $(z,t)$.
\end{proof}

\begin{proof}[Proof of Theorem~\ref{thm:DL_all_coal}]

\smallskip\noindent
\textbf{Item~\ref{itm:DL_allsigns_coal} (Coalescence):} Let $g_1$ and $g_2$ be $\dir \sigg$ Busemann geodesics from $(x,s)$ and $(y,t)$, respectively, and take $s \le t$ without loss of generality. Let $a = (g_1(t) \wedge y) - 1$ and $b = (g_1(t) \vee y) + 1$. By Theorem~\ref{thm:g_basic_prop}\ref{itm:DL_SIG_mont_x}, for all $u \ge t$,
\be \label{112}
g_{(a,t)}^{\dir \sig,R}(u) \le g_1(u) \wedge g_2(u) \le g_1(u)\vee g_2(u) \le g_{(b,t)}^{\dir \sig,L}(u).
\ee
By Theorem~\ref{thm:DL_Buse_summ}\ref{itm:DL_unif_Buse_stick}, there exists $\alpha$, sufficiently close to $\dir$, (from the left for $\sigg = -$ and from the right for $\sigg = +$) such that $\W_{\dir  \sig}(b,t;a,t) = \W_{\alpha \sig}(b,t;a,t)$. By Lemma~\ref{lem:DL_LR_coal}, $g_{(a,t)}^{\dir \sig,R}$ coalesces with $g_{(b,t)}^{\dir \sig,L}$. Then, for $u$ large enough, all inequalities in~\eqref{112} are equalities, and $g_1$ and $g_2$ coalesce.  

If the first point of intersection is not $(y,t)$, then $g_1(t) \neq y$, and the coalescence point of $g_1$ and $g_2$ is the first point of intersection by Theorem~\ref{thm:g_basic_prop}\ref{itm:DL_SIG_mont_x}.

\smallskip\noindent \textbf{Item~\ref{itm:DL_split_return} (Geodesics coalesce when they meet):} Let $(x,s) \in \NU_0^{\dir \sig}$, and let $g_1$ and $g_2$ be two distinct $\dir \sig$ Busemann geodesics from $(x,s)$. The set $\text{GNEQ} := \{t>s:g_1(t) \neq g_2(t)\}$ is therefore nonempty and infinite by continuity of $g_1$ and $g_2$. Assume, by way of contradiction, that $\text{GNEQ}$ is not an open interval. By continuity of geodesics, $\text{GNEQ}$ cannot be a closed or half-closed interval, so $\text{GNEQ}$ is not path connected. Thus, there exists $t_1 < t_2 < t_3$ so that 
\[
g_1(t_1) \neq g_2(t_1),\quad g_1(t_2) = g_2(t_2),\quad \text{ and}\quad  g_1(t_3) \neq g_2(t_3).
\]
 The geodesics $g_1|_{[t_1,\infty)}$ and $g_2|_{[t_1,\infty)}$ started from $(g_1(t_1),t_1)$ and $(g_2(t_1),t_1)$, respectively, are both Busemann geodesics by their construction in Theorem~\ref{thm:DL_SIG_cons_intro}. Since the geodesics $g_1|_{[t_1,\infty)}$ and $g_2|_{[t_1,\infty)}$  start at different spatial locations (namely $g_1(t_1)$ and $g_2(t_1)$) along the same time level $t_1$, they cannot intersect at either of their starting points.  By Item~\ref{itm:DL_allsigns_coal}, the two geodesics $g_1|_{[t_1,\infty)}$ and $g_2|_{[t_1,\infty)}$ must coalesce, and the first point of intersection is the coalescence point. Since $g_1(t_2)  = g_2(t_2)$, this implies that $g_1(t) = g_2(t)$ for all $t > t_2$, a contradiction to the existence of $t_3$.

\smallskip\noindent  \textbf{Item~\ref{itm:unif_coal} (Uniformity of coalescence):} Let $\dir \in \R$, $\sigg \in \{-,+\}$, and let the compact set $K$ be given. Let $S$ be the smallest integer greater than $\max\{s: (x,s) \in K\}$. Set 
\[
A := \inf\{g_{(x,s)}^{\dir \sig,L}(S): (x,s) \in K\},\qquad\text{and}\qquad B := \sup\{g_{(x,s)}^{\dir \sig,R}(S):(x,s) \in K\}.
\]
By Lemma~\ref{lem:bounded_maxes}, $-\infty < A \le B < \infty$.  Then, by Theorem~\ref{thm:g_basic_prop}\ref{itm:DL_SIG_mont_x}, whenever $g$ is a $\dir \sig$ geodesic starting from $(x,s) \in K$,
\[
g_{(A,S)}^{\dir \sig,L}(t) \le g(t) \le g_{(B,S)}^{\dir \sig,R}(t) \qquad\text{for all }t \ge S.
\]
To complete the proof, let $T$ be the time at which $g_{(A,S)}^{\dir \sig,L}$ and $g_{(B ,S)}^{\dir \sig,R}$ coalesce, which is guaranteed to be finite by Item~\ref{itm:DL_allsigns_coal}.
\end{proof}

For two initial points on a horizontal level, as $\dir$ varies, a constant  Busemann process corresponds to a  constant    coalescence point of the geodesics.
The non-uniqueness of geodesics  
requires us to be careful about the choice of left and right geodesic. 
\begin{definition} \label{def:coal_pt}
For $s \in \R$ and $x < y$, let $\mbf z^{\dir \sig}(y,s;x,s)$ be the coalescence point of $g_{(y,s)}^{\dir \sig,L}$ and $g_{(x,s)}^{\dir \sig,R}$.
\end{definition}

\begin{theorem} \label{thm:DL_eq_Buse_cpt_paths}
On a single event of probability one, for all reals $\alpha < \beta$, $s$, and $x < y$, the following are equivalent.
\begin{enumerate}[label=\rm(\roman{*}), ref=\rm(\roman{*})]  \itemsep=3pt
\item \label{itm:DL_buse_eq}$\W_{\alpha +}(y,s;x,s) = \W_{\beta -}(y,s;x,s)$.
\item \label{itm:DL_coal_pt_equal} $\mbf z^{\alpha +}(y,s;x,s) = \mbf z^{\beta -}(y,s;x,s)$.
\item \label{itm:DL_paths} There exist $t > s$ and $z \in \R$ such that there are paths $g_1:[s,t] \to \R$ {\rm(}connecting $(x,s)$ and $(z,t)${\rm)} and $g_2:[s,t] \to \R$ {\rm(}connecting $(y,s)$ to $(z,t)${\rm)} such that for all $\dir \in (\alpha,\beta)$, $\sigg \in \{-,+\}$, and $u \in [s,t)$,
\be \label{124}\begin{aligned} 
g_1(u) &= g_{(x,s)}^{\dir \sig,R}(u) = g_{(x,s)}^{\alpha +,R}(u) = g_{(x,s)}^{\beta -,R}(u) \\
&< g_2(u) =g_{(y,s)}^{\dir \sig,L}(u) = g_{(y,s)}^{\alpha +,L}(u) = g_{(y,s)}^{\beta -,L}(u).
\end{aligned} \ee
\end{enumerate}
\end{theorem}

\begin{proof}
 
\ref{itm:DL_buse_eq}$\Rightarrow$\ref{itm:DL_coal_pt_equal} follows from Lemma~\ref{lem:DL_LR_coal}. 

\smallskip \noindent \ref{itm:DL_coal_pt_equal}$\Rightarrow$\ref{itm:DL_buse_eq}: Assume $(z,t) := \mbf z^{\alpha +}(y,s;x,s) = \mbf z^{\beta -}(y,s;x,s)$. By additivity (Theorem~\ref{thm:DL_Buse_summ}\ref{itm:DL_Buse_add}) and Theorem~\ref{thm:DL_SIG_cons_intro}\ref{itm:DL_all_SIG},
\begin{align*}
\W_{\alpha +}(y,s;x,s) &= \W_{\alpha +}(y,s;z,t) - \W_{\alpha +}(x,s;z,t) \\
&= \Ll(y,s;z,t) - \Ll(x,s;z,t) 
\\ &= \W_{\beta -}(y,s;z,t) - \W_{\beta -}(x,s;z,t)
= \W_{\beta -}(y,s;x,s).
\end{align*}

\smallskip \noindent \ref{itm:DL_coal_pt_equal}$\Rightarrow$\ref{itm:DL_paths}: Let $(z,t)$ be as in the proof of \ref{itm:DL_coal_pt_equal}$\Rightarrow$\ref{itm:DL_buse_eq}. By Theorem~\ref{thm:DL_SIG_cons_intro}\ref{itm:DL_LRmost_geod}, the restriction of $g_{(x,s)}^{\alpha +,R}$ and $g_{(x,s)}^{\beta -,R}$ to the domain $[s,t]$ are both rightmost geodesics between $(x,s)$ and $(z,t)$, and therefore they agree on this restricted domain. Similarly, $g_{(y,s)}^{\alpha +,L}$ and $g_{(y,s)}^{\beta -,L}$ agree on the domain $[s,t]$. By the monotonicity of Theorem~\ref{thm:g_basic_prop}\ref{itm:DL_mont_dir}, and since $(z,t)$ is the common coalescence point,~\eqref{124} holds for $u \in [s,t)$, as desired. 

\smallskip \noindent \ref{itm:DL_paths}$\Rightarrow$\ref{itm:DL_coal_pt_equal} is immediate.
\end{proof}

\begin{theorem} \label{thm:Buse_pm_equiv}
On a single event of probability one, for all reals $s,\dir \in \R$, and $x < y$, the following are equivalent. 
\begin{enumerate} [label=\rm(\roman{*}), ref=\rm(\roman{*})]  \itemsep=3pt
    \item \label{itm:DL_pm_Buse_eq} $\W_{\dir -}(y,s;x,s) = \W_{\dir +}(y,s;x,s).$
    \item \label{itm:DL_pm_coal_pt} $\mbf z^{\dir -}(y,s;x,s) = \mbf z^{\dir +}(y,s;x,s)$.
    \item \label{itm:DL_disjoint_paths} $g_{(x,s)}^{\dir-,R}(t) = g_{(y,s)}^{\dir +,L}(t)$ for some $t > s$, i.e., the paths $g_{(x,s)}^{\dir -,R}$ and $g_{(y,s)}^{\dir +,L}$  intersect.
\end{enumerate}
\end{theorem}
\begin{remark}
In Item~\ref{itm:DL_disjoint_paths}, if $\dir \in \DLBusedc$, then despite intersecting, the geodesics   $g_{(x,s)}^{\dir -,R}$ and $g_{(y,s)}^{\dir +,L}$ cannot coalesce. This follows from Theorem~\ref{thm:DL_good_dir_classification}, which gives  a full classification of the directions in which all semi-infinite geodesics coalesce. 
\end{remark}
\begin{proof}[Proof of Theorem~\ref{thm:Buse_pm_equiv}]
\ref{itm:DL_pm_Buse_eq}$\Rightarrow$\ref{itm:DL_pm_coal_pt}: If $\W_{\dir-}(y,s;x,s) = \W_{\dir+}(y,s;x,s)$, then Theorem~\ref{thm:DL_Buse_summ}\ref{itm:DL_unif_Buse_stick} implies that for some $\alpha < \dir < \beta$, $\W_{\alpha +}(y,s;x,s) = \W_{\beta -}(y,s;x,s)$. Then, we apply \ref{itm:DL_buse_eq}$\Rightarrow$\ref{itm:DL_paths} of Theorem~\ref{thm:DL_eq_Buse_cpt_paths} to conclude that for some $t > s$ and $z \in \R$, 
\[
g_{(x,s)}^{\dir -,R}(u) = g_{(x,s)}^{\dir +,R}(u) < g_{(y,s)}^{\dir -,L}(u) = g_{(y,s)}^{\dir +,L}(u),\qquad\text{for }u \in [s,t),
\]
whereas for $u = t$, all terms above equal some common value $z$. Therefore, $(z,t) = \mbf z^{\dir -}(y,s;x,s) = \mbf z^{\dir +}(y,s;x,s)$.

\smallskip \noindent \ref{itm:DL_pm_coal_pt}$\Rightarrow$\ref{itm:DL_pm_Buse_eq}: 
Similarly as in the proof of \ref{itm:DL_coal_pt_equal}$\Rightarrow$\ref{itm:DL_buse_eq} of Theorem~\ref{thm:DL_eq_Buse_cpt_paths}, if $(z,t) = \mbf z^{\dir -}(y,s;x,s) = \mbf z^{\dir +}(y,s;x,s)$, then
$
\W_{\dir -}(y,s;x,s) = 
\Ll(y,s;z,t) - \Ll(x,s;z,t) 
= \W_{\dir +}(y,s;x,s).
$

\smallskip \noindent \ref{itm:DL_pm_coal_pt}$\Rightarrow$\ref{itm:DL_disjoint_paths}: Assume $(z,t) = \mbf z^{\dir -}(y,s;x,s) = \mbf z^{\dir +}(y,s;x,s)$. Then, $g_{(x,s)}^{\dir -,R}(t) = z = g_{(y,s)}^{\dir +,L}(t)$. 

\smallskip \noindent \ref{itm:DL_disjoint_paths}$\Rightarrow$\ref{itm:DL_pm_coal_pt}: 
 Assume that $g_{(x,s)}^{\dir-,R}(t) = g_{(y,s)}^{\dir +,L}(t)$ for some $t > s$. Let $t$ be the minimal such time, and let $(z,t)$ be the point where the geodesics first intersect. By Theorem~\ref{thm:g_basic_prop}, Items~\ref{itm:DL_mont_dir} and~\ref{itm:DL_SIG_mont_x}, for $u > s$,
 \be \label{406}
 g_{(x,s)}^{\dir-,R}(u) \le g_{(x,s)}^{\dir +,R}(u) \wedge g_{(y,s)}^{\dir -,L}(u) \le  g_{(x,s)}^{\dir +,R}(u) \vee g_{(y,s)}^{\dir -,L}(u)  \le  g_{(y,s)}^{\dir +,L}(u).
 \ee
 In particular, when $u = t$, all inequalities in~\eqref{406} are equalities. Further, since  $g_{(x,s)}^{\dir -,R}$,$g_{(x,s)}^{\dir +,R}$ are rightmost geodesics between $(x,s)$ and $(z,t)$ (Theorem~\ref{thm:DL_SIG_cons_intro}\ref{itm:DL_LRmost_geod}), $g_{(x,s)}^{\dir -,R}(u) = g_{(x,s)}^{\dir +,R}(u)$ for $u \in [s,t]$. Similarly, $g_{(y,s)}^{\dir -,L}(u) = g_{(y,s)}^{\dir +,L}(u)$ for $u \in [s,t]$. Since $t$ was chosen minimally for $g_{(x,s)}^{\dir-,R}(t) = g_{(y,s)}^{\dir +,L}(t)$, we have $(z,t) = \mbf z^{\dir -}(y,s;x,s) = \mbf z^{\dir +}(y,s;x,s)$. 
\end{proof}

\begin{proof}[Proof of Theorem~\ref{thm:DL_good_dir_classification} (Classification of coalescence directions)]
\ref{itm:DL_good_dir}$\Rightarrow$\ref{itm:DL_LR_all_agree}: If $\dir \notin \DLBusedc$, then $W_{\dir -} = \W_{\dir +}$, so \ref{itm:DL_LR_all_agree} follows by the construction of the Busemann geodesics from the Busemann functions. 

\smallskip \noindent \ref{itm:DL_LR_all_agree}$\Rightarrow$\ref{itm:DL_good_dir_coal}: Since a geodesic in direction $\dir$ from $(x,s)$ must pass through each horizontal level $t > s$, it is sufficient to show that, for $s \in \R$ and $x < y$, whenever $g_1$ is a semi-infinite  geodesic from $(x,s)$ in direction $\dir$ and $g_2$ is a semi-infinite geodesic from $(y,s)$ in direction $\dir$, $g_1$ and $g_2$ coalesce. Assuming~\ref{itm:DL_LR_all_agree} and using Theorem~\ref{thm:all_SIG_thm_intro}\ref{itm:DL_LRmost_SIG}, for all $t > s$,
\[
g_{(x,s)}^{\dir +,L}(t) = g_{(x,s)}^{\dir -,L}(t) \le g_1(t) \wedge g_2(t) \le g_1(t) \vee g_2(t) \le g_{(y,s)}^{\dir +,R}(t).
\]
By Theorem~\ref{thm:DL_all_coal}\ref{itm:DL_allsigns_coal}, $g_{(x,s)}^{\dir +,L}$ and $g_{(y,s)}^{\dir +,R}$ coalesce, so all inequalities above are equalities for large $t$, and $g_1$ and $g_2$ coalesce. 

\smallskip \noindent \ref{itm:DL_good_dir_coal}$\Rightarrow$\ref{itm:DL_good_dir}: We prove the contrapositive. If $\dir \in \DLBusedc$, then by Theorem~\ref{thm:DL_Buse_summ}\ref{itm:DL_Buse_gen_mont}-\ref{itm:Buse_KPZ_description}, \\$\W_{\dir -}(y,s;x,s) < \W_{\dir +}(y,s;x,s)$ for some $x < y$ and $s \in \R$. By \ref{itm:DL_pm_Buse_eq}$\Leftrightarrow$\ref{itm:DL_disjoint_paths} of Theorem~\ref{thm:Buse_pm_equiv}, $g_{(x,s)}^{\dir-,R}(t) < g_{(y,s)}^{\dir +,L}(t)$ for all $t > s$. In particular, $g_{(x,s)}^{\dir-,R}$ and $g_{(y,s)}^{\dir +,L}$ do not coalesce. 

\smallskip \noindent \ref{itm:DL_LR_all_agree}$\Rightarrow$\ref{itm:DL_good_dir_unique_geod}: By definition of $\NU_0$, whenever $p\notin \NU_0$, $g_p^{\dir \sig,L} = g_{p}^{\dir \sig,R}$ for $\dir \in \R$ and $\sigg \in \{-,+\}$. Hence, assuming $p \notin \NU_0$ and $g_{p}^{\dir -,R} = g_{p}^{\dir +,R}$, we also have $g_{p}^{\dir-,L} = g_{p}^{\dir +,R}$, so there is a unique  geodesic from $p$ in direction $\dir$ by Theorem~\ref{thm:all_SIG_thm_intro}\ref{itm:DL_LRmost_SIG}.

\smallskip \noindent \ref{itm:DL_good_dir_unique_geod}$\Rightarrow$\ref{itm:DL_good_dir_pt_unique}: By Theorem~\ref{thm:DLNU}\ref{itm:DL_NU_p0}, on the event $\Omega_2$, $\NU_0$ contains no points of $\Q^2$, and therefore, $\NU_0$ is not all of $\R^2$.

\smallskip \noindent \ref{itm:DL_good_dir_pt_unique}$\Rightarrow$\ref{itm:DL_good_dir_L_unique} and \ref{itm:DL_good_dir_pt_unique}$\Rightarrow$\ref{itm:DL_good_dir_R_unique} are direct consequences of Theorem~\ref{thm:all_SIG_thm_intro}\ref{itm:DL_LRmost_SIG}: If there is a unique semi-infinite geodesic in direction $\dir$ from a point $p \in \R^2$, then $g_{p}^{\dir -,L} = g_{p}^{\dir +,L} = g_{p}^{\dir -,R} = g_{p}^{\dir +,R}$.

\smallskip \noindent \ref{itm:DL_good_dir_L_unique}$\Rightarrow$\ref{itm:DL_LR_all_agree}: Let $p$ be a point from which $g_{p}^{\dir-,L} = g_{p}^{\dir +,L}$, and call this common geodesic $g$. Let $q$ be an arbitrary point in $\R^2$. By Theorem~\ref{thm:DL_all_coal}\ref{itm:DL_allsigns_coal}, $g_{q}^{\dir -,L}, g_{q}^{\dir +,L}, g_{q}^{\dir - ,R}$, and $g_{q}^{\dir +,R}$ each coalesce with $g$, so $g_{q}^{\dir -,L}$ and $g_{q}^{\dir +,L}$ coalesce. Since both geodesics are the leftmost geodesics between their points by Theorem~\ref{thm:DL_SIG_cons_intro}\ref{itm:DL_LRmost_geod}, they must be the same. Similarly, $g_{q}^{\dir-,R} = g_{q}^{\dir +,R}$. 

\smallskip \noindent \ref{itm:DL_good_dir_R_unique}$\Rightarrow$\ref{itm:DL_LR_all_agree}: follows by the same proof.

\smallskip \noindent \textbf{Item~\ref{itm:DL_allBuse}:} Let $\dir \in \R \setminus \DLBusedc$, and let $g$ be a  semi-infinite geodesic in direction $\dir$, starting from a point $(x,s) \in \R^2$. By Lemma~\ref{lem:L_and_Buse_ineq} and Theorem~\ref{thm:DL_SIG_cons_intro}\ref{itm:arb_geod_cons}, it is sufficient to show that for sufficiently large $t$, 
\be \label{647}
\Ll(x,s;g(t),t) = \W_{\dir}(x,s;g(t),t).
\ee
(we dropped the $\pm$ distinction since $\W_{\dir -} = \W_{\dir +}$). By Item~\ref{itm:DL_good_dir_coal}, $g$ coalesces with $g_{(x,s)}^{\dir,R}$. Then, for sufficiently large $t$, $g(t) = g_{(x,s)}^{\dir,R}(t)$ and by Theorem~\ref{thm:DL_SIG_cons_intro}\ref{itm:DL_all_SIG},~\eqref{647} holds. 
\end{proof}

\subsection{Remaining proofs from Section~\ref{sec:Buse_geod_results} and Proof of Theorem~\ref{thm:DLSIG_main}} \label{sec:Buseextraproofs}
We complete some unfinished business. 
\begin{proof}[Proof of Items~\ref{itm:BuseLim1}--\ref{itm:global_attract} of Theorem~\ref{thm:DL_Buse_summ} and the mixing in Theorem~\ref{thm:Buse_dist_intro}\ref{itm:stationarity}]
We continue to work on the event $\Omega_2$.

\smallskip \noindent \textbf{Item~\ref{itm:BuseLim1} of Theorem~\ref{thm:DL_Buse_summ} (Busemann limits I):}  By Theorem~\ref{thm:DL_good_dir_classification}\ref{itm:DL_allBuse}, if $\dir \notin \DLBusedc$, all $\dir$-directed semi-infinite geodesics are Busemann geodesics, and they all coalesce.  By Theorem~\ref{thm:DL_all_coal}\ref{itm:unif_coal}, there exists a level $T$ such that all geodesics from points starting in the compact set $K$ have coalesced by time $T$. Let $(Z,T)$ denote the location of the point of the common geodesics at time $T$. Let $r_t = (z_t,u_t)_{t \in \R_{\ge 0}}$ be any net with $u_t \to \infty$ and $z_t/u_t \to \dir$. By Theorem~\ref{thm:all_SIG_thm_intro}\ref{itm:finite_geod_stick}, for all sufficiently large $t$ and $p \in K$, all geodesics from $p$ to $r_t$ pass through $(Z,T)$. Then, for $p,q \in K$, 
\[
\Ll(p;r_t) - \Ll(q;r_t) = \Ll(p;Z,T) + \Ll(Z,T;r_t) - (\Ll(q;Z,T) + \Ll(Z,T;r_t)).
\]
By Theorems~\ref{thm:DL_SIG_cons_intro}\ref{itm:arb_geod_cons}\ref{itm:weight_of_geod} and~\ref{thm:DL_Buse_summ}\ref{itm:DL_Buse_add}, the right-hand side is equal to 
\[
\W_\dir(p;Z,T) - \W_\dir(q;Z,T) = \W_\dir(p;q). 
\]

\smallskip\noindent
\textbf{Item~\ref{itm:BuseLim2} of Theorem~\ref{thm:DL_Buse_summ} (Busemann limits II):} By Theorem~\ref{thm:DLBusedc_description}\ref{itm:DL_dc_set_count}, on the event $\Omega_2$, $\DLBusedc$ contains no rational directions. Then, for arbitrary $\dir \in \R$, $s \in \R$, $x < y \in \R$, $\alpha,\beta \in \Q$ with $\alpha < \dir < \beta$, and a net $(z_r,u_r)$ with $u_r \to \infty$ and $z_r/u_r \to \dir$, for sufficiently large $r$, $\alpha u_r < z_r < \beta u_r$. Theorem~\ref{thm:DL_Buse_summ}\ref{itm:BuseLim1} gives the existence of the limits in the first and last lines below, while the monotonicity of Lemma~\ref{lem:DL_crossing_facts}\ref{itm:DL_crossing_lemma} justifies the first and last inequalities:
\begin{align*}
\W_\alpha(y,s;x,s) &= \lim_{r \to \infty} \Ll(y,s;\alpha u_r,u_r) - \Ll(x,s;\alpha u_r,u_r) \\
&\le \liminf_{r \to \infty} \Ll(y,s;z_r,u_r) - \Ll(x,s;z_r,u_r) \\ 
&\le \limsup_{r \to \infty} \Ll(y,s;z_r,u_r) - \Ll(x,s;z_r,u_r)\\
&\le \lim_{r \to \infty} \Ll(y,s;\beta u_r,u_r) -\Ll(x,s;\beta u_r,u_r)= \W_{\beta}(y,s;x,s).
\end{align*}
Sending $\Q \ni \alpha \nearrow \dir$ and $\Q \ni \beta \searrow \dir$ and using Item~\ref{itm:DL_unif_Buse_stick} completes the proof.

\smallskip\noindent \textbf{Item~\ref{itm:global_attract} of Theorem~\ref{thm:DL_Buse_summ} (Global attractiveness):}   We follow a similar proof to the attractiveness in Theorem~\ref{thm:invariance_of_SH}. Let $\dir \notin \DLBusedc$ and assume $\h \in \UC$ is a function satisfying the drift condition~\eqref{eqn:drift_assumptions}. Recall that we define
\be \label{hst}
h_{s,t}(x) = \sup_{z \in \R}\{\Ll(x,s;z,t) + \h(z)\}.
\ee
For $a > 0$ and $s < t$, Theorems~\ref{thm:DL_Buse_summ}\ref{itm:DL_unif_Buse_stick} and~\ref{thm:DLBusedc_description}\ref{itm:DL_dc_set_count} allows us to choose $\ve = \ve(\dir) > 0$ small enough  so that $\dir \pm 2\ve \in \Q $ (and thus $\dir \pm 2\ve \notin \DLBusedc$),  and so for all $x \in [-a,a]$,
\be \label{pmeq}
\W_{\dir \pm 2\ve}(x,s;0,s) = \W_\dir(x,s;0,s).
\ee
By Theorem~\ref{thm:DL_all_coal}\ref{itm:unif_coal}, there exists a random $T = T(a,\dir \pm \ve)$ such that all $\dir - 2\ve$ Busemann geodesics have coalesced by time $T$ and all $\dir + 2\ve$ Busemann geodesics have coalesced by time $T$. For $t > T$, let $g^{\dir \pm 2\ve}(t)$ be locations of these two common geodesics at time $t$. By Theorem~\ref{thm:DL_SIG_cons_intro}\ref{itm:arb_geod_cons}\ref{itm:geo_dir}, $g^{\dir \pm 2\ve}(t)/t \to \dir \pm 2\ve$. By the reflected version of Equation~\eqref{downexit} in Lemma~\ref{lem:unq}, there exists $t_0(a,\ve(\dir),s)$ so that for $t > t_0$, whenever $x \in [-a,a]$ and $z$ is a maximizer in~\eqref{hst},
$
 g^{\dir - 2\ve}(t) < z < g^{\dir + 2\ve}(t).
$
Then, by Lemma~\ref{lem:DL_crossing_facts}\ref{itm:KPZ_crossing_lemma}, for such large $t$,
\[
\W_{\dir - 2\ve}(x,s;0,s) \le  h_{s,t}(x) - h_{s,t}(0) \le \W_{\dir + 2\ve}(x,s;0,s),
\]
while for $-a \le x \le 0$, the equalities reverse. Combined with~\eqref{pmeq}, this completes the proof. 

\smallskip\noindent \textbf{Item~\ref{itm:stationarity} of Theorem~\ref{thm:Buse_dist_intro} (Mixing):}  This proof follows a similar idea as that in Lemma 7.5 of~\cite{Bakhtin-Cator-Konstantin-2014}, and the key is that, within a compact set, the Busemann functions are equal to differences of the directed landscape for large enough $t$. Then, we use Lemma~\ref{lm:horiz_shift_mix}, which states that, as a projection of $\{\Ll,\W\}$, the directed landscape $\Ll$ is mixing under the transformation $T_{z;a,b}$.  
Set $r_z = (az,bz)$.  By a standard $\pi-\lambda$ argument, it suffices to show that for $\dir_1,\ldots\dir_k \in \R$ (ignoring the sign $\sigg$ since $\dir_i \notin \DLBusedc$ a.s.), all compact sets $K := K_1 \times K_2^k \subseteq \Rup \times (\R^4)^k$, and all Borel sets $A,B \in C(K,\R)$,
\begin{align*}
&\lim_{z \to \infty} \Pp\Bigl(\{\Ll, \W_{\dir_{1:k}}\}\big|_K \in A, \{T_{z;a,b} \Ll, T_{z;a,b}\W_{\dir_{1:k}}\}\big|_K \in B\Bigr) \\
&\qquad\qquad= \Pp\bigl( \{\Ll, \W_{\dir_{1:k}}\}\big|_K \in A\bigr) \Pp\bigl(\{\Ll, \W_{\dir_{1:k}}\}\big|_K \in B \bigr),
\end{align*}
where we use the shorthand notation 
\[
\{\Ll, \W_{\dir_{1:k}}\}\big|_K := \{\Ll(v), \W_{\dir_i}(p;q):1 \le i \le k,(v,p,q) \in K\},
\]
and $T_{z;a,b}$ acts on $\Ll$ and $\W$ as projections of $\{\Ll,\W\}$.
By Theorem~\ref{thm:DL_Buse_summ}\ref{itm:BuseLim1}, we may choose $t > 0$ sufficiently large so that 
\be \label{busc1}
\Pp(\W_{\dir_i}(p;q) = \Ll(p;(t\dir,t)) - \Ll(q;(t\dir,t)) \;\forall (p,q) \in K_2, 1 \le i \le k) \ge 1 - \ve. 
\ee
By stationarity of the process under space-time shifts, we also have that for such large $t$ and all $z \in \R$, 
\be \label{busc2}
\Pp(T_{z;a,b} \W_{\dir_i}(p;q) = T_{z;a,b} [\Ll(p ;(t\dir,t)) - \Ll(q;(t\dir,t))] \;\forall (p,q) \in K_2, 1 \le i \le k) \ge 1 - \ve
\ee
Let $C_{z,t}$ be the intersection of the events in~\eqref{busc1} over $1 \le i \le k$ with the event~\eqref{busc2}.  Then for large enough $t$, $\Pp(C_{z,t}) \ge 1 - 2\ve$, and
\begin{align*}
    &\Bigl|\Pp\Bigl(\{\Ll, \W_{\dir_{1:k}}\}|_K \in A, \{T_{z;a,b} \Ll, T_{z;a,b}\W_{\dir_{1:k}}\}|_K \in B\Bigr)  \\
    &\qquad- \Pp\Bigl( \{\Ll, \W_{\dir_{1:k}}\}|_K \in A\Bigr) \Pp\Bigl(\{\Ll, \W_{\dir_{1:k}}\}|_K \in B \Bigr) \Bigr|  \\
    &\le \Bigl|\Pp\Bigl(\{\Ll, \W_{\dir_{1:k}}\}|_K \in A, \{T_{z;a,b} \Ll, T_{z;a,b}\W_{\dir_{1:k}}\}|_K \in B, C_{z,t}\Bigr) \\
    &\qquad -\Pp\Bigl( \{\Ll, \W_{\dir_{1:k}}\}|_K \in A,C_{z,t}\Bigr) \Pp\Bigl(\{\Ll, \W_{\dir_{1:k}}\}|_K \in B,C_{z,t} \Bigr) \Bigr| + C\ve \\
    &= \Bigl|\Pp\Bigl(\{\Ll(v), \Ll(p;(t\dir_{1:k},t)) - \Ll(q;(t\dir_{1:k},t))\}|_K \in A, \\
    &\qquad\qquad \{T_{z;a,b} \Ll(v), T_{z;a,b}[\Ll(p;(t\dir_{1:k},t)) - \Ll(q;(t\dir_{1:k},t))]\}|_K \in B, C_{z,t}\Bigr) \\
    &\qquad -\Pp\Bigl( \{\Ll(v), \Ll(p;(t\dir_{1:k},t)) - \Ll(q;(t\dir_{1:k},t))\}|_K \in A,C_{z,t}\Bigr) \\ &\qquad\qquad\times \Pp\Bigl(\{\Ll(v), \Ll(p;(t\dir_{1:k},t)) - \Ll(q;(t\dir_{1:k},t))\}|_K \in B,C_{z,t} \Bigr) \Bigr| + C\ve  \\
    &\le \Bigl|\Pp\Bigl(\{\Ll(v), \Ll(p;(t\dir_{1:k},t)) - \Ll(q;(t\dir_{1:k},t))\}|_K \in A, \\
    &\qquad\qquad \{T_{z;a,b} \Ll(v), T_{z;a,b}[\Ll(p;(t\dir_{1:k},t)) - \Ll(q;(t\dir_{1:k},t))]\}|_K \in B\Bigr) \\
    &\qquad -\Pp\Bigl( \{\Ll(v), \Ll(p;(t\dir_{1:k},t)) - \Ll(q;(t\dir_{1:k},t))\}|_K \in A\Bigr) \\ &\qquad\qquad\times \Pp\Bigl(\{\Ll(v), \Ll(p;(t\dir_{1:k},t)) - \Ll(q;(t\dir_{1:k},t))\}|_K \in B \Bigr) \Bigr| + C'\ve,
\end{align*}
where the constants $C,C'$ came as the cost of adding and removing the high probability event $C_{z,t}$. The proof is complete by sending $z \to \infty$ and using the mixing of $\Ll$ under the shift $T_{z;a,b}$ (Lemma \ref{lm:horiz_shift_mix}).
\end{proof}

\begin{proof}[Proof of Theorem~\ref{thm:DLSIG_main}]
\textbf{Item~\ref{itm:all_dir} (All geodesics have a direction):} First, we show that, on $\Omega_2$, if $g$ is a semi-infinite geodesic starting from $(x,s)$, then
\be \label{594}
-\infty < \liminf_{t \to \infty} t^{-1}{g(t)} \le \limsup_{t \to \infty} t^{-1}{g(t)} < \infty.
\ee
We show the rightmost inequality, the leftmost being analogous. Assume, by way of contradiction, that $\limsup_{t \to \infty} g(t)/t = \infty$. By the directedness of Theorem~\ref{thm:DL_SIG_cons_intro}\ref{itm:DL_all_SIG}, $\forall \dir \in \R$ there exists an infinite sequence $t_i \to \infty$ such that  $g(t_i) > g_{(x,s)}^{\dir +,L}(t_i)$ for all $i$. Since $g_{(x,s)}^{\dir +,L}$ is the leftmost geodesic between any two of its points (Theorem~\ref{thm:DL_SIG_cons_intro}\ref{itm:DL_LRmost_geod}), we must have $g(t) \ge g_{(x,s)}^{\dir+,L}(t)$ $\forall \dir \in \R$ and  $t \in \R$. By Theorem~\ref{thm:g_basic_prop}\ref{itm:limits_to_inf}, $g(t) = \infty$ $\forall t > s$, a contradiction. 

Having established~\eqref{594},   assume by way of contradiction that
\[
\liminf_{t \to \infty} t^{-1} {g(t)} < \limsup_{t \to \infty} t^{-1}{g(t)}.
\]
Choose some $\dir$ strictly between the two values above. By the directedness of Theorem~\ref{thm:DL_SIG_cons_intro}\ref{itm:DL_all_SIG}, there exists a sequence $t_i \to \infty$ such that $g_{(x,s)}^{\dir +,R}(t_i) < g(t_i)$ for $i$ even and $g_{(x,s)}^{\dir +,R}(t_i) > g(t_i)$ for $i$ odd. This cannot occur since  $g_{(x,s)}^{\dir +,R}$ is the rightmost geodesic between any two of its points.

By  Theorem~\ref{thm:DL_SIG_cons_intro}\ref{itm:DL_all_SIG}, for each $\dir \in \R$ and $(x,s) \in \R^2$, $g_{(x,s)}^{\dir +,R}$, for example, is a semi-infinite geodesic from $(x,s)$ in direction $\dir$, justifying the claim that there is at least one semi-infinite geodesic from each point and in every direction. 

\smallskip\noindent \textbf{Item~\ref{itm:good_dir_coal} (Coalescence):} The first statement  follows from the equivalences \ref{itm:DL_good_dir}$\Leftrightarrow$\ref{itm:DL_good_dir_coal}$\Leftrightarrow$\ref{itm:DL_good_dir_unique_geod} of Theorem~\ref{thm:DL_good_dir_classification}. 
By Theorem~\ref{thm:DLNU}\ref{itm:DL_NU_p0}, $\Pp(p \in \NU_0)=0$ $\forall p \in \R^2$. This and Fubini's theorem imply that the set $\NU_0$ almost surely has planar Lebesgue measure zero. 


\smallskip\noindent \textbf{Item~\ref{itm:bad_dir_split} (Non-uniqueness in exceptional directions):} This follows from Remark~\ref{rmk:split_from_all_p}. 
\end{proof}

\section{Random measures and their supports}\label{sec:meas_supp}
This section studies further the   points with disjoint geodesics in the same direction, discussed in Theorem~\ref{thm:Split_pts} and Remark~\ref{rmk:supports}. Recall the  functions $f_{s,\dir}(x) = \W_{\dir +}(x,s;0,s) - \W_{\dir -}(x,s;0,s)$ defined in~\eqref{fsdir} and the sets $\Split_{s,\dir}$ from \eqref{Split_sdir}:
\be\label{eqn:gen_split_set4}\begin{aligned}
    \Split_{s,\dir} &:= \{x \in \R:  \exists \text{ 
    \textbf{disjoint}}   \text{   }\text{semi-infinite  geodesics from  }(x,s) \text{ in direction }\dir\}\\
    \Split &:= \bigcup_{s \tspb\in\tspb \R, \, \dir\tspb \in\tspb \DLBusedc} \Split_{s,\dir} \times \{s\}. 
    \end{aligned}\ee

Each  $\dir \in \R$ is a direction of discontinuity with probability zero.  Conditioning   on   $\dir \in \DLBusedc$ is done through the  Palm kernel  from the theory of random measures  (see   \cite{Kallenberg-book} for background).  The next theorem is proved   in Section~\ref{sec:Palm}, together with a study of the random point process $\{(\tau_\dir, \dir)\}_{\dir\tspb\in \tspb\DLBusedc}$. The Palm conditioning is made precise in Theorems \ref{thm:Lac} and \ref{thm:indep_loc}.

\begin{theorem} \label{thm:BusePalm}
For $\dir \in \R$   consider the random function $f_\dir := f_{0,\dir}$ from~\eqref{fsdir}. Let
\[
\tau_\dir = \inf\{x > 0: f_{\dir}(x) > 0 \}\qquad\text{and}\qquad \bck{{\tau}_\dir} = \inf\{x > 0: -f_{\dir}(-x) > 0\}
\]
denote the points to the right and left of the origin beyond which $\W_{\dir +}(\aabullet,0;0,0)$ and $\W_{\dir -}(\aabullet,0;0,0)$ separate, if ever. Then,  conditionally on $\dir \in \DLBusedc$ in the appropriate Palm sense, the restarted functions 
\[
x\mapsto f_{\dir}(x + \tau_\dir) -f_{\dir}(\tau_\dir)
\quad\text{ and } \quad 
x \mapsto -f_{\dir}(-x - \bck{\tau_{\dir}}) + f_{\dir}(-\bck{\tau_\dir}),\quad x \in\R_{\ge0},  
\]
 are equal in distribution to two independent running maximums of Brownian motion with diffusivity $2$ and zero drift. In particular, they are equal in distribution to two independent appropriately normalized versions of Brownian local time. See Figure~\ref{fig:loc_time}.
\end{theorem}

\begin{figure}[t]
\centering
\includegraphics[height = 3in]{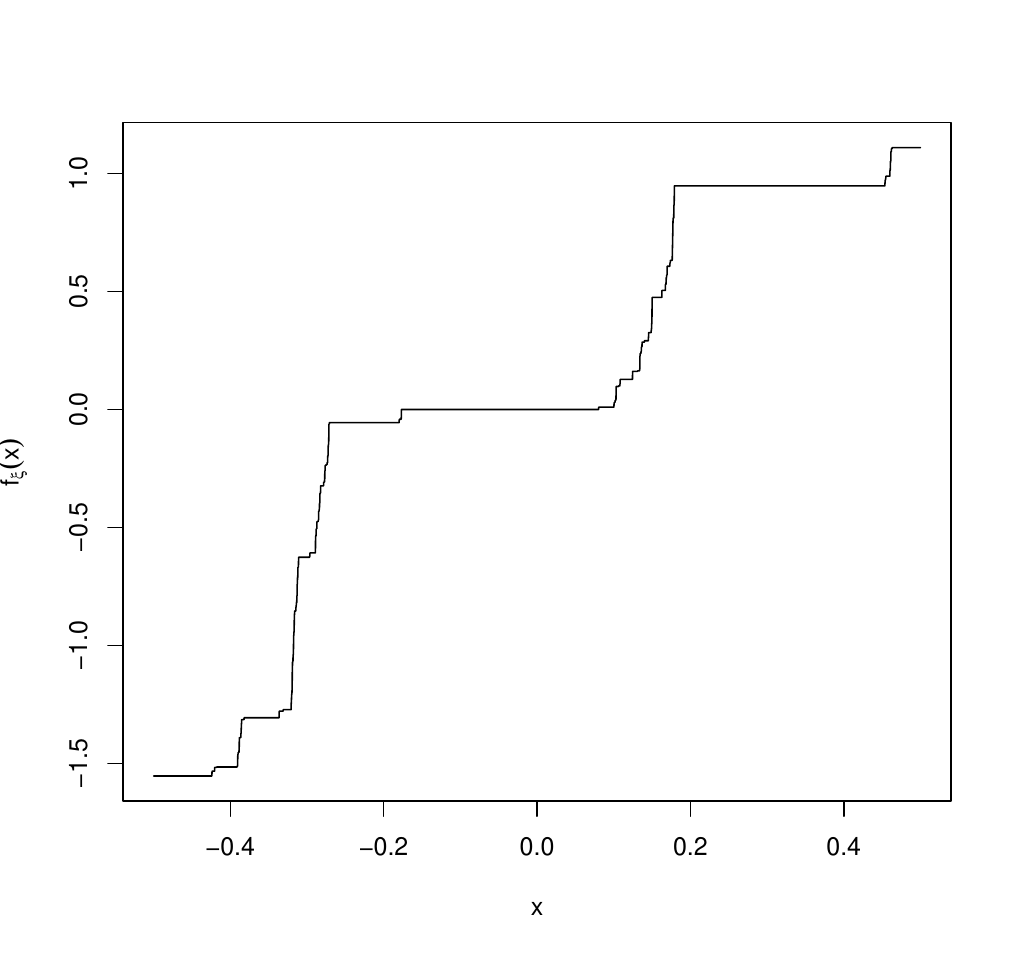}
\caption{\small The Busemann difference profile $f_\dir(x)$. The function vanishes   in a nondegenerate  random neighborhood 
of $x = 0$ and evolves as two independent Brownian local times to the left and right {\rm(}Theorem \ref{thm:BusePalm}{\rm)}.} 
\label{fig:loc_time}
\end{figure}

As described in the next  theorem, $\Split_{s,\dir}$ is the support of a random measure, up to the removal of an at most  countable set.

\begin{theorem} \label{thm:random_supp}

On a single event of full probability, the function $f_{s,\dir}$ is nondecreasing simultaneously for all $s \in \R$ and $\dir \in \DLBusedc$.  Denote the set of local variation of $f_{s,\dir}$  by  
\be \label{Dsdir}  \D_{s,\dir} =  \{ x\in\R:  f_{s,\dir}(x - \ve) < f_{s,\dir}(x + \ve)\; \forall \ve > 0 \}.
\ee
Then, on a single event of full probability, simultaneously for each $s \in \R$ and $\dir \in \DLBusedc$,
\be \label{eqn:supp_set}
\D_{s,\dir} = \Split_{s,\dir}^L \cup \Split_{s,\dir}^R \;\subseteq  \;\Split_{s,\dir},
\ee
where for $S \in \{L,R\}$,
\be \label{eqn:split_LR_sdir}
\Split_{s,\dir}^S := \{x \in \R: g_{(x,s)}^{\dir -,S} \text{ and } g_{(x,s)}^{\dir +,S} \text{ are disjoint}\}.
\ee
 $(\Split_{s,\dir} \setminus \D_{s,\dir}) \times \{s\}$ is contained in the at most countable set $\NU_1^{\dir -} \cap \tspb\NU_1^{\dir +} \cap \,\Hh_s$. 
\end{theorem}
\begin{remark} \label{rmk:NUsupp}
Presently, we do not know if $\D_{s,\dir}$ equals $\Split_{s,\dir}$.  Since  $\NU_1^{\dir -} \cap \NU_1^{\dir +} \subseteq \NU_1$ and  $\NU_1 \cap \,\Hh_s$ is at most countable (Theorem~\ref{thm:DLNU}\ref{itm:DL_NU_count}), $\Split_{s,\dir}$ and $\D_{s,\dir}$ have the same Hausdorff dimension for all $s \in \R$ and $\dir \in \DLBusedc$. 
In Section~\ref{sec:last_proofs} we prove that this Hausdorff dimension is $\f{1}{2}$ on an $s$-dependent probability one event (as Theorem~\ref{thm:Split_pts}\ref{itm:Hasudorff1/2}). 
\end{remark}

The remainder of this section develops the theory needed to prove Theorems~\ref{thm:BusePalm} and~\ref{thm:random_supp} and ultimately Theorem~\ref{thm:Split_pts}. Sections~\ref{sec:Palm} and~\ref{sec:decoup} develop the Palm kernel theory necessary for Theorem~\ref{thm:BusePalm}. The proofs of Theorems~\ref{thm:BusePalm},~\ref{thm:random_supp},~\ref{thm:Split_pts} are in Section~\ref{sec:last_proofs}, along with the unfinished business of Theorem~\ref{thm:DLBusedc_description}\ref{itm:Busedc_t}.


\subsection{Random measures and Palm kernels} \label{sec:Palm}
 To study Palm conditioning, we represent the Busemann process $\{\W_{\dir+}(\aabullet,0,0,0)\}_{\dir\in\R}$ by the stationary horizon $\{G_\dir(\aabullet)\}_{\dir\in\R}$, as permitted by Theorem \ref{thm:Buse_dist_intro}\ref{itm:SH_Buse_process}. 
Define the process of jumps 
\[
	\shdif := \{\shdif_\dir\}_{\dir\tsp\in\tsp\R} = \{G_\dir-G_{\dir-}\}_{\dir\tsp\in\tsp\R}
\]
where $G_{\dir -} = \lim_{\alpha \nearrow \dir} G_\alpha$.  Either  
$\shdif_\dir$ vanishes identically or $\shdif_\dir$ is a   nondecreasing continuous function that vanishes in a nondegenerate (random) neighborhood of the origin.  By a combination of Theorem \ref{thm:SH10}\ref{itm:SH_sc}--\ref{itm:SH_sc2}, 
\be\label{H65}
\{\shdif_{\dir+\eta}(y+x)-\shdif_{\dir+\eta}(y): x \in \R\}_{\dir \in \R}\;\deq\;\{\shdif_{\dir}(x):x \in \R\}_{\dir \in \R}  \quad \forall\tsp y, \eta\in\R.  
\ee
We study the functions $\shdif_\dir(x)$ first for $x\ge0$.  Approximate $\shdif$ by a process $\shdif^{N}$ defined on dyadic rational $\dir$. For  $N\in \Z_{>0}$ let 
\begin{equation}\label{Hd} 
	\shdif^{N}_{\dir_i}=G_{\dir_i}-G_{\dir_{i-1}} \qquad\text{for } \  \dir_i=\dir^N_i=i2^{-N} \ \text{ and } \  i\in \Z.
\end{equation}For $i \in \Z$, let
\begin{equation}\label{taui}
	\tau_{\dir_i}^{N}= \inf\{x > 0: \shdif^{N}_{\dir_i}(x) > 0\}.
\end{equation}
Since the $G_{\dir_i}$ have different drifts for different values of $i$, $\tau_{\dir_i}^{N}<\infty$ almost surely. For $f\in C(\R)$ and $\tau \in \R$,   let 
$[f]^{\tau}\in C(\R_{\ge0})$ denote the restarted function
\begin{equation}
[f]^{\tau}(x) =f(\tau+x) -f(\tau) \  \text{ for }  x\in [0,\infty) .
\end{equation}

Denote by $\mathcal{D}^{\alpha}$
the distribution on $C(\R_{\ge0})$  of  the running maximum of a Brownian motion with drift $\alpha\in\R$ and diffusivity $2$. That is, if $X$ denotes standard Brownian motion, then 
\[   \mathcal{D}^{\alpha}(A)  =  \Pp\bigl\{  \bigl[ \sup_{0\leq u\leq s} 2X(u)+\alpha u\bigr]_{s\in[0,\infty)} \in A  \bigr\}  
\]
for Borel sets $A\subset C(\R_{\ge0})$.  When the drift vanishes  ($\alpha=0$)  we  abbreviate  $\mathcal{D}=\mathcal{D}^0$.

\begin{lemma} \label{lem:WBSM}
Let $B^\alpha = \{B^\alpha(x): x \ge 0\}$ be a Brownian motion with drift $\alpha$ and diffusivity $2$. Let $W$ be an almost surely negative random variable independent of $B^\alpha$. Let 
\[\theta = \inf\{x > 0:  W+B^{\alpha}(x) \ge 0 \}.
\]
Then, for all $x > 0$,
\be \label{condD}
\Pp\Big(\Big[\sup_{0\leq s\leq \theta+u} W+B^{\alpha}(s)\Big]^+_{u\in[0,\infty)}\in \aabullet \,\Big|\,\theta=x\Big) = \D^\alpha(\aabullet).
\ee
In particular, 
\be \label{uncondD}
\Pp\Big(\Big[\sup_{0\leq s\leq \theta+u} W+B^{\alpha}(s)\Big]^+_{u\in[0,\infty)}\in \aabullet\Big) = \D^\alpha(\aabullet)
\ee
\end{lemma}
\begin{proof}

Let $A\in\mathcal B(C(\R_{\ge0}))$ and  $\theta>0$.  Below, notice that  $B^{\alpha}(\theta)=-W$. Then, noting that $\theta$ is a stopping time with respect to the filtration $\mathcal{F}_y=\sigma\big(W,\{B^{\alpha}(x)\}_{x\in[0,y]}\big)$, we use   the strong Markov property to restart at time $\theta$. 
	\begin{equation}\label{Oh2}
	\begin{aligned}
	&\quad\; \Pp\Big(\Big[\sup_{0\leq s\leq \theta+u} W+B^{\alpha}(s)\Big]^+_{u\in[0,\infty)}\in A \,\Big|\,\theta=x\Big)\\
	&=\Pp\Big( \Bigl[\,\sup_{\theta\leq s\leq \theta+u} W+B^{\alpha}(s)\Bigr]_{u\in[0,\infty)} \in A \,\Big|\,\theta=x\Big)\\
	&=\Pp\Big( \Bigl[\,\sup_{0\leq s\leq u} B^{\alpha}(\theta+s)- B^{\alpha}(\theta)\Bigr]_{u\in[0,\infty)} \in A \,\Big|\,\theta=x\Big)\\
	&=\Pp\Big( \Bigl[\,\sup_{0 \leq s\leq u} B^{\alpha}(s)\Bigr]_{u\in[0,\infty)} \in A\Big)=\mathcal{D}^{\alpha}(A).
	\end{aligned}
	\end{equation}
	The claim of~\eqref{condD} has now been verified. Equation~\eqref{uncondD} follows. 
\end{proof}

\begin{corollary} \label{cor:discrete_restart}
	Let  $\alpha_N=2^{-N+1}$. Then for all $i\in \Z$ and  $x>0$,  
	\begin{equation}\label{eq4}
	\Pp\big(\big[\shdif^{N}_{\dir_i}\big]^{\tau^{N}_{\dir_i}}\in \aabullet\,\,\big|\,\tau^{N}_{\dir_i}=x\big)= \mathcal{D}^{\alpha_N}(\aabullet) .
	\end{equation}
\end{corollary}
\begin{proof}
	From the definition of the stationary horizon (Definition \ref{def:SH}) one can deduce that,   for each $i\in\Z$, the process $\shdif^{N}_{\dir_i}$  
	has the same distribution as the process
	\begin{equation} \label{W+B}
	\wt{J}^{N}(y)=\Big[\sup_{0\leq x\leq y} W+B^{\alpha_N}(x)\Big]^+ 
	\end{equation}
	where $B^{\alpha_N}$ is a Brownian motion with drift $\alpha_N=2^{-N+1}$ and diffusivity $2$, and  $W$ is an almost surely  negative random variable   independent of $B^{\alpha_N}$. Define
	\begin{equation}\label{thetaN}
	\theta^N = \inf\{x > 0: \wt J^{N}(x) > 0 \}= \inf\{x > 0:  W+B^{\alpha_N}(x) \ge 0 \}.
	\end{equation}
	Hence, now $(\shdif^{N}_{\dir_i},\tau^{N}_{\dir_i})\deq (\wt{J}^{N},\theta^N)$, and the result follows from Lemma~\ref{lem:WBSM}.
\end{proof}

For $\dir\in \R$ let  
\begin{equation}\label{eq9}
\tau_\dir=\inf\{x\ge 0:\shdif_\dir(x) > 0\}. 
\end{equation}
The connection with the discrete counterpart in \eqref{taui} is 
\be\label{tau45}  
\tau^N_{\dir_i}= \min\{ \tau_\dir:  \dir\in(\dir_{i-1}, \dir_i]\}.  
\ee
On the space  $\R_{\ge0}\times \R$ define the random  point measure and its  mean measure  
\begin{equation}\label{SHpp}  
	\SHpp=\sum_{(\tau_\dir,\dir):\tau_{\dir}<\infty} \delta_{(\tau_\dir,\dir)}
	\quad\text{ and } \quad 
	\lambda_{\SHpp}(\aabullet):=
	\E[ \tspb{\SHpp}(\aabullet)\tspb]. 
\end{equation}
The point process $\SHpp$ records the jump  directions $\dir$  and the points $\tau_\dir$ where $G_\dir$ and  $G_{\dir-}$ separate on $\R_{\ge0}$.   Theorem~\ref{thm:SH10}\ref{itm:SH_j} ensures that $\SHpp$ and $\lambda_{\SHpp}$ are locally finite.  
It will cause no confusion to use the same  symbol $\Gamma$ to denote the random set:  \[ 
	\SHpp=\{(\tau_\dir,\dir): \dir\in\R, \tau_{\dir}<\infty\} . 
\] 
Then also $\lambda_{\SHpp}(\aabullet)=
	\E(\tspb|\SHpp\cap \aabullet|\tspb)$ where $| \aabullet |$ denotes cardinality.  
The counterparts for the approximating process are 
\begin{equation} \label{SHppdis}
\SHpp^{(N)}=\{(\tau^{N}_{\dir_i},\dir_i):i\in \Z, \tau^{N}_{\dir_i}<\infty\}
\quad\text{and}\quad 
\lambda_{\SHpp}^{(N)}(\aabullet):=\mathbb{E}(|\SHpp^{(N)}\cap  \aabullet|).
\end{equation}

The dyadic partition in \eqref{Hd} imposes a certain monotonicity as $N$ increases: $\tau_\xi$ values can be added but not removed. The $\xi$-coordinates that are not dyadic rationals  move as the partition refines.  So we have 
\be\label{tau67} 
\{  \tau^{N}_{\dir_i} : (\tau^{N}_{\dir_i},\dir_i) \in \SHpp^{(N)}\} 
\subset 
\{  \tau^{N+1}_{\dir_i} : (\tau^{N+1}_{\dir_i},\dir_i) \in \SHpp^{(N+1)}\}
\subset 
\{  \tau_{\dir} : (\tau_{\dir},\dir) \in \SHpp\}.  
\ee


\begin{lemma}\label{lm:ac}
	The measure $\lambda_{\SHpp}$ and Lebesgue measure $m$ are mutually absolutely continuous on $\R_{>0}\times\R$.  
	The Radon-Nikodym derivative is given by 
	\be\label{RN128} 
	\f{d\lambda_{\SHpp}}{dm}(\tau,\dir) = \sqrt{\f{2}{\pi \tau}} \qquad \text{for } \ (\tau,\dir)\in\R_{>0}\times\R .
	\ee
\end{lemma}
\begin{proof}
From Theorem~\ref{thm:SH10}\ref{itm:exp}, for $\dir \in \R, \tau > 0$, and $\delta > 0$, 
\[ 
\lambda_{\SHpp}\big((\tau,\tau+\delta]\times [\xi-\delta,\xi+\delta]\big) = 4\sqrt{\f{2}{\pi}}\tspb\delta\tspb\bigl(\sqrt{\tau + \delta} - \sqrt{\tau}\tspc\bigr) = \int_{\dir - \delta}^{\dir + \delta} \int_{\tau}^{\tau + \delta} \sqrt{\f{2}{\pi x}}\,dx \,d\alpha.  \qedhere
 \] 
\end{proof}


By \eqref{RN128},  $\lambda_{\SHpp}$ does not have a finite marginal on the $\dir$-component, as expected since the jump directions are dense. Hence, below we do Palm conditioning on the pair $(\tau_{\dir},\dir)\in\R_{>0}\times\XiSH$ and not on the jump directions $\dir\in\XiSH$ alone.

\begin{lemma}\label{lm:SH5}  Let  $A \subseteq C(\R_{\ge0})$ be a Borel set. Then for any open rectangle $R=(a,b)\times(c,d)\subseteq \R_{\ge0}\times\R$, 
	\begin{equation}\label{Omr2}
	\mathbb{E}\Bigl[ \; \sum_{(\tau,\,\dir)\,\in\,  \SHpp}\ind_{A}([\shdif_\dir]^\tau)\tspb\ind_{R}(\tau,\dir)\Bigr] =\lambda_{\SHpp}(R)\tspb\mathcal{D}(A). 
	\end{equation}
\end{lemma}
\begin{proof}  It suffices to prove \eqref{Omr2} for  continuity sets $A$ of the distribution $\mathcal{D}$ of the type $ A = \{f\in C(\R_{\ge0}): f|_{[0,k]} \in A_k\}$
for $k > 0$ and Borel $A_k \subseteq C[0,k]$.  
Such sets form a $\pi$-system that generates the Borel $\sigma$-algebra of $C(\R_{\ge0})$.


We prove \eqref{Omr2} for $\shdif^{N}$. Below the values $\dir_i=i2^{-N}$ are not random and hence can come outside the expectation. 	Condition on $\tau^N_{\dir_i}$ and  use \eqref{eq4}:  
	\begin{equation}\label{Omr3}
	\begin{aligned}
	&\E\Big(\sum_{(\tau^N_{\dir_i}\!,\,\dir_i)\in R\cap\SHpp^{(N)}} \ind_A\big([\shdif^{N}_{\dir_i}]^{\tau^{N}_{\dir_i}}\big)\Big)
	=\E\Big(\sum_{\dir_i\in(c,d)} \ind_A\big([\shdif^{N}_{\dir_i}]^{\tau^{N}_{\dir_i}}\big)\ind_{(a,b)}\big(\tau^{N}_{\dir_i}\big)\Big)\\
	&=\sum_{\dir_i\in(c,d)} \E\bigg(\ind_{(a,b)}\big(\tau^{N}_{\dir_i}\big)\, \E\Big[\ind_A\big([\shdif^{N}_{\dir_i}]^{\tau^{N}_{\dir_i}}\big)\tspb\Big|\tspb\tau^N_{\dir_i}\Big]\bigg) \\&
	\stackrel{\eqref{eq4}}{=}\sum_{\dir_i\in(c,d)}\Pp\big(\tau^{N}_{\dir_i}\in (a,b)\big)\mathcal{D}^{\alpha_N}(A)
	=\mathcal{D}^{\alpha_N}(A)\tspb \lambda_{\SHpp}^{(N)}(R).
	\end{aligned}
	\end{equation}
	
	To conclude the proof, we check that \eqref{Omr2} arises as we  let $N\to\infty$ in the first and last member of the string of equalities above.  $\mathcal{D}^{\alpha_N}(A)\to\mathcal{D}(A)$ by the  continuity of $\alpha\mapsto \mathcal{D}^\alpha$ in the weak topology  and the assumption that $A$ is a continuity set.

	As an intermediate step, we verify that  $\forall k > 0$, 
	$\ind_{\mathcal{U}_N^k}\to 1$ almost surely 
	for the events  
	\begin{align}\label{Uset}  
	    \mathcal{U}_N^k&=\bigl\{\,|\SHpp^{(N)}\cap R| = |\SHpp\cap R| \text{ \,and for every $(\tau,\dir)\in \SHpp\cap R$  there is a unique} \\
	    &\qquad 
	    \text{  $(\tau^N_{\dir_i},\dir_i)\in \SHpp^{(N)}\cap R$
	    such that  $[\shdif^{N}_{\dir_i}]^{\tau^{N}_{\dir_i}}\big|_{[0,k]}=[\shdif_{\dir}]^{\tau_{\dir}}\big|_{[0,k]}$}\bigr\}. \nonumber
	\end{align}
	Almost surely,  $\SHpp\cap R$ is finite and none of its points lie on the boundary of $R$.  For any such realization, the condition in braces holds when (i) all points $(\tau_\xi, \xi)\in\SHpp\cap R$ lie in distinct rectangles $(a,b)\times(\xi_{i-1}, \xi_i]\subset(a,b)\times(c,d)$, (ii) when no point $(\tau^N_{\dir_i},\dir_i)\in \SHpp^{(N)}\cap R$ is generated by a point $(\tau_\xi, \xi)\in\SHpp$ outside $R$, and (iii) when $N$ is large enough so that for the unique $i$ with $\dir_i < \dir \le \dir_{i + 1}$, $G_{\dir-}(x) = G_{\dir_i}(x)$ and $G_{\dir+}(x) = G_{\dir_{i + 1}}(x)$ for all $x \in [0,\tau_\dir + k]$. By Theorem~\ref{thm:SH10}\ref{itm:SH_j}, this happens for all the finitely many $(\tau,\dir) \in \Gamma \cap R$ when the mesh $2^{-N}$ is fine enough.  Thus, for each  $k > 0$, almost every realization lies eventually in $\mathcal{U}_N^k$.

\smallskip

We prove that 
	$\lambda_{\SHpp}^{(N)}(R)\to 		\lambda_{\SHpp}(R)$.   The paragraph above gave $|\SHpp^{(N)}\cap R| \to  |\SHpp\cap R|$ almost surely. We also have  $|\SHpp^{(N)}\cap R| \le   |\SHpp\cap ((a,b)\times(c-1,d))|$ because \eqref{tau45} shows that each  point $(\tau^N_{\dir_i},\dir_i)$ that is not matched to a unique point $(\tau_\dir, \dir)\in\SHpp\cap  R$ must be generated by some point $(\tau_\dir, \dir)\in \SHpp\cap  ((a,b)\times(c-1,d))$.  The limit $\lambda_{\SHpp}^{(N)}(R)\to 		\lambda_{\SHpp}(R)$ comes now from dominated convergence.

	\smallskip
	
	It remains  to show that 
\[ 
	\E\Big(\sum_{(\tau^N_{\dir_i}\!,\,\dir_i)\in R\cap\SHpp^{(N)}} \ind_{ A}([\shdif^{N}_{\dir_i}]^{\tau^{N}_{\dir_i}})\Big)
	\underset{N\to\infty}\longrightarrow
	\E\Big(\sum_{(\tau_{\dir},\dir)\in R\cap\SHpp} \ind_{ A}([\shdif_\dir]^{\tau_\dir})\Big).
\] 
This follows by choosing $k > 0$ so that $A$ depends only on the domain $[0,k]$. Then, the difference in absolute values in the display below vanishes on $\mathcal{U}_N^k$.
\[
\begin{aligned}
	    &\lim_{N\to\infty} \E\bigg[ \,\Big|\sum_{(\tau^N_{\dir_i}\!,\,\dir_i)\in R\cap\SHpp^{(N)}} \ind_{ A}\bigl([\shdif^{N}_{\dir_i}]^{\tau^{N}_{\dir_i}}\bigr)-\sum_{(\tau_{\dir},\dir)\in R\cap\SHpp} \ind_{ A}\bigl([\shdif_{\dir}]^{\tau_{\dir}}\bigr)\Big|
	    \cdot (\ind_{\mathcal{U}_N^k}+\ind_{(\mathcal{U}_N^k)^c})\bigg]    \\
&\qquad \qquad \le  \lim_{N\to\infty}  
2\tspb \E\bigl[ \tspb |\SHpp\cap ((a,b)\times(c-1,d))| \cdot \ind_{(\mathcal{U}^k_N)^c} \bigr] =0,	 
	    \end{aligned}
	    \]
     and the last equality follows by dominated convergence.
\end{proof}

\smallskip 

To capture the distribution of $[\shdif_\xi]^{\tau_\xi}$, we augment the point measure  $\SHpp$ of \eqref{SHpp} to a point measure on    the space  $\R_{\ge0}\times \R\times C(\R_{\ge0})$:   
\begin{equation}\label{SHHpp}  
	\SHHpp=\sum_{(\tau_\dir,\,\dir)\, \in \,\SHpp} \delta_{(\tau_\dir,\,\dir,\, [\shdif_\xi]^{{\scaleobj{1.5}{\tau}}_{\!\!\xi}})}. 
\end{equation}
 The \textit{Palm kernel} of $[\shdif_\xi]^{\tau_\xi}$  with respect to $\SHpp$ is the stochastic kernel $Q$  from  $\R_{\ge0}\times \R$ into  $C(\R_{\ge0})$ that satisfies the following identity:  for every bounded Borel function $\Psi$ on $\R_{\ge0}\times \R\times C(\R_{\ge0})$ that is supported on  $B\times C(\R_{\ge0})$ for some bounded Borel set $B\subset \R_{\ge0}\times \R$,  
\begin{equation}\label{Opalm}\begin{aligned} 
\E \sum_{(\tau_\dir,\,\dir)\, \in \, B\tspa \cap\tspa\SHpp} \Psi\bigl(\tau_\dir,\,\dir,\, [\shdif_\xi]^{\tau_\xi}\bigr)     
&\; = \; \E\!\!\int\limits_{\R_{\ge0}\times \R\times C(\R_{\ge0})} \!\! \Psi(\tau, \xi, h) \,\SHHpp(d\tau, d\xi, dh) \\
&= \;  \int\limits_{\R_{\ge0}\times\mathbb{R}} \lambda_{\SHpp}(d\tau, d\xi)
\int\limits_{C(\R_{\ge0})}  Q(\tau, \xi, dh)\, \Psi(\tau, \xi, h)  .
\end{aligned} \end{equation}
The first equality above is a restatement of the definition of $\SHHpp$ and included to make the next proof transparent.   
The key result of this section is this  characterization of $Q$.  

\begin{theorem}\label{thm:Lac}
For 
Lebesgue-almost every $(\tau, \xi)$,  $Q(\tau, \xi, \aabullet)=\mathcal{D}(\aabullet)$,   the distribution of the running maximum of a Brownian motion with diffusivity 2. 
 \end{theorem}
 
 \begin{proof} This  comes from Lemma \ref{lm:SH5}: 
  take  $\Psi(\tau, \xi, h)=\ind_R(\tau, \xi)\ind_A(h)$  in \eqref{Opalm} and note that the left-hand side of \eqref{Omr2} is exactly the left-hand side of \eqref{Opalm}. Lemma \ref{lm:ac} turns   $\lambda_{\SHpp}$-almost everywhere into Lebesgue-almost everywhere. 
 \end{proof}

Denote the set of directions $\dir$ for which $G_\dir$ and $G_{\dir-}$ separate on $\R_{\ge0}$   by  
\begin{equation}\label{XiSHdef} 
	\begin{aligned}
		\XiSH&=\{\dir\in\R:\tau_\dir < \infty\}.
	\end{aligned}
\end{equation}
\begin{theorem}\label{thm:Lac5}
Let $A\subseteq C(\R_{\ge0})$ be a Borel  set such that  $\mathcal{D}(A)=0$. Then
\begin{equation}\label{Oeq}
    \Pp\big(\exists \dir\in \XiSH :  [\shdif_\dir]^{\tau_\dir}\in A\big)=0.
\end{equation}
\end{theorem}
\begin{proof}  Let   $R_N=(0,N)\times(-N,N)$.  Since $\xi\in\XiSH$ means that $\tau_\xi<\infty$, we have 
 \begin{align*} 
    &\Pp\big(\exists \dir\in \XiSH :  [\shdif_\dir]^{\tau_\dir}\in A\big)
    =  \lim_{N\to\infty}  \Pp\big(\exists \dir\in \XiSH : (\tau_\dir,\dir)\in R_N,  [\shdif_\dir]^{\tau_\dir}\in A\big) \\
& \qquad\quad    \leq  \lim_{N\to\infty}  \mathbb{E}\sum_{(\tau,\,\dir)\in  \SHpp}\ind_{A}([\shdif_\dir]^\tau) \tspb\ind_{R_N}(\tau, \xi)   
\overset{\eqref{Omr2}}=  \lim_{N\to\infty}  \lambda_{\SHpp}(R_N) \tspb \mathcal{D}(A) =0.
\qquad \ \qedhere  \end{align*} 
\end{proof}

We show that \eqref{XiSHdef} captures all   $\xi$ at which a jump happens  on the real line. 

\begin{corollary} \label{cor:dcLR}
With probability one, $\XiSH = \{\dir \in \R: \shdif_\dir(x) \neq 0 \text{ for some }x \in \R\}$. Furthermore, for each $\dir \in \XiSH$, $\lim_{x \to \pm \infty} \shdif_\dir(x) = \pm \infty$.
\end{corollary}
\begin{proof}
By Theorem~\ref{thm:Lac5} and the associated fact for the running max of a Brownian motion, 
\be \label{XIshto+inf}
\Pp(\forall \dir \in \XiSH, \lim_{x \to +\infty} \shdif_\dir(x) = +\infty) = 1.
\ee
By definition, $\XiSH = \{\dir \in \R: \shdif_\dir(x) \neq 0 \text{ for some }x >0\}$. 
Now, we show that if $\shdif_\dir(x) \neq 0$ for some $x < 0$, then $\shdif_\dir(x) \neq 0$ for some $x > 0$. If not, then there exist $\dir \in \R$ and $m \in \Z_{<0}$ such that $[\shdif_\dir]^{m}|_{[0,\infty)} \neq 0$, but $[\shdif_\dir]^m|_{[-m,\infty)}$ is constant. In particular, $[\shdif_\dir]^{m}|_{[0,\infty)}$ is bounded. Let $\tau^m_\dir = \inf \{x > 0: [\shdif_\dir]^m(x) > 0\}$.  Then, $[\shdif_\dir]^{m}|_{[0,\infty)} \neq 0$ iff  $\tau^m_\dir<\infty$, and we have
\be \label{XiLR}
\begin{aligned}
    & \Pp\bigl(\XiSH \neq \{\dir \in \R: \shdif_\dir(x) \neq 0 \text{ for some }x \in \R\}\bigr) \\
    &\quad 
    \le \sum_{m \in \Z_{<0}}\Pp\bigl(\exists \dir \in \R:  \; \tau^m_\dir<\infty  
    \text{ but }  [\shdif_\dir]^m|_{[0,\infty)} \text{ is bounded} \bigr) = 0. 
\end{aligned}
\ee
The probability equals zero by~\eqref{XIshto+inf} because by shift invariance~\eqref{H65}, $[\shdif]^m\deq \shdif$.
To finish,~\eqref{XIshto+inf} proves the limits for $x \to +\infty$. The limits as $x \to -\infty$ then follow from~\eqref{XIshto+inf} and the reflection invariance of Corollary~\ref{cor:SH_reflect}.
\end{proof}

Let $\nu_f$ denote the Lebesgue-Stieltjes measure of a non-decreasing function $f$ on $\R$.   Denote the  support of $\nu_f$ by $\text{supp}(\nu_f)$. The Hausdorff dimension of a set $A$ is denoted by $\dim_H(A)$.


\begin{corollary} \label{cor:SHHaus1/2} Consider the Lebesgue-Stieltjes measure $\nu_{\shdif_\dir}$ for $\dir\in \XiSH$ on the entire real line.  Then we have 
\begin{equation}\label{eq8}
	\Pp\big\{\forall \dir\in \XiSH :  \dim_H \big(\text{\rm supp}(\nu_{\shdif_\dir})\big)=1/2\big\}=1.
\end{equation}
\end{corollary}
\begin{proof}
	First, note that 
	\[ 
		\big\{\exists \dir\in \XiSH : \text{dim}_H\big(\text{supp}(\nu_{\shdif_\dir})\big)\neq \tfrac12\big\}\subseteq\!\bigcup_{m\in\Z_{\le0}}\!\!\big\{\exists \dir\in \XiSH : \text{dim}_H\big(\text{supp}(\nu_{\shdif_\dir})\cap[m,\infty)\big)\neq \tfrac12\big\}.
	\]  
By \eqref{H65}, it is enough to take $m=0$ and show that 
 \[
 \Pp\big(\exists \dir\in \XiSH :\text{dim}_H\big(\text{supp}(\nu_{\shdif_\dir})\cap[0,\infty)\big)\neq 1/2\big)=0.
 \]
 This last claim follows from  Theorem \ref{thm:Lac5} because the event in question has zero probability for the running maximum of Brownian motion 
 (\cite{taylor_1955}, see also \cite{morters_peres_2010}, Theorem 4.24 and Exercise 4.12). 
\end{proof}

\begin{remark}
Representation of the  difference of  Busemann functions   as the running maximum of random walk 
goes back to \cite{bala-busa-sepp-20}. It was used  in \cite{busa-ferr-20} to capture the local universality  of geodesics. 
The representation of the difference profile as the running maximum of  Brownian motion in the point-to-point setup emerges from the Pitman transform   \cite{Ganguly-Hegde-2021,Dauvergne-22}.  Theorem 1 and Corollary 2 in~\cite{Ganguly-Hegde-2021} are   point-to-point analogues of  our Theorem \ref{thm:Lac} and Corollary \ref{cor:SHHaus1/2}. Their proof is different from ours. Although an analogue of the Pitman transform exists in the stationary case \cite[Section 3]{Busani-2021}, comparing the running maximum of a Brownian motion to the profile requires different tools in the two settings.  
\end{remark}

\subsection{Decoupling} \label{sec:decoup}  	By Corollary~\ref{cor:dcLR},  whenever $\dir$ is a jump direction, the difference profiles for both positive and negative $x$ are nontrivial.  
We extend   Theorem~\ref{thm:Lac} to show that    these two  difference profiles  are independent and equal in distribution.  
We spell out only the   modifications needed in  the arguments of the previous section.    For the difference profile on the left, define for $x\ge0$
\[ 
		\bck{\shdif}_{\xi}(x):=-\shdif_{\xi}(-x) 
\quad\text{and}\quad 
		\bck{\tau_\xi}:=\inf \{x >  0:\bck{\shdif}_{\xi} > 0\}.
\] 
For $N\in \Z_{>0}$ and  $\xi_i$ as in \eqref{Hd}, the discrete approximations are 
\[ 
	\bck{\shdif}^N_{\xi_i}(x):=-\shdif^N_{\xi_i}(-x)  
	\quad\text{and}\quad 
	\bck{\tau_{\xi_i}}^N:=\inf\{x >  0:\bck{\shdif}^N_{\xi_i}(x) > 0\}.
\] 
The measures $\bck{\SHpp}$, $\lambda_{\bck{\SHpp}}$,  $\bck{\SHpp}^{(N)}$, and $\lambda_{\bck{\SHpp}}^{(N)}$ are defined  as in~\eqref{SHpp} and~\eqref{SHppdis}, but now with   $(\bck{\tau_\dir},\dir)$ and $(\bck{\tau_{\dir_i}},\dir_i)$. 
	Extend the measure $\SHHpp$ of \eqref{SHHpp} with a component for the left profile: 
\[
\SHHpp'=\sum_{(\bck{\tau_\dir},\,\dir)\, \in \,\bck{\SHpp}} \delta_{(\bck{\tau_\dir},\,\dir,\; [\shdif_\xi]^{{\scaleobj{1.5}{\tau}}_{\!\!\dir}}, \; [\bck{\shdif}_\xi]^{\bck{\scaleobj{1.5}{\tau}}_{\!\!\dir}}) }. 
\]
Since  $\tau_\dir < \infty$ if and only if $\bck{\tau_\dir} < \infty$ (Corollary~\ref{cor:dcLR}),  it is immaterial whether we sum over  $({\tau}_\dir,\dir)$ or  $(\bck{\tau_\dir},\dir)$.  The latter is more convenient for the next calculations.

	 The \textit{Palm kernel} of $\big([\shdif_\xi]^{\tau_\xi},[\bck{\shdif}_\xi]^{\bck{\tau_\dir}}\big)$  with respect to $\bck{\SHpp}$ is the stochastic kernel $Q^2$  from  $\R_{\ge0}\times \R$ into  $C(\R_{\ge0})\times C(\R_{\ge0})$ that satisfies the following identity:  for every bounded Borel function $\Psi$ on $\R_{\ge0}\times \R\times C(\R_{\ge0})\times C(\R_{\ge0})$ that is supported on  $B\times C(\R_{\ge0})\times C(\R_{\ge0})$ for some bounded Borel set $B\subset \R_{\ge0}\times \R$,  
	\begin{equation}\label{Opalm2}
	\begin{aligned} 
	&\quad \,\E \Bigl[ \; \sum_{(\bck{\tau_\dir},\,\dir)\, \in \, B\tspa \cap\tspa\bck{\SHpp}} \Psi\bigl(\bck{\tau_\dir},\,\dir,\, [\shdif_\xi]^{\tau_\xi},[\bck{\shdif}_\xi]^{\bck{\tau_\dir}}\bigr)  \Bigr]  \\ 
			&= \int\limits_{\R_{\ge 0} \times \R} \lambda_{\bck{\SHpp}}(d\bck{\tau}, d\xi)
			\int\limits_{C(\R_{\ge0})\times C(\R_{\ge 0})}  Q^2(\bck{\tau}, \xi, dh^1,dh^2)\, \Psi(\bck{\tau}, \xi, h^1,h^2).  
	\end{aligned} 
	\end{equation}

	\begin{theorem} \label{thm:indep_loc}
	For 
	Lebesgue-almost every $(\tau, \xi)$,  $Q^2(\tau, \xi, \aabullet)=(\mathcal{D} \otimes \mathcal{D})(\aabullet)$,   the product of the distribution of the running maximum of a Brownian motion with diffusivity 2. In particular, for any Borel set $A \subseteq C(\R_{\ge 0}) \times C(\R_{\ge 0})$ such that $(\D \otimes \D)(A) = 0$,
	\[
        \Pp\bigl\{ \exists \dir \in \XiSH: \bigl([\shdif_\xi]^{\tau_\xi},[\bck{\shdif}_\xi]^{\bck{\tau_\dir}}\bigr)   \in A\bigr\} = 0.	
	\]
\end{theorem}
\begin{proof}
By definition of the stationary horizon (Definition~\ref{def:SH}), as functions in $C(\R)$,
\be \label{supdif}
\shdif^N_{\dir_i}(y) \deq  \sup_{-\infty < x \le y}\{B^{\alpha_N}(x)\} - \sup_{-\infty < x \le 0}\{B^{\alpha_N}(x)\},
\ee
where $B^{\alpha_N}$ is a two-sided Brownian motion with drift $\alpha_N$ and diffusivity $2$, with $B^{\alpha_N}(0) = 0$. By adjusting our probability space if needed, we will assume that such a process $B^{\alpha_N}$ exists on our space and $\shdif_{\dir_i}$ is given as~\eqref{supdif}. 
Define two independent $\sigma$-algebras
\[
\F_- = \sigma(B^{\alpha_N}(x): x \le 0),\qquad\text{and}\qquad \F_+ = \sigma(B^{\alpha_N}(x): x \ge 0).
\]

When $y > 0$, we may write
\be \label{shdif2}
\shdif_{\dir_i}^N(y) = \Bigl[W + \sup_{0 \le x \le y} B^{\alpha_N}(x)\Bigr]^+,
\ee
where $W = -\sup_{-\infty < x \le 0}\{B^{\alpha_N}(x)\} \in \F_-$, and $\sup_{0 \le x \le y} B^{\alpha_N}(x) \in \F_+$. Then, conditional on $\F_-$, $W$ is constant while the law of $B^{\alpha_N}(x)$ for $x \ge 0$ is unchanged. Then, by~\eqref{shdif2} and Equation~\eqref{uncondD} of Lemma~\ref{lem:WBSM} in the special case where $W$ is constant (using the exact same reasoning as in the proof of Corollary~\ref{cor:discrete_restart}),
\be \label{F-cond}
\Pp(\big[\shdif^{N}_{\dir_i}\big]^{\tau^{N}_{\dir_i}} \in \aabullet\,|\,\F_-) = \D^{\alpha_N}(\aabullet).
\ee

 For a fixed $i$,  $\bck{\shdif}^N_{\xi_i}$ and  $\shdif^N_{\xi_i}$ have the same distribution as functions on $\R$. This comes  by   first applying  Corollary~\ref{cor:SH_reflect} and then \eqref{H65}, shifting the directions by $\xi_{i-1}+\xi_i$: 
\begin{align*}
 \bck{\shdif}^N_{\xi_i}(x) &= -{\shdif}^N_{\xi_i}(-x) = - G_{\xi_i}(-x) +   G_{\xi_{i-1}}(-x) 
 \deq - G_{-\xi_i}(x) +   G_{-\xi_{i-1}}(x)\\
 &\deq  -  G_{\xi_{i-1}}(x) + G_{\xi_i}(x)  
 ={\shdif}^N_{\xi_i}(x) . 
\end{align*}
	By~\eqref{supdif}, $(\bck{ \shdif}_{\dir_i}^N,\bck{\tau}_{\dir_i}^N) \in \F_-$.  We mimic the calculation in \eqref{Omr3}, for two Borel sets $A_1,A_2\subseteq C(\R_{\ge 0})$ and an open rectangle $R=(a,b)\times(c,d)\subseteq \R_{\ge0}\times\R$: 
		\begin{align}\label{Omr4}
			&\quad\;\E\Big(\sum_{(\bck{\tau}^N_{\!\!\dir_i},\;\dir_i)\in R\cap\bck{\SHpp}^{(N)}} \ind_{A_1}\big([\shdif^{N}_{\dir_i}]^{\tau^{N}_{\dir_i}}\big)\ind_{A_2}\big([\bck{\shdif}^{N}_{\dir_i}]^{\bck{\tau}^{N}_{\dir_i}}\big)\Big) \\
			&= \sum_{\dir_i\in(c,d)} \E\bigg(\ind_{A_2}\big([\bck{\shdif}^{N}_{\dir_i}]^{\bck{\tau}^{N}_{\dir_i}}\big)\;\ind_{(a,b)}\big(\bck{\tau}^{N}_{\dir_i}\big) \E\Big[\big(\ind_{A_1}\big([\shdif^{N}_{\dir_i}]^{\tau^{N}_{\dir_i}}\big)\tspb\Big|\F_-\Big]\bigg) \nonumber\\
			&\stackrel{\eqref{F-cond}}{=} \sum_{\dir_i\in(c,d)}  \E\bigg( \E\Big[\ind_{A_2}\big([\bck{\shdif}^{N}_{\dir_i}]^{\bck{\tau}^{N}_{\dir_i}}\big)\;\ind_{(a,b)}\big(\bck{\tau}^{N}_{\dir_i}\big)\tspb\Big|\bck{\tau}_{\dir_i}^N\Big] \bigg)\D^{\alpha_N}(A_1) \nonumber\\
			&= \sum_{\dir_i\in(c,d)}  \E\bigg(\ind_{(a,b)}\big(\bck{\tau}^{N}_{\dir_i}\big) \E\Big[\ind_{A_2}\big([\bck{\shdif}^{N}_{\dir_i}]^{\bck{\tau}^{N}_{\dir_i}}\big)\;\tspb\Big|\bck{\tau}_{\dir_i}^N\Big] \bigg)\D^{\alpha_N}(A_1) \nonumber \\
			&\stackrel{\eqref{eq4}}{=} \sum_{\dir_i\in(c,d)} \Pp\big(\bck{\tau}^{N}_{\dir_i}\in (a,b)\big)\D^{\alpha_N}(A_1)\D^{\alpha_N}(A_2) \nonumber
					\\&
			=\mathcal{D}^{\alpha_N}(A_1)\mathcal{D}^{\alpha_N}(A_2)\tspb \lambda_{\bck{\SHpp}}^{(N)}(R).\nonumber
	\end{align}
As in the proof of Lemma \ref{lm:SH5}, we derive from  the above that 
\begin{equation}\label{Q2}
	\E\Big(\sum_{(\bck{\tau}_{\dir},\dir)\in R\cap\bck{\SHpp}} \ind_{A_1}\big([\shdif_{\dir}]^{\tau_{\dir}}\big)\ind_{A_2}\big([\bck{\shdif}_{\dir}]^{\bck{\tau}_{\dir}}\big)\Big)=\mathcal{D}(A_1)\mathcal{D}(A_2)\tspb \lambda_{\bck{\SHpp}}(R), 
\end{equation}
through the convergence of line~\eqref{Omr4} to the left-hand side of \eqref{Q2}. Instead of the events $\mathcal{U}_N^k$ in \eqref{Uset}, consider 
\begin{align*}  
	\mathcal{\wt U}_N^k&=\bigl\{\,|\bck{\SHpp}^{(N)}\cap R| = |\bck{\SHpp}\cap R|, \text{ \,and   $\forall \tspb(\bck{\tau},\dir)\in \bck{\SHpp}\cap R$, \  $\exists$  unique $(\tau^N_{\dir_i},\dir_i)\in \bck{\SHpp}^{(N)}\cap R$
		} \\
		&\qquad\qquad\text{ such that  $[\shdif^{N}_{\dir_i}]^{\tau^{N}_{\dir_i}}\big|_{[0,k]}=[\shdif_{\dir}]^{\tau_{\dir}}\big|_{[0,k]}$ and $[\bck{\shdif}^{N}_{\dir_i}]^{\bck{\tau}^{N}_{\dir_i}}\big|_{[0,k]}=[\bck{\shdif}_{\dir}]^{\bck{\tau}_{\dir}}\big|_{[0,k]}$}\bigr\}. 
\end{align*}
For each $k>0$,   $\ind_{\mathcal{\wt U}_N^k}\to 1$ almost surely, as it did for \eqref{Uset}.   Indeed, there are finitely many pairs $(\bck{\tau},\dir) \in \bck{\SHpp} \cap R$, and each has a finite forward splitting time $\tau$. All these can be confined  in a common compact rectangle.
From here, the proof 
continues as for  Lemma \ref{lm:SH5} and Theorem~\ref{thm:Lac5}. 
\end{proof}

\subsection{Remaining proofs} \label{sec:last_proofs}
It remains to prove Theorems~\ref{thm:DLBusedc_description}\ref{itm:Busedc_t},~\ref{thm:BusePalm}, and~\ref{thm:Split_pts}. Recall the definition of the function from~\eqref{fsdir}: 
$
f_{s,\dir}(x) = \W_{\dir+}(x,s;0,s) - \W_{\dir -}(x,s;0,s). 
$
\be \label{omega3} \begin{aligned} 
&\text{Let $\Omega_3$ be the subset of $\Omega_2$ on which the following holds: for each $T \in \Z$,}\\ 
&\text{whenever  $\dir \in \R$ is such that $f_{T,\,\dir} \neq 0$, then}   
\lim_{x \to \pm \infty} f_{T,\,\dir}(x) = \pm \infty.
\end{aligned} \ee
By Theorem~\ref{thm:Buse_dist_intro}\ref{itm:SH_Buse_process} and Corollary~\ref{cor:dcLR}, $\Pp(\Omega_3) = 1$. 
\begin{proof}[Proof of Theorem~\ref{thm:DLBusedc_description}\ref{itm:Busedc_t}] We work on the full-probability event  $\Omega_3$.   The statement \eqref{bad_ub} to be proved  is  $\dir \in \DLBusedc \iff \forall s\in\R:\lim_{x \to \pm \infty} f_{s,\dir}(x) = \pm \infty$.
 If, for \textit{any} $s$, $f_{s,\dir} \to \pm \infty$ as $x \to \pm \infty$, then $\W_{\dir-}(x,s;0,s) \neq \W_{\dir +}(x,s;0,s)$ for $|x|$ sufficiently large, and $\dir \in \DLBusedc$. It remains to prove the converse statement.
From~\eqref{881}, 
\[
\DLBusedc = \bigcup_{T \in \Z} \{\dir \in \R: \W_{\dir - }(x,T;0,T) \neq \W_{\dir +}(x,T;0,T) \text{ for some }x \in \R\}.
\]
To finish the proof of~\eqref{bad_ub}, by definition of $\Omega_3$, it suffices to show these two statements:
\begin{enumerate}[label=\rm(\roman{*}), ref=\rm(\roman{*})]  \itemsep=3pt
    \item If $f_{s,\dir} \neq 0$ for some $s,\dir \in \R$ then $f_{T,\,\dir} \neq 0$ for all $T > s$. 
    \item For $T \in \Z, \dir \in \R$, if $f_{T,\,\dir} \neq 0$, then for all $s < T$, $\lim_{x \to \pm \infty} f_{s,\dir}(x) = \pm \infty$.
\end{enumerate}

Part  (i) follows from the equality below.   By~\eqref{880}, for $s < T$,
\be \label{883}\begin{aligned} 
\W_{\dir \sig}(x,s;0,s) &= \sup_{z \in \R}\{\Ll(x,s;z,T) + \W_{\dir \sig}(z,T;0,T)\}\\[-3pt]
&\qquad\qquad 
- \sup_{z \in \R} \{\Ll(0,s;z,T) + \W_{\dir \sig}(z,T;0,T)\}. 
\end{aligned} \ee

To prove (ii), we show the limits as $x \to +\infty$, and the limits as $x \to -\infty$ follow analogously. Let $T \in \Z,\dir \in \R$ be such that $f_{T,\,\dir} \neq 0$, and let $R > 0$. By definition of the event $\Omega_3$, we may choose $Z > 0$ sufficiently large so that $\inf_{z \ge Z}\{f_{T,\,\dir}(z)\} \ge R$. Then, by Equation~\eqref{373} of Theorem~\ref{thm:g_basic_prop}\ref{itm:DL_SIG_conv_x}, for all sufficiently large $x$ and $\sigg \in \{-,+\}$,
\[
\sup_{z \in \R}\{\Ll(x,s;z,T) + \W_{\dir \sig}(z,T;0,T)\} = \sup_{z \ge Z}\{\Ll(x,s;z,T) + \W_{\dir \sig}(z,T;0,T)\}.
\]
Let 
\[
A := \sup_{z \in \R}\{\Ll(0,s;z,T) + \W_{\dir +}(z,T;0,T)\}- \sup_{z \in \R}\{\Ll(0,s;z,T) + \W_{\dir -}(z,T;0,T)\},
\]
and note that this does not depend on $x$. Then,
by~\eqref{883},
\begin{align*}
    -f_{s,\dir}(x) &= \sup_{z \ge Z}\{\Ll(x,s;z,T) + \W_{\dir -}(z,T;0,T)\} \\
    &\qquad\qquad 
    - \sup_{z \ge Z}\{\Ll(x,s;z,T) + \W_{\dir +}(z,T;0,T)\} + A \\
    &\le \sup_{z \ge Z}\{\W_{\dir-}(z,T;0,T) - \W_{\dir +}(z,T;0,T) \} +A \\&
    = -\inf_{z \ge Z}\{f_{T,\,\dir}(z)\} + A \le -R + A,
\end{align*}
so that $f_{s,\dir}(x) \ge R - A$. Since $A$ is constant in $x$ and $R$ is arbitrary, the desired result follows. 

Note that~\eqref{bad_ub} immediately proves~\eqref{eqn:dcset_union1} in the case $x = 0$. The general case follows from additivity of the Busemann functions (Theorem~\ref{thm:DL_Buse_summ}\ref{itm:DL_Buse_add}) and~\eqref{bad_ub}.
\end{proof}

\begin{proof}[Proof of Theorem~\ref{thm:BusePalm} (Local time description of the difference profile)]
This comes by Theorem~\ref{thm:indep_loc} since    $\{\W_\dir(\abullet,0;0,0)\}_{\dir \in \R}$  $\deq G$ (Theorem~\ref{thm:Buse_dist_intro}\ref{itm:SH_Buse_process}), with probability one $\dir \in \DLBusedc$ iff $\tau_{\dir} < \infty$ iff $\bck{\tau_\dir} < \infty$ (Theorem~\ref{thm:DLBusedc_description}\ref{itm:Busedc_t}, Corollary~\ref{cor:dcLR}), and  the  running maximum  process  and  the local time process of a Brownian motion are equal in  distribution (L\'evy~\cite{Levy_book}). 
\end{proof}

For the convenience of the reader, we repeat definitions  \eqref{Split_sdir}--\eqref{eqn:gen_split_set} and~\eqref{Dsdir},\eqref{eqn:split_LR_sdir}.  As before, $S \in \{L,R\}$. 
\begin{align*} 
        \Split_{s,\dir} &= \{x \in \R: \text{there exist disjoint semi-infinite  geodesics from  }(x,s) \text{ in direction }\dir\}, \\
        \Split &= \bigcup_{s \in \R, \dir \in \DLBusedc} \Split_{s,\dir} \times \{s\},
        \qquad
        \Split_{s,\dir}^{S} = \{x \in \R: g_{(x,s)}^{\dir -,S} \text{ and } g_{(x,s)}^{\dir +,S} \text{ are disjoint}\},  \\
        \Split^S &= \bigcup_{\dir \in \DLBusedc, s \in \R} \Split_{s,\dir}^S \times \{s\},   \qquad\text{and}\qquad 
         \D_{s,\dir} = \{ x\in\R:  f_{s,\dir}(x - \ve) < f_{s,\dir}(x + \ve)\; \forall \ve > 0 \}. 
    \end{align*}
\begin{remark} \label{rmk:splitsetseq}
In contrast with  $\Split$ in~\eqref{eqn:gen_split_set4},   the sets $\Split^S$ are concerned only with leftmost ($S=L$) and rightmost ($S=R$) Busemann geodesics.
In   BLPP, the analogues of $\Split^L$ and $\Split^R$ are both equal to the set of initial points from which some geodesic travels initially vertically (Theorems 2.10 and 4.30 in~\cite{Seppalainen-Sorensen-21b}). Furthermore, in BLPP,  the analogue of this set contains $\NU_0$. We do not presently know whether either is true in DL. 
\end{remark}

\begin{proof}[Proof of Theorem~\ref{thm:random_supp}]
The full-probability event is $\Omega_2$ in~\eqref{omega2}.  The monotonicity of the function $f_{s,\dir}$ follows from~\eqref{801}.  We now prove that $\D_{s,\dir} = \Split_{s,\dir}^L \cup \Split_{s,\dir}^R$. Assume that $y \notin \D_{s,\dir}$. Then, there exist $a < y < b$ such that $f_{s,\dir}$ is constant on $[a,b]$. Hence, for $a \le x < y$,
\[
\W_{\dir +}(x,s;0,s) - \W_{\dir -}(x,s;0,s) = \W_{\dir +}(y,s;0,s) - \W_{\dir -}(y,s;0,s),
\]
and by additivity (Theorem~\ref{thm:DL_Buse_summ}\ref{itm:DL_Buse_add}), $\W_{\dir-}(y,s;x,s) = \W_{\dir +}(y,s;x,s)$. Choose $t > s$ sufficiently small so that $g_{(x,s)}^{\dir +,R}(t) < g_{(y,s)}^{\dir -,L}(t)$. By Lemma~\ref{lem:Buse_equality_coal}, $g_{(y,s)}^{\dir -,L}(u) = g_{(y,s)}^{\dir +,L}(u)$ for $u \in [s,t]$.  By a symmetric argument, instead choosing a point $x > y$, $g_{(y,s)}^{\dir -,R}$ and $g_{(y,s)}^{\dir +,R}$ agree near the starting point $(y,s)$. Hence, $y \notin \Split_{s,\dir}^L \cup \Split_{s,\dir}^R$.

Next, assume that $y \in \D_{s,\dir}$. Then, for all $x < y < z$,
\[
\W_{\dir +}(x,s;0,s) - \W_{\dir -}(x,s;0,s) < \W_{\dir +}(z,s;0,s) - \W_{\dir -}(z,s;0,s)
\]
and hence 
either (i)  $\W_{\dir -}(y,s;x,s) < \W_{\dir +}(y,s;x,s)$    for all $x < y$  or (ii) $\W_{\dir -}(z,s;y,s) < \W_{\dir +}(z,s;y,s)$  for all $z > y.$

We show that $g_{(y,s)}^{\dir -,L}$ and $g_{(y,s)}^{\dir +,L}$ are disjoint in the first case.     A symmetric proof shows that  $g_{(y,s)}^{\dir -,R}$ and $g_{(y,s)}^{\dir +,R}$ are disjoint in the second case. So 
assume $\W_{\dir -}(y,s;x,s) < \W_{\dir +}(y,s;x,s)$ for all $x < y$.  Sending $x \nearrow y$, $g_{(x,s)}^{\dir -,R}$ converges to $g_{(y,s)}^{\dir -,L}$ by Theorem~\ref{thm:g_basic_prop}\ref{itm:DL_SIG_conv_x}.   Assume, by way of contradiction, that $g_{(y,s)}^{\dir -,L}(u) = g_{(y,s)}^{\dir +,L}(u)$ for some $u > s$. This implies then $g_{(y,s)}^{\dir -,L}(t) = g_{(y,s)}^{\dir +,L}(t)$ for all $t \in [s,u]$ since both paths are the leftmost geodesic between any two of their points (Theorem~\ref{thm:DL_SIG_cons_intro}\ref{itm:DL_LRmost_geod}). For $t \ge s$, the convergence $g_{(x,s)}^{\dir -,R}(t) \to g_{(y,s)}^{\dir -,L}(t)$ is monotone by Theorem~\ref{thm:g_basic_prop}\ref{itm:DL_SIG_mont_x}. Since geodesics are continuous paths, Dini's theorem implies that, as $x \nearrow y$, $g_{(x,s)}^{\dir -,R}(t)$ converges to $g_{(y,s)}^{\dir -,L}(t) = g_{(y,s)}^{\dir + ,L}(t)$ uniformly in $t \in [s,u]$. Lemma~\ref{lm:BGH_disj} implies that, for sufficiently close $x < y$, $g_{(x,s)}^{\dir -,R}$ and $g_{(y,s)}^{\dir +,L}$ are not disjoint. This contradicts \ref{itm:DL_pm_Buse_eq}$\Leftrightarrow$\ref{itm:DL_paths} of Theorem~\ref{thm:Buse_pm_equiv} since we assumed $\W_{\dir -}(y,s;x,s) < \W_{\dir +}(y,s;x,s)$ for all $x < y$.

Lastly, we show that $(\Split_{s,\dir} \setminus \D_{s,\dir}) \times \{s\} \subseteq \NU_1^{\dir - } \cap \NU_1^{\dir +} \cap\, \Hh_s$. Let $x \in \Split_{s,\dir} \setminus \D_{s,\dir}$. By Theorem~\ref{thm:all_SIG_thm_intro}\ref{itm:DL_LRmost_SIG}, $g_{(x,s)}^{\dir -,L}$ is the leftmost $\dir$-directed geodesic from $(x,s)$, and $g_{(x,s)}^{\dir +,R}$ is the rightmost. Since $x \in \Split_{s,\dir}$, these two geodesics must be disjoint. Since $x \notin \D_{s,\dir}$, $g_{(x,s)}^{\dir -,L}$ and $g_{(x,s)}^{\dir +,L}$ are not disjoint, and $g_{(x,s)}^{\dir-,R}$ and $g_{(x,s)}^{\dir +,R}$ are not disjoint. Since the leftmost/rightmost semi-infinite geodesics are leftmost/rightmost geodesics between their points (Theorem~\ref{thm:DL_SIG_cons_intro}\ref{itm:DL_LRmost_geod}), there exists $\ve > 0$ such that for $t \in (s,s + \ve)$,
\[
g_{(x,s)}^{\dir -,L}(t) = g_{(x,s)}^{\dir +,L}(t) < g_{(x,s)}^{\dir -,R}(t) = g_{(x,s)}^{\dir +,R}(t),
\]
so, recalling the definition~\eqref{NU1}, $(x,s) \in \NU_1^{\dir -} \cap \NU_1^{\dir +} \cap\, \Hh_s$. 
\end{proof}

\begin{lemma} \label{lem:rm_geod}
Given $\omega \in \Omega_2$ and  $(x,s;y,u) \in \Rup$, let $g:[s,u] \to \R$ be the leftmost {\rm(}resp.\ rightmost{\rm)} geodesic between $(x,s)$ and $(y,u)$. Then, $(g(t),t) \in \Split^L$ ${\rm(} \text{resp.\ } \Split^R {\rm)}$ for some $t \in [s,u)$. Furthermore, among the directions $\dir$ for which $g_{(x,s)}^{\dir-,L}$ and $g_{(x,s)}^{\dir+,L}$ separate at some $t \in [s,u)$, there is a unique direction $\wh \dir$ such that 
\[
g_{(x,s)}^{\wh \dir-,L}(u) \le y < g_{(x,s)}^{\wh \dir+,L}(u).
\]
The same  holds with $L$ replaced by $R$ and the strict and weak inequalities swapped.
\end{lemma}
\begin{proof}
We prove the statement for leftmost geodesics.  The proof  for rightmost geodesics is analogous. Set
\be \label{468}
\wh \dir := \sup \{\dir \in \R: g_{(x,s)}^{\dir \sig,L}(u) \le y\} = \inf \{\dir \in \R: g_{(x,s)}^{\dir \sig,L}(u) > y\}.
\ee
The monotonicity of Theorem~\ref{thm:g_basic_prop}\ref{itm:DL_mont_dir} guarantees that the second equality holds, and that the definition is independent of the choice of $\sigg \in \{-,+\}$. Theorem~\ref{thm:g_basic_prop}\ref{itm:limits_to_inf} guarantees that $\wh \dir \in \R$. By definition of $\wh \dir$ and the monotonicity of Theorem~\ref{thm:g_basic_prop}\ref{itm:DL_mont_dir}, $g_{(x,s)}^{\alpha \sig,L}(u) \le y = g(u) < g_{(x,s)}^{\beta \sig,L}(u)$  whenever $\alpha < \wh \dir < \beta$ and $\sigg \in \{-,+\}$.  
But by Theorem~\ref{thm:g_basic_prop}\ref{itm:DL_SIG_unif}, the $\beta\sigg$ and $\wh\xi+$ geodesics agree locally when  $\beta$ is close enough to $\wh\xi$. We can conclude  that 
\be \label{423}
g_{(x,s)}^{\wh \dir -,L}(u) \le y = g(u) < g_{(x,s)}^{\wh \dir +,L}(u).
\ee
Since all three are leftmost geodesics (recall Theorem~\ref{thm:DL_SIG_cons_intro}\ref{itm:DL_LRmost_geod} for the Busemann geodesics), 
\be \label{523}
g_{(x,s)}^{\wh \dir -,L}(t)\le g(t) \le g_{(x,s)}^{\wh \dir +,L}(t)\qquad\text{for } t \in [s,u].
\ee
By~\eqref{423} the paths $g_{(x,s)}^{\wh \dir -,L}$ and $g_{(x,s)}^{\wh \dir +,L}$ must separate at some time $t \in [s,u)$. Furthermore, once $g_{(x,s)}^{\wh \dir -,L}$ splits from $g_{(x,s)}^{\wh \dir +,L}$ at a point $(z_1,t_1)$, the geodesics must stay apart. Otherwise, they would meet again at a point $(z_2,t_2)$, and Theorem~\ref{thm:DL_SIG_cons_intro}\ref{itm:DL_LRmost_geod} implies that both paths are the leftmost geodesic between $(z_1,t_1)$ and $(z_2,t_2)$. See Figure~\ref{fig:splitting}.  Set $\hat t = \inf\{t > s: g_{(x,s)}^{\wh \dir-,L}(t) < g_{(x,s)}^{\wh \dir +,L}(t)\}$. Then, $g_{(x,s)}^{\wh \dir-,L}(t) < g_{(x,s)}^{\wh \dir +,L}(t)$ for all $t > \hat t$. By~\eqref{523} and continuity of geodesics, $g_{(x,s)}^{\wh \dir-,L}(t) = g(t) = g_{(x,s)}^{\wh \dir +,L}(t)$ for $t \in [s,\hat t\tspa]$, and so $(g(\hat t),\hat t\tspa) \in \Split^L$. 
\end{proof}

\begin{figure}[t]
    \centering    \includegraphics[height = 1.5 in]{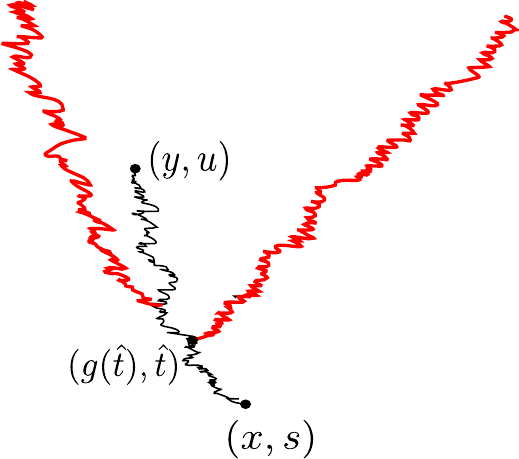}
    \caption{\small The black/thin path is the path $g$. The red/thick paths are the semi-infinite geodesics $g_{(x,s)}^{\wh \dir-,L}$ and $g_{(x,s)}^{\wh \dir+,L}$ after they split from $g$. Once the red paths split, they cannot return, or else there would be two leftmost geodesics from $(g(\hat t),\hat t)$ to the point where they come back together.}
    \label{fig:splitting}
\end{figure}

\begin{proof}[Proof of Theorem~\ref{thm:Split_pts}]

\smallskip\noindent \textbf{Item~\ref{itm:split_dense} ($\Split$ is dense):} Work on the full-probability event   $\Omega_2$. Since $\Split \supseteq \Split^L \cup \Split^R$, it suffices to show that for $(x,s) \in \R^2$   there is  a sequence $(y_n,t_n) \in \Split^L$ converging to $(x,s)$. Let $g$ be the leftmost geodesic from $(x,s)$ to $(x,s + 1)$. Then $\forall n \ge 1$, $g|_{[s,s + n^{-1}]}$ is the leftmost geodesic from $(x,s)$ to $(x,s + n^{-1})$. By Lemma~\ref{lem:rm_geod},  $\forall n \in \Z_{> 0}$  $\exists (x_n,t_n) \in \Split^L$ such that $x_n = g(t_n)$ and $s \le t_n \le s + n^{-1}$. The proof is complete by continuity of geodesics.

\smallskip\noindent \textbf{Item~\ref{itm:splitp0} ($\Pp(p \in \Split) = 0$ for all $p \in \R^2$):} If there exist disjoint semi-infinite geodesics from $(x,s)$, then for each level $t > s$, there exist disjoint geodesics from $(x,s)$ to some points $(y_1,t),(y_2,t)$. For each fixed $(x,s)$, with probability one, this occurs for no such points by~\cite[Remark 1.12]{Bates-Ganguly-Hammond-22}.

\smallskip\noindent \textbf{Item~\ref{itm:Hasudorff1/2} (Hausdorff dimension of $\Split_{s,\dir})$:} Since $s$ is fixed, it suffices to take $s = 0$. By Theorem~\ref{thm:Buse_dist_intro}\ref{itm:SH_Buse_process}, $\{\W_{\dir +}(\abullet,0;0,0)\}\deq G$, and by Theorem~\ref{thm:DLBusedc_description}\ref{itm:Busedc_t}, $\dir \in \DLBusedc$ if and only if $f_{0,\dir} \neq 0$. Therefore, Corollary~\ref{cor:SHHaus1/2} implies that, with probability one, $\dim_H(\D_{0,\dir}) = \f{1}{2}$ for all $\dir \in \DLBusedc$. By Remark~\ref{rmk:NUsupp}, $\Pp(\dim_H(\Split_{0,\dir}) = \f{1}{2} \;\forall \dir \in \DLBusedc) = 1$.

\smallskip\noindent \textbf{Item~\ref{itm:nonempty} ($\Split_{s,\dir}$ is nonempty and unbounded for all $s$):} 
By Theorem~\ref{thm:DLBusedc_description}\ref{itm:Busedc_t}, on the event $\Omega_3$, whenever $\dir \in \DLBusedc$, for all $s \in \R$, $f_{s,\dir}(x) \to \pm \infty$ as $x \to \pm \infty$. Since $f_{s,\dir}$ is continuous (Theorem~\ref{thm:DL_Buse_summ}\ref{itm:general_cts}), the set $\D_{s,\dir}$ is unbounded in both directions. The proof is complete since $\D_{s,\dir} \subseteq \Split_{s,\dir}$ by definition.
\end{proof}

\section{Open problems} \label{sec:op}
We enumerate  open problems that arise from this paper and mention solutions that have appeared since this paper was first posted. 
\begin{enumerate}  [label=\rm(\roman{*}), ref=\rm(\roman{*})]  \itemsep=3pt
    \item Prove convergence to SH for the Busemann process of some model other than   exponential LPP \cite{Busani-2021} and  BLPP \cite{Seppalainen-Sorensen-21b} (For BLPP, convergence has been shown only for  finite-dimensional distributions). In our work~\cite{Busa-Sepp-Sore-22b} that came after the first version of this paper, we show convergence of the TASEP speed process from~\cite{Amir_Angel_Valko11} to the SH. In this particle systems context, there are no Busemann functions, but there is a notion of coupled invariant measures.  In the long term, a true statement for KPZ universality should include convergence of its coupled invariant measures to the stationary horizon. 
    \item Recall definitions  \eqref{Split_sdir}--\eqref{eqn:gen_split_set} and  Remark~\ref{rmk:supports}. Can one describe the size of the sets $\Split_{s,\dir}$ globally instead of just on a fixed horizontal line, as in Theorem~\ref{thm:Split_pts}? Does  $\Split_{s,\dir}$ have Hausdorff dimension $\f{1}{2}$ simultaneously for all $s \in \R$ and $\dir \in \DLBusedc$?   The support of the Airy difference profile along a vertical line was recently studied in~\cite{Ganguly-Zhang-2022a}. What properties does  the set   $\Split$ have along a vertical line? 
    
    \item Are all semi-infinite geodesics Busemann geodesics?  (Theorem~\ref{thm:DL_good_dir_classification}\ref{itm:DL_allBuse} covers the case  $\dir \notin \DLBusedc$.) Equivalently,   does  every   semi-infinite geodesic in direction $\dir \in \DLBusedc$ coalesce with a $\dir-$ or $\dir+$ geodesic? 
    \item For $\dir \in \R$ and $\sigg \in \{-,+\}$, is  $\NU_1^{\dir \sig}$ a strict subset of $\NU_0^{\dir \sig}$? (Recall definitions \eqref{NU0}-~\eqref{NU1}.)  That is, are there $\dir \sig$ geodesics that stick together for some time, separate, then come back together, or must they separate immediately? See Figure~\ref{fig:NU}. After the posting of the first version of this paper, it was shown in two independent works \cite{Bhatia-23,Dauvergne-23} that the two sets are equal.
    \item The set $\NU_0$   is  countably infinite   on each horizontal line and hence globally  uncountable  (Theorem~\ref{thm:DLNU}). What is the Hausdorff dimension of $\NU_0$? It has since been shown in~\cite{Bhatia-23} that for fixed direction $\dir \in \R$, $\NU_0^\dir$ almost surely has Hausdorff dimension $\f{4}{3}$. By Theorem~\ref{thm:DLNU}, the full set $\NU_0$ also has Hausdorff dimension $\f{4}{3}$. 
    \item In  BLPP,   the analogue of the inclusion $\NU_0 \subseteq \Split$ holds \cite{Seppalainen-Sorensen-21b}. The reason is that  in BLPP, the analogue of the set $\Split$ is   the set of initial points from  which some finite geodesic begins with a vertical step.   We do not have such a description in DL. Does the    inclusion still hold? 
    \item Are the sets $\Split^L$ and $\Split^R$ defined in~\eqref{eqn:split_LR_sdir} equal, as is the case for the analogous sets in BLPP? See Remark~\ref{rmk:splitsetseq}. 
\end{enumerate}

\begin{appendix}
\section{Maximizers of continuous functions}
For a function $f:\R \to \R$, we denote its increments by $f(x,y) := f(y) - f(x).$ For two functions $f,g:\R \to \R$, we say that $f \li g$ if $f(x,y) \le g(x,y)$ for all $x < y$ in $\R$.
\begin{lemma} \label{lemma:max_monotonicity}
Let $f,g:\R \to \R$ be continuous functions satisfying $f(x)\vee g(x) \to -\infty$ as $x \to \pm \infty$ and $f \li g$. Let $x_f^L$ and $x_f^R$ be the leftmost and rightmost maximizers of $f$ over $\R$, and similarly defined for $g$. Then,  $x_f^L \le x_g^L$ and $x_f^R \le x_g^R$.
\end{lemma}
\begin{proof}
By the definition of  $x_g^R$   and by the assumption $f \li g$, for all $x > x_g^R$
\[
f(x_g^R,x) \le g(x_g^R,x) < 0.
\]
Hence, the rightmost maximizer of $f$ must be less than or equal to $x_g^R$. We get the statement for leftmost maximizers by considering the functions $x \mapsto f(-x)$ and $g \mapsto g(-x)$.
\end{proof}

\begin{lemma} \label{lem:ext_mont}
Assume that $f,g:\R \to \R$ satisfy $f \li g$. Then, for $a \le x \le y \le b$,
\[
0 \le g(x,y) - f(x,y) \le g(a,b) - f(a,b).
\]
\end{lemma}
\begin{proof}
The first inequality follows immediately from the assumption $f \li g$. The second follows from the inequality
\[
f(a,b) - f(x,y) = f(a,x) + f(y,b) \le g(a,x) + g(y,b) = g(a,b) - g(x,y). \qedhere
\]
\end{proof}

\begin{lemma} \label{lemma:convergence of maximizers from converging functions}
Let $S \subseteq \R^n$, and let  $f_n:S \rightarrow \R$ be a sequence of continuous functions, converging uniformly to the function $f:S \rightarrow \R$. Assume that there exists a sequence $\{c_n\}$, of maximizers of $f_n$, converging to some $c \in S$. Then, $c$ is a maximizer of $f$. 
\end{lemma}
\begin{proof}
$f_n(c_n) \ge f_n(x)$ for all $x \in S$, so  it suffices to show that $f_n(c_n) \rightarrow f(c)$. This follows from the uniform convergence of $f_n$ to $f$, the continuity of $f$, and
\[
|f_n(c_n) - f(c)| \le |f_n(c_n) - f(c_n)| +|f(c_n) - f(c)|. \qedhere
\]
\end{proof}


\section{Directed landscape and the KPZ fixed point} \label{sec:DL_KPZ_appendix}
The directed landscape satisfies the following symmetries.
\begin{lemma}[\cite{Directed_Landscape}, Lemma 10.2 and \cite{Dauvergne-Virag-21}, Proposition 1.23] \label{lm:landscape_symm}
As a random continuous function of $(x,s;y,t) \in \Rup$, the directed landscape $\Ll$ satisfies the following distributional symmetries, for all  $r,c \in \R$ and $q > 0$.
\begin{enumerate} [label=\rm(\roman{*}), ref=\rm(\roman{*})]  \itemsep=4pt
    \item {\rm(Space-time stationarity)}  \label{itm:time_stat} \ \ $\Ll(x,s;y,t) \deq \Ll(x+c,s + r;y+c,t + r).
    $
    \item {\rm(Skew stationarity)} \label{itm:skew_stat}
    \ \ $
    \Ll(x,s;y,t) \deq \Ll(x + cs,s;y + ct,t) -2c(x - y) + (t- s)c^2.  
    $
    \item \label{itm:DL_reflect} {\rm(Spatial and temporal reflections)} 
    \ \ $
    \Ll(x,s;y,t) \deq \Ll(-x,s;-y,t) \deq \Ll(y,-t;x,-s).
    $
    \item \label{itm:DL_rescaling} {\rm(Rescaling)} 
    \ \ $
    \Ll(x,s;y,t) \deq q\Ll(q^{-2}x,q^{-3}s;q^{-2}y,q^{-3}t).
    $
\end{enumerate}
\end{lemma}


\begin{lemma}[\cite{Directed_Landscape}, Corollary 10.7] \label{lem:Landscape_global_bound}
There exists a random constant $C$ such that for all $v = (x,s;y,t) \in \Rup$, we have 
\[
\Bigl|\Ll(x,s;y,t) + \f{(x - y)^2}{t - s}\Bigr| \le C (t - s)^{1/3} \log^{4/3} \Bigl(\f{2(\|v\| + 2)}{t - s}\Bigr)\log^{2/3}(\|v\| + 2),
\]
where $\|v\|$ is the Euclidean norm.
\end{lemma}
\begin{lemma}[\cite{Dauvergne-22}, Proposition 2.6] \label{lem:DL_erg_coup}
For every $i = 1,\ldots,k$ and $\ve > 0$, let 
\be \label{Kie}
K_{i,\ve} = \{(x,s;y,t) \in \Rup: s,t \in [0,\ve],x,y \in [i -1/4,i + 1/4]\}.
\ee
Then, there exists a coupling of $k + 1$ copies of the directed landscape $\Ll_0,\Ll_1,\ldots,\Ll_k$ so that $\Ll_1,\ldots,\Ll_k$ are independent, and almost surely, there exists a random $\ve > 0$ such that for $1 \le i \le k$, $\Ll_0|_{K_{i,\ve}} = \Ll_i|_{K_{i,\ve}}$.
\end{lemma}

On a measure space $(\Omega,\F,\Pp)$, a measure-preserving transformation $T$ satisfies $T^{-1}E \in \F$ and $\Pp(T^{-1}E) = \Pp(E)$ for all $E \in \F$. Such a transformation is said to be \textit{ergodic} if $\Pp(E) \in\{0,1\}$ whenever $T^{-1}E = E$. The transformation $T$ is said to be \textit{mixing} if, for all $A,B \in \F$, $\Pp(A \cap T^{-k}B) \to \Pp(A) \Pp(B)$ as $k \to \infty$. By setting $A = B$, one sees that mixing implies ergodicity. 
\begin{lemma} \label{lm:horiz_shift_mix}
For $a,b \in \R$, not both $0$ and $z >0$, consider the shift operator $T_{z;a,b}$ acting on the directed landscape $\Ll$ as
\[
T_{z;a,b} \Ll(x,s;y,t) = \Ll(x + az,s + bz;y + az;t + bz),
\]
where both sides are understood as a process on $\Rup$.  Then, $\Ll$ is mixing under this transformation. That is, for all Borel subsets $A,B$ of the space $C(\Rup,\R)$,
\[
\Pp(\Ll \in A, T_{z;a,b}\Ll \in B) \overset{z \to \infty}{\longrightarrow} \Pp(\Ll \in A)\Pp(\Ll \in B).
\]
In words, the directed landscape is mixing (and therefore ergodic) under space-time shifts in any planar direction.
\end{lemma}
\begin{proof}
This key to the proof is Lemma~\ref{lem:DL_erg_coup}, and we thank Duncan Dauvergne for pointing this out to us. We prove the case $a \neq 0$, and the case $a = 0$ and $b \neq 0$ will be proven separately. Further, since we send $z \to \infty$, it suffices to show the result for $a = 1$, for then the result also holds for arbitrary $a$ and $b = ab$. With this simplification, we use the shorthand notation $T_{z;b} = T{z:1,b}$. By Lemma~\ref{lm:landscape_symm}\ref{itm:time_stat}, $\Ll$ is stationary under the shift $T_{z;b}$. By Dynkin's $\pi$-$\lambda$ theorem, it suffices to show that for a compact set $K\subseteq \Rup$ and Borel sets $A\subseteq C(K,\R)$ and $B \subseteq C(K,\R)$,  
\[
\Pp(\Ll|_{K} \in A,(T_{z;b}\Ll)|_{K} \in B) \overset{z \to \infty}{\longrightarrow} \Pp(\Ll|_{K} \in A)\Pp(\Ll|_{K} \in B).
\]
Further, by temporal stationarity, it suffices to assume that $\inf\{s: (x,s;y,t) \in K\} \ge 0$. Consider the coupling $\Ll_0,\Ll_1,\Ll_2$ of Lemma~\ref{lem:DL_erg_coup} with $k = 2$. Then, using the rescaling and spatial stationarity of Lemma~\ref{lm:landscape_symm}, 
\begin{align}
   &\quad\;\Pp(\Ll|_{K} \in A,(T_{z;b}\Ll)|_{K} \in B) \nonumber \\
   &= \Pp(\Ll_0(x,s;y,t)|_{K} \in A, \Ll_0(x + z,s + bz;y + z,t + bz)|_{K} \in B)\nonumber \\
   &= \Pp(z^{1/2}\Ll_0(z^{-1} x,z^{-3/2}s  ;z^{-1} y,z^{-3/2} t  )|_{K} \in A, \nonumber \\ 
   &\qquad\qquad z^{1/2}\Ll_0(z^{-1} x + 1,z^{-3/2}s + z^{-1/2} b;z^{-1} y + 1,z^{-3/2} t + z^{-1/2}b)|_{K} \in B) \nonumber \\
  &= \Pp(z^{1/2}\Ll_0(z^{-1} x + 1,z^{-3/2}s  ;z^{-1} y + 1 ,z^{-3/2} t )|_{K} \in A, \nonumber \\
  &\qquad\qquad z^{1/2}\Ll_0(z^{-1} x + 2,z^{-3/2}s + z^{-1/2}b  ;z^{-1} y + 2,z^{-3/2} t + z^{-1/2}b  )|_{K} \in B). \label{433}
\end{align}
Specifically, we used the rescaling property with $q = z^{1/2}$ in the second equality, and in the third equality, we shifted the entire process by $1$ in the spatial direction.
Above, the restrictions $|_{K_j}$ mean that $(x,s;y,t) \in K_j$.  
Since $K$ is compact and we assumed $s \ge 0$ for all $(x,s;y,t) \in K$, for any $\ve > 0$, there exists $Z > 0$ such that for $z > Z$,
\be \label{small_cpct}
\begin{aligned}
  &\{(z^{-1} x + 1,z^{-3/2}s;z^{-1} y + 1,z^{-3/2} t): (x,s;y,t) \in K\} \subseteq   K_{1,\ve}, \quad \text{and}\\
  &\{(z^{-1} x + 2,z^{-3/2}s + z^{-1/2}b;z^{-1} y + 2,z^{-3/2} t + z^{-1/2}b): (x,s;y,t) \in K\} \subseteq K_{2,\ve},
  \end{aligned}
\ee
where $K_{i,\ve}$ are defined in~\eqref{Kie}. Let $C_z$ be the event where both these containments hold for the random $\ve > 0$ in Lemma~\ref{lem:DL_erg_coup}, and let $D_z$ be the event in~\eqref{433}. Then, $\Pp(C_z) \to 1$ as $z \to +\infty$. Then, for any $\delta > 0$, whenever $z$ is sufficiently large so that $1 - \Pp(C_z) = \delta > 0$,
\begin{align*}
    &\quad \;\Big|\Pp(D_z) - \Pp(\Ll|_{K} \in A)\Pp(\Ll|_{K} \in B)\Big| \\
    &\le \Big|\Pp(D_z \cap C_z) -\Pp(\Ll|_{K} \in A)\Pp(\Ll|_{K} \in B)\Big| + \delta \\
    &=\Big|\Pp(z^{1/2}\Ll_1(z^{-1} x + 1,z^{-3/2}s;z^{-1} y + 1,z^{-3/2} t)|_{K} \in A,  \\
  &\qquad\qquad z^{1/2}\Ll_2(z^{-1} x + 2,z^{-3/2}s + z^{-1/2}b;z^{-1} y + 2,z^{-3/2} t + z^{-1/2}b)|_{K} \in B, C_z) \\
  &\qquad\qquad\qquad- \Pp(\Ll_1|_{K} \in A)\Pp(\Ll_1|_{K} \in B)\Big| + \delta \\
  &\le \Big|\Pp(z^{1/2}\Ll_0(z^{-1} x + 1,z^{-3/2}s;z^{-1} y + 1,z^{-3/2} t)|_{K} \in A,  \\
  &\qquad\qquad z^{1/2}\Ll_0(z^{-1} x + 2,z^{-3/2}s + z^{-1/2}b;z^{-1} y + 2,z^{-3/2} t + z^{-1/2}b)|_{K} \in B) \\
  &\qquad\qquad\qquad- \Pp(\Ll_1|_{K} \in A)\Pp(\Ll_2|_{K} \in B)\Big| + 2\delta \\
  &= \Big|\Pp(\Ll_1|_{K} \in A, \Ll_2|_{K} \in B) - \Pp(\Ll_1|_{K} \in A)\Pp(\Ll_2|_{K} \in B)\Big| + 2\delta = 2\delta,
\end{align*}
completing the proof since $\delta$ is arbitrary. Specifically, in the two inequalities, we added and removed the event $C_z$ at the cost of $\delta$. In the first equality above, we used the fact that the containments~\eqref{small_cpct} hold on $C_z$ and  $\Ll_0|_{K_i,\ve} = \Ll_i|_{K_i,\ve}$ for $i = 1,2$. In the last line, we reversed the application of the rescaling and spatial stationarity, then used the independence of $\Ll_1$ and $\Ll_2$ from Lemma~\ref{lem:DL_erg_coup}.

The statement for the vertical shift operator when $a = 0$ is simpler. For a compact set $K$, the processes $\Ll|_{K}$ and $T_{z;0,b} \Ll|_{K}$ are independent for sufficiently large $b$ by the independent increments property of the directed landscape, and the desired result follows. 
\end{proof}


Recall the definition of the state space $\UC$~\eqref{UCdef} for the KPZ fixed point. Recall that the KPZ fixed point $h_t(\abullet;\h)$ with initial data $\h$ at time $0$ can be represented as
\[
h_t(y;\h) = \sup_{x \in R}\{\h(x) + \Ll(x,0;y,t)\} \qquad\text{for }t > 0.
\]
The KPZ fixed point satisfies the semi-group property. That is, for $0 < s < t$,
\[
h_t(y;\h) = \sup_{x \in \R} \{h_s(x;\h) + \Ll(x,s;y,t) \}.
\]
In this sense, we may say that $h_t$ has initial data $\h$ sampled at time $s < t$, in which case, we write
\[
h_t(y;\h) = \sup_{x \in \R} \{\h(x) + \Ll(x,s;y,t) \}.
\]

\begin{lemma}[\cite{Directed_Landscape,Basu-Ganguly-Hammond-21,Ganguly-Hegde-2021,Pimentel-21b}] \label{lem:DL_crossing_facts}
Let $\Ll:\Rup \to \R$ be a continuous function satisfying the metric composition law~\eqref{eqn:metric_comp} and such that maximizers in~\eqref{eqn:metric_comp} exist. Then,
\begin{enumerate} [label=\rm(\roman{*}), ref=\rm(\roman{*})]  \itemsep=3pt
    \item \label{itm:DL_crossing_lemma} Whenever $s < t$, $x_1 < x_2$, $y_1 < y_2$,
\[
\Ll(x_2,s;y_1,t) - \Ll(x_1,s;y_1,t) \le \Ll(x_2,s;y_2,t) - \Ll(x_1,s;y_2,t).
\]
\end{enumerate} 
Let $\h^1,\h^2 \in \UC$, and For $i = 1,2$ and $t > 0$, set 
\be \label{992}
h_t(y;\h^i) = \sup_{x \in \R}\{\h^i(x) + \Ll(x,0;y,t)\}.
\ee 
Then, assuming that maximizers in~\eqref{992} exist, the following hold.
\begin{enumerate}[resume, label=\rm(\roman{*}), ref=\rm(\roman{*})]  \itemsep=3pt
    \item \label{itm:KPZ_attractiveness} If $\h^1 \li \h^2$, then $h_t(\aabullet;\h^1)\li h_t(\aabullet;\h^2)$ for all $t > 0$.
    \item \label{itm:KPZ_crossing_lemma} For $t > 0$ and $i = 1,2$, set 
\[
Z_t(y;\h^i) = \max \argmax_{x \in \R}\{\h^i(x) + \Ll(x,0;y,t)\}.
\]
Then, if $x < y$ and $Z_t(y;\h^1) \le Z_t(x;\h^2)$,
\[
h_t(y;\h^1) - h_t(x;\h^1) \le h_t(y;\h^2) - h_t(x;\h^2).
\]
\end{enumerate}
\end{lemma}

We prove the following technical lemmas that are used in the proofs of the main theorems.  
\begin{lemma}\label{lem:unq}
	 Fix $\xi\in \R$ and $a>0$. Consider the KPZ fixed point starting at time $s$ from a function $\h \in \UC$. 
For $t > s$, let $Z_\h^{a,s,t}\in\R$ denote the set of exit points from the time horizon $\Hh_s$ of the geodesics associated with $\h$ and that terminate in $\{t\}\times [-a,a]$. That is,
	\be \label{exitpt}
		Z_\h^{a,s,t}=\bigcup_{y\in [-a,a]}\argmax_{x\in\R} \{\h(x)+\Ll(x,s;y,t)\}.
	\ee
	Then, on the full probability event of Lemma~\ref{lem:Landscape_global_bound}, whenever $\h \in \UC$ satisfies condition~\eqref{eqn:drift_assumptions}, and when  $\ve>0$, $a > 0$, and $s \in \R$, there exists  a random $t_0 = t_0(\ve,a,s) > s \vee 0$ such that for any $t> t_0$,
	\be \label{upexit}
	Z_\h^{a,s,t}\subset \big[(\xi-\ve)t,(\xi +\ve)t\big].
	\ee
	In particular, if $\h$ is a random function almost surely satisfying condition~\eqref{eqn:drift_assumptions}, then this random $t_0$ exists almost surely, and
\[
	\lim_{t \to \infty} \Pp\Big(Z_\h^{a,s,t}\subset \big[(\xi-\ve)t,(\xi +\ve)t\big]\Big) = 1.
\]  
Furthermore, an analogous statement holds on the same full-probability event if $t$ is held fixed and $s \to -\infty$. That is, there exists a random $s_0 = s_0(\ve,a,t)< t \wedge 0 $ such that for any $s < s_0$,
\be \label{downexit}
Z_\h^{a,s,t}\subset \big[-(\xi-\ve)s,-(\xi +\ve)s\big]
\ee
\end{lemma}
\begin{proof}
We show~\eqref{upexit}, and~\eqref{downexit} follows by an analogous proof. The idea of the proof is that  $\h(x)+\Ll(x,0;y,t)$ is a noisy version of $2\xi x-\frac{x^2}{t}$ and  $\argmax_{x\in\R} (2\xi x-\frac{x^2}{t})=\xi t$, but the noise  cannot change the exit point by much when $t$ is large. The drift conditions~\eqref{eqn:drift_assumptions} are used in the proof to ensure that for $t$ large enough, all maximizers are positive for $\dir > 0$ and negative for $\dir < 0$. Below, we prove the result for $\xi>0$, and the proof for $\xi<0$ follows by symmetry. The case $\xi=0$ will be proven separately. Fix $\ve>0$. Suppressing the dependence on $a,s$, set
\begin{equation} \label{F}
	F(x;t)=C_{\text{DL}}(t-s)^{1/3}\log^2\big(2\sqrt{a^2+x^2+s^2 + t^2}+4\big),
\end{equation}
where $C_{\text{DL}}$ is the random positive constant from Lemma \ref{lem:Landscape_global_bound}. 
\begin{lemma} \label{lem:dFbd}
For $\ve > 0,a > 0,$ and $s < t \in \R$, there exists $t_1 = t_1(\ve,a,s) > s$ such that for all $t > t_0$,  $|F'(x;t)| < \ve$ uniformly for all $x \in \R$. In addition, for $t > t_0$, $F'(x;t) > 0$ for $x > 0$ while $F'(x;t) < 0$ for $x < 0$.
\end{lemma}
\begin{proof}
A quick computation shows that
\[
	F'(x;t)=\f{2 C_\text{DL}(t-s)^{1/3} x\log\big(2\sqrt{a^2+x^2+t^2 + s^2}+4 \big)}{(\sqrt{a^2+x^2+t^2 + s^2}+2)\sqrt{a^2+x^2+t^2 + s^2}}.
\]
We note that
$
|x|(a^2 + x^2 +t^2 + s^2)^{1/2} \le 1,
$
and that $\log(x)/x$ is decreasing for $x > e$. 
Hence, for all $a > 0,s < t,\ve > 0$, and $x \in \R$,
\[
|F'(x;t)| \le 2C(t - s)^{1/3} \f{\log(2|t| +4)}{|t| + 2} \overset{t \to \infty}{\longrightarrow} 0. \qedhere
\]
\end{proof}

Back to the main proof,  from the drift assumption~\eqref{eqn:drift_assumptions}, for each $\ve > 0$, there exists $R_\ve > 0$ so that for all $x \ge R_\ve$, $|\f{\h(x)}{x} - \dir| \le \f{\ve}{2}$. Let $C_\ve = \sup_{0 \le R_\ve} \h(x)$. By the global bound $\h(x) \le a + b|x|$ (assumed in the definition of $\UC$), $C_\ve < \infty$. Further, observe that $-\f{(x - y)^2}{t - s} = -\f{x^2}{t - s} + \f{2xy}{t - s} - \f{y^2}{t - s}$, and $|\f{2xy}{t - s} - \f{y^2}{t - s}| \le \ve + \f{\ve}{2}|x|$ for large $t$, uniformly for $y \in [-a,a]$. Using this and the bounds of Lemma~\ref{lem:Landscape_global_bound}, for $t  > s + 1$ sufficiently large (depending on $\ve,a$), for all $y \in [-a,a]$ and $x \ge 0$,
	\begin{equation}\label{ub1'}
		\begin{aligned}
		 &\h(x)+\Ll(x,s;y,t)\leq M_{U}(x;t):=  C_\ve + 2\xi x + \ve x-\frac{x^2}{t - s}+ \ve + F(x;t),
	\end{aligned}
	\end{equation}
 Further, for all $x\ge R_\ve$, and sufficiently large $t > s + 1$,
\be \label{ub2'}
 	 \h(x)+\Ll(x,s;y,t)\geq M_{L}(x;t):= 2\xi x - \ve x -\frac{x^2}{t - s} - \ve -F(x;t). 
\ee

By the assumption~\eqref{eqn:drift_assumptions}, we may choose $\gamma$ so that 
$-2\dir < 2\gamma < \liminf_{x \to -\infty} \f{\h(x)}{x} \le + \infty$. Then, applying a similar procedure as before and adjusting the constant $C_\ve$ if needed, for all $y \in [-a,a]$ and $x \le 0$,
	\be \label{ub3'}
	\h(x) + \Ll(x,s;y,t) \le  M_U^-(x;t) := C_\ve  + 2\gamma x - \f{x^2}{t - s} + F(x;t).
	\ee

We start by using these bounds to show that when $t > s$ is sufficiently large,
\be \label{max_right}
\sup_{x \ge 0}\{\h(x) + \Ll(x,s;y,t)\} > \sup_{x \le 0}\{\h(x) + \Ll(x,s;y,t)\}, \qquad\forall y \in [-a,a].
\ee
so that all maximizers of $\h(x) + \Ll(x,s;y,t)$ over $x \in \R$ are nonnegative. First, we observe that for $t$ large enough so that $\dir(t - s) \ge R_\ve$, for all $y \in [-a,a]$,
\begin{equation}\label{Mlm}
\begin{aligned}
	\quad \;\sup_{x \ge 0}\{\h(x) + \Ll(x,s;y,t)\}\ge M_L(\xi (t-s),t) 
	= (\dir^2 - \dir \ve)(t - s) + o(t).
	\end{aligned}
\end{equation}

Next, using~\eqref{ub3'}, we obtain
\begin{align*}
&\quad \;\sup_{x \le 0}\{\h(x) + \Ll(x,s;y,t) \} \\ &\le \sup_{x \le 0}\{2\gamma x - \f{x^2}{t-s}  + C_\ve + F(x;t)\} \\
&\le\sup_{x \le 0}\{2\gamma x - 2\ve x - \f{x^2}{t-s}\}+ \sup_{x \le 0}\{2\ve x + F(x;t) \} + C_\ve. =  (\gamma - \ve)^2(t - s) + o(t).
\end{align*}
To justify the last equality, we note that the first supremum on the RHS above is equal to  $(\gamma - \ve)^2(t - s)$, while for large enough $t$, Lemma~\ref{lem:dFbd} implies that the function inside the second supremum is increasing, so the maximum is achieved at $x = 0$. Note that $F(0;t) = o(t)$. Since $\gamma > -\dir$, by choosing $\ve > 0$ small enough, a comparison with~\eqref{Mlm} verifies~\eqref{max_right} for sufficiently large $t$.

Next, we find (approximately) where the maximizers of $M_{U}$ on $x\geq 0$ are. The function $M_U(x;t)$ has leading order $-(x - y)^2/(t-s)$, so maximizers exist on $x \ge 0$. A quick computation shows that $M_U(0;t) = o(t)$, so by~\eqref{Mlm}, all maximizers are strictly positive for sufficiently large $t$. Hence, for any maximizer $x$, $M_U'(x;t) = 0$. 
First, note that
\begin{equation}
	M'_{U}(x;t)=2(\xi+\ve)-\frac{2x}{t - s}+F'(x;t).
\end{equation}
By Lemma~\ref{lem:dFbd}, for sufficiently large $t > s$, $0 < F'(x;t) < \ve$ for all $x > 0$. Then, for such $t$, and $y \in [-a,a]$ 
\begin{equation}
	 \{x \ge 0 :M'_U(x;t)=0\}\subseteq \big((\xi + \ve) (t - s), (\xi+2\ve)(t - s)\big).
\end{equation}
and that
\begin{equation}\label{d}
	\begin{aligned}
		&M'_U(x;t)<0 \quad \forall x\geq (\xi + 2\ve)(t - s)\\
		& M'_U(x;t)>0 \quad \forall x\leq   (\xi + \ve) (t - s).
	\end{aligned}
\end{equation}
Next we consider the supremum of $x\mapsto M_U$ outside the interval 
\[
I_{\ve}:=[(\xi-2\sqrt{\xi \ve})(t - s),(\xi+2\sqrt{\xi\ve})(t - s)].
\]
By choosing $\ve > 0$ small enough, then for $t$ large enough (depending on $\ve,a$), 
\begin{equation}\label{sub}
 \big((\xi - \ve) (t - s),(\xi+2\ve)(t - s)\big) \subseteq I_{\ve}.
\end{equation}
 From \eqref{d} and \eqref{sub}, we see that to determine the supremum of $M_U$ outside $I_{\ve}$ it is enough to take the maximum of $M_U$ at the endpoints of $I_{\ve}$. Plugging the end points of the interval on the right-hand side of  \eqref{sub} in $M_U$, for small enough $\ve$,
\begin{equation}\label{sup1}
	\begin{aligned}
	&M_U((\xi-2\sqrt{\xi\ve})(t - s);t)=\big[\xi^2 -3\xi\ve - 2\dir^{1/2}\ve^{3/2}\big](t - s) + o(t), \qquad\text{and}\\
	&M_U( (\xi+2\sqrt{\xi}\ve^{1/2})(t -s);t)=\big[\xi^2-3\xi\ve+2\xi^{1/2}\ve^{3/2}\Big](t - s) + o(t).
	\end{aligned}
\end{equation}
 It follows that for $\ve$ small enough and $t>t_0(\ve,a,s,C_{\text{DL}},C_\ve)$,
\begin{equation}\label{Mum}
	\sup_{x\notin I_{\ve}} M_U(x;t)\leq \max\{M_U((\xi-2\sqrt{\xi\ve})(t-s)),M_U((\xi+2\sqrt{\xi\ve})(t -s))\}\leq (\xi^2-2\xi\ve)(t-s).
\end{equation}
From \eqref{ub1'},~\eqref{max_right}, and~\eqref{sub}, for sufficiently large $t$,
\begin{equation}\label{sub2}
	\begin{aligned}
		&\big\{\sup_{x\notin I_{\ve}}  M_U(x;t)<\sup_{x\in I_{\ve}} M_L(x;t)\big\}\\
	&\subseteq\big\{\sup_{x\notin I_{\ve}}  \h(x)+\Ll(x,s;y,t)<\sup_{x\in I_{\ve}} \h(x)+\Ll(x,s;y,t)\quad \forall y\in[-a,a]\big\}\subseteq\{Z_\h^{a,s,t}\subset I_{\ve}\}.
	\end{aligned}
\end{equation}
\eqref{Mlm} and \eqref{Mum}, imply that, almost surely, there is a random $t_0 = t_0(\ve,a,s) > 0$ so that for $t > t_0$,
\begin{equation}
	\sup_{x\notin I_{\ve}}  M_U(x;t)<\sup_{x\in I_{\ve}} M_L(x;t),
\end{equation}
so by replacing $\ve$ with $\ve^2/(4\dir)$, the inclusion~\eqref{sub2} completes the proof in the case $\dir > 0$.

Now we prove the separate $\xi = 0$ case. This time, we set  $I_{\ve}=[2\sqrt{\ve}\tspb (t-s), 2\sqrt{\ve}\tspb (t-s)]$. Fix a point $x^\star \in \R$ so that $\h(x^\star) > -\infty$. Then, for $t$ large enough $x^\star \in I_{\ve}$, and
\be \label{ub4'}
\sup_{x \in I_{\ve}} \h(x) + \Ll(x,s;y,t) \ge \h(x^\star) + \Ll(x^\star,s;y,t)\ge \h(x^\star) - \f{(x^\star - y)^2}{t - s} +F(x^\star,t) = o(t).
\ee
 
 By the assumption~\eqref{eqn:drift_assumptions} and upper semi-continuity, following a similar argument as in the previous case, for all $y \in [-a,a]$ and $x \in \R$,
\[
\begin{aligned}
\h(x)+\Ll(x,s;y,t)\leq M_{U}(x;t) &:= -\frac{x^2}{t-s}+\ve x+C_\ve+F(x;t)
\end{aligned}
\]
A similar proof as before shows that 
\[
\sup_{x \notin I_{\ve}} M_U(x;t) \le \max\{M_U(-2\sqrt \ve(t - s)),M_U(2\sqrt \ve (t - s))\} = -3\ve (t-s) + o(t),
\]
and comparison with~\eqref{ub4'} completes the proof. 
\end{proof}

We believe the following Lemma is well-known, but we do not have a precise reference. In particular,~\cite{KPZfixed} states that the KPZ fixed point preserves the space of linearly bounded continuous functions 
and gives regularity estimates for the KPZ fixed point.

\begin{lemma} \label{lem:max_restrict}
Let $\h \in \UC$ be initial data for the KPZ fixed point sampled at time $s \in \R$.  
For all $t > s$, and $y \in \R$, set
\be \label{KPZs}
h_t(y;\h) = \sup_{x \in \R}\{\h(x) + \Ll(x,s;y,t)\}.
\ee
Then, on the full-probability event of Lemma~\ref{lem:Landscape_global_bound}, the following hold.
\begin{enumerate}[label=\rm(\roman{*}), ref=\rm(\roman{*})]  \itemsep=3pt
    \item \label{itm:KPZcont} If $\h$ is continuous, then $(t,y) \mapsto h_t(y;\h)$ is continuous on $(s,\infty) \times \R$.
    \item \label{itm:KPZ_unif_line} For each compact set $K \subseteq \R_{> s}$, there exist constants $A = A(a,b,K)$ and $B = B(a,b,K)$ such that for all $t \in K$ and all $y \in \R$, $h_t(y;\h) \le A + B|y|$. If we assume that $\h(x) \ge -a - b|x|$ for some constants $a,b > 0$, then we also obtain the bound $h_t(y;\h) \ge -A - B|y|$ for all $t \in K$ and $y \in \R$ \rm{(}the upper bound $\h(x) \le a + b|x|$ is assumed in the definition of $\UC$\rm{)}.
    \item \label{itm:KPZrestrict}
    If there exists $a,b > 0$ so that $|\h(x)| \le a + b|x|$ for all $x$, then for any $t > s$, $\delta > 0$, there exists $Y = Y(t,\delta) > 0$ so that when $|y| \ge Y$, all maximizers of $\h(x) + \Ll(x,s;y,t)$ over $x \in \R$ lie in the interval $(y - |y|^{1/2 + \delta},y + |y|^{1/2+ \delta})$
\end{enumerate} 
\end{lemma}
\begin{proof}
\medskip \noindent \textbf{Item~\ref{itm:KPZcont}}
By definition of $\UC$, there exists constants $a,b > 0$ so that $\h(x) \le a + b|x|$ for all $x \in \R$. Combined with the bounds on the directed landscape in Lemma~\ref{lem:Landscape_global_bound}, this implies that when $(y,t)$ varies over a compact set, the supremum in~\eqref{KPZs} can be taken uniformly over a  common compact set. Then, continuity of $\h$ and $\Ll$ gives the continuity of $h$.

\medskip \noindent \textbf{Item~\ref{itm:KPZ_unif_line}:}
By Lemma~\ref{lem:Landscape_global_bound}, for each $x \in \R$,
\be \label{414}
\h(x) + \Ll(x,s;y,t) \le  a + b|x| -\tspc \f{(x - y)^2}{t - s} + C(t - s)^{1/3} \log^{2}\Bigl(\f{2\sqrt{x^2 + y^2 + s^2 + t^2} + 4}{(t - s)\wedge 1}\Bigr).
\ee
 The $\log$ term in~\eqref{414}.  can be bounded by an affine function uniformly for $t \in K$ and $x,y \in \R$. Then, for constants $a_1 = a_1(a,b,K), b_1 = b_1(a,b,K)$, and $b_2 = b_2(a,b,K)$,
\begin{align*}
    h_t(y;\h)
    &\le \sup_{x \in \R}\Bigl\{-\f{(x - y)^2}{t - s} + a_1 + b_1|x| + b_2|y| \Bigr\}  \\
    &\le \sup_{x \in \R}\Bigl\{-\f{(x - y)^2}{t - s} + a_1 + b_1x + b_2|y| \Bigr\}  \vee \Bigl\{-\f{(x - y)^2}{t - s} + a_1 - b_1 x + b_2|y| \Bigr\}\\
    &\le a_1 + b_2|y| + \Bigl(b_1 y +\f{b_1^2 t}{4}\Bigr) \vee \Bigl(-b_1y + \f{b_1^2t}{4}\Bigr),
\end{align*}
giving a linear bound, uniformly for $t \in K$.

The lower bound is simpler: By Lemma~\ref{lem:Landscape_global_bound} and the assumption $\h(x) \ge -a - b|x|$ for all $x \in \R$,
\be \label{413}
\begin{aligned}
h_t(y;\h) &= \sup_{x \in \R}\{\h(x) + \Ll(x,s;y,t)\} \\ &\ge \h(y) + \Ll(y,s;y,t) \ge -a - b|y| - C(t - s)^{1/3}\log^2\Bigl(\f{2\sqrt{2y^2 + t^2 + s^2} + 4}{(t - s)\wedge 1}\Bigr),
\end{aligned}
\ee
and again, the $\log$  term can be bounded by an affine function, uniformly for $t \in K$ and $y \in \R$.

\medskip \noindent \textbf{Item~\ref{itm:KPZrestrict}:} By comparing~\eqref{414} to~\eqref{413}, when $|y|$ is sufficiently large, for $x \notin (y - |y|^{1/2 + \delta},y + |y|^{1/2 + \delta})$,~\eqref{414} is strictly less than $h_t(y;\h)$, so maximizers cannot lie outside the interval $(y - |y|^{1/2 + \delta},y + |y|^{1/2 + \delta})$.

\end{proof}

\begin{lemma} \label{lem:KPZ_preserve_lim}
The following holds simultaneously for all initial data and all $t > s$ on the event of probability one from Lemma~\ref{lem:Landscape_global_bound}. Let $\h \in \UC$ be initial data for the KPZ fixed point, sampled at time $s$. 
For $t > s$, let $h_t$ be defined as in~\eqref{KPZs}.
Then, simultaneously for all $t > s$,
\be \label{hliminfbd1}
\liminf_{x \to +\infty} \f{h_t(x;\h)}{x} \ge \liminf_{x \to + \infty} \f{\h(x)}{x},\qquad\text{and}\qquad \limsup_{x \to - \infty} \f{h_t(x;\h)}{x} \le \limsup_{x \to -\infty} \f{\h(x)}{x}.
\ee

Furthermore, assuming that $\h:\R \to \R$ is \textbf{continuous} and satisfies
\be \label{liminfsupfinite}
\liminf_{x \to \pm \infty} \f{\h(x)}{x} > -\infty \qquad\text{and}\qquad \limsup_{x \to \pm \infty} \f{\h(x)}{x} < +\infty,
\ee
then also
\be \label{hliminfbd2}
\limsup_{x \to + \infty} \f{h_t(x;\h)}{x} \le \limsup_{x \to +\infty} \f{\h(x)}{x},\qquad\text{and}\qquad \liminf_{x \to -\infty} \f{h_t(x;\h)}{x} \ge \liminf_{x \to - \infty} \f{\h(x)}{x}.
\ee

In particular, for \textbf{continuous} initial data $\h$ satisfying~\eqref{liminfsupfinite}, if either \rm{(}or both\rm{)} of the limits
$
\lim_{x \to \pm \infty} \f{\h(x)}{x}
$
exist \rm{(}potentially with different limits on each side\rm{)}, then for $t > s$,
\[
\lim_{x \to \pm \infty} \f{h_t(x;\h)}{x} = \lim_{x \to \pm \infty} \f{\h(x)}{x}.
\]
\end{lemma}
\begin{proof}
We start with~\eqref{hliminfbd1} by proving the first inequality, and the other is analogous.
If $\liminf_{x \to +\infty} \f{\h(x)}{x} = -\infty$, there is nothing to show. Otherwise, let $\dir_1 \in \R$ be an arbitrary number less than $\liminf_{x \to +\infty} \f{\h(x)}{x}$. Let $y$ be sufficiently large and positive so that $\h(y)\ge \dir_1 y$. Then, using Lemma~\ref{lem:Landscape_global_bound}, for such sufficiently large positive $y$,
\be \label{105}
\begin{aligned}
\sup_{x \in \R}\{\h(x) + \Ll(x,s;y,t)\} \ge \h(y) + \Ll(y,s;y,t) \ge \dir_1 y - C(t-s)^{1/3} \log^{2}\Bigl(\f{2\sqrt{2y^2 + t^2 + s^2} + 4}{(t - s)\wedge 1}\Bigr),
\end{aligned}
\ee
where $C$ is a constant. Therefore,
$
\liminf_{y \to \infty} \f{h_t(y;\h)}{y} \ge \dir_1,
$
but this is true for all $\dir_1 < \liminf_{x \to +\infty} \f{\h(x)}{x}$, so 
\[
\liminf_{y \to \infty} \f{h_t(y;\h)}{y} \ge \liminf_{x \to +\infty} \f{\h(x)}{x}.
\]

Next, we turn to proving~\eqref{hliminfbd2}. Again, we prove the first inequality and the second follows analogously. Set $\dir_2 = \limsup_{x \to +\infty} \f{\h(x)}{x}$ and let $\ve > 0$. By continuity,  the assumption~\eqref{liminfsupfinite} on the asymptotics of $\h$ implies there exist constants $a,b > 0$ so that $|\h(x)| \le a + b|x|$ for all $x \in \R$.  Lemma~\ref{lem:max_restrict}\ref{itm:KPZrestrict} implies that for $\ve  > 0$ and sufficiently large $y > 0$, 
\begin{align*}
&\quad\; \sup_{x \in \R}\{\h(x) + \Ll(x,s;y,t)\} \\
&\le \sup_{x \in (y - y^{2/3},y + y^{2/3})}\Big\{(\dir_2 + \ve)x -\f{(x - y)^2}{t - s} + C(t - s)^{1/3} \log^{2}\Bigl(\f{2\sqrt{x^2 + y^2 + t^2 + s^2} + 4}{(t - s)\wedge 1}\Bigr) \Big\} \\
&\le \sup_{x \in (y - y^{2/3},y + y^{2/3})}\Big\{(\dir_2 + \ve)x-\f{(x - y)^2}{t - s } + \ve(x + y)  \Big\} \\
&= (\dir_2 + 3\ve)y + C(\ve,s,t,\dir_2),
\end{align*}
and so
\[
\limsup_{y \to \infty} \f{h_t(y;\h)}{y} \le \dir_2 + 3\ve. \qedhere
\]
\end{proof}

\subsection{Geodesics in the directed landscape}
 We start by citing some results from~\cite{Bates-Ganguly-Hammond-22} and~\cite{Dauvergne-Sarkar-Virag-2020}. 
\begin{lemma}[\cite{Bates-Ganguly-Hammond-22}, Theorem 1.18] \label{lm:BGH_disj}
There exists a single event of full probability on which, for any compact set $K \subseteq \Rup$, there is a random $\ve > 0$ such that the following holds. If $v_1 = (x,s;y,u) \in K$ and $v_2 = (z,s;w,u) \in K$ admit geodesics $\gamma_1$ and $\gamma_2$ satisfying $|\gamma_1(t) - \gamma_2(t)| \le \ve$ for all $t \in [s,u]$, then $\gamma_1$ and $\gamma_2$ are not disjoint, i.e., $\gamma_1(t) = \gamma_2(t)$ for some $t \in [s,u]$. 
\end{lemma}

Let $g$ be a geodesic from $(x,s)$ to $(y,u)$. Define the graph of this geodesic as
\[
\graph g := \{(g(t),t): t \in [s,u]\}.
\]

\begin{lemma}[\cite{Dauvergne-Sarkar-Virag-2020}, Lemma 3.1] \label{lem:precompact}
The following holds on a single event of full probability. Let $(p_n;q_n) \to (p,q) = (x,s;y,t) \in \Rup$, and let $g_n$ be any sequence of geodesics from $p_n$ to $q_n$. Then, the sequence of graphs $\graph g_n$ is precompact in the Hausdorff metric, and any subsequential limit of $\graph g_n$ is the graph of a geodesic from $p$ to $q$.  
\end{lemma}

\begin{lemma}[\cite{Dauvergne-Sarkar-Virag-2020}, Lemma 3.3] \label{lem:overlap}
The following holds on a single event of full probability. Let $(p_n;q_n) = (x_n,s_n;y_n,u_n) \in \Rup \to (p;q) = (x,s;y,u) \in \Rup$, and let $g_n$ be any sequence of geodesics from $p_n$ to $q_n$. Suppose that either
\begin{enumerate} [label=\rm(\roman{*}), ref=\rm(\roman{*})]  \itemsep=3pt
    \item \label{uniqn} For all $n$, $g_n$ is the unique geodesic from $(x_n,s_n)$ to $(y_n,u_n)$ and $\graph g_n \to \graph g$ for some geodesic $g$ from $p$ to $q$, or
    \item \label{uniqueg} There is a unique geodesic $g$ from $p$ to $q$.
    \end{enumerate}
    Then, the \textbf{overlap}
    \[
    O(g_n,g) := \{t \in [s_n,u_n]\cap [s,u]: g_n(t) = g(t)\}    
    \]
    is an interval for all $n$ whose endpoints converge to $s$ and $u$. 
\end{lemma}
\begin{remark}
We note that condition~\ref{uniqn} is slightly different from that stated in~\cite{Dauvergne-Sarkar-Virag-2020}. There, it is assumed instead that $(x_n,s_n;y_n,u_n) \in \Q^4 \cap \Rup$ for all $n$. The only use of this requirement in the proof is to ensure that there is a unique geodesic from $(x_n,s_n)$ to $(y_n,u_n)$ for all $n$, so there is no additional justification needed for the statement we use here. 
\end{remark}

\begin{lemma} \label{lem:geod_pp}
On the intersection of the full probability events from Lemmas~\ref{lem:precompact} and~\ref{lem:overlap}, the following holds. For all ordered triples $s < t < u$ and compact sets $K \subseteq \R$, the set  
\be \label{distinct}
\{g(t): g \text{ is the unique geodesic between }(x,s) \text{ and }(y,u) \text{ for some } x \in K, y \in K\}
\ee
is finite. 
\end{lemma}
Lemma~\ref{lem:geod_pp} is known. Its derivation from  Lemma~\ref{lm:BGH_disj} and some results of \cite{Dauvergne-Sarkar-Virag-2020} are  shown in~\cite{Busa-Sepp-Sore-22arXiv}.  Lemma 3.12 in \cite{Ganguly-Zhang-2022a} (posted after our first version)  provides a stronger quantitative statement, but we do not need it for our purposes. This stronger estimate can be traced back to the work of Basu, Hoffman, and Sly~\cite{SlyNonexistenceOB} using integrable methods in exponential LPP. 
\begin{proof}[Proof of Lemma~\ref{lem:geod_pp}]
Assume, without loss of generality, that $K$ is a closed interval $[a,b]$. We observe that by planarity, all geodesics from $(x,s)$ to $(y,u)$ for $x,y \in K$ lie between the leftmost geodesic from $(a,s)$ to $(a,u)$ and the rightmost geodesic from $(b,s)$ to $(b,u)$. Hence, the set~\eqref{distinct} is contained in a compact set, and it suffices to show that the set has no limit points.

Assume, to the contrary, that there exists a point $(\hat x,s;\hat y,u) \in (K \times \{s\}\times K \times \{u\}) \cap \Rup$ with unique geodesic $\hat g$ such that there exists a sequence of points $x_n,y_n \in K$ such that for all $n$, the geodesic $g_n$ from $(x_n,s)$ to $(y_n,t)$ is unique and so that $g_n(t) \to \hat g(t)$ but $g_n(t) \neq \hat g(t)$ for all $n$.
By compactness, there exists a convergent subsequence $(x_{n_k},y_{n_k}) \to (x,y)$. By Lemma~\ref{lem:precompact}, there exists a further subsequence $(x_{n_{k_\ell}},y_{n_{k_\ell}})$ such that the geodesic graphs $\G g_{n_{k_\ell}}$ converge to the graph of some geodesic $\G g$ from $(x,s)$ to $(y,u)$ in the Hausdorff metric. Since $g_n(t) \to \hat g(t)$, we have $g(t) = \hat g(t)$. By Lemma~\ref{lem:overlap}, the overlap $O(g_{n_{k_\ell}},g)$ is an interval whose endpoints converge to $s$ and $u$, so $g_{n_{k_\ell}}(t) = g(t) = \hat g(t)$ for sufficiently large $\ell$, contradicting the definition of the sequence $g_n$. 
\end{proof}

\section{Exponential last-passage percolation} \label{sec:LPP}
\subsection{Discrete last-passage percolation} \label{sec:LPP_bd_queue}
Let $\{Y_{\mbf x}\}_{\mbf x \in \Z^2}$ be a collection of nonnegative i.i.d random variables, each associated to a vertex on the integer lattice.  For $\mbf x \le \mbf y \in \Z \times \Z$, define the last-passage time as
\be\label{d100} 
d(\mbf x,\mbf y) = \sup_{\mbf x_\centerdot \in \Pi_{\mbf x,\mbf y}} \sum_{k = 0}^{|\mbf y - \mbf x|_1} Y_{\mbf x_k}, 
\ee
where $\Pi_{\mbf x, \mbf y}$ is the set of up-right paths $\{\mbf x_k\}_{k = 0}^{n}$ that satisfy $\mbf x_0 = \mbf x,\mbf x_{n} = \mbf y$, and $\mbf x_k - \mbf x_{k - 1} \in \{\mbf e_1,\mbf e_2\}$. A maximizing path is called a geodesic. We call this model discrete last-passage percolation (LPP). The most tractable case of discrete LPP is given when $Y_{\mbf x} \sim \operatorname{Exp}(1)$, and we refer to this model as the exponential corner growth model or CGM. We will consider this model specifically for the remainder of this appendix.

\subsection{Stationary LPP in the quadrant} \label{sec:LPP_quad}
Choose $\mbf x \in \Z^2$ and consider the quadrant $\mbf x + \Z_{\ge 0}^2$. Fix a parameter $\rho \in (0,1)$. 
Let $\{Y_{\mbf z}: \mbf z \in \mbf x + \Z_{> 0}^2\}$ be i.i.d.\ $\Exp(1)$ weights in the bulk of the quadrant, and let $\{I_{\mbf x + k\mbf e_1},J_{\mbf x + \ell \mbf e_2}: k,\ell \in \Z_{>0}\}$ be mutually independent boundary weights such that $I_{\mbf x + k\mbf e_1} \sim \Exp(\rho)$ and $J_{\mbf x + \ell \mbf e_2} \sim \Exp(1 - \rho)$. These weights are defined under a probability measure $\Pp^\rho$. We define the increment-stationary process $d_{\mbf x}^\rho$ as follows. First, on the boundary, $d_{\mbf x}^\rho(\mbf x) = 0$, and for $k,\ell \ge 1$, $d_{\mbf x}^\rho(\mbf x + k \mbf e_1) = \sum_{i = 1}^k I_{\mbf x + i\mbf e_1}$ and $d_{\mbf x}^\rho(\mbf x + \ell \mbf e_2) = \sum_{j = 1}^\ell J_{\mbf x + j \mbf e_2}$.  In the bulk, for $\mbf y = \mbf x + (m,n) \in \mbf x + \Z_{>0}^2$, 
\be \label{eqn:stat_LPP}
d_{\mbf x}^\rho(\mbf y) = \max_{1 \le k \le m}\Biggl\{ \Biggl(\sum_{i = 1}^k I_{\mbf x + i\mbf e_1}\Biggr) + d(\mbf x + k\mbf e_1 + \mbf e_2,\mbf y)\Biggr\}  \bigvee \max_{1 \le \ell \le n}\Biggl\{\Biggl(\sum_{j = 1}^\ell J_{\mbf x + j \mbf e_2}\Biggl) + d(\mbf x + \ell \mbf e_2 + \mbf e_1,\mbf y)\Biggr\}
\ee
In this model, we can define also define geodesics from $\mbf x$ to $\mbf y \in \mbf x + \Z_{\ge 0}^2$ that travel for some time along the boundary and then enter the bulk. Because exponential random variables have continuous distribution, the maximizing paths for both bulk LPP and stationary LPP are almost surely unique.  For $\mbf y \in  \mbf x + \Z_{>0}^2$, if the unique geodesic for the stationary model enters the bulk from the horizontal boundary, define $\tau_1^{\mbf x}(\mbf y)$ as the unique value $k$ that maximizes in~\eqref{eqn:stat_LPP}. Otherwise, define $\tau_1^{\mbf x}(\mbf y) = 0$. Similarly, define $\tau_{2}^{\mbf x}(\mbf y)$ as the exit location from the vertical boundary, or $0$ if the geodesic exits from the horizontal boundary.

This model is increment-stationary in the sense that, for any down-right path $\{\mbf y_i\}$ in $\mbf x + \Z_{\ge 0}^2$, the increments $d^\rho_{\mbf x,\mbf y_{i + 1}} - d_{\mbf x,\mbf y_i}^\rho$ are mutually independent, and for $\mbf y \in \mbf x + \Z_{>0} \times \Z_{\ge 0}$ and $\mbf z \in \mbf x + \Z_{\ge 0} \times \Z_{> 0}$,
\[
I_{\mbf y}^{\mbf x} := d_{\mbf x}^\rho(\mbf y)- d_{\mbf x}^\rho(\mbf y - \mbf e_1) \sim \Exp(\rho)\qquad\text{and}\qquad J_{\mbf z}^{\mbf x} = d_{\mbf x}^\rho(\mbf z) - d_{\mbf x}^{\rho}(\mbf z - \mbf e_2) \sim \Exp(1 - \rho).
\]
See Theorem 3.1 in~\cite{Sepp_lecture_notes} for a proof.

\subsection{LPP in the half-plane with boundary conditions and queues}
We also define the last-passage model in the upper half-plane with a horizontal boundary. Let $h = (h(k))_{k \in \Z}$ be a real sequence.   For $m \in \Z$ let $d^h(m,0) = h(m)$,  and for $n > 0$  
\be \label{eqn:LPP_bd}
d^h(m,n) = \sup_{-\infty < k \le m}\{h(k) + d((k,1),(m,n)) \}.
\ee
We assume that $h$ is such that the supremum is almost surely finite and achieved at a finite $k$. We define the \textit{exit point} $Z^h(m,n)$ as
\be \label{eqn:exit_pt}
Z^h(m,n) = \max\{k \in \Z:h(k) + d((k,1),(m,n)) = d^h(m,n) \}.
\ee
Geodesics in this model are defined as follows: the path consists of the backwards-infinite horizontal ray $\{(k,0): k \le Z := Z^h(m,n)\}$, an upward step from $(Z,0)$ to $(Z,1)$, and then the LPP path in the bulk from $(Z,1)$ to $(m,n)$.

The half-plane model with boundary satisfies superadditivity. That is, for $r \in \{1,2\}$, $\mbf x \in \Z \times \Z_{\ge 0}$ and $\mbf x + \mbf e_r \le \mbf z$ coordinate-wise,
\be \label{205}
d^h(\mbf z) \ge d^h(\mbf x) + d^h(\mbf x + \mbf e_r,\mbf z).
\ee

The model with boundary condition can be constructed from queuing mappings, which we now define. Let $I = (I_k)_{k \in \Z}$ and $\omega = (\omega_k)_{k \in \Z}$  be sequences of nonnegative numbers such that
\[
\lim_{m \to -\infty} \sum_{i = m}^0 (\omega_i - I_{i + 1}) = -\infty. 
\]
Let $F = (F_k)_{k \in \Z}$ be a function on $\Z$ satisfying $I_k = F_k - F_{k - 1}$. Define the output sequence $\wt F = (\wt F_\ell)_{\ell \in \Z}$ by
\be \label{202}
\wt F_\ell = \sup_{-\infty < k \le \ell} \Biggl\{F_k + \sum_{i = k}^\ell \omega_i \Biggr\},\qquad \ell \in \Z.
\ee
Now, define the sequences $\wt I = (\wt I_\ell)_{\ell \in \Z}$ and $J = (J_k)_{k \in \Z}$ by
\[
\wt I_\ell = \wt F_\ell - \wt F_{\ell - 1},\qquad \text{and}\qquad J_k = \wt F_k -  F_k.
\]
In queuing terms, $I_k$ is the time between the arrivals of customers $k - 1$ and $k$, $\omega_k$ is the service time of customer $k$, $\wt I_\ell$ is the interdeparture time between customers $\ell - 1$ and $\ell$, and $J_k$ is the sojourn time of customer $k$. We use the mappings $D$ and $S$ to describe this queuing process. That is, 
\[
\wt I = D(\omega,I),\qquad\text{and}\qquad J = S(\omega,I).
\]

Exponentially distributed arrival times are invariant for the queue with exponential service times. This is made precise is the following lemma.
\begin{lemma}[\cite{Fan-Seppalainen-20}, Lemma B.2] \label{lem:output_of_queue}
Let $0 < \rho < \tau$. Let $(I_k)_{k \in \Z}$ and $\{\omega_j\}_{j \in \Z}$ be mutually independent random variables with $I_k \sim \Exp(\rho)$ and $\omega_j \sim \Exp(\tau)$. Let $\wt I =D(\omega,I)$ and $J = S(\omega,I)$. Then, $\{\wt I_j\}_{j \in \Z}$ is an i.i.d. sequence of $\Exp(\rho)$ random variables, and for each $k \in \Z$, $\{\wt I_j\}_{j \le k}$ and $J_k$ are mutually independent with $J_k \sim \Exp(\tau - \rho)$.
\end{lemma}

The following lemma shows how to construct the LPP model in the half-plane with boundary from the queuing mappings.
\begin{lemma} \label{lem:D_and_LPP_bd}
Consider last-passage percolation for the environment $\{ Y_{\mbf x}\}_{\mbf x \in \Z^2}$. For $n \ge 1$, let $Y^n = \{ Y_{m,n}\}_{m \in \Z}$ be the weights along the horizontal level $n$. Let $h$ be a function on $\Z$ that denotes initial data for the LPP model with boundary. Define the sequence $I^0 = (I^0_i)_{i \in \Z}$ by $I^0_i = h(i) - h(i - 1)$. Let $I^1 = D( Y^1,I^0)$, and for $n > 1$, let $I^n = D(Y^n,I^{n - 1})$ and $J^n = S(Y^n,I^{n - 1})$. Then, for each $n \ge 1$ and $m \in \Z$,
\be \label{203}
I_m^n = d^h(m,n) - d^h(m - 1,n),\qquad\text{and}\qquad J_m^n = d^h(m,n) - d^h(m,n - 1).
\ee
\end{lemma}
\begin{proof}
For $m \in \Z$,
\[
d^h(m,1) = \sup_{-\infty < k \le m} \{h(k) + d((k,1),(m,1))\} = \sup_{-\infty < k \le m}\Bigl\{h(k) + \sum_{i = k}^m Y_{i,1}\Bigr\},
\]
and so by~\eqref{202} and the definitions below, Equation~\eqref{202} holds for $n = 1$, follow from the definition. Now, assume that the statements hold for some $n \ge 1$. Then, $(d^h(m,n))_{m \in \Z}$ is a function whose increments are given by $I^n$. Then, by definition of $D$, for $m \in \Z$,
\begin{align*}
I_m^{n + 1} &= [D(Y^{n + 1},I^n)]_m = \sup_{-\infty < k \le m}\Bigl\{d^h(k,n) + \sum_{i = k}^m Y_{i,n + 1}\Bigr\}  - \sup_{-\infty < k \le m - 1}\Bigl\{d^h(k,n) + \sum_{i = k}^{m - 1} Y_{i,n + 1}\Bigr\} \\
&= \sup_{-\infty <\ell \le  k \le m} \Bigl\{h(\ell) + d((\ell,1),(m,n) +\sum_{i = k}^m Y_{i,n + 1}    \Bigr\} \\
&\qquad - \sup_{-\infty <\ell \le  k \le m - 1} \Bigl\{h(\ell) + d((\ell,1),(m,n) +\sum_{i = k}^m Y_{i,n + 1}    \Bigr\} \\
&= d^h(m,n + 1) - d^h(m - 1,n + 1).
\end{align*}
The last equality is the dynamic programming principle. Similarly,
\begin{align*}
    J_m^n &= [S(Y^{n + 1},I^n)]_m = \sup_{-\infty < k \le m}\Bigl\{d^h(k,n) + \sum_{i = k}^m Y_{i,n + 1}\Bigr\} - d^h(m,n) \\
    &=d^h(m,n + 1) - d^h(m,n). \qedhere
\end{align*}
\end{proof}

To construct the stationary boundary condition, fix a parameter $\rho \in (0,1)$, and let $h$ be defined so that $h(0) = 0$ and $ \{h(k) - h(k - 1)\}_{k \in \Z}$ is a sequence of i.i.d. $\Exp(\rho)$ random variables, independent of the i.i.d. $\Exp(1)$ bulk variables $\{Y_{\mbf x}\}_{\mbf x \in \Z \times \Z_{>0}}$. Let $\wh \Pp^\rho$ be the probability measure of these random variables. Abusing notation, we denote LPP in the half-plane with this initial data simply as $d^\rho$.

For $\mbf y \in \Z \times \Z_{\ge 0}$ and $\mbf z \in \Z \times \Z_{> 0}$, we define
\be \label{IJ}
I_{\mbf y} = d^\rho(\mbf y) - d^\rho(\mbf y - \mbf e_1),\qquad\text{and}\qquad J_{\mbf z} = d^\rho(\mbf z) - d^\rho(\mbf z - \mbf e_2).
\ee

The stationary model in the quadrant is simply a projection of the stationary model in the half-plane. This is made precise in the following lemma.
\begin{lemma} \label{lemma:LPP_coupling}
Let $I_{\mbf y}$ and $J_{\mbf z}$ be defined as in~\eqref{IJ}. Fix $\mbf x = (k,0)$, where $k \in \Z$. Then, $\{J_{\mbf x + j \mbf e_2}\}_{j \ge 1}$ is a sequence of i.i.d. $\Exp(1 - \rho)$ random variables, independent of the i.i.d. $\Exp(\rho)$ random variables $\{I_{\mbf x + i \mbf e_1}\}_{i \ge 1}$. With $I_{\mbf x + i \mbf e_1}$ and $J_{\mbf x + j \mbf e_2}$ defined, let the process $\{d_{\mbf x}^{\rho}(\mbf y): \mbf y \in \mbf x + \Z_{\ge 0}^2\}$ be defined as in~\eqref{eqn:stat_LPP}. Then, under $\wh \Pp^\rho$, for any $\mbf y \in \mbf x + \Z_{> 0}$, the portion of the almost surely unique geodesic to $\mbf y$ for the process $d^\rho$ that lies in $\mbf x + \Z_{> 0}^2$  coincides with the portion of the geodesic from $\mbf x$ to $\mbf y$ for the process $d_{\mbf x}^\rho$ that lies in $\mbf x + \Z_{> 0}^2$.   
\end{lemma}
\begin{proof}
 Let $I^0 = \{I_{i \mbf e_1}\}$, and for $n \ge 1$, let $Y^n =\{Y_{m,n}\}_{m \in \Z}$, and $I^n = D(Y^n,I^{n - 1})$. By Lemma~\ref{lem:D_and_LPP_bd}, $J_{\mbf x + \mbf e_2} = [S(Y^1,I^0)]_k$. For $k$ fixed, we note that the sequence 
 \[
 F_\ell = \begin{cases}
 -\sum_{i = \ell + 1}^k I_{i \mbf e_1}, &\ell \le k \\
 \sum_{i = k + 1}^\ell I_{i\mbf e_1}, &\ell > k.
 \end{cases}
 \]
 satisfies $F_k = 0$ and $F_\ell - F_{\ell - 1} = I_{\ell\mbf e_1} = I^0_\ell$ for $\ell \in \Z$.
 Then, by definition of the mappings $S$ and $D$,
\[
J_{\mbf x + \mbf e_2} = \sup_{-\infty < j \le k}\Biggl\{-\sum_{i = j + 1}^k I_{i\mbf e_1}  + \sum_{i = j}^k Y^n_i\Biggr\},
\]
while, for $\ell \le k$,
\[
I^\ind_\ell = \sup_{-\infty < j \le \ell}\Biggl\{-\sum_{i = j + 1}^k I_{i\mbf e_1}  + \sum_{i = j}^\ell Y^n_i\Biggr\} - \sup_{-\infty < j \le \ell - 1} \Biggl\{-\sum_{i = j + 1}^k I_{i\mbf e_1}  + \sum_{i = j}^\ell Y^n_i\Biggr\}.
\]
Therefore, $\{(I^\ind_\ell)_{\ell \le k},J_{\mbf x + \mbf e_2}\}$ is a measurable function of $(I_{i \mbf e_1})_{i \le k}$ and $Y^1$, and is therefore independent of $\{(I_{\mbf x + i\mbf e_1})_{i \ge 1},Y^2,Y^3,\ldots\}$. By Lemma~\ref{lem:output_of_queue}, $I^1$ is an i.i.d. sequence of $\Exp(\rho)$ random variables, $J_{\mbf x + \mbf e_2} \sim \Exp(1 - \rho)$, and $(\{I^\ind_\ell\}_{\ell \le k},J_{\mbf x + \mbf e_2})$ are mutually independent. 

Now, assume by way of induction, that for some $n \ge 1$,
\be \label{204}
\{(I^n_\ell)_{\ell \le k}, J_{\mbf x + \mbf e_2},\ldots, J_{\mbf x + n\mbf e_2}, (I_{\mbf x + i\mbf e_1})_{i \ge 1},Y^{n + 1},Y^{n + 2},\ldots\}.
\ee
are mutually independent, and $I^n$ is an i.i.d. sequence of $\Exp(\rho)$ random variables. 
Using the same reasoning as in the base case via Lemmas~\ref{lem:D_and_LPP_bd} and~\ref{lemma:LPP_coupling}, $\{(I^{n + 1}_{\ell \le k},J_{\mbf x + (n + 1)\mbf e_1}\}$ is a measurable function of $(I_{i}^n)_{i \le k}$ and $Y^{n + 1}$. Thus, from~\eqref{204}, we have
\[
\{(I^{n + 1}_\ell)_{\ell \le k}, J_{\mbf x + \mbf e_2},\ldots, J_{\mbf x + (n + 1)\mbf e_2}, (I_{\mbf x + i\mbf e_1})_{i \ge 1},Y^{n + 2},Y^{n + 2},\ldots\}
\]
are mutually independent, $I^{n +1}$ is a sequence of i.i.d. $\Exp(\rho)$ random variables, and $J_{\mbf x + (n + 1)\mbf e_2} \sim \Exp(1 - \rho)$.

For the second part of the lemma, this follows the same reasoning as Lemma B.3 in~\cite{Balzs2019NonexistenceOB} and Lemma A.1 in~\cite{Sepp_lecture_notes}. Suppose that the geodesic to $\mbf y$ for $d^\rho$ enters the quadrant $\mbf x + \Z_{>0}^2$ through the edge $(\mbf w,\mbf z)$ with $\mbf w = \mbf x + \ell \mbf e_r$ for $t \in \{1,2\}$, and suppose that the geodesic from $\mbf x$ to $\mbf y$ for $d^\rho_{\mbf x}$ enters the boundary through $(\wt{\mbf w},\wt{\mbf z})$ with $\wt{\mbf w} = \mbf x + p \mbf e_s$ for $s \in \{1,2\}$. For $1 \le i \le \ell$,  set $\eta_i = I_{\mbf x + i \mbf e_1} = d^\rho(\mbf x + i\mbf e_1) - d^\rho(\mbf x + (i - 1)\mbf e_1)$ if $t = 1$, and $\eta_i = J_{\mbf x + i\mbf e_2} = d^\rho(\mbf x + j\mbf e_2) - d^\rho(\mbf x + (j - 1)\mbf e_2)$ if $t = 2$. For $1 \le j \le p$, set $\wt \eta_j = I_{\mbf x + j \mbf e_1}$ if $s = 1$ and $\wt \eta_j = J_{\mbf x + j \mbf e_2}$ if $s = 2$. Then, using~\eqref{205} in the last inequality below, 
\begin{align*}
d^\rho(\mbf y) &= d^\rho(\mbf w) + d(\mbf z,\mbf y) = d^\rho(\mbf x) + \sum_{i = 1}^\ell \eta_i + d(\mbf z,\mbf y) \\
&\le d^\rho(\mbf x) + d^\rho_{\mbf x}(\mbf y) = d^\rho(\mbf x) + \sum_{j = 1}^p \wt \eta_j + d(\wt{\mbf z},\mbf y) \\
&= d^\rho(\wt{\mbf w}) + d(\wt{\mbf z},\mbf y) \le d^\rho(\mbf y).
\end{align*}
Thus, all inequalities are equalities. Since geodesics are almost surely unique in both models, the desired conclusion follows. 
\end{proof}

\subsection{KPZ scaling of the exponential CGM}
The following lemma states that the exit point of stationary LPP obeys the KPZ wandering exponent $2/3$.   
\begin{lemma}[\cite{Emrah-Janjigian-Seppalainen-20}, Theorem 2.5 (See also~\cite{Bhatia-2020}, Theorem 2.5,~\cite{BasuSarkarSly_Coalescence}, Theorem 3,~\cite{Martin-Sly-Zhang-21}, Lemma 2.8 and~\cite{Seppalainen-Shen-2020}, Corollary 3.6 and Remark 2.5b) \label{lemma:quad_exit_pt}]
Recall the stationary LPP model in the quadrant from Section~\ref{sec:LPP_quad}, and Let $\tau_1^{\mbf x}$ and $\tau_2^{\mbf x}$ be the exit times from the boundary. Let $K = [a,b] \subseteq (0,1)$. Then, there exist positive constants $N_0,C$ depending only on $K$ such that for all $N \ge N_0$, $b > 0$, and $\rho \in K$,
\begin{align}
    \Pp^\rho\Biggl(\tau_{\mbf x}^1\Bigl(\mbf x + (\lfloor N \rho^2  - b N^{2/3}  \rfloor, \lfloor N(1 - \rho)^2 \rfloor) \Bigl) \ge 1\Biggl) &\le e^{-C b^3},\qquad\text{and} \label{eqn:h_exit_bd} \\
    \Pp^\rho\Biggl(\tau_{\mbf x}^2\Bigl(\mbf x + (\lfloor N \rho^2 + b N^{2/3} \rfloor, \lfloor N(1 - \rho)^2 \rfloor) \Bigl) \ge 1\Biggl) &\le e^{-C b^3}. \label{eqn:v_exit_bd}
\end{align}
\end{lemma}

\begin{lemma} \label{lemma:line_exit_pt}
For $N \ge 1$ and $\rho \in (0,1)$, Consider LPP with $\Exp(1)$ bulk weights and boundary conditions for the stationary model in the half-plane as defined above, where $|\rho_N - \rho| \le cN^{-1/3}$ and $c$ is some real-valued constant.  Assume these are all coupled together under some probability measure $\Pp$. Let $Z^{\rho_N}$ denote shorthand notation for the exit point defined in~\eqref{eqn:exit_pt} with initial profile given by sums of i.i.d. $\Exp(\rho_N)$ random variables. Then, for any compact set $K \subseteq \R$ and $t > 0$, there exists a constant $C_1,C_2 > 0$, depending on $x,t,\rho,K$ such that, for all $y \in K$, 
\[
\limsup_{N \to \infty} \Pp(|Z^{\rho_N}(\lfloor tN \rho^2 + yN^{2/3} \rfloor ,\lfloor tN (1-\rho)^2 \rfloor )| \ge MN^{2/3})  \le C_1e^{-C_2 M^3},\qquad\text{ for all } M > 0.
\]
\end{lemma}
\begin{proof}
\begin{figure}[t]
    \centering
    \includegraphics[height = 2in]{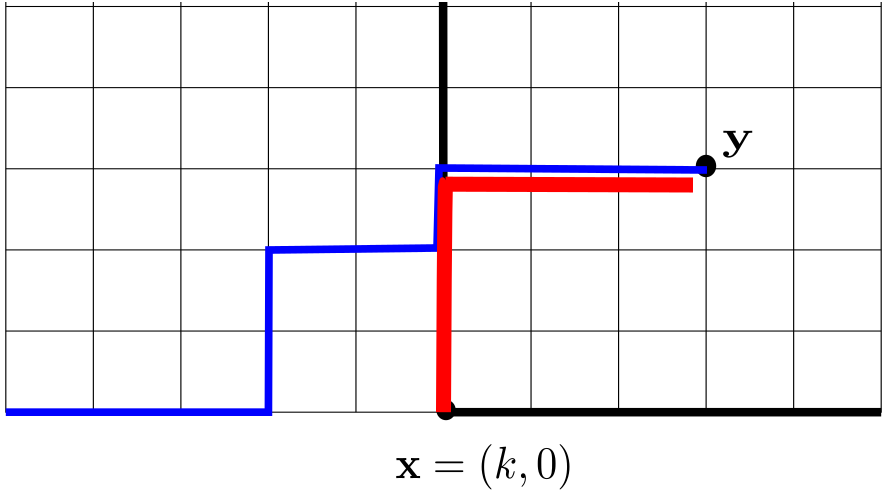}
    \caption{\small We couple the stationary model on the half-plane with the stationary model in the quadrant $\mbf x + \Z_{\ge 0}$. Inside the quadrant $\mbf x + \Z_{>0}^2$, the geodesics for the two models agree. The exit time from the initial horizontal line for the stationary geodesic  in the half-plane (blue/thin) is less than or equal to $k$ if and only if the stationary geodesic in the quadrant (red/thick) exits from the vertical boundary.}
    \label{fig:exit_time_hp}
\end{figure}

We first show that 
\[
\limsup_{N \to \infty} \Pp(Z^{\rho_N}(\lfloor tN \rho^2 + yN^{2/3} \rfloor ,\lfloor tN (1- \rho)^2 \rfloor) \le  -MN^{2/3})  \le C_1e^{-C_2M^3}.
\]
Fix $y \in \R$, and let $N$ be large enough so that $\lfloor tN\rho^2 + N^{2/3}y \rfloor  > \lfloor - MN^{2/3} \rfloor$. Set $\mbf x = (\lfloor -MN^{2/3} \rfloor ,0)$. With this choice of $\mbf x$, consider the coupling of $d^{\rho_N}$ and  
$d_{\mbf x}^{\rho_N}$ described in Lemma~\ref{lemma:LPP_coupling}, where geodesics for the two models coincide in the quadrant $\mbf x + \Z_{>0}^2$. In particular, under this coupling, $Z^{\rho_N}(\lfloor tN\rho^2 + N^{2/3}y \rfloor ,\lfloor tN(1-\rho)^2 \rfloor ) \le \lfloor -MN^{2/3} \rfloor$ if and only if $\tau_2^{\mbf x}(\lfloor tN\rho^2 + N^{2/3} y \rfloor, \lfloor tN(1-\rho)^2 \rfloor) \ge 1$. See Figure~\ref{fig:exit_time_hp}. Then, since $|\rho_N - \rho| \le cN^{-1/3}$, we have
\begin{align*}
tN(1-\rho)^2 &= \f{tN(1-\rho)^2}{(1 - \rho_N)^2}(1 - \rho_N)^2  =  O(N) ,\;\;\text{and}\;\; \\
tN\rho^2 + yN^{2/3}  &= -MN^{2/3} +  \f{tN \rho_N^2 (1- \rho_N)^2}{(1 - \rho_N)^2} + MN^{2/3} +  O(N^{2/3}),
\end{align*}
where the $O(N^{2/3})$ term depends only on the compact set $K$. 
Hence,~\eqref{eqn:v_exit_bd} of Lemma~\ref{lemma:quad_exit_pt} completes the proof.
The proof that
\[
\limsup_{N \to \infty} \Pp(Z^{\rho_N}(\lfloor tN\rho^2 + N^{2/3}y \rfloor ,\lfloor tN (1-\rho)^2\rfloor ) > MN^{2/3})  \le C_1e^{-C_2M^3}
\]
follows analogously, this time setting $\mbf x = (\lfloor MN^{2/3} \rfloor ,0)$, and replacing the use of~\eqref{eqn:v_exit_bd} from  Lemma~\ref{lemma:quad_exit_pt} with~\eqref{eqn:h_exit_bd}. 
\end{proof}

It was shown in~\cite{Dauvergne-Virag-21} that exponential last-passage percolation converges to the directed landscape. We cite their theorem here
\begin{theorem}[\cite{Dauvergne-Virag-21}, Theorem 1.7] \label{thm:conv_to_DL}
Let $d$ denote last-passage percolation \eqref{d100} with i.i.d.\ $\Exp(1)$ weights. Choose a parameter $\rho \in (0,\infty)$, and define the parameters $\chi,\alpha,\beta,\tau$ by the relations
\[
\chi^3 = \f{(\sqrt \rho + 1)^4}{\sqrt \rho},\quad \alpha = (\sqrt \rho + 1)^2,\; \quad \beta = 1 + \f{1}{\sqrt \rho},\quad \f{\chi}{\tau^2} = \f{1}{4\rho^{3/2}}.
\]
Let $v = (\rho,1)$ and $u = (\tau,0)$. Then, for any sequence $\sigma \to \infty$, there is a coupling of identically distributed  copies $d_\sigma$ of $d$ and the directed landscape $\Ll$ so that 
\[ 
    d_\sigma(x \sigma^2 u + s\sigma^3 v,y \sigma^2u + t\sigma^3v)
    =\alpha \sigma^3(t-s) +\beta \tau \sigma^2(y - x) + \chi \sigma(\Ll + o_\sigma)(x,s;y,t).
\] 
Above $d_\sigma$ is interpreted as an appropriately interpolated version of the LPP process, and $o_\sigma:\Rup \to \R$ is a random continuous function such that for every compact  $K\subset \Rup$, there exists a constant $c > 0$ such that 
\begin{align*}
	\sup_K |o_\sigma|\rightarrow 0 \ \text{ almost surely},\qquad\text{and}\qquad \E\big[c\sup_K (o_\sigma^-)^3+(o_\sigma^+)\big]\rightarrow 1.
\end{align*}
In particular, choosing $\rho = 1$ and setting $\sigma = N^{1/3}$,
there exists a coupling of the directed landscape $\Ll$ and identically distributed copies $d_N$, of $d$, such that 
\begin{multline*}
d_N\big((sN+2^{5/3}xN^{2/3},sN),(tN+2^{5/3}yN^{2/3},tN)\big) \\
=4N(t-s)+2^{8/3}N^{2/3}(y-x)+2^{4/3}N^{1/3}(\mathcal{L}+\wt o_N)(x,s;y,t),
\end{multline*}
where $\wt o_N := o_{N^{1/3}}$ as defined above. 
\end{theorem}

\subsection{Existence and distribution of the Busemann process}
\label{sec:cgm-bus}
 In the case of the exponential CGM, Busemann functions are known to exist and are indexed by  direction vectors $\mbf u$. We index the direction in terms of a real parameter $\rho \in(0,1)$: 
\[
\mbf u(\rho) = \Bigg(\frac{\rho^2}{\rho^2 + (1 - \rho)^2} , \frac{(1 - \rho)^2}{\rho^2 + (1 - \rho)^2}\Bigg).
\]
Then for a fixed  $\rho \in (0,1)$ and $\mbf x,\mbf y \in \Z^2$, the following limit exists almost surely:
\[
B^\rho_{\mbf x,\mbf y} = \lim_{n \rightarrow \infty} d(-n\mbf u(\rho),\mbf y) - d(-n\mbf u(\rho),\mbf x). 
\]
The Busemann functions can be extended to right- and left-continuous processes defined for all directions, as \cite{Janjigian-Rassoul-Seppalainen-19}. Here, we notice that the geodesics are travelling asymptotically to the southwest, whereas the geodesics we construct in the present paper are travelling to the northeast. In distribution, we can obtain one formulation from the other by a simple reflection.  
The geodesics to the southwest give rise to queuing relations that are more natural in the discrete model, so we use this formulation here.  

Define the following state space of $n$-tuples of bi-infinite nonnegative sequences:
\begin{align*}
\Y^n  &= \Bigl\{(I^1,\ldots,I^n) \in (\R_{\ge0}^\Z)^n: \lim_{m \to -\infty} \f{1}{m} \sum_{i = m}^0 I_i^k < \lim_{m \to \infty} \f{1}{m}\sum_{i = m}^0 I_i^{k + 1},\,\,\text{ for } 1 \le k \le n - 1   \Bigr\}. 
\end{align*}
In the definitions above, the limits are assumed to exist. We extend the mapping $D$ to maps that take more than two sequences as inputs. For $k \ge 1$, define the maps $D^{(k)}: \Y^k \to \R_{\ge0}^\Z$ inductively as follows. Define $D^{(1)}(I^1) = I^1$, and for $k > 1$, 
\[
D^{(k)}(I^1,\ldots,I^k) = D(I^1,D^{(k - 1)}(I^2,I^3,\ldots,I^k)).
\]
Furthermore, define the map $\D^{(n)}:\Y^n \to \Y^n$ as 
\[
[\D^{(n)}(I^1,\ldots,I^n)]_i = D^{(i)}(I^1,\ldots,I^i).
\]
On the space $\Y^n$, we define the measure $\nu^{\boldsymbol \rho^n}$ as follows: $(I^1,\ldots,I^n) \sim \nu^{\boldsymbol \rho^n}$ if $(I^1,\ldots,I^n)$ are mutually independent, and for $1 \le i \le n$, $I^i$ is a sequence of i.i.d.\  exponential random variables with rate $\rho_i$.
We define the measure $\mu^{\boldsymbol \rho^n}$ as 
\be \label{mu_def}
\mu^{\boldsymbol \rho^n} = \nu^{\boldsymbol \rho^n} \circ (\D^{(n)})^{-1}.
\ee
We now cite two theorems.
\begin{theorem}[\cite{Fan-Seppalainen-20}, Theorem 5.4] \label{thm:mu_invariant}
Let $\boldsymbol \rho^n = (\rho_1,\ldots,\rho_n)$ with $1 > \rho_1 > \cdots > \rho_n > 0$ and assume  $(I^1,\ldots,I^n) \sim \mu^{\boldsymbol \rho^n}$. Let $I^0$ be a sequence of i.i.d. exponential random variables with rate $1$, independent of $(I^1,\ldots,I^n)$. Then,
\[
(D(I^0,I^1),\ldots,D(I^0,I^n)) \sim \mu^{\boldsymbol \rho^n}.
\]
\end{theorem}

\begin{theorem}[\cite{Fan-Seppalainen-20}, Theorem 3.2] \label{thm:exp_Buse_dist}
For $\rho \in (0,1)$, define the sequence $I^\rho$ as $I^{\rho}_i = B_{(i -1)\mbf e_1,i\mbf e_1}^\rho$. Let $\boldsymbol \rho^n = (\rho_1,\ldots,\rho_n)$ with $1 > \rho_1 > \cdots > \rho_n > 0$. Then, 
\[
(I^{\rho_1},\ldots,I^{\rho_n}) \sim \mu^{\boldsymbol \rho^n}.
\]
\end{theorem}

\section{Stationary horizon} \label{sec:stat_horiz}
To describe the stationary horizon, we introduce some notation from~\cite{Busani-2021}. The map $\Phi:C(\R) \times C(\R) \to C(\R)$ is defined as \[
\Phi(f,g)(y) = \begin{cases}
f(y) + \Big[W_0(f - g) + \inf_{0 \le x \le y} (f(x) - g(x))\Big]^{-} &y \ge 0 \\
f(y) - \Big[W_y(f - g) + \inf_{y < x \le 0} \Big(f(x) - f(y) - [g(x) - g(y)]\Big)\Big]^{-} &y < 0,
\end{cases}
\]
where 
\[
W_y(f) = \sup_{-\infty < x \le y}[f(y) - f(x)].
\]
We note that the map $\Phi$ is well-defined only on the appropriate space of functions where the supremums are all finite. By Lemma 9.2 in~\cite{Seppalainen-Sorensen-21b}, when $f(0) = g(0) = 0$, 
\be \label{Phialt}
\Phi(f,g)(y) = f(y) + \sup_{-\infty <x \le y }\{g(x) - f(x)\} - \sup_{-\infty < x \le 0}\{g(x) - f(x)\}
\ee
This map extends to maps $\Phi^k:C(\R)^k \to C(\R)^k$ as follows. \begin{enumerate}
    \item $\Phi^1(f_1)(x) = f_1(x)$. 
    \item $\Phi^2(f_1,f_2)(x) = [f_1(x),\Phi(f_1,f_2)(x)]$,\qquad\text{and for }$k \ge 3,$ 
    \item $\Phi^k(f_1,\ldots,f_k)(x) = [f_1(x),\Phi(f_1,[\Phi^{k - 1}(f_2,\ldots,f_k)]_1)(x),\ldots,\Phi(f_1,[\Phi^{k -1}(f_2,\ldots,f_k)]_{k - 1})(x)]$.
\end{enumerate}

\begin{definition} \label{def:SH}
The stationary horizon $\{G_\dir\}_{\dir \in \R}$ is a process with state space $C(\R)$ and with paths in the Skorokhod space $D(\R,C(\R))$ of right-continuous functions $\R \to C(\R)$ with left limits. $C(\R)$ has the Polish topology of uniform convergence on compact sets. The law of the stationary horizon is characterized as follows: For real numbers $\dir_1 < \cdots < \dir_k$, the $k$-tuple  $(G_{\dir_1},\ldots,G_{\dir_k})$ of continuous function  has the same law as $\Phi^k(f_1,\ldots,f_k)$, where $f_1,\ldots,f_k$ are independent two-sided Brownian motions with drifts $2\dir_1,\ldots,2\dir_k$, and each with diffusion coefficient $\sqrt 2$ (as defined in point \eqref{def:2BMcmu} in Section \ref{sec:notat}).   
\end{definition}
\begin{remark}
The transformation $\Phi^k$ is such that for each $\dir \in \R$, $G_{\dir}$ is also a two-sided Brownian motion with  diffusion coefficient $\sqrt 2$ and drift $2 \dir$.
\end{remark}

For $N \in \N$, we define $F_{\abullet}^N \in D(\R,C(\R))$ to be such that, for each $\dir \in \R$, $F_\dir^N$ is the linear interpolation of the  function $\Z\ni i \mapsto B_{0,i\mbf e_1}^{1/2 - 2^{-4/3}\dir N^{-1/3}}$, where $B$ is defined in Section~\ref{sec:cgm-bus}. Then, for $\dir \in \R$, we define
\be \label{GN}
G_{\dir}^N(x) = 2^{-4/3}N^{-1/3}\Bigl[F_{\dir}^N(2^{5/3} N^{2/3}x ) - 2^{8/3}N^{2/3} x\Bigr].
\ee
\begin{remark}
The parameterization here is different from the one used in~\cite{Busani-2021}, because in the present paper, $G_\dir$ is a Brownian motion with diffusivity $\sqrt 2$ and drift $2\dir$, while in~\cite{Busani-2021}, $G_{\mu}$ has diffusivity $2$ and drift $\mu$. We can check that we have the desired diffusivity and drift parameters directly. Let $\wt G$ be the version of the stationary horizon as defined in~\cite{Busani-2021}. It is constructed as follows: for each $\mu \in \R$, $\wt F_\mu$ is the linear interpolation of the function $\Z \ni i \mapsto B_{0,i \mbf e_1}^{1/2 - 4^{-1}\mu N^{-1/3}}$, and $\wt G_\mu(x)$ is the limit as $N \to \infty$ of
\[
N^{-1/3}\Bigl[\wt F_\mu^N(N^{2/3}x) - 2N^{2/3}x\Bigr],
\]
which is distributed as $2B(x) + \mu x$,
where $B$ is a standard Brownian motion. Then, after taking a limit from~\eqref{GN}, 
\begin{align*}
    G_{\dir}(x) = 2^{-4/3}\wt G_{4 \cdot 2^{-4/3} \dir}(2^{5/3} x) \deq 2^{-4/3}[2B(2^{5/3}x) + 2^{2/3} \dir (2^{5/3} x)  ] \deq 2^{1/2} B(x) + 2\dir x,
\end{align*}
where the last equality comes from Brownian scaling. Hence, the resulting object has diffusivity $\sqrt 2$ and drift $2\dir$, as desired. Furthermore, using the scaling relations of Theorem~\ref{thm:SH10}\ref{itm:SH_sc} below, $G_\dir(x) \deq \wt G_{4\dir}(x/2)$. 
\end{remark}



The main theorem of~\cite{Busani-2021} is the following:
\begin{theorem} \label{thm:conv_to_SH}
As $N \to \infty$, the process $G^N$ converges in distribution to $G$ on the path space $D(\R,C(\R))$. In particular, for any finite collection $\dir_1,\ldots,\dir_n$, 
\[
(G_{\dir_1}^N,\ldots,G_{\dir_n}^N) \Longrightarrow (G_{\dir_1},\ldots,G_{\dir_n}),
\]
where the convergence holds in distribution in the sense of uniform convergence on compact sets of functions in $C(\R)^n$.
\end{theorem}

The first author~\cite{Busani-2021} first proved this finite-dimensional convergence and then showed tightness of the process to conclude the existence of a limit taking values in $D(\R,C(\R))$. The second and third authors~\cite{Seppalainen-Sorensen-21b} discovered that the stationary horizon is also the Busemann process of Brownian last-passage percolation, up to an appropriate scaling and reflection (see Theorem 5.3 in~\cite{Seppalainen-Sorensen-21b}).

The following collects several facts about the stationary horizon from these two papers. For notation, let $G_{\dir+} = G_\dir$, and let $G_{\dir -}$ be the limit of $G_{\alpha}$ as $\alpha \nearrow \dir$. 
\begin{theorem}[\cite{Busani-2021}, Theorem 1.2; \cite{Seppalainen-Sorensen-21b}, Theorems 3.9, 3.11, 3.15, 7.20 and Lemma 3.6] \label{thm:SH10} $ $  The following hold for the stationary horizon.
\begin{enumerate} [label=\rm(\roman{*}), ref=\rm(\roman{*})]  \itemsep=3pt
    \item\label{itm:SHpm} For each $\dir \in \R$, with probability one, $G_{\dir -} = G_{\dir +}$, and $G_\dir$ is a two-sided Brownian motion with diffusion coefficient $\sqrt 2$ and drift $2\dir$
    \item \label{itm:SH_sc} For $c > 0$ and $\nu \in \R$, \ 
    $ 
    \{cG_{c (\dir + \nu)}(c^{-2}x) - 2\nu x  : x\in \R\}_{\dir \in \R} \,\deq\, \{G_\dir(x): x \in \R\}_{\dir \in \R}.
    $ 
    \item\label{itm:SH_sc2} Spatial stationarity holds in the sense that,  for $y \in \R$, 
    \[\{G_{\dir}(x):x \in \R\}_{\dir \in \R} \deq \{G_{\dir}(y,x + y): x \in \R\}_{\dir \in \R}.\]
    \item Fix $x > 0$ , $\dir_0 \in \R$,  $\dir > 0$, and $z \ge 0$. Then, 
    \begin{align*}
    &\Pp\bigl(\sup_{a,b \in [-x,x]}|G_{\dir_0 + \dir}(a,b) - G_{\dir_0 }(a,b)| \le z\bigr) = \Pp\bigl(G_{\dir_0 + \dir}(-x,x) - G_{\dir_0}(-x,x) \le z\bigr) \\
    &\quad = \Phi\Bigl(\f{z - 4\dir x}{2\sqrt {2x}}\Bigr) + e^{\dir z}\biggl(\Bigl(1 + {\dir}z + 4\dir^2 x \Bigr)\Phi\Bigl(-\f{z + 4\dir x}{2 \sqrt{2 x}}\Bigr) - 2\dir\sqrt{{x}/{\pi}\tspa} \tspb e^{-\f{(z + 4\dir x)^2}{8\sqrt x}}\biggr)
    \end{align*}
     where $\Phi$ is the standard normal distribution function. 
    \item \label{itm:exp} For $x < y$ and $\alpha < \beta$, with $\#$ denoting the cardinality, 
    \[
    \E[\#\{\dir \in (\alpha,\beta): G_{\dir-}(x,y) < G_{\dir +}(x,y) \}] = 2\sqrt{{2}/{\pi}\tsp}
    (\beta - \alpha)\sqrt{y - x}.
    \]
\end{enumerate}
Furthermore, the following hold on a single event of full probability.
\begin{enumerate} [resume, label=\rm(\roman{*}), ref=\rm(\roman{*})]  \itemsep=3pt
    \item \label{itm:SH_j} For $x_0 > 0$   define  the process $G^{x_0} \in D(\R,C[-x_0,x_0])$   by restricting each function $G_\xi$ to $[-x_0,x_0]$: $G^{x_0}_\dir=G_\xi\vert_{[-x_0,x_0]}$. Then, $\xi\mapsto G^{x_0}_\xi$ is a $C[-x_0,x_0]$-valued  jump process with finitely many jumps in any compact interval, but countably infinitely many jumps in $\R$. The number of jumps in a compact interval has finite expectation given in item \ref{itm:exp} above, and each direction $\dir$ is a jump direction with probability $0$. In particular, for each $\dir \in \R$ and compact set $K$, there exists a random $\ve = \ve(\dir,K)>0$ such that for all $\dir - \ve < \alpha < \dir < \beta < \dir  + \ve$, $\sigg \in \{-,+\}$, and all $x \in K$, $G_{\dir -}(x) = G_{\alpha}(x)$ and $G_{\dir +}(x) = G_{\beta}(x)$.
    \item \label{itm:SH_mont} For $x_1 < x_2$, $\dir \mapsto G_{\dir}(x_1,x_2)$ is a non-decreasing jump process, converging to $\pm \infty$ as $\dir \to \pm \infty$.
    \item Let $\alpha < \beta$. The function $x\mapsto G_\beta(x)-G_\alpha(x)$ is nondecreasing.  There exist finite $S_1 = S_1(\alpha,\beta)$ and $S_2 = S_2(\alpha,\beta)$ with $S_1 < 0 < S_2$ such that $G_{\alpha }(x) = G_{\beta }(x)$ for $x \in [S_1,S_2]$ and $G_{\alpha }(x) \ne G_{\beta }(x)$ for $x \notin [S_1,S_2]$.
    \item \label{itm:bad_dir_contained} 
    Let $\alpha < \beta$,  $S_1 = S_1(\alpha,\beta)$ and $S_2 = S_2(\alpha,\beta)$.  Then $\exists\tspa \zeta, \eta\in[\alpha, \beta]$ such that,     \begin{align*} 
    &\text{$G_{\zeta -}(x) = G_{\zeta +}(x)$ for $x \in [-S_1,0]$, and $G_{\zeta -}(x) > G_{\zeta +}(x)$ for $x < S_1$, and } \\
    &\text{$G_{\eta -}(x) = G_{\eta +}(x)$ for $x \in [0,S_2]$, and $G_{\eta -}(x) < G_{\eta +}(x)$ for $x > S_2$.}
    \end{align*} 
        In particular, the set $\{\dir \in \R: G_{\dir} \neq G_{\dir-}\}$ is dense in $\R$.
\end{enumerate}
\end{theorem}

Theorem~\ref{thm:invariance_of_SH} gives the following  previously unknown property of SH. 
\begin{corollary} \label{cor:SH_reflect}
The distribution of the stationary horizon on $D(\R,C(\R))$ satisfies the following reflection property:
\[
\{G_{(-\dir)-}(-\tspb\aabullet)\}_{\dir \in \R} \deq \{G_\dir(\aabullet)\}_{\dir \in \R}.
\]
\end{corollary}
\begin{proof}
By the spatial reflection invariance of the directed landscape (Lemma~\ref{lm:landscape_symm}\ref{itm:DL_reflect}), $\{G_{(-\dir)-}(-\aabullet)\}_{\dir \in \R}$ is an invariant distribution for the KPZ fixed point such that each marginal satisfies the limit assumptions~\eqref{eqn:drift_assumptions}. The result follows from the uniqueness part of  Theorem~\ref{thm:invariance_of_SH}.
\end{proof}

\end{appendix}

\bibliographystyle{alpha}
\bibliography{references_file}
\end{document}